\documentclass[11pt]{amsart}%
\usepackage[pdftex]{graphicx,color}
\usepackage{comment}
\usepackage{amsmath,amsthm, amssymb}
\usepackage{amsfonts}
\usepackage{hyperref}
\usepackage{cite}
\usepackage[toc,page]{appendix}

\providecommand{\U}[1]{\protect\rule{.1in}{.1in}}
\newtheorem{theorem}{Theorem}[section]

\newtheorem{corollary}[theorem]{Corollary}

\newtheorem{definition}[theorem]{Definition}
\newtheorem{example}[theorem]{Example}

\newtheorem{lemma}[theorem]{Lemma}
\newtheorem{notation}[theorem]{Notation}

\newtheorem{proposition}[theorem]{Proposition}
\newtheorem{remark}[theorem]{Remark}

\newtheorem{ass}{Assumption}

\def \pt{/\!/}

\numberwithin{equation}{section}

\newcommand{\GraphicsDirectory}{./svg/}
\newcommand{\executeiffilenewer}[3]{%
\ifnum\pdfstrcmp{\pdffilemoddate{#1}}%
{\pdffilemoddate{#2}}>0%
{\immediate\write18{#3}}\fi%
}

\graphicspath{{\GraphicsDirectory}}
\excludecomment{Aside}
\begin{document}
\title{On truncated logarithms of flows on a Riemannian manifold}
\date{\today}
\author{Bruce K. Driver}

\begin{abstract}
This paper gives quantitative global estimates between a time dependent flow
on a Riemannian manifold $\left(  M\right)  $ and the flow of a vector field
constructed by truncating the formal Magnus expansion for the logarithm of the
flow. As a corollary, we also find quantitative estimates between the
composition of the flows of two given time independent vector fields on $M$
and the flow of a truncated version of the Baker-Cambel-Hausdorff-Dynkin
expansion associated to the two given vector fields.

\end{abstract}
\maketitle
\tableofcontents

\section{Introduction}\label{sec.1}

For the purposes of this paper, let $M$ be a connected manifold without
boundary, $\Gamma\left(  TM\right)  \ $be the space of smooth vector fields on
$M,$ $g$ be a Riemannian metric on $M,$ $d=d_{g}$ be the induced length metric
on $M$ (see Notation \ref{not.2.1}), $\nabla=\nabla^{g}$ be the associated
Levi-Civita covariant derivative, and $R=R^{g}$ be the curvature tensor of
$\nabla$ (see Definition \ref{def.2.5}).

\begin{definition}
[Complete vector fields]\label{def.1.1}Let $J=\left[  0,T\right]  $ (or
possibly some other interval) and $J\ni t\rightarrow Y_{t}\in\Gamma\left(
TM\right)  $ be a smoothly varying time dependent vector field. We say that
$Y$ is \textbf{complete }provided for every $s\in J$ and $m\in M,$ there
exists a solution, $\sigma:J\rightarrow M$ solving the ordinary differential
equation (ODE for short),%
\begin{equation}
\dot{\sigma}\left(  t\right)  =Y_{t}\left(  \sigma\left(  t\right)  \right)
\text{ with }\sigma\left(  s\right)  =m, \label{e.1.1}%
\end{equation}
where $\dot{\sigma}\left(  0\right)  $ and $\dot{\sigma}\left(  T\right)  $
are to be interpreted as the appropriate one-sided derivative. [See Corollary
\ref{cor.2.12} below for some necessary conditions on $\left(  M,g\right)  $
and $Y_{\cdot}$ which imply that $Y$ is complete.]
\end{definition}

\begin{definition}
[Flows]\label{def.1.2}If $J=\left[  0,T\right]  \ni t\rightarrow Y_{t}%
\in\Gamma\left(  TM\right)  $ is a smoothly varying time dependent complete
vector field, let $\mu_{t,s}^{Y}\left(  m\right)  :=\sigma\left(  t\right)  $
where $\sigma\left(  \cdot\right)  $ is the solution to Eq. (\ref{e.1.1}).
Thus the \textbf{flow associated to }$Y,$ $\mu_{t,s}^{Y}:M\rightarrow M$ for
$s,t\in J,$ satisfies,%
\[
\frac{d}{dt}\mu_{t,s}^{Y}\left(  m\right)  =Y_{t}\circ\mu_{t,s}^{Y}\text{ with
}\mu_{s,s}^{Y}=Id_{M}.
\]
When $Y_{t}=Y$ is independent of $t,$ we denote $\mu_{t,0}^{Y}$ by $e^{tY}$ so
that $\mu_{t,s}^{Y}=e^{\left(  t-s\right)  Y}$ for all $s,t\in\mathbb{R}.$
\end{definition}

The \textbf{logarithm problem} in this context is the question of finding
vector fields, $Z_{t}\in\Gamma\left(  TM\right)  ,$ so that $\mu_{t,0}%
^{Y}=e^{Z_{t}}.$ This problem seems to have first been formally studied by
Magnus \cite{Magnus1954} in the context of linear differential equations
although the special case encoded in the Baker-Cambel-Hausdorff-Dynkin formula
is much older. For a short derivation of Magnus result see \cite{Arnal2018},
and for an extensive review of the Magnus' expansion and its many
generalizations and applications see the survey article, \cite{Blanes2009}.
Although it is not within the author's ability to give a systematic review of
the many uses of Magnus' idea, in order to understand the breadth of
applications let me give a sampling references involving quantum physics,
control theory, geometric numerical integration, and stochastic analysis. An
early reference in quantum physics is \cite{Wilcox1967}. An example in control
theory is \cite{Kawski2000} who is discussing taking logarithms of the
Neumann-Dyson series and their non-linear extensions described by K.T. Chen
\cite{ChenKT1957a,ChenKT1958a} and Fliess \cite{Fliess1981}. The following
references,
\cite{Cai2016,DAmbrosio2014,Hairer2003,Hochbruck2010,Lundervold2011a,Lundervold2015a,Munthe-Kaas1995}
along with the survey articles \cite{Budd1999,Iserles2011} and the monographs
\cite{Hairer2006c,Blanes2016}, give only a sporadic sampling of the extensive
literature in geometric numerical integration theory. For references from
stochastic analysis, see
\cite{BenArous1989,BenArous1988,Castell1993,Castell1996,Inahama2010a,Inahama2017,Takanobu1988,Takanobu1990}
which pertain to approximating stochastic flows and see
\cite{Sussmann1986,Sussmann1988} for a couple of example involving stochastic
control theory. Hopefully the reader sees from this sampling of references how
ubiquitous the \textquotedblleft logarithm problem\textquotedblright\ has become.

A key starting point for this paper and many of the references above is
Strichartz's \cite[Eq. (G.C-B-H-D)]{Strichartz1987a} (and also see
\cite{Bialynicki-Birula1969,Mielnik1970}) formal series solution,
\begin{equation}
Z_{t}\sim\sum_{m=1}^{\infty}\sum_{\sigma\in P_{m}}\left(  \frac{\left(
-1\right)  ^{e\left(  \sigma\right)  }}{m^{2}\binom{m-1}{e\left(
\sigma\right)  }}\right)  \int_{\Delta_{m}\left(  t\right)  }\left(
-1\right)  ^{m}\operatorname{ad}_{Y_{\tau_{\sigma\left(  m\right)  }}}%
\dots\operatorname{ad}_{Y_{\sigma\left(  2\right)  }}Y_{\tau_{\sigma\left(
1\right)  }}d\mathbf{\tau,}\label{e.1.2}%
\end{equation}
to the logarithm problem, i.e. $Z_{t}\in\Gamma\left(  TM\right)  $
\textquotedblleft represented\textquotedblright\ by the series above formally
solves $\mu_{t,0}^{Y}=e^{Z_{t}}.$ In Eq. (\ref{e.1.2}), $P_{m}$ is the set of
permutations of $\left\{  1,2,\dots,m\right\}  ,$%
\[
\Delta_{m}\left(  t\right)  =\left\{  0\leq\tau_{1}\leq\tau_{2}\leq\dots
\leq\tau_{m}\leq t\right\}  ,\text{ and}%
\]%
\[
e\left(  \sigma\right)  :=\#\left\{  j<m:\sigma\left(  j\right)
>\sigma\left(  j+1\right)  \right\}
\]
is the number of \textquotedblleft errors\textquotedblright\ in the ordering
of $\sigma\left(  1\right)  ,\dots,\sigma\left(  m\right)  .$

If $M$ is a Lie group and $Y_{t}$ is a family of left invariant vector fields
on $M$ then the expansion in Eq. (\ref{e.1.2}) will converge when $t$ is
sufficiently close to $0$ and the resulting sum will solve the logarithm
problem in this context. There is a rather vast literature exploring when such
series expansions actually converge, see for example
\cite{Moan2008,Biagi2014,Lakos2017,Curry2018} just to give a very thin sample.
In the general context of arbitrary time dependent vector fields, the
expansion in Eq. (\ref{e.1.2}) will typically not converge. In this paper, our
goal is not to discuss convergence of the series but rather to estimate the
errors made by truncating the expansion in Eq. (\ref{e.1.2}).

For $n\in\mathbb{N},$ let $Z_{t}^{\left(  n\right)  }$ denote the series
expansion in Eq. (\ref{e.1.2}) where the first sum, $\sum_{m=1}^{\infty},$ is
truncated to $\sum_{m=1}^{n}.$ Roughly speaking, the main object of this paper
is (under certain added hypothesis on $Y_{t})$ to estimate the distance
between $\mu_{t,0}^{Y}$ and $e^{Z_{t}^{\left(  n\right)  }}.$ The rest of this
introduction will be devoted to summarizing the main results of this paper and
in particular to stating Theorems \ref{thm.1.30} and \ref{thm.1.32} below.

\subsection{Basic flow estimates\label{sec.1.1}}

\begin{notation}
\label{not.1.3}If $\rho_{m}\geq0$ for all $m\in M,$ we let
\[
\rho_{M}:=\sup_{m\in M}\rho_{m}.
\]
If $J$ is a compact subinterval of $\mathbb{R}$ and $M\times J\ni\left(
m,t\right)  \rightarrow\rho_{m}\left(  t\right)  \geq0$ is a continuous
function we let,%
\[
\left\vert \rho\right\vert _{J}^{\ast}:=\int_{J}\rho_{M}\left(  t\right)
dt=\int_{J}\sup_{m\in M}\rho_{m}\left(  t\right)  dt.
\]
When $J=\left[  0,t\right]  $ for some $t>0$ we will simply write $\left\vert
\rho\right\vert _{t}^{\ast}$ for $\left\vert \rho\right\vert _{\left[
0,t\right]  }^{\ast}.$ Note that if $J\ni t\rightarrow\rho\left(  t\right)
\geq0$ does not depend on $m\in M,$ then
\[
\rho_{t}^{\ast}=\int_{J}\rho\left(  t\right)  dt=\left\Vert \rho\right\Vert
_{L^{1}\left(  J,m\right)  }%
\]
where $m$ is Lebesgue measure on $\mathbb{R}.$
\end{notation}

\begin{notation}
\label{not.1.4}For $X\in\Gamma\left(  TM\right)  ,$ $m\in M,$ and $v,w\in
T_{m}M,$ let
\[
\nabla_{v\otimes w}^{2}X:=\nabla_{v}\left(  \nabla_{W}X\right)  -\nabla
_{\nabla_{v}W}X
\]
where $W\in\Gamma\left(  TM\right)  $ is chosen so that $W\left(  m\right)
=w.$
\end{notation}

See Definition \ref{def.2.5} and Remark \ref{rem.2.6} below for an alternative
but equivalent definition of $\nabla^{2}X$ as well as the verification that
$\nabla_{v\otimes w}^{2}X$ is well defined bilinear form on $T_{m}M\times
T_{m}M.$

\begin{notation}
[Tensor Norms]\label{not.1.5}If $X\in\Gamma\left(  TM\right)  $ and $m\in M,$
let
\begin{align*}
\left\vert X\right\vert _{m}  &  :=\left\vert X\left(  m\right)  \right\vert
_{g}\\
\left\vert \nabla X\right\vert _{m}  &  :=\sup_{\left\vert v_{m}\right\vert
=1}\left\vert \nabla_{v_{m}}X\right\vert _{g},\\
\left\vert \nabla^{2}X\right\vert _{m}  &  =\sup_{\left\vert v_{m}\right\vert
=1=\left\vert w_{m}\right\vert }\left\vert \nabla_{v_{m}\otimes w_{m}}%
^{2}X\right\vert _{g},\\
\left\vert R\left(  X,\cdot\right)  \right\vert _{m}  &  :=\sup_{\left\vert
v_{m}\right\vert =1=\left\vert w_{m}\right\vert }\left\vert R\left(  X\left(
m\right)  ,v_{m}\right)  w_{m}\right\vert _{g},\text{ and}\\
H_{m}\left(  X\right)   &  :=\left\vert \nabla^{2}X\right\vert _{m}+\left\vert
R\left(  X,\bullet\right)  \right\vert _{m}.
\end{align*}

\end{notation}

Let us give a few examples of how this notation will be used.

\begin{example}
\label{ex.1.6}Suppose that $J\ni t\rightarrow X_{t}\in\Gamma\left(  TM\right)
$ is a continuously varying time dependent vector field, then%
\begin{align*}
\left\vert X_{t}\right\vert _{M}  &  :=\sup_{m\in M}\left\vert X_{t}%
\right\vert _{m},\text{\quad\ }\left\vert X_{\cdot}\right\vert _{J}^{\ast
}:=\int_{J}\sup_{m\in M}\left\vert X_{t}\right\vert _{m}dt,\\
\left\vert \nabla^{2}X_{t}\right\vert _{M}  &  :=\sup_{m\in M}\left\vert
\nabla^{2}X_{t}\right\vert _{m},\text{\quad\ }\left\vert \nabla^{2}X_{\cdot
}\right\vert _{J}^{\ast}:=\int_{J}\sup_{m\in M}\left\vert \nabla^{2}%
X_{t}\right\vert _{m}dt,\\
H_{M}\left(  X_{t}\right)   &  =\sup_{m\in M}\left[  \left\vert \nabla
^{2}X_{t}\right\vert _{m}+\left\vert R\left(  X_{t},\bullet\right)
\right\vert _{m}\right]  \leq\left\vert \nabla^{2}X_{t}\right\vert
_{M}+\left\vert R\left(  X_{t},\bullet\right)  \right\vert _{M},
\end{align*}
and%
\[
H\left(  X_{\cdot}\right)  _{J}^{\ast}=\int_{J}\sup_{m\in M}H_{m}\left(
X_{t}\right)  dt\leq\left\vert \nabla^{2}X_{\cdot}\right\vert _{J}^{\ast
}+\left\vert R\left(  X_{\cdot},\cdot\right)  \right\vert _{J}^{\ast}.
\]

\end{example}

The next theorem is a combination of Theorem \ref{thm.2.30} and Corollary
\ref{cor.2.31} below.

\begin{theorem}
\label{thm.1.7}Let $J=\left[  0,T\right]  \ni t\rightarrow X_{t,}Y_{t}%
\in\Gamma\left(  TM\right)  $ be two smooth complete time dependent vector
fields on $M$ and $\mu^{X}$ and $\mu^{Y}$ be their corresponding flows. Then
for $m\in M$ and $t>0$ (for notational simplicity) we have the following
estimates,%
\begin{align*}
d\left(  \mu_{t,0}^{X}\left(  m\right)  ,\mu_{t,0}^{Y}\left(  m\right)
\right)   &  \leq\int_{0}^{t}e^{\int_{s}^{t}\left\vert \nabla X_{\sigma
}\right\vert _{\mu_{\sigma,s}^{X}\left(  m\right)  }d\sigma}\cdot\left\vert
Y_{s}-X_{s}\right\vert _{\mu_{s,0}^{Y}\left(  m\right)  }~ds\\
&  \leq e^{\left\vert \nabla X_{\cdot}\right\vert _{t}^{\ast}}\left\vert
Y_{\cdot}-X_{\cdot}\right\vert _{t}^{\ast},
\end{align*}%
\begin{align*}
d\left(  \mu_{t,0}^{Y}\left(  m\right)  ,m\right)   &  \leq\int_{0}%
^{t}\left\vert Y_{s}\right\vert _{\mu_{s,0}^{Y}\left(  m\right)  }%
~ds\leq\left\vert Y\right\vert _{t}^{\ast},\text{ and}\\
d\left(  \mu_{t,0}^{Y}\left(  m\right)  ,m\right)   &  \leq\int_{0}^{t}%
e^{\int_{s}^{t}\left\vert \nabla Y_{\sigma}\right\vert _{\mu_{\sigma,s}%
^{Y}\left(  m\right)  }d\sigma}\cdot\left\vert Y_{s}\right\vert _{m}~ds\leq
e^{\left\vert \nabla Y\right\vert _{t}^{\ast}}\int_{0}^{t}\left\vert
Y_{s}\left(  m\right)  \right\vert ds.
\end{align*}

\end{theorem}

We are also interested in estimating the distance between the differentials,
$\left(  \mu_{t,0}^{X}\right)  _{\ast}$ and $\left(  \mu_{t,0}^{Y}\right)
_{\ast},$ of $\mu_{t,0}^{X}$ and $\mu_{t,0}^{Y}.$ To do so we endow $TM$ with
its \textquotedblleft natural\textquotedblright\ Riemannian metric induced
from the Riemannian metric, $g,$ on $M$ (see Definition \ref{def.5.1} of
Section \ref{sec.5} below) and let $d^{TM}$ be the induced length metric on
$TM.$ The next theorem is a combination of Theorem \ref{thm.7.2} and Corollary
\ref{cor.7.3} below.

\begin{theorem}
\label{thm.1.8}If $J=\left[  0,T\right]  \ni t\rightarrow X_{t,}Y_{t}\in
\Gamma\left(  TM\right)  $ are smooth complete (see Definition \ref{def.1.1})
time dependent vector fields on $M$ and $\mu^{X}$ and $\mu^{Y}$ be their
corresponding flows, then
\begin{align*}
\sup_{v\in TM:\left\vert v\right\vert =1}d^{TM}  &  \left(  \left(  \mu
_{t,0}^{X}\right)  _{\ast}v,\left(  \mu_{t,0}^{Y}\right)  _{\ast}v\right) \\
&  \leq e^{2\left\vert \nabla X\right\vert _{t}^{\ast}+\left\vert \nabla
Y\right\vert _{t}^{\ast}}\cdot\left(  \left(  1+H\left(  X_{\cdot}\right)
_{t}^{\ast}\right)  \left\vert Y-X\right\vert _{t}^{\ast}+\left\vert
\nabla\left[  Y-X\right]  \right\vert _{t}^{\ast}\right)  ,
\end{align*}
and
\[
\sup_{v\in TM:\left\vert v\right\vert =1}d^{TM}\left(  \left(  \mu_{t,0}%
^{Y}\right)  _{\ast}v,v\right)  \leq e^{\left\vert \nabla Y\right\vert
_{t}^{\ast}}\cdot\left(  \left\vert Y\right\vert _{t}^{\ast}+\left\vert \nabla
Y\right\vert _{t}^{\ast}\right)  .
\]

\end{theorem}

The next proposition starts to indicate how the two previous theorems fit into
the logarithm approximation theorem.

\begin{proposition}
\label{pro.1.9}Suppose that $\left[  0,T\right]  \ni t\rightarrow X_{t}%
\in\Gamma\left(  TM\right)  $ is a complete time dependent vector field and
$\left[  0,T\right]  \ni t\rightarrow Z_{t}\in\Gamma\left(  TM\right)  $ is
another time dependent vector field such that $Z_{t}$ is complete for each
fixed $t$ and $Z_{0}\equiv0.$ Then
\begin{equation}
d\left(  \mu_{t,0}^{X}\left(  m\right)  ,e^{Z_{t}}\left(  m\right)  \right)
\leq\int_{0}^{t}e^{\int_{s}^{t}\left\vert \nabla X_{\sigma}\right\vert
_{\mu_{\sigma,s}^{X}\left(  m\right)  }d\sigma}\cdot\left\vert W_{s}^{Z}%
-X_{s}\right\vert _{e^{Z_{s}}\left(  m\right)  }~ds \label{e.1.3}%
\end{equation}
where
\[
W_{t}^{Z}:=\int_{0}^{1}e_{\ast}^{sZ_{t}}\dot{Z}_{t}\circ e^{-sZ_{t}}%
ds=\int_{0}^{1}\operatorname{Ad}_{e^{sZ_{t}}}\dot{Z}_{t}~ds\in\Gamma\left(
TM\right)  .
\]

\end{proposition}

\begin{proof}
By Corollary \ref{cor.2.24} below, which states,
\[
\frac{d}{dt}e^{Z_{t}}=W_{t}^{Z}\circ e^{Z_{t}}\text{ with }e^{Z_{0}}=Id
\]
and so $\mu_{t,0}^{W^{Z}}=e^{Z_{t}}$ for all $t\in\left[  0,T\right]  .$ Thus
the estimate in Eq. (\ref{e.1.3}) follows by applying Theorem \ref{thm.1.7}
with $Y_{t}=W_{t}^{Z}.$
\end{proof}

Because of Proposition \ref{pro.1.9}, in order to find good approximate
logarithms for the flow, $\mu^{X},$ we should choose $Z_{t}\in\Gamma\left(
TM\right)  $ so that $Z_{0}=0$ and $\left\vert W_{s}^{Z}-X_{s}\right\vert
_{e^{Z_{s}}\left(  m\right)  }$ is small. Ideally we would like to choose $Z$
so that $W_{s}^{Z}=X_{s}$ but this is not possible in general. However,
formally solving the equation $W_{s}^{Z}=X_{s}$ for $Z$ would lead to the
expansion in Eq. (\ref{e.1.2}). In order to get precise estimates we are now
going to make more assumptions (in the spirit of control theory) on what we
allow for our choice of $X_{t}.$ These additional assumptions and necessary
notations will be explained in the next subsection.

\subsection{Free niltpotent Lie groups and dynamical systems\label{sec.1.2}}

\begin{definition}
[Tensor Algebras]\label{def.1.10}Let $T\left(  \mathbb{R}^{d}\right)
:=\oplus_{k=0}^{\infty}\left[  \mathbb{R}^{d}\right]  ^{\otimes k}$ be the
tensor algebra over $\mathbb{R}^{d}$ so the general element of $\omega\in
T\left(  \mathbb{R}^{d}\right)  $ is of the form
\[
\omega=\sum_{k=0}^{\infty}\omega_{k}\text{ with }\omega_{k}\in\left(
\mathbb{R}^{d}\right)  ^{\otimes k}\text{ for }k\in\mathbb{N}_{0}%
\]
where we assume $\omega_{k}=0$ for all but finitely many $k.$ Multiplication
is the tensor product and associated to this multiplication is the Lie
bracket,%
\begin{equation}
\left[  A,B\right]  _{\otimes}:=A\otimes B-B\otimes A\text{ for all }A,B\in
T\left(  \mathbb{R}^{d}\right)  . \label{e.1.4}%
\end{equation}

\end{definition}

\begin{definition}
[Free Lie Algebra]\label{not.1.11}The \textbf{free Lie algebra over
}$\mathbb{R}^{d}$ will be taken to be the Lie-subalgebra, $F\left(
\mathbb{R}^{d}\right)  ,$ of $\left(  T\left(  \mathbb{R}^{d}\right)  ,\left[
\cdot,\cdot\right]  _{\otimes}\right)  $ generated by $\mathbb{R}^{d}.$
\end{definition}

\begin{remark}
\label{rem.1.12}If $\left(  \mathfrak{g},\left[  \cdot,\cdot\right]  \right)
$ is a Lie algebra and $V\subset\mathfrak{g}$ is a subspace, then using
Jacobi's identity one easily shows that Lie sub-algebra $\left(
\operatorname*{Lie}\left(  V\right)  \right)  $ of $\mathfrak{g}$ generated by
$V$ may be described as;%
\[
\operatorname*{Lie}\left(  V\right)  =\operatorname*{span}\cup_{k=1}^{\infty
}\left\{  \operatorname{ad}_{v_{1}}\dots\operatorname{ad}_{v_{k-1}}v_{k}%
:v_{1},\dots,v_{k}\in V\right\}  ,
\]
where $\operatorname{ad}_{A}B:=\left[  A,B\right]  $ for all $A,B\in
\mathfrak{g}.$ As a consequence of this remark it follows that $F\left(
\mathbb{R}^{d}\right)  $ is a $\mathbb{N}_{0}$-graded algebra with
\[
F\left(  \mathbb{R}^{d}\right)  =\oplus_{k=0}^{\infty}F_{k}\left(
\mathbb{R}^{d}\right)  \text{ where }F_{k}\left(  \mathbb{R}^{d}\right)
=F\left(  \mathbb{R}^{d}\right)  \cap\left[  \mathbb{R}^{d}\right]  ^{\otimes
k}\subset F\left(  \mathbb{R}^{d}\right)  .
\]
According to this grading, if $A\in F\left(  \mathbb{R}^{d}\right)  $ we let
$A_{k}\in F_{k}\left(  \mathbb{R}^{d}\right)  $ denote the projection of $A$
into $F_{k}\left(  \mathbb{R}^{d}\right)  .$
\end{remark}

See \cite{ReutenauerBook}, for general background information on free Lie
algebras. The spaces $T\left(  \mathbb{R}^{d}\right)  $ and $F\left(
\mathbb{R}^{d}\right)  $ are infinite dimensional. We are going to be most
interested in the finite dimensional truncated versions of these algebras.

\begin{definition}
[Truncated Tensor Algebras]\label{def.1.13}Given $\kappa\in\mathbb{N},$ let
\[
T^{\left(  \kappa\right)  }\left(  \mathbb{R}^{d}\right)  :=\oplus
_{k=0}^{\kappa}\left[  \mathbb{R}^{d}\right]  ^{\otimes k}\subset T\left(
\mathbb{R}^{d}\right)
\]
which is algebra under the multiplication rule,%
\[
AB=\sum_{k=0}^{\kappa}\left(  AB\right)  _{k}=\sum_{k=0}^{\kappa}\sum
_{j=0}^{k}A_{j}\otimes B_{k-j}~\text{ }\forall~A,B\in T^{\left(
\kappa\right)  }\left(  \mathbb{R}^{d}\right)
\]
and a Lie algebra under the bracket operation, $\left[  A,B\right]  :=AB-BA$
for all $A,B\in T^{\left(  \kappa\right)  }\left(  \mathbb{R}^{d}\right)  .$
\end{definition}

\begin{notation}
\label{not.1.14}Let $\pi_{\leq\kappa}:T\left(  \mathbb{R}^{d}\right)
\rightarrow T^{\left(  \kappa\right)  }\left(  \mathbb{R}^{d}\right)  $ and
$\pi_{>\kappa}:=I_{T\left(  \mathbb{R}^{d}\right)  }-\pi_{\leq\kappa}:T\left(
\mathbb{R}^{d}\right)  \rightarrow\oplus_{k=\kappa+1}^{\infty}\left[
\mathbb{R}^{d}\right]  ^{\otimes k}$ be the projections associated to the
direct sum decomposition,
\[
T\left(  \mathbb{R}^{d}\right)  =T^{\left(  \kappa\right)  }\left(
\mathbb{R}^{d}\right)  \oplus\left(  \oplus_{k=\kappa+1}^{\infty}\left[
\mathbb{R}^{d}\right]  ^{\otimes k}\right)  .
\]
Further let
\begin{equation}
\mathfrak{g}^{\left(  \kappa\right)  }=\oplus_{k=1}^{\kappa}\left[
\mathbb{R}^{d}\right]  ^{\otimes k} \label{e.1.6}%
\end{equation}
which is a two sided ideal as well as a Lie sub-algebra of $T^{\left(
\kappa\right)  }\left(  \mathbb{R}^{d}\right)  .$
\end{notation}

With this notation the multiplication and Lie bracket on $T^{\left(
\kappa\right)  }\left(  \mathbb{R}^{d}\right)  $ may be described as,%
\[
AB=\pi_{\leq\kappa}\left(  A\otimes B\right)  \text{ and }\left[  A,B\right]
=\pi_{\leq\kappa}\left[  A,B\right]  _{\otimes}.
\]

\begin{notation}
[Induced Inner product]\label{not.1.15}The usual dot product on $\mathbb{R}%
^{d}$ induces an inner product, $\left\langle \cdot,\cdot,\right\rangle $ on
$T^{\left(  \kappa\right)  }\left(  \mathbb{R}^{d}\right)  $ uniquely
determined by requiring $T^{\left(  \kappa\right)  }\left(  \mathbb{R}%
^{d}\right)  :=\oplus_{k=0}^{\kappa}\left[  \mathbb{R}^{d}\right]  ^{\otimes
k}$ to be an orthogonal direct sum decomposition, $\left\langle
1,1\right\rangle =1$ for $1\in\left[  \mathbb{R}^{d}\right]  ^{\otimes0},$
and
\[
\left\langle v_{1}v_{2}\dots v_{k},w_{1}w_{2}\dots w_{k}\right\rangle
=\left\langle v_{1},w_{1}\right\rangle \left\langle v_{2},w_{2}\right\rangle
\dots\left\langle v_{k},w_{k}\right\rangle
\]
for any $v_{j},w_{j}\in\mathbb{R}^{d}$ and $1\leq k\leq\kappa.$ We let
$\left\vert A\right\vert :=\sqrt{\left\langle A,A\right\rangle }$ denote the
associated Hilbertian norm of $A\in T^{\left(  \kappa\right)  }\left(
\mathbb{R}^{d}\right)  .$
\end{notation}

Often, it turns out to be more convenient (see Proposition \ref{pro.3.24}
below) to measure the size of $A\in\mathfrak{g}^{\left(  \kappa\right)  }$
using the following \textquotedblleft homogeneous norms.\textquotedblright\

\begin{definition}
[Homogeneous norms]\label{def.1.16}For $A\in\mathfrak{g}^{\left(
\kappa\right)  }\subset T^{\left(  \kappa\right)  }\left(  \mathbb{R}%
^{d}\right)  ,$ let
\[
N\left(  A\right)  :=\max_{1\leq k\leq\kappa}\left\vert A_{k}\right\vert
^{1/k}%
\]
and for $f\in C\left(  \left[  0,t\right]  ,\mathfrak{g}^{\left(
\kappa\right)  }\right)  $ let
\[
N_{t}^{\ast}\left(  f\right)  :=\max_{1\leq k\leq\kappa}\left\vert
f_{k}\right\vert _{t}^{\ast1/k}=\max_{1\leq k\leq\kappa}\left(  \int_{0}%
^{t}\left\vert f_{k}\left(  \tau\right)  \right\vert d\tau\right)  ^{1/k}%
\]
be the \textbf{homogeneous }$L^{1}$\textbf{-norm of }$f$\textbf{. [}Note that
$N\left(  A\right)  $ is the best constant such that $\left\vert
A_{k}\right\vert \leq N\left(  A\right)  ^{k}$ for $1\leq k\leq\kappa.]$
\end{definition}

Let us observe that for $t\in\mathbb{R},$
\begin{equation}
N\left(  tA\right)  =\max_{1\leq k\leq\kappa}\left[  \left\vert t\right\vert
^{1/k}\left\vert A_{k}\right\vert ^{1/k}\right]  \leq\max_{1\leq k\leq\kappa
}\left[  \left\vert t\right\vert ^{1/k}\right]  \cdot N\left(  A\right)
\leq\left(  1\vee\left\vert t\right\vert \right)  \cdot N\left(  A\right)
\label{e.1.7}%
\end{equation}
and if $\delta_{t}:T^{\left(  \kappa\right)  }\left(  \mathbb{R}^{d}\right)
\rightarrow T^{\left(  \kappa\right)  }\left(  \mathbb{R}^{d}\right)  $ is the
dilation operator defined by $\delta_{t}\left(  A\right)  =\sum_{k=0}^{\kappa
}t^{k}A_{k},$ then%
\begin{equation}
N\left(  \delta_{t}A\right)  =\max_{1\leq k\leq\kappa}\left[  \left\vert
t^{k}A_{k}\right\vert ^{1/k}\right]  =\left\vert t\right\vert N\left(
A\right)  . \label{e.1.8}%
\end{equation}

\begin{definition}
[Free Nilpotent Lie Algebra]\label{not.1.17}The \textbf{step }$\kappa$\textbf{
free Nilpotent Lie algebra} on $\mathbb{R}^{d}$ may then be realized as the
Lie sub-algebra, $F^{\left(  \kappa\right)  }\left(  \mathbb{R}^{d}\right)  ,$
of $\mathfrak{g}^{\left(  \kappa\right)  }$ generated by $\mathbb{R}%
^{d}\subset T^{\left(  \kappa\right)  }\left(  \mathbb{R}^{d}\right)  .$
\end{definition}

Again, a simple consequence of Remark \ref{rem.1.12} is that, as vector
spaces, $F^{\left(  \kappa\right)  }\left(  \mathbb{R}^{d}\right)  =\pi
_{\leq\kappa}\left(  F\left(  \mathbb{R}^{d}\right)  \right)  $ and
$F^{\left(  \kappa\right)  }\left(  \mathbb{R}^{d}\right)  $ is graded as%
\[
F^{\left(  \kappa\right)  }\left(  \mathbb{R}^{d}\right)  =\oplus
_{k=0}^{\kappa}F_{k}^{\left(  \kappa\right)  }\left(  \mathbb{R}^{d}\right)
\]
where
\[
F_{k}^{\left(  \kappa\right)  }\left(  \mathbb{R}^{d}\right)  :=F^{\left(
\kappa\right)  }\left(  \mathbb{R}^{d}\right)  \cap\left[  \mathbb{R}%
^{d}\right]  ^{\otimes k}\subset F^{\left(  \kappa\right)  }\left(
\mathbb{R}^{d}\right)  \text{ for }1\leq k\leq\kappa.
\]

The set,
\begin{equation}
G^{\left(  \kappa\right)  }\left(  \mathbb{R}^{d}\right)  :=1+\mathfrak{g}%
^{\left(  \kappa\right)  }\subset T^{\left(  \kappa\right)  }\left(
\mathbb{R}^{d}\right)  , \label{e.1.9}%
\end{equation}
forms a group under the multiplication rule of $T^{\left(  \kappa\right)
}\left(  \mathbb{R}^{d}\right)  $ which is a Lie group with Lie algebra,
$\operatorname*{Lie}\left(  G^{\left(  \kappa\right)  }\right)  =\mathfrak{g}%
^{\left(  \kappa\right)  }.$ Moreover, the exponential map,
\[
\mathfrak{g}^{\left(  \kappa\right)  }\ni\xi\rightarrow e^{\xi}=\sum
_{k=0}^{\kappa}\frac{\xi^{k}}{k!}\in G^{\left(  \kappa\right)  }\left(
\mathbb{R}^{d}\right)  ,
\]
is a diffeomorphism whose inverse is given by%
\begin{equation}
\log\left(  1+\xi\right)  =\sum_{k=1}^{\kappa}\frac{\left(  -1\right)  ^{k+1}%
}{k}\xi^{k}. \label{e.1.10}%
\end{equation}
[See Section \ref{sec.3} for more details.] We will mostly only use the
following subgroup of $G^{\left(  \kappa\right)  }\left(  \mathbb{R}%
^{d}\right)  .$

\begin{definition}
[Free Nilpotent Lie Groups]\label{not.1.18}For $\kappa\in\mathbb{N},$ let
$G_{\text{geo}}^{\left(  \kappa\right)  }\left(  \mathbb{R}^{d}\right)
\subset G^{\left(  \kappa\right)  }$ be the simply connected Lie subgroup of
$G^{\left(  \kappa\right)  }=1\oplus_{k=1}^{\kappa}\left[  \mathbb{R}%
^{d}\right]  ^{\otimes k}$ whose Lie algebra is $F^{\left(  \kappa\right)
}\left(  \mathbb{R}^{d}\right)  .$ This subgroup is a step-$\kappa$ (free)
nilpotent Lie group which we refer to as the \textbf{geometric sub-group }of
$G^{\left(  \kappa\right)  }.$
\end{definition}

It is well known as a consequence of the Baker-Campel-Dynken-Hausdorff formula
(see Proposition \ref{pro.3.12} of Section \ref{sec.3}) that the exponential
map restricted to $F^{\left(  \kappa\right)  }\left(  \mathbb{R}^{d}\right)
,$%
\[
F^{\left(  \kappa\right)  }\left(  \mathbb{R}^{d}\right)  \ni\xi\rightarrow
e^{\xi}=\sum_{k=0}^{\kappa}\frac{\xi^{k}}{k!}\in G_{\text{geo}}^{\left(
\kappa\right)  }\left(  \mathbb{R}^{d}\right)  ,
\]
is again diffeomorphism.

\begin{notation}
\label{not.1.19}Let $\mathrm{LD}\left(  C^{\infty}\left(  M,\mathbb{R}\right)
\right)  $ denote the algebra of smooth linear differential operators from
$C^{\infty}\left(  M,\mathbb{R}\right)  .$
\end{notation}

As usual we view the smooth vector fields, $\Gamma\left(  TM\right)  ,$ on $M$
as a subspace of $\mathrm{LD}\left(  C^{\infty}\left(  M,\mathbb{R}\right)
\right)  .$

\begin{definition}
[Dynamical systems]\label{def.1.20}A $d$\textbf{-dimensional dynamical system}
on $M$ is a linear map, $\mathbb{R}^{d}\ni w\rightarrow V_{w}\in\Gamma\left(
TM\right)  .$
\end{definition}

A $d$-dimensional dynamical system on $M$ is completely determined by knowing
$\left\{  V_{e_{j}}\right\}  _{j=1}^{d}\subset\Gamma\left(  TM\right)  $ where
$\left\{  e_{j}\right\}  _{j=1}^{d}$ is the standard basis for $\mathbb{R}%
^{d}.$ The tensor algebra, $T\left(  \mathbb{R}^{d}\right)  ,$ of Definition
\ref{def.1.10} satisfies the following universal property; if $V:\mathbb{R}%
^{d}\rightarrow\mathcal{A}$ is a linear map, $\mathcal{A}$ is another
associative algebra with identity, then $V$ extends uniquely to an algebra
homomorphism from $T\left(  \mathbb{R}^{d}\right)  $ to $\mathcal{A}$ which we
still denote by $V.$ The extension is uniquely determined by $V_{1}%
=1_{\mathcal{A}}$ and $V_{v_{1}\otimes\dots\otimes v_{k}}=V_{v_{1}}\dots
V_{v_{k}}$ for all $v_{i}\in\mathbb{R}^{d}$ and $k\in\mathbb{N}.$ The
following example is of primary importance to this paper.

\begin{example}
\label{ex.1.21}Every $d$-dimensional dynamical system on $M,$ $\mathbb{R}%
^{d}\ni w\rightarrow V_{w}\in\Gamma\left(  TM\right)  \subset\mathrm{LD}%
\left(  C^{\infty}\left(  M,\mathbb{R}\right)  \right)  ,$ extends to an
algebra homomorphism from $T\left(  \mathbb{R}^{d}\right)  $ to $\mathrm{LD}%
\left(  C^{\infty}\left(  M,\mathbb{R}\right)  \right)  .$ We will still
denote this extension by $V.$ Because of Remark \ref{rem.1.12}, it is easy to
see that $V\left(  F\left(  \mathbb{R}^{d}\right)  \right)  \subset
\Gamma\left(  TM\right)  $ and $V|_{F\left(  \mathbb{R}^{d}\right)  }:F\left(
\mathbb{R}^{d}\right)  \rightarrow\Gamma\left(  TM\right)  $ is a Lie algebra homomorphism.
\end{example}

\begin{notation}
[Extension of $V$ to $F^{\left(  \kappa\right)  }\left(  \mathbb{R}%
^{d}\right)  $]\label{not.1.22}The restriction, $V|_{F^{\left(  \kappa\right)
}\left(  \mathbb{R}^{d}\right)  },$ of $V$ to the subspace $F^{\left(
\kappa\right)  }\left(  \mathbb{R}^{d}\right)  $ of $F\left(  \mathbb{R}%
^{d}\right)  $ will be denoted by $V^{\left(  \kappa\right)  }:F^{\left(
\kappa\right)  }\left(  \mathbb{R}^{d}\right)  \rightarrow\Gamma\left(
TM\right)  .$
\end{notation}

\begin{remark}
\label{rem.1.23}It is \textbf{not} generally true that $V^{\left(
\kappa\right)  }:=V|_{F^{\left(  \kappa\right)  }\left(  \mathbb{R}%
^{d}\right)  }:F^{\left(  \kappa\right)  }\left(  \mathbb{R}^{d}\right)
\rightarrow\Gamma\left(  TM\right)  $ is a Lie algebra homomorphism. In order
for this to be true we must require that $\operatorname{ad}_{V_{a_{\kappa}}%
}\dots\operatorname{ad}_{V_{a_{1}}}V_{a_{0}}=0$ for all $\left\{
a_{j}\right\}  _{j=0}^{\kappa}\subset\mathbb{R}^{d},$ i.e. $\left\{
V_{a}:a\in\mathbb{R}^{d}\right\}  $ should generate a \textbf{step-}$\kappa
$\textbf{ nilpotent Lie sub-algebra} of $\Gamma\left(  TM\right)  .$
\end{remark}

\begin{definition}
[Dynamical System Norms]\label{def.1.24}If $V$ is a dynamical system and
$\kappa\in\mathbb{N},$ we let%
\begin{align}
\left\vert V^{\left(  \kappa\right)  }\right\vert _{M}  &  :=\left\{
\left\vert V_{A}\right\vert _{M}:A\in F^{\left(  \kappa\right)  }\left(
\mathbb{R}^{d}\right)  \text{ with }\left\vert A\right\vert =1\right\}
,\label{e.1.11}\\
\left\vert \nabla V^{\left(  \kappa\right)  }\right\vert _{M}  &  :=\left\{
\left\vert \nabla V_{A}\right\vert _{M}:A\in F^{\left(  \kappa\right)
}\left(  \mathbb{R}^{d}\right)  \text{ with }\left\vert A\right\vert
=1\right\}  ,\label{e.1.12}\\
\left\vert \nabla^{2}V^{\left(  \kappa\right)  }\right\vert _{M}  &
:=\left\{  \left\vert \nabla^{2}V_{A}\right\vert _{M}:A\in F^{\left(
\kappa\right)  }\left(  \mathbb{R}^{d}\right)  \text{ with }\left\vert
A\right\vert =1\right\}  ,\text{ and}\label{e.1.13}\\
H_{M}\left(  V^{\left(  \kappa\right)  }\right)   &  :=\sup\left\{
H_{M}\left(  V_{A}\right)  :A\in F^{\left(  \kappa\right)  }\left(
\mathbb{R}^{d}\right)  \text{ with }\left\vert A\right\vert =1\right\}
\label{e.1.14}%
\end{align}
where we allow for the possibility that any of these expressions might be
infinite. [Recall that $H_{M}\left(  V_{A}\right)  $ is defined in Notation
\ref{not.1.5} and Example \ref{ex.1.6}.]
\end{definition}

\subsection{Approximate logarithm theorems\label{sec.1.3}}

\begin{definition}
[See Definition \ref{def.3.6}]\label{def.1.25}For $\xi\in C^{1}\left(  \left[
0,T\right]  ,F^{\left(  \kappa\right)  }\left(  \mathbb{R}^{d}\right)
\right)  ,$ let $g^{\xi}\in C^{1}\left(  \left[  0,T\right]  ,G_{geo}\right)
$ denote the solution to the ODE,%
\begin{equation}
\dot{g}^{\xi}\left(  t\right)  =g^{\xi}\left(  t\right)  \dot{\xi}\left(
t\right)  \text{ with }g^{\xi}\left(  0\right)  =1 \label{e.1.15}%
\end{equation}
and
\begin{equation}
C^{\xi}\left(  t\right)  :=\log\left(  g^{\xi}\left(  t\right)  \right)
=\sum_{k=1}^{\kappa}\frac{\left(  -1\right)  ^{k+1}}{k}\left(  g^{\xi}\left(
t\right)  -1\right)  ^{k}\in F^{\left(  \kappa\right)  }\left(  \mathbb{R}%
^{d}\right)  . \label{e.1.16}%
\end{equation}

\end{definition}

\begin{notation}
\label{not.1.26}For $f,g\in C^{1}\left(  M,M\right)  ,$ let
\begin{align*}
d_{M}\left(  f,g\right)   &  :=\sup_{m\in M}d\left(  f\left(  m\right)
,g\left(  m\right)  \right)  \text{ and }\\
d_{M}^{TM}\left(  f_{\ast},g_{\ast}\right)   &  :=\sup_{v\in TM:\left\vert
v\right\vert =1}d^{TM}\left(  f_{\ast}v,g_{\ast}v\right)
\end{align*}
where again $d^{TM}$ is defined in Section \ref{sec.5} below.
\end{notation}

\begin{definition}
[$\kappa$-complete]\label{def.1.27}We say that a dynamical system,
$\mathbb{R}^{d}\ni w\rightarrow V_{w}\in\Gamma\left(  TM\right)  ,$ is
$\kappa$\textbf{-complete }if for any $\xi\in C^{1}\left(  \left[  0,T\right]
,F^{\left(  \kappa\right)  }\left(  \mathbb{R}^{d}\right)  \right)  $ the time
dependent vector-field, $\left[  0,T\right]  \ni t\rightarrow V_{\dot{\xi
}\left(  t\right)  }\in\Gamma\left(  TM\right)  ,$ is complete as defined in
Definition \ref{def.1.1}.
\end{definition}

\begin{ass}
\label{ass.1}Unless otherwise stated, the dynamical system $V:\mathbb{R}%
^{d}\rightarrow\Gamma\left(  TM\right)  $ is assumed to be $\kappa$-complete.
\end{ass}

The next two theorems are the main theorems of this paper. The first theorem
is a combination of Theorem \ref{thm.4.12}, Eq. (\ref{e.4.18}), and Corollary
\ref{cor.4.16}. To simplify the statements we first introduce the following notation.

\begin{notation}
\label{not.1.28}For $\lambda\geq0$ and $m,n\in\mathbb{N}$ with $m<n,$ let
\begin{align*}
Q_{[m,n]}\left(  \lambda\right)   &  :=\max\left\{  \lambda^{k}:k\in
\mathbb{N}\cap\left[  m,n\right]  \right\}  =\max\left\{  \lambda^{m}%
,\lambda^{n}\right\}  \text{ and}\\
Q_{(m,n]}\left(  \lambda\right)   &  =Q_{[m+1,n]}\left(  \lambda\right)
:=\max\left\{  \lambda^{k}:k\in\mathbb{N}\cap(m,n]\right\}  =\max\left\{
\lambda^{m+1},\lambda^{n}\right\}  .
\end{align*}

\end{notation}

\begin{notation}
\label{not.1.29}Given two functions, $f\left(  x\right)  $ and $g\left(
x\right)  ,$ depending on some parameters indicated by $x,$ we write $f\left(
x\right)  \lesssim g\left(  x\right)  $ if there exists a constant, $C\left(
\kappa\right)  ,$ only possibly depending on $\kappa$ so that $f\left(
x\right)  \leq C\left(  \kappa\right)  g\left(  x\right)  $ for the allowed
values of $x.$ Similarly we write $f\left(  x\right)  \asymp g\left(
x\right)  $ if both $f\left(  x\right)  \lesssim g\left(  x\right)  $ and
$g\left(  x\right)  \lesssim f\left(  x\right)  $ hold.
\end{notation}

\begin{theorem}
\label{thm.1.30}There is a constant $c\left(  \kappa\right)  <\infty$ such
that%
\begin{align*}
d_{M}  &  \left(  \mu_{T,0}^{V_{\dot{\xi}}},e^{V_{\log\left(  g^{\xi}\left(
T\right)  \right)  }}\right) \\
&  \lesssim\left\vert V^{\left(  \kappa\right)  }\right\vert _{M}\left\vert
\nabla V^{\left(  \kappa\right)  }\right\vert _{M}e^{c\left(  \kappa\right)
\left\vert \nabla V^{\left(  \kappa\right)  }\right\vert _{M}Q_{\left[
1,\kappa\right]  }\left(  N_{T}^{\ast}\left(  \dot{\xi}\right)  \right)
}Q_{(\kappa,\kappa+1]}\left(  N_{T}^{\ast}\left(  \dot{\xi}\right)  \right)
\end{align*}
for every $\xi\in C^{1}\left(  \left[  0,T\right]  ,F^{\left(  \kappa\right)
}\left(  \mathbb{R}^{d}\right)  \right)  .$ Moreover, if $A,B\in F^{\left(
\kappa\right)  }\left(  \mathbb{R}^{d}\right)  ,$ then%
\[
d_{M}\left(  e^{V_{B}},Id_{M}\right)  \leq\left\vert V^{\left(  \kappa\right)
}\right\vert \left\vert B\right\vert \leq\left\vert V^{\left(  \kappa\right)
}\right\vert Q_{\left[  1,\kappa\right]  }\left(  N\left(  B\right)  \right)
\]
and
\begin{align*}
d_{M}  &  \left(  e^{V_{B}}\circ e^{V_{A}},e^{V_{\log\left(  e^{A}%
e^{B}\right)  }}\right) \\
&  \lesssim\mathcal{K}_{0}N\left(  A\right)  N\left(  B\right)  Q_{\left[
\kappa-1,2\kappa-2\right]  }\left(  N\left(  A\right)  +N\left(  B\right)
\right)
\end{align*}
where
\[
\mathcal{K}_{0}:=\left\vert V^{\left(  \kappa\right)  }\right\vert
_{M}\left\vert \nabla V^{\left(  \kappa\right)  }\right\vert _{M}e^{c\left(
\kappa\right)  \left\vert \nabla V^{\left(  \kappa\right)  }\right\vert
_{M}Q_{\left[  1,\kappa\right]  }\left(  N\left(  A\right)  +N\left(
B\right)  \right)  }.
\]

\end{theorem}

\begin{remark}
[Dialating Theorem \ref{thm.1.30}]\label{rem.1.31}If we define the dilation
homomorphism, $\delta_{\lambda}:T^{\left(  \kappa\right)  }\left(
\mathbb{R}^{d}\right)  \rightarrow T^{\left(  \kappa\right)  }\left(
\mathbb{R}^{d}\right)  ,$ where $\delta_{\lambda}A=\sum_{k=0}^{\kappa}%
\lambda^{k}A_{k}$ for $\lambda>0$ and $A\in T^{\left(  \kappa\right)  }\left(
\mathbb{R}^{d}\right)  ,$ then $N_{T}^{\ast}\left(  \lambda\dot{\xi}\right)
=\lambda N_{T}^{\ast}\left(  \dot{\xi}\right)  $ and hence it follows from
Theorem \ref{thm.1.30} that%
\[
d_{M}\left(  \mu_{T,0}^{V_{\delta_{\lambda}\dot{\xi}}},e^{V_{\log\left(
g^{\delta_{\lambda}\xi}\left(  T\right)  \right)  }}\right)  =O\left(
\lambda^{\kappa+1}\right)  \text{ and }\lambda\rightarrow0.
\]
If is also easy to verify, 1) $N\left(  \delta_{\lambda}A\right)  =\lambda
N\left(  A\right)  $ for all $A\in F^{\left(  \kappa\right)  }\left(
\mathbb{R}^{d}\right)  ,$ 2) $g^{\delta_{\lambda}\xi}=\delta_{\lambda}\left(
g^{\xi}\right)  ,$%
\[
\text{3)~}\log\left(  g^{\delta_{\lambda}\xi}\right)  =\log\left(
\delta_{\lambda}\left(  g^{\xi}\right)  \right)  =\delta_{\lambda}\log\left(
g^{\xi}\right)  ,
\]
and 4) $\delta_{\lambda}\dot{\xi}\left(  t\right)  =\lambda\dot{\xi}\left(
t\right)  $ in the special case where $\xi\left(  t\right)  \in\mathbb{R}%
^{d}\subset F^{\left(  \kappa\right)  }\left(  \mathbb{R}^{d}\right)  .$
\end{remark}

The next theorem (which is a combination of Theorem \ref{thm.8.4}, Eq.
(\ref{e.4.19}), and Corollary \ref{cor.8.5}) is an analogue of Theorem
\ref{thm.1.30} for the differentials of $\mu_{T,0}^{V_{\dot{\xi}}}$ of
$e^{V_{\log\left(  g^{\xi}\left(  T\right)  \right)  }}.$

\begin{theorem}
\label{thm.1.32}If $\xi\in C^{1}\left(  \left[  0,T\right]  ,F^{\left(
\kappa\right)  }\left(  \mathbb{R}^{d}\right)  \right)  ,$ then%
\[
d_{M}^{TM}\left(  \mu_{T,0\ast}^{V_{\dot{\xi}}},e_{\ast}^{V_{\log\left(
g^{\xi}\left(  T\right)  \right)  }}\right)  \leq\mathcal{K}\cdot
Q_{(\kappa,2\kappa]}\left(  N_{T}^{\ast}\left(  \dot{\xi}\right)  \right)  ,
\]
where
\[
\mathcal{K=K}\left(  T,\left\vert V^{\left(  \kappa\right)  }\right\vert
_{M},\left\vert \nabla V^{\left(  \kappa\right)  }\right\vert _{M},\left\vert
\nabla^{2}V^{\left(  \kappa\right)  }\right\vert _{M},\left\vert R\left\langle
V_{\cdot},\bullet\right\rangle \right\vert _{M},N_{T}^{\ast}\left(  \dot{\xi
}\right)  \right)
\]
is a (fairly complicated) increasing function of each of its arguments.
Moreover, if $A,B\in F^{\left(  \kappa\right)  }\left(  \mathbb{R}^{d}\right)
,$ then%
\[
d_{M}^{TM}\left(  e^{V_{B}},Id_{M}\right)  \leq\left\vert V^{\left(
\kappa\right)  }\right\vert \left\vert B\right\vert \leq\left\vert V^{\left(
\kappa\right)  }\right\vert Q_{\left[  1,\kappa\right]  }\left(  N\left(
B\right)  \right)
\]
and%
\begin{align*}
d_{M}^{TM}  &  \left(  \left[  e^{V_{B}}\circ e^{V_{A}}\right]  _{\ast
},e_{\ast}^{V_{\log\left(  e^{A}e^{B}\right)  }}\right) \\
&  \leq\mathcal{K}_{1}\cdot N\left(  A\right)  N\left(  B\right)
Q_{(\kappa-1,2\left(  \kappa-1\right)  ]}\left(  N\left(  A\right)  +N\left(
B\right)  \right)  .
\end{align*}
where
\[
\mathcal{K}_{1}=\mathcal{K}_{1}\left(  \left\vert V^{\left(  \kappa\right)
}\right\vert _{M},\left\vert \nabla V^{\left(  \kappa\right)  }\right\vert
_{M},H_{M}\left(  V^{\left(  \kappa\right)  }\right)  ,N\left(  A\right)  \vee
N\left(  B\right)  \right)  .
\]

\end{theorem}

This paper separates into two parts. The first part consisting of Sections
\ref{sec.2} --\ref{sec.4} which develops the results needed to prove Theorem
\ref{thm.1.30} estimating error between the flow $\mu_{T,0}^{V_{\dot{\xi}}}$
and $e^{V_{\log\left(  g^{\xi}\left(  T\right)  \right)  }}.$ The second part
of the paper consists of Sections \ref{sec.5} -- \ref{sec.8} where the tools
are developed to estimate the error between the differentials of $\mu
_{T,0}^{V_{\dot{\xi}}}$ and $e^{V_{\log\left(  g^{\xi}\left(  T\right)
\right)  }}$ given in Theorem \ref{thm.1.32}. The computations in the second
part are necessarily more complicated and this is where curvature of $M$
enters the scene. Lastly, the Appendix \ref{sec.9} gathers some basic Gronwall
type estimates used in the body of this paper.

\subsection{Acknowledgments}

The author is very grateful to Masha Gordina for many illuminating
conversations on this work and to her hospitality and to that of the
mathematics department at the University of Connecticut where this work was
started while I was on sabbatical in the Fall of 2017.

\section{Geometric notation and background\label{sec.2}}

\subsection{Riemannian distance\label{sec.2.1}}

Given $-\infty<a<b<\infty,$ a path, $\sigma\in C\left(  \left[  a,b\right]
\rightarrow M\right)  $ is said to be \textbf{absolutely continuous} provided
for any chart $x$ on $M$ and closed positive length subinterval,
$J\subset\left[  a,b\right]  $ such that $\sigma\left(  J\right)
\subset\mathcal{D}\left(  x\right)  $ ($\mathcal{D}\left(  x\right)  $ is the
domain of $x)$ we have then $x\circ\sigma|_{J}:J\rightarrow\mathbb{R}^{d}$ is
absolutely continuous.

\begin{notation}
\label{not.2.1}For $-\infty<a<b<\infty,$ let $AC\left(  \left[  a,b\right]
\rightarrow M\right)  $ denote the \textbf{absolutely continuous paths} from
$\left[  a,b\right]  $ to $M$. Moreover, if $p,q\in M,$ let
\[
AC_{p,q}\left(  \left[  a,b\right]  \rightarrow M\right)  :=\left\{  \sigma\in
AC\left(  \left[  a,b\right]  \rightarrow M\right)  :\sigma\left(  a\right)
=p\text{ and }\sigma\left(  b\right)  =q\right\}  .
\]
The\textbf{ length}, $\ell_{M}\left(  \sigma\right)  ,$ of a path in
$\sigma\in AC\left(  \left[  a,b\right]  \rightarrow M\right)  $ is defined
by
\[
\ell_{M}\left(  \sigma\right)  :=\int_{a}^{b}\left\vert \dot{\sigma}\left(
t\right)  \right\vert dt
\]
and (as usual) the \textbf{distance }between $m,m^{\prime}\in M$ is defined by%
\[
d\left(  m,m^{\prime}\right)  :=\inf\left\{  \ell_{M}\left(  \sigma\right)
:\sigma\in AC_{m,m^{\prime}}\left(  \left[  0,1\right]  \rightarrow M\right)
\right\}  .
\]
Given $v\in TM,$ let $\sigma_{v}\left(  t\right)  $ be the geodesic in $M$
such that $\dot{\sigma}_{v}\left(  0\right)  =v,$ $\exp\left(  v\right)
=\sigma_{v}\left(  1\right)  \in M$ for those $v\in TM$ such that $\sigma
_{v}\left(  1\right)  $ exists, and for $m\in M$ we let $\exp_{m}%
:=\exp|_{T_{m}M}:T_{m}M\rightarrow M.$
\end{notation}

Throughout this paper we will use the following geometric notations.

\begin{notation}
[Metric vector bundles and connections]\label{not.2.2}Let $\left(  M,g\right)
$ be a Riemannian manifold, $\pi:E\rightarrow M$ be a real Hermitian vector
bundle over $M$ (with fiber dimension, $D)$ with the fiber metric denoted by,
$\left\langle \cdot,\cdot\right\rangle _{E}.$ We further assume that $E$ is
equipped with a metric compatible covariant derivative, $\nabla=\nabla^{E}$.
[Typically we are interested in the setting where $E=TM$ in which case we
always take $\nabla=\nabla^{TM}$ to be the Levi-Civita covariant derivative on
$TM.]$ Further,

\begin{enumerate}
\item let $E_{m}:=\pi^{-1}\left(  \left\{  m\right\}  \right)  $ be the fiber
over $m$ which is isomorphic to $\mathbb{R}^{D},$

\item let $\pt_{t}^{\nabla}\left(  \sigma\right)  :E_{\sigma\left(  a\right)
}\rightarrow E_{\sigma\left(  t\right)  }$ denote parallel translation along a
curve $\sigma\in C^{1}\left(  \left[  a,b\right]  \rightarrow M\right)  $ or
more generally along $\sigma\in AC\left(  \left[  a,b\right]  \rightarrow
M\right)  $ -- the space of $M$-valued absolutely continuous paths on $\left[
a,b\right]  ,$ and

\item if $\xi\left(  t\right)  \in E_{\sigma\left(  t\right)  }$ for
$t\in\left[  a,b\right]  ,$ let
\[
\nabla_{t}\xi\left(  t\right)  =\frac{\nabla\xi}{dt}\left(  t\right)
:=\pt_{t}\left(  \sigma\right)  \frac{d}{dt}\left[  \pt_{t}\left(
\sigma\right)  ^{-1}\xi\left(  t\right)  \right]  .
\]

\end{enumerate}
\end{notation}

By assumption, for every $m\in M,$ there exists an open neighborhood $\left(
W\right)  $ of $m$ and a smooth function $W\times\mathbb{R}^{D}\ni\left(
m,\alpha\right)  \rightarrow u\left(  m\right)  \alpha\in E$ such that
$u\left(  m\right)  :\mathbb{R}^{D}\rightarrow E_{m}$ is an isometric
isomorphism of inner product spaces. We refer to $\left(  u,W\right)  $ as
\textbf{a (local) orthogonal frame of }$E.$ We also let $SO\left(
\mathbb{R}^{D}\right)  $ be the group of $D\times D$ real orthogonal matrices
with determinant equal to $1$ and let $so\left(  \mathbb{R}^{D}\right)  $ be
its Lie algebra of $D\times D$ real skew-symmetric matrices.

\begin{remark}
[Local model for $E$]\label{rem.2.3}As described just above, after choosing a
local orthogonal frame, we may identify (locally) $E$ with the trivial bundle
$W\times\mathbb{R}^{D}$ where $W$ is an open subset of $W.$ In this local
model we have;

\begin{enumerate}
\item $\pi\left(  m,\alpha\right)  =m$ for all $m\in W$ and $\alpha
\in\mathbb{R}^{d}.$

\item $\left\langle \left(  m,\alpha\right)  ,\left(  m,\beta\right)
\right\rangle =\alpha\cdot\beta$ for all $m\in W$ and $\alpha,\beta
\in\mathbb{R}^{d}.$

\item There exists and $so\left(  \mathbb{R}^{D}\right)  $-valued one form,
$\Gamma,$ such that if $S\left(  m\right)  =\left(  m,\alpha\left(  m\right)
\right)  $ is a section of $E$ and $v\in T_{m}W,$ then
\[
\nabla_{v}S=\left(  m,d\alpha\left(  v_{m}\right)  +\Gamma\left(
v_{m}\right)  \alpha\left(  m\right)  \right)  .
\]

\item If $\sigma\in C^{1}\left(  \left[  \alpha,\beta\right]  \rightarrow
W\right)  ,$ then $\pt_{t}\left(  \sigma\right)  \left(  \sigma\left(
a\right)  ,\alpha\right)  =\left(  \sigma\left(  t\right)  ,g\left(  t\right)
\alpha\right)  $ where $g\left(  t\right)  \in SO\left(  \mathbb{R}%
^{D}\right)  $ is the solution to the ordinary differential equation,
\[
\dot{g}\left(  t\right)  +\Gamma\left(  \dot{\sigma}\left(  t\right)  \right)
g\left(  t\right)  =0\text{ with }g\left(  a\right)  =I_{\mathbb{R}^{D}}.
\]

\item If $\xi\left(  t\right)  =\left(  \sigma\left(  t\right)  ,\alpha\left(
t\right)  \right)  $ is a $C^{1}$-path in $E,$ then%
\begin{equation}
\frac{\nabla\xi}{dt}\left(  t\right)  =\left(  \sigma\left(  t\right)
,\dot{\alpha}\left(  t\right)  +\Gamma\left(  \dot{\sigma}\left(  t\right)
\right)  \alpha\left(  t\right)  \right)  . \label{e.2.1}%
\end{equation}

\end{enumerate}

For completeness, here is the verification of Eq. (\ref{e.2.1});%
\begin{align*}
\frac{d}{dt}\left[  \pt_{t}\left(  \sigma\right)  ^{-1}\xi\left(  t\right)
\right]   &  =\frac{d}{dt}\left(  \sigma\left(  a\right)  ,g\left(  t\right)
^{-1}\alpha\left(  t\right)  \right) \\
&  =\left(  \sigma\left(  a\right)  ,g\left(  t\right)  ^{-1}\dot{\alpha
}\left(  t\right)  -g\left(  t\right)  ^{-1}\dot{g}\left(  t\right)  g\left(
t\right)  ^{-1}\alpha\left(  t\right)  \right) \\
&  =\left(  \sigma\left(  a\right)  ,g\left(  t\right)  ^{-1}\dot{\alpha
}\left(  t\right)  +g\left(  t\right)  ^{-1}\Gamma\left(  \dot{\sigma}\left(
t\right)  \right)  \alpha\left(  t\right)  \right) \\
&  =\pt_{t}\left(  \sigma\right)  ^{-1}\left(  \sigma\left(  t\right)
,\dot{\alpha}\left(  t\right)  +\Gamma\left(  \dot{\sigma}\left(  t\right)
\right)  \alpha\left(  t\right)  \right)  .
\end{align*}

\end{remark}

The next elementary lemma illustrates how the structures in Notation
\ref{not.2.2} fit together.

\begin{lemma}
\label{lem.2.4}If $\xi:\left[  a,b\right]  \rightarrow E$ is a $C^{1}$-curve
and $\sigma:=\pi\circ\xi\in C^{1}\left(  \left[  a,b\right]  ,M\right)  ,$
then
\begin{equation}
\left\vert \left\vert \xi\left(  b\right)  \right\vert -\left\vert \xi\left(
a\right)  \right\vert \right\vert \leq\left\vert \pt_{b}\left(  \sigma\right)
^{-1}\xi\left(  b\right)  -\xi\left(  a\right)  \right\vert \leq\int_{a}%
^{b}\left\vert \frac{\nabla}{dt}\xi\left(  t\right)  \right\vert dt.
\label{e.2.2}%
\end{equation}

\end{lemma}

\begin{proof}
By the metric compatibility of $\nabla,$ $\left\vert \xi\left(  b\right)
\right\vert =\left\vert \pt_{b}\left(  \sigma\right)  ^{-1}\xi\left(
b\right)  \right\vert $ and therefore%
\[
\left\vert \left\vert \xi\left(  b\right)  \right\vert -\left\vert \xi\left(
a\right)  \right\vert \right\vert =\left\vert \left\vert \pt_{b}\left(
\sigma\right)  ^{-1}\xi\left(  b\right)  \right\vert -\left\vert \xi\left(
a\right)  \right\vert \right\vert \leq\left\vert \pt_{b}\left(  \sigma\right)
^{-1}\xi\left(  b\right)  -\xi\left(  a\right)  \right\vert
\]
which proves the first inequality in Eq. (\ref{e.2.2}). By the fundamental
theorem of calculus and the definition of $\frac{\nabla}{dt},$%
\begin{align*}
\pt_{b}\left(  \sigma\right)  ^{-1}\xi\left(  b\right)  -\xi\left(  a\right)
&  =\int_{a}^{b}\frac{d}{dt}\left[  \pt_{t}\left(  \sigma\right)  ^{-1}%
\xi\left(  t\right)  \right]  dt\\
&  =\int_{a}^{b}\pt_{t}\left(  \sigma\right)  ^{-1}\frac{\nabla}{dt}\xi\left(
t\right)  dt.
\end{align*}
The second inequality in Eq. (\ref{e.2.2}) now follows from this identity and
the triangle inequality for vector valued integrals,
\[
\left\vert \int_{a}^{b}\pt_{t}\left(  \sigma\right)  ^{-1}\frac{\nabla}{dt}%
\xi\left(  t\right)  dt\right\vert \leq\int_{a}^{b}\left\vert \pt_{t}\left(
\sigma\right)  ^{-1}\frac{\nabla}{dt}\xi\left(  t\right)  \right\vert
dt=\int_{a}^{b}\left\vert \frac{\nabla}{dt}\xi\left(  t\right)  \right\vert
dt.
\]

\end{proof}

\begin{definition}
\label{def.2.5}If $X\in\Gamma\left(  TM\right)  $ and $v_{m},w_{m}\in T_{m}M,$
let%
\[
\nabla_{v_{m}\otimes w_{m}}^{2}X:=\frac{d}{dt}|_{0}\left[  \pt_{t}\left(
\sigma\right)  ^{-1}\left(  \nabla_{\pt_{t}\left(  \sigma\right)  w_{m}%
}X\right)  \right]
\]
where $\sigma\left(  t\right)  \in M$ is chosen so that $\dot{\sigma}\left(
0\right)  =v_{m}.$ In this notation the curvature tensor may be defined by%
\[
R\left(  v_{m},w_{m}\right)  \xi_{m}=\nabla_{v_{m}\otimes w_{m}}^{2}%
X-\nabla_{w_{m}\otimes v_{m}}^{2}X,
\]
where $X\in\Gamma\left(  TM\right)  $ is any vector field such that $X\left(
m\right)  =\xi_{m}\in T_{m}M.$
\end{definition}

\begin{remark}
\label{rem.2.6}If $W,X\in\Gamma\left(  TM\right)  $ and $v_{m}=\dot{\sigma
}\left(  0\right)  \in T_{m}M,$ then
\begin{align}
\nabla_{v_{m}}\nabla_{W}X  &  =\frac{d}{dt}|_{0}\pt_{t}\left(  \sigma\right)
^{-1}\left(  \nabla_{W}X\right)  \left(  \sigma\left(  t\right)  \right)
\nonumber\\
&  =\frac{d}{dt}|_{0}\left[  \pt_{t}\left(  \sigma\right)  ^{-1}%
\nabla_{W\left(  \sigma\left(  t\right)  \right)  }X\right] \nonumber\\
&  =\frac{d}{dt}|_{0}\left[  \pt_{t}\left(  \sigma\right)  ^{-1}%
\nabla_{\pt_{t}\left(  \sigma\right)  \left[  \pt_{t}\left(  \sigma\right)
^{-1}W\left(  \sigma\left(  t\right)  \right)  \right]  }X\right] \nonumber\\
&  =\frac{d}{dt}|_{0}\left[  \pt_{t}\left(  \sigma\right)  ^{-1}%
\nabla_{\pt_{t}\left(  \sigma\right)  W\left(  m\right)  }X\right]  +\frac
{d}{dt}|_{0}\left[  \nabla_{\left[  \pt_{t}\left(  \sigma\right)
^{-1}W\left(  \sigma\left(  t\right)  \right)  \right]  }X\right] \nonumber\\
&  =\nabla_{v_{m}\otimes W\left(  m\right)  }^{2}X+\nabla_{\nabla_{v_{m}}W}X.
\label{e.2.3}%
\end{align}
This shows two things; 1) that $\nabla_{v_{m}\otimes w_{m}}^{2}X$ is
independent of the choice of curve, $\sigma\left(  t\right)  $ such that
$\dot{\sigma}\left(  0\right)  =v_{m}$ since
\[
\nabla_{v_{m}\otimes W\left(  m\right)  }^{2}X=\nabla_{v_{m}}\nabla
_{W}X-\nabla_{\nabla_{v_{m}}W}X,
\]
and 2) that with this definition of $\nabla^{2}X$ the natural product rule
derived in Eq. (\ref{e.2.3}) holds.
\end{remark}

\begin{definition}
\label{def.2.7}For $f\in C^{1}\left(  M,M\right)  $ let $f_{\ast
}:TM\rightarrow TM$ be the differential of $f,$
\begin{align*}
\left\vert f_{\ast}\right\vert _{m}  &  :=\sup_{v\in T_{m}M:\left\vert
v\right\vert =1}\left\vert f_{\ast}v\right\vert \text{ for each }m\in M,\text{
and}\\
\left\vert f_{\ast}\right\vert _{M}  &  :=\sup_{m\in M}\left\vert f_{\ast
}\right\vert _{m}=\sup_{v\in TM:\left\vert v\right\vert =1}\left\vert f_{\ast
}v\right\vert .
\end{align*}

\end{definition}

\begin{definition}
\label{def.2.8}We say $f\in C\left(  M,M\right)  $ is \textbf{Lipschitz} if
there exists $K=K\left(  f\right)  <\infty$ such that
\begin{equation}
d\left(  f\left(  m\right)  ,f\left(  m^{\prime}\right)  \right)  \leq
Kd\left(  m,m^{\prime}\right)  \text{ }\forall~m,m^{\prime}\in M.
\label{e.2.4}%
\end{equation}
The best smallest $K\in\left[  0,\infty\right]  $ such that Eq. (\ref{e.2.4})
holds is denoted by $\operatorname{Lip}\left(  f\right)  ,$ i.e.
\[
\operatorname{Lip}\left(  f\right)  :=\sup_{m\neq m^{\prime}}\frac{d\left(
f\left(  m\right)  ,f\left(  m^{\prime}\right)  \right)  }{d\left(
m,m^{\prime}\right)  }.
\]
We will write $\operatorname{Lip}\left(  f\right)  =\infty$ if $f$ is not Lipschitz.
\end{definition}

\begin{lemma}
\label{lem.2.9}If $f\in C^{1}\left(  M,M\right)  $ then $\operatorname{Lip}%
\left(  f\right)  =\left\vert f_{\ast}\right\vert _{M}$.
\end{lemma}

\begin{proof}
Let $m,m^{\prime}\in M$ and $\sigma\in AC\left(  \left[  0,1\right]
,M\right)  $ such that $\sigma\left(  0\right)  =m$ and $\sigma\left(
1\right)  =m^{\prime}.$ Then $f\circ\sigma\in AC\left(  \left[  0,1\right]
,M\right)  $ and $\frac{d}{dt}f\left(  \sigma\left(  t\right)  \right)
=f_{\ast}\dot{\sigma}\left(  t\right)  $ for a.e. $t$ and therefore,%
\begin{align*}
d\left(  f\left(  m\right)  ,f\left(  m^{\prime}\right)  \right)   &  \leq
\ell\left(  f\circ\sigma\right)  =\int_{0}^{1}\left\vert f_{\ast}\dot{\sigma
}\left(  t\right)  \right\vert dt\\
&  \leq\int_{0}^{1}\left\vert f_{\ast}\right\vert _{M}\left\vert \dot{\sigma
}\left(  t\right)  \right\vert dt=\left\vert f_{\ast}\right\vert _{M}\ell
_{M}\left(  \sigma\right)  .
\end{align*}
Taking the infimum of this inequality over all $\sigma\in AC_{m,m^{\prime}%
}\left(  \left[  0,1\right]  ,M\right)  $ then shows
\[
d\left(  f\left(  m\right)  ,f\left(  m^{\prime}\right)  \right)
\leq\left\vert f_{\ast}\right\vert _{M}d\left(  m,m^{\prime}\right)
\]
which implies $\operatorname{Lip}\left(  f\right)  \leq\left\vert f_{\ast
}\right\vert _{M}.$

For the opposite inequality let $m\in M,$ $v\in T_{m}M$ with $\left\vert
v\right\vert =1,$ and let $\sigma_{v}\left(  t\right)  :=\exp_{m}\left(
tv\right)  $ for $t$ near $0.$ Further let $\gamma\left(  t\right)  $ be the
smooth curve in $T_{f\left(  m\right)  }M$ satisfying $\gamma\left(  0\right)
=0_{f\left(  m\right)  }$ and $f\left(  \sigma_{v}\left(  t\right)  \right)
=\exp_{f\left(  m\right)  }\left(  \gamma\left(  t\right)  \right)  .$ It then
follows that
\[
\dot{\gamma}\left(  0\right)  =\left(  \exp_{f\left(  m\right)  }\right)
_{\ast}\dot{\gamma}\left(  0\right)  =f_{\ast}\dot{\sigma}_{v}\left(
0\right)  =f_{\ast}v\text{ }%
\]
and for $t$ sufficiently close to $0\in\mathbb{R},$ that%
\[
\left\vert \gamma\left(  t\right)  \right\vert =d\left(  f\left(  m\right)
,f\left(  \sigma_{v}\left(  t\right)  \right)  \right)  \leq\operatorname{Lip}%
\left(  f\right)  d\left(  m,\sigma_{v}\left(  t\right)  \right)
=\operatorname{Lip}\left(  f\right)  \left\vert v\right\vert \left\vert
t\right\vert =\operatorname{Lip}\left(  f\right)  \left\vert t\right\vert .
\]
Since
\[
\lim_{t\rightarrow0}\frac{1}{t}\gamma\left(  t\right)  =\lim_{t\rightarrow
0}\frac{1}{t}\left[  \gamma\left(  t\right)  -\gamma\left(  0\right)  \right]
=\dot{\gamma}\left(  0\right)  =f_{\ast}v
\]
we may conclude that
\[
\left\vert f_{\ast}v\right\vert =\lim_{t\rightarrow0}\left\vert \frac{1}%
{t}\gamma\left(  t\right)  \right\vert \leq\operatorname{Lip}\left(  f\right)
.
\]
As $v\in TM$ was arbitrary, it follows that $\left\vert f_{\ast}\right\vert
_{M}\leq\operatorname{Lip}\left(  f\right)  .$
\end{proof}

\begin{lemma}
\label{lem.2.10}If $X\in\Gamma\left(  TM\right)  $ satisfies, $\left\vert
\nabla X\right\vert _{M}<\infty,$ then
\begin{equation}
\left\vert \left\vert X\left(  p\right)  \right\vert -\left\vert X\left(
m\right)  \right\vert \right\vert \leq\left\vert \nabla X\right\vert _{M}\cdot
d\left(  p,m\right)  \text{ }\forall~m,p\in M, \label{e.2.5}%
\end{equation}
i.e. $\operatorname{Lip}\left(  \left\vert X\left(  \cdot\right)  \right\vert
\right)  \leq\left\vert \nabla X\right\vert _{M}.$
\end{lemma}

\begin{proof}
Let $\sigma\in C^{1}\left(  \left[  0,1\right]  ,M\right)  $ satisfy
$\sigma\left(  0\right)  =m$ and $\sigma\left(  1\right)  =p$ and define
$\xi\left(  t\right)  =X\left(  \sigma\left(  t\right)  \right)  $ and note
that
\[
\left\vert \frac{\nabla}{dt}\xi\left(  t\right)  \right\vert =\left\vert
\nabla_{\dot{\sigma}\left(  t\right)  }X\right\vert \leq\left\vert \nabla
X\right\vert _{M}\left\vert \dot{\sigma}\left(  t\right)  \right\vert .
\]
Therefore by Lemma \ref{lem.2.4},
\begin{equation}
\left\vert \left\vert X\left(  p\right)  \right\vert -\left\vert X\left(
m\right)  \right\vert \right\vert \leq\int_{0}^{1}\left\vert \frac{\nabla}%
{dt}\xi\left(  t\right)  \right\vert dt\leq\int_{0}^{1}\left\vert \nabla
X\right\vert _{M}\left\vert \dot{\sigma}\left(  t\right)  \right\vert
dt=\left\vert \nabla X\right\vert _{M}\cdot\ell\left(  \sigma\right)  .
\label{e.2.6}%
\end{equation}
Taking the infimum of the last term in this inequality over all paths joining
$m$ to $p$ gives Eq. (\ref{e.2.5}).
\end{proof}

\begin{theorem}
[Distance estimates]\label{thm.2.11}Suppose that $\left[  0,T\right]  \ni
t\rightarrow Y_{t}\in\Gamma\left(  TM\right)  $ smoothly varying time
dependent vector field and $\left(  a,b\right)  \subset\left[  0,T\right]  $
and $\sigma:\left(  a,b\right)  \rightarrow M$ solves
\[
\dot{\sigma}\left(  t\right)  =Y_{t}\left(  \sigma\left(  t\right)  \right)
\text{ for all }t\in\left(  a,b\right)
\]
where $\dot{\sigma}\left(  a\right)  $ and $\dot{\sigma}\left(  b\right)  $
are interpreted as appropriate one sided derivatives. Then for any
$s,t\in\left(  a,b\right)  ,$
\begin{align*}
d\left(  \sigma\left(  t\right)  ,\sigma\left(  s\right)  \right)   &
\leq\left\vert Y\right\vert _{J\left(  s,t\right)  }^{\ast}\leq\left\vert
Y\right\vert _{T}^{\ast}\text{ and}\\
d\left(  \sigma\left(  t\right)  ,\sigma\left(  s\right)  \right)   &  \leq
e^{\left\vert \nabla Y\right\vert _{J\left(  s,t\right)  }^{\ast}}\left\vert
Y_{\cdot}\left(  m\right)  \right\vert _{J\left(  s,t\right)  }^{\ast}\leq
e^{\left\vert \nabla Y\right\vert _{T}^{\ast}}\cdot\left\vert Y_{\cdot}\left(
m\right)  \right\vert _{T}^{\ast}.\text{ }%
\end{align*}

\end{theorem}

\begin{proof}
Without loss of generality we may assume that $s\leq t.$ Since $d\left(
\sigma\left(  t\right)  ,\sigma\left(  s\right)  \right)  $ is no more than
the length of $\sigma|_{\left[  s,t\right]  }$ we immediately find,%
\[
d\left(  \sigma\left(  t\right)  ,\sigma\left(  s\right)  \right)  \leq
\int_{s}^{t}\left\vert \dot{\sigma}\left(  \tau\right)  \right\vert d\tau
=\int_{s}^{t}\left\vert Y_{\tau}\left(  \sigma\left(  \tau\right)  \right)
\right\vert d\tau\left\vert Y\right\vert _{J\left(  s,t\right)  }^{\ast}%
\leq\left\vert Y\right\vert _{T}^{\ast}%
\]
which gives the first inequality. To prove the second inequality we use the
estimate in Eq. (\ref{e.2.6}) with $X=Y_{t}$ to find,
\begin{align}
\left\vert \dot{\sigma}\left(  t\right)  \right\vert  &  =\left\vert
Y_{t}\left(  \sigma\left(  t\right)  \right)  \right\vert \leq\left\vert
Y_{t}\left(  \sigma\left(  s\right)  \right)  \right\vert +\left\vert \nabla
Y_{t}\right\vert _{M}\ell\left(  \sigma|_{\left[  s,t\right]  }\right)
\nonumber\\
&  =\left\vert Y_{t}\left(  \sigma\left(  s\right)  \right)  \right\vert
+\left\vert \nabla Y_{t}\right\vert _{M}\int_{s}^{t}\left\vert \dot{\sigma
}\left(  r\right)  \right\vert dr. \label{e.2.7}%
\end{align}
If we define%
\[
\psi\left(  \tau\right)  :=\int_{s}^{s+\tau}\left\vert \dot{\sigma}\left(
r\right)  \right\vert dr\text{ for }0\leq\tau\leq b-s,
\]
then the inequality in Eq. (\ref{e.2.7}) may be rewritten as,%
\[
\dot{\psi}\left(  \tau\right)  =\left\vert \dot{\sigma}\left(  s+\tau\right)
\right\vert \leq\left\vert Y_{s+\tau}\left(  m\right)  \right\vert +\left\vert
\nabla Y_{s+\tau}\right\vert _{M}\psi\left(  \tau\right)  \text{ with }%
\psi\left(  0\right)  =0.
\]
By Gronwall's inequality (see Proposition \ref{pro.9.1}) and a simple change
of variables we find,%
\begin{align*}
\int_{s}^{s+\tau}\left\vert \dot{\sigma}\left(  r\right)  \right\vert
dr=\psi\left(  \tau\right)  \leq &  \int_{0}^{\tau}e^{\int_{r}^{\tau
}\left\vert \nabla Y_{s+\tau}\right\vert _{M}ds}\cdot\left\vert Y_{s+r}\left(
m\right)  \right\vert dr\\
&  =\int_{0}^{\tau}e^{\int_{s+r}^{s+\tau}\left\vert \nabla Y_{\sigma
}\right\vert _{M}d\sigma}\cdot\left\vert Y_{s+r}\left(  m\right)  \right\vert
dr.
\end{align*}
Choosing $\tau$ so that $s+\tau=t$ and making another translational change of
variables yields%
\begin{align*}
d\left(  \sigma\left(  t\right)  ,\sigma\left(  s\right)  \right)   &
\leq\int_{s}^{t}\left\vert \dot{\sigma}\left(  r\right)  \right\vert
dr\leq\int_{0}^{t-s}e^{\int_{s+r}^{t}\left\vert \nabla Y_{\sigma}\right\vert
_{M}d\sigma}\cdot\left\vert Y_{s+r}\left(  m\right)  \right\vert dr\\
&  =\int_{s}^{t}e^{\int_{\tau}^{t}\left\vert \nabla Y_{\sigma}\right\vert
_{M}ds}\cdot\left\vert Y_{\tau}\left(  m\right)  \right\vert d\tau\leq
e^{\left\vert \nabla Y\right\vert _{J\left(  s,t\right)  }^{\ast}}\left\vert
Y_{\cdot}\left(  m\right)  \right\vert _{J\left(  s,t\right)  }^{\ast}%
\end{align*}
which gives the second inequality.
\end{proof}

\begin{corollary}
\label{cor.2.12}If $\left(  M,g\right)  $ is a complete Riemannian manifold
and either $\left\vert Y\right\vert _{T}^{\ast}<\infty$ or $\left\vert \nabla
Y\right\vert _{T}^{\ast}<\infty,$ then $Y$ is complete.
\end{corollary}

\begin{proof}
Suppose that $s\in\left[  0,T\right]  $ and $m\in M$ are given and that
$\sigma:\left(  a,b\right)  \rightarrow M$ is a maximal solution to the ODE,%
\[
\dot{\sigma}\left(  t\right)  =Y_{t}\left(  \sigma\left(  t\right)  \right)
\text{ with }\sigma\left(  s\right)  =m.
\]
In order to handle both cases at once, let $R:=\left\vert Y\right\vert
_{T}^{\ast}$ or $R=e^{\left\vert \nabla Y\right\vert _{T}^{\ast}}\left\vert
Y_{\cdot}\left(  m\right)  \right\vert _{T}^{\ast}$ so that according to
Theorem \ref{thm.2.11}, $\sigma\left(  t\right)  \in K:=\overline{B\left(
0,R\right)  }$ for all $t\in\left(  a,b\right)  .$ Since $R<\infty$ and $M$ is
complete we know that $K$ is compact and hence%
\[
\left\vert \sigma^{\prime}\left(  t\right)  \right\vert =\left\vert
Y_{t}\left(  \sigma\left(  t\right)  \right)  \right\vert \leq C_{K}%
:=\max_{0\leq s\leq T~\&~m\in K}\left\vert Y_{s}\left(  m\right)  \right\vert
<\infty.
\]
Thus it follows that
\[
d\left(  \sigma\left(  t\right)  ,\sigma\left(  s\right)  \right)  \leq
C_{K}\left\vert t-s\right\vert \text{ for }s,t\in\left(  a,b\right)  .
\]
From this we conclude that $\,\lim_{t\uparrow b}\sigma\left(  t\right)  $
exists as $\left\{  \sigma\left(  t\right)  :t\uparrow\tau\right\}  $ is a
Cauchy and $\left(  M,g\right)  $ is complete and similarly, $\lim
_{t\downarrow a}\sigma\left(  t\right)  $ exists and we may extend $\sigma$ to
a continuous function on $\left[  a,b\right]  .$

We now claim that the one sided derivatives of $\sigma\left(  t\right)  $ at
$t=a$ and $t=b$ exist and are given by $Y_{a}\left(  \sigma\left(  a\right)
\right)  $ and $Y_{b}\left(  \sigma\left(  b\right)  \right)  $ respectively.
Indeed, if $\lim_{t\uparrow b}\sigma\left(  t\right)  =p=:\sigma\left(
b\right)  $ and $f\in C^{\infty}\left(  M\right)  ,$ then for $a<t<b$%
\begin{align*}
f\left(  \sigma\left(  b\right)  \right)  -f\left(  \sigma\left(  t\right)
\right)   &  =\lim_{\tau\uparrow b}f\left(  \sigma\left(  \tau\right)
\right)  -f\left(  \sigma\left(  t\right)  \right) \\
&  =\lim_{\tau\uparrow b}\int_{t}^{\tau}df\left(  \sigma^{\prime}\left(
r\right)  \right)  dr\\
&  =\lim_{\tau\uparrow b}\int_{t}^{\tau}\left(  Y_{r}f\right)  \left(
\sigma\left(  r\right)  \right)  dr=\int_{t}^{b}\left(  Y_{r}f\right)  \left(
\sigma\left(  r\right)  \right)  dr
\end{align*}
and hence%
\[
\lim_{t\uparrow b}\frac{f\left(  \sigma\left(  b\right)  \right)  -f\left(
\sigma\left(  t\right)  \right)  }{b-t}=\lim_{t\uparrow b}\frac{1}{b-t}%
\int_{t}^{b}\left(  Y_{r}f\right)  \left(  \sigma\left(  r\right)  \right)
dr=\left(  Y_{b}f\right)  \left(  \sigma\left(  b\right)  \right)  .
\]
Since this holds for all $f\in C^{\infty}\left(  M\right)  ,$ it follows that
$\sigma$ has a left derivative at $b$ given by $Y_{b}\left(  \sigma\left(
b\right)  \right)  .$ Similarly, one shows the right derivative of
$\sigma\left(  t\right)  $ exists at $t=a$ and is given by $Y_{a}\left(
\sigma\left(  a\right)  \right)  .$

To complete the proof, for the sake of contradiction, suppose that $b<T.$ By
local existence of ODEs we may find $\gamma:\left(  b-\varepsilon
,b+\varepsilon\right)  \rightarrow M$ such that
\[
\dot{\gamma}\left(  t\right)  =Y_{t}\left(  \gamma\left(  t\right)  \right)
\text{ with }\gamma\left(  b\right)  =p=\sigma\left(  b\right)  .
\]
The path
\[
\tilde{\sigma}\left(  t\right)  :=\left\{
\begin{array}
[c]{ccc}%
\sigma\left(  t\right)  & \text{if} & 0\leq t\leq b\\
\gamma\left(  t\right)  & \text{if} & b\leq t<b+\varepsilon
\end{array}
\right.
\]
then satisfies the ODE on a longer time interval which violates the maximality
of the solution and so we in fact must have $b=T.$ Similarly, one shows that
$a$ must be $0$ as well and hence $Y$ is complete.
\end{proof}

\begin{corollary}
[$\left\vert \nabla X\right\vert _{M}<\infty$ growth implications]%
\label{cor.2.13}If $X\in\Gamma\left(  TM\right)  $ is a complete time
independent vector field, then for all $m\in M,$%
\begin{align}
d\left(  e^{X}\left(  m\right)  ,m\right)   &  \leq\left\vert X\right\vert
_{M},\text{ and }\label{e.2.8}\\
d\left(  e^{X}\left(  m\right)  ,m\right)   &  \leq\left\vert X\left(
m\right)  \right\vert \cdot e^{\left\vert \nabla X\right\vert _{M}}.
\label{e.2.9}%
\end{align}

\end{corollary}

\begin{proof}
This follows immediately from Theorem \ref{thm.2.11} with $Y_{t}=X$ for all
$t$ and $\sigma\left(  t\right)  =e^{tX}\left(  m\right)  .$ The inequalities
in the theorem are applied with $t=1$ and $s=0.$
\end{proof}

\subsection{Flows\label{sec.2.2}}

The next theorem recalls some basic properties of flows associated to complete
time dependent vector fields.

\begin{theorem}
\label{thm.2.14}Suppose that $J=\left[  0,T\right]  \ni t\rightarrow Y_{t}%
\in\Gamma\left(  TM\right)  $ is a smoothly varying complete vector field on
$M$ and for fixed $s\in J$ and $m\in M,$ $J\ni t\rightarrow\mu_{t,s}\left(
m\right)  $ is the unique solution to the ODE,%
\[
\frac{d}{dt}\mu_{t,s}\left(  m\right)  =Y_{t}\left(  \mu_{t,s}\left(
m\right)  \right)  \text{ with }\mu_{s,s}\left(  m\right)  =m.
\]
Then;

\begin{enumerate}
\item $J\times J\times M\ni\left(  t,s,m\right)  \rightarrow\mu_{t,s}\left(
m\right)  \in M$ is smooth.

\item For all $r,s,t\in J,$%
\begin{equation}
\mu_{t,s}\circ\mu_{s,r}=\mu_{t,r}. \label{e.2.10}%
\end{equation}

\item For all $s,t\in J,$ $\mu_{t,s}\in\mathrm{Diff}\left(  M\right)  ,$
$\mu_{t,s}^{-1}=\mu_{s,t},$ the map $\left(  s,t,m\right)  \rightarrow
\mu_{t,s}^{-1}\left(  m\right)  $ is smooth and
\begin{equation}
\frac{d}{dt}\mu_{t,s}^{-1}=\dot{\mu}_{s,t}=-\left(  \mu_{s,t}\right)  _{\ast
}Y_{t}=-\left(  \mu_{t,s}^{-1}\right)  _{\ast}Y_{t}. \label{e.2.11}%
\end{equation}

\item For all $s,t\in J,$
\begin{equation}
\mu_{t,s}=\mu_{t,0}\circ\mu_{s,0}^{-1}. \label{e.2.12}%
\end{equation}

\end{enumerate}
\end{theorem}

\begin{proof}
We take each item in turn.

\begin{enumerate}
\item The smoothness of $\left(  t,s,m\right)  \rightarrow\mu_{t,s}\left(
m\right)  $ is a basic consequence of the fact that ODE's depend smoothly on
parameters and initial conditions. For example $\sigma\left(  \tau\right)
=\mu_{T\left(  \tau\right)  ,s}\left(  m\right)  $ and $T\left(  \tau\right)
=\tau+s,$ then
\[
\frac{d}{d\tau}\left(
\begin{array}
[c]{c}%
\sigma\left(  \tau\right) \\
T\left(  \tau\right)
\end{array}
\right)  =\left(
\begin{array}
[c]{c}%
Y_{\tau}\left(  \sigma\left(  \tau\right)  \right) \\
1
\end{array}
\right)  \text{ with }\left(
\begin{array}
[c]{c}%
\sigma\left(  0\right) \\
T\left(  0\right)
\end{array}
\right)  =\left(
\begin{array}
[c]{c}%
m\\
T\left(  0\right)  =s
\end{array}
\right)
\]
and hence $\sigma\left(  \tau,m,s\right)  =\mu_{\tau+s,s}\left(  m\right)  $
depends smoothly on $\left(  \tau,m,s\right)  $ and hence $\mu$ depends
smoothly on all of its variables.

\item To prove Eq. (\ref{e.2.10}) simply notice that $t\rightarrow\mu
_{t,s}\circ\mu_{s,r}$ and $t\rightarrow\mu_{t,r}$ both satisfy the
differential equation,%
\[
\frac{d}{dt}\nu_{t}=Y_{t}\circ\nu_{t}\text{ with }\nu_{s}=\mu_{s,r}%
\]

\item Taking $r=t$ in Eq. (\ref{e.2.10}) gives,%
\begin{equation}
\mu_{t,s}\circ\mu_{s,t}=\mu_{t,t}=Id_{M}\text{ for all }s,t\in J.
\label{e.2.13}%
\end{equation}
By interchanging $s$ and $t$ in the above equation may also be written as
\begin{equation}
\mu_{s,t}\circ\mu_{t,s}=Id_{M}\text{ for all }s,t\in J. \label{e.2.14}%
\end{equation}
From these last two equations we see that $\mu_{t,s}\in\mathrm{Diff}\left(
M\right)  $ for all $s,t\in J$ and moreover that $\mu_{t,s}^{-1}=\mu_{s,t}$
which also shows the map $\left(  s,t,m\right)  \rightarrow\mu_{t,s}%
^{-1}\left(  m\right)  $ is smooth. To prove Eq. (\ref{e.2.11}) we
differentiate Eq. (\ref{e.2.13}) with respect to $t$ to find,%
\begin{align*}
0  &  =\dot{\mu}_{t,s}\circ\mu_{s,t}+\left(  \mu_{t,s}\right)  _{\ast}\dot
{\mu}_{s,t}=Y_{t}\circ\mu_{t,s}\circ\mu_{s,t}+\left(  \mu_{t,s}\right)
_{\ast}\dot{\mu}_{s,t}\\
&  =Y_{t}+\left(  \mu_{t,s}\right)  _{\ast}\dot{\mu}_{s,t}%
\end{align*}
and hence
\[
\dot{\mu}_{s,t}=-\left(  \mu_{t,s}\right)  _{\ast}^{-1}Y_{t}=-\left(
\mu_{s,t}\right)  _{\ast}Y_{t}.
\]

\item This one is easily deduced by what has already been proved;%
\[
\mu_{t,s}=\mu_{t,0}\circ\mu_{0,s}=\mu_{t,0}\circ\mu_{s,0}^{-1}.
\]

\end{enumerate}
\end{proof}

\subsection{$\mathrm{Diff}\left(  M\right)  $-Adjoint Action\label{sec.2.3}}

\begin{definition}
[Adjoint actions and Lie derivatives]\label{def.2.16}If $f\in\mathrm{Diff}%
\left(  M\right)  $ and $Y\in\Gamma\left(  TM\right)  ,$ let
$\operatorname{Ad}_{f}Y=f_{\ast}Y\circ f^{-1}\in\Gamma\left(  TM\right)  ,$
i.e. $\operatorname{Ad}_{f}Y$ is the vector field defined by%
\begin{equation}
\left(  \operatorname{Ad}_{f}Y\right)  \left(  m\right)  =f_{\ast}Y\left(
f^{-1}\left(  m\right)  \right)  \in T_{m}M\text{ }\forall~m\in M.
\label{e.2.17}%
\end{equation}
If $X,Y\in\Gamma\left(  TM\right)  ,$ then the \textbf{Lie derivative }of $Y$
with respect to $X$ is
\begin{equation}
L_{X}Y:=\frac{d}{dt}|_{0}\operatorname{Ad}_{e^{-tX}}Y=\frac{d}{dt}|_{0}%
e_{\ast}^{-tX}Y\circ e^{tX}. \label{e.2.18}%
\end{equation}
We further let
\begin{equation}
\operatorname{ad}_{X}:=\frac{d}{dt}|_{0}\operatorname{Ad}_{e^{tX}}=-L_{X}.
\label{e.2.19}%
\end{equation}

\end{definition}

\begin{remark}
\label{rem.2.17}The following identities are well known and easy to prove.

\begin{enumerate}
\item If $f,g\in\mathrm{Diff}\left(  M\right)  $ then $\operatorname{Ad}%
_{f\circ g}=\operatorname{Ad}_{f}\operatorname{Ad}_{g}.$

\item The Lie derivative, $L_{X}Y,$ is again a vector field on $M$ which may
also be computed using
\[
L_{X}Y=\left[  X,Y\right]  =XY-YX.
\]

\item If $X,Y\in\Gamma\left(  TM\right)  $ and $f\in\mathrm{Diff}\left(
M\right)  ,$ then%
\begin{equation}
\operatorname{Ad}_{f}\left[  X\,Y\right]  =\left[  \operatorname{Ad}%
_{f}X,\operatorname{Ad}_{f}Y\right]  . \label{e.2.20}%
\end{equation}

\end{enumerate}

For example, to verify Eq. (\ref{e.2.20}), observe that $\operatorname{Ad}%
_{f}Y$ is the unique vector field on $M$ such that
\begin{equation}
f_{\ast}Y=\left(  \operatorname{Ad}_{f}Y\right)  \circ f, \label{e.2.21}%
\end{equation}
i.e. such that $Y$ and $\operatorname{Ad}_{f}Y$ are \textquotedblleft%
$f$-related.\textquotedblright\ Since the commutator of two $f$-related vector
fields are $f$-related, it follows that $\left[  X,Y\right]  $ and $\left[
\operatorname{Ad}_{f}X,\operatorname{Ad}_{f}Y\right]  $ are $f$-related, i.e.
\[
f_{\ast}\left[  X,Y\right]  =\left[  \operatorname{Ad}_{f}X,\operatorname{Ad}%
_{f}Y\right]  \circ f\implies\operatorname{Ad}_{f}\left[  X\,Y\right]
=f_{\ast}\left[  X,Y\right]  \circ f^{-1}=\left[  \operatorname{Ad}%
_{f}X,\operatorname{Ad}_{f}Y\right]  .
\]

\end{remark}

\begin{proposition}
[Adjoint flow equations]\label{pro.2.18}Suppose $Y\in\Gamma\left(  TM\right)
$ and $\nu_{t}\in\mathrm{Diff}\left(  M\right)  $ is smoothly varying in $t.$
If we define%
\[
W_{t}:=\dot{\nu}_{t}\circ\nu_{t}^{-1}\in\Gamma\left(  TM\right)  \text{ and
}\tilde{W}_{t}=\left(  \nu_{t}^{-1}\right)  _{\ast}\dot{\nu}_{t}\in
\Gamma\left(  TM\right)  ,
\]
then the \textbf{adjoint flows} of $Y$, $\operatorname{Ad}_{\nu_{t}}Y$ and
$\operatorname{Ad}_{\nu_{t}^{-1}}Y,$ satisfy%
\begin{equation}
\frac{d}{dt}\operatorname{Ad}_{\nu_{t}}Y=\left[  \operatorname{Ad}_{\nu_{t}%
}Y,W_{t}\right]  =\operatorname{Ad}_{\nu_{t}}\left[  Y,\tilde{W}_{t}\right]
\label{e.2.22}%
\end{equation}
and%
\begin{equation}
\frac{d}{dt}\operatorname{Ad}_{\nu_{t}^{-1}}Y=\left[  \tilde{W}_{t}%
,\operatorname{Ad}_{\nu_{t}^{-1}}Y\right]  =\operatorname{Ad}_{\nu_{t}^{-1}%
}\left[  W_{t},Y\right]  . \label{e.2.23}%
\end{equation}

\end{proposition}

\begin{proof}
Let $Y_{t}:=\operatorname{Ad}_{\nu_{t}}Y$ for $t\in\mathbb{R}$ (as in Eq.
(\ref{e.2.21})) so that,
\[
Y_{t}\circ\nu_{t}=\left(  \operatorname{Ad}_{\nu_{t}}Y\right)  \circ\nu
_{t}=\nu_{t\ast}Y.
\]
Hence, if $\varphi\in C^{\infty}\left(  M,\mathbb{R}\right)  ,$ then
\[
\left(  Y_{t}\varphi\right)  \circ\nu_{t}=\left(  \nu_{t\ast}Y_{t}\right)
\varphi=Y\left(  \varphi\circ\nu_{t}\right)
\]
and differentiating this equation in $t$ gives%
\[
\left(  \dot{Y}_{t}\varphi\right)  \circ\nu_{t}+\left(  W_{t}Y_{t}%
\varphi\right)  \circ\nu_{t}=Y\left(  \left(  W_{t}\varphi\right)  \circ
\nu_{t}\right)  =\left(  \nu_{t\ast}Y\right)  \left(  W_{t}\varphi\right)
=\left(  Y_{t}W_{t}\varphi\right)  \circ\nu_{t}.
\]
The last equation is equivalent to the first equality in Eq. (\ref{e.2.22}).
Since $\operatorname{Ad}_{\nu_{t}}\tilde{W}_{t}=W_{t},$%
\[
\left[  Y_{t},W_{t}\right]  =\left[  \operatorname{Ad}_{\nu_{t}}%
Y,\operatorname{Ad}_{\nu_{t}}\tilde{W}_{t}\right]  =\operatorname{Ad}_{\nu
_{t}}\left[  Y,\tilde{W}_{t}\right]
\]
which gives the second equality in Eq. (\ref{e.2.22}).

As, by Theorem \ref{thm.2.14},%
\[
\frac{d}{dt}\nu_{t}^{-1}=-\nu_{t\ast}^{-1}W_{t}=-\nu_{t\ast}^{-1}W_{t}\circ
\nu_{t}\circ\nu_{t}^{-1}=-\tilde{W}_{t}\circ\nu_{t}^{-1},
\]
it follows from Eq. (\ref{e.2.22}) that%
\begin{align*}
\frac{d}{dt}\left(  \operatorname{Ad}_{\nu_{t}^{-1}}Y\right)   &  =\left[
\operatorname{Ad}_{\nu_{t}^{-1}}Y,-\tilde{W}_{t}\right]  =\left[  \tilde
{W}_{t},\operatorname{Ad}_{\nu_{t}^{-1}}Y\right]  \text{ and}\\
\frac{d}{dt}\left(  \operatorname{Ad}_{\nu_{t}^{-1}}Y\right)   &
=\operatorname{Ad}_{\nu_{t}^{-1}}\left[  Y,-W_{t}\right]  =\operatorname{Ad}%
_{\nu_{t}^{-1}}\left[  W_{t},Y\right]  ,
\end{align*}
which proves both equalities Eq. (\ref{e.2.23}).
\end{proof}

\begin{corollary}
\label{cor.2.19}If $Y,X\in\Gamma\left(  TM\right)  $ with $X$ being complete,
then
\[
\frac{d}{dt}\operatorname{Ad}_{e^{tX}}Y=\operatorname{Ad}_{e^{tX}%
}\operatorname{ad}_{X}Y=\operatorname{ad}_{X}\operatorname{Ad}_{e^{tX}}Y
\]
where $\operatorname{ad}_{X}=-L_{X}.$
\end{corollary}

\begin{proof}
The result follows directly from Eq. (\ref{e.2.22}) by taking $\nu_{t}=e^{tX}$
and noting that $W_{t}=X=\tilde{W}_{t}$ in this case as $\operatorname{Ad}%
_{e^{tX}}X=X$ for all $t\in\mathbb{R}.$
\end{proof}

\subsection{Vector field differentiation of flows\label{sec.2.4}}

\begin{definition}
[Differentiating $\mu^{X}$ in $X$]\label{def.2.20}Let $X_{t},Y_{t}\in
\Gamma\left(  TM\right)  $ be smoothly varying time dependent vector fields on
$M.$ We say $\mu^{X}$ \textbf{is differentiable relative to} $Y$ if there
exists $\left\{  X_{t}^{\varepsilon}\right\}  _{\varepsilon,t}\subset
\Gamma\left(  TM\right)  $ such that $\left(  t,\varepsilon,m\right)
\rightarrow X_{t}^{\varepsilon}\left(  m\right)  \in TM$ is smooth and
$X_{\cdot}^{\varepsilon}$ is complete for $\varepsilon$ near $0,$ $X_{t}%
^{0}=X_{t},$ and $\frac{d}{d\varepsilon}|_{0}X_{t}^{\varepsilon}=Y_{t}.$ If
all of this holds we let
\begin{equation}
\partial_{Y}\mu_{t,s}^{X}:=\frac{d}{d\varepsilon}|_{0}\mu_{t,s}%
^{X^{\varepsilon}}. \label{e.2.24}%
\end{equation}

\end{definition}

\begin{theorem}
\label{thm.2.21}If $X_{t},Y_{t}\in\Gamma\left(  TM\right)  $ are as in
Definition \ref{def.2.20} so that $\partial_{Y}\mu_{t,s}^{X}=\frac
{d}{d\varepsilon}|_{0}\mu_{t,s}^{X^{\varepsilon}}$ exists, then%
\begin{align}
\partial_{Y}\mu_{t,s}^{X}  &  =\int_{s}^{t}\mu_{t,\tau\ast}^{X}\left[
Y_{\tau}\circ\mu_{\tau,s}^{X}\right]  d\tau\label{e.2.25}\\
&  =\left(  \int_{s}^{t}\operatorname{Ad}_{\mu_{t,\tau}^{X}}Y_{\tau}%
d\tau\right)  \circ\mu_{t,s}^{X}\label{e.2.26}\\
&  =\left(  \mu_{t,s}^{X}\right)  _{\ast}\int_{s}^{t}\operatorname{Ad}%
_{\mu_{s,\tau}^{X}}Y_{\tau}d\tau. \label{e.2.27}%
\end{align}

\end{theorem}

\begin{proof}
Let $V_{t,s}:=\left(  \mu_{s,t}^{X}\right)  _{\ast}\partial_{Y}\mu_{t,s}%
^{X}\in\Gamma\left(  TM\right)  $ so that
\begin{equation}
\partial_{Y}\mu_{t,s}^{X}=\frac{d}{d\varepsilon}|_{0}\mu_{t,s}^{X^{\varepsilon
}}=\left(  \mu_{t,s}^{X}\right)  _{\ast}V_{t,s} \label{e.2.28}%
\end{equation}
Notice that Eq. (\ref{e.2.28}) is equivalent to, for all $f\in C^{\infty
}\left(  M\right)  $,
\begin{align}
\frac{d}{d\varepsilon}|_{0}\left[  f\circ\mu_{t,s}^{X^{\varepsilon}}\right]
&  =df\left(  \frac{d}{d\varepsilon}|_{0}\mu_{t,s}^{X^{\varepsilon}}\right)
=df\left(  \partial_{Y}\mu_{t,s}^{X}\right) \nonumber\\
&  =df\left(  \left(  \mu_{t,s}^{X}\right)  _{\ast}V_{t,s}\right)
=V_{t,s}\left[  f\circ\mu_{t,s}^{X}\right]  . \label{e.2.29}%
\end{align}
So, on one hand,%
\begin{align}
\frac{d}{d\varepsilon}|_{0}\frac{d}{dt}\left[  f\circ\mu_{t,s}^{X^{\varepsilon
}}\right]   &  =\frac{d}{dt}\frac{d}{d\varepsilon}|_{0}\left[  f\circ\mu
_{t,s}^{X^{\varepsilon}}\right]  =\frac{d}{dt}V_{t,s}\left[  f\circ\mu
_{t,s}^{X}\right] \nonumber\\
&  =\dot{V}_{t,s}\left[  f\circ\mu_{t,s}^{X}\right]  +V_{t,s}\left[
X_{t}f\circ\mu_{t,s}^{X}\right]  . \label{e.2.30}%
\end{align}
On the other hand,
\[
\frac{d}{dt}\left[  f\circ\mu_{t,s}^{X^{\varepsilon}}\right]  =\left(
X_{t}^{\varepsilon}f\right)  \circ\mu_{t,s}^{X^{\varepsilon}}%
\]
and differentiating this equation in $\varepsilon$ implies while using Eq.
(\ref{e.2.29}) with $f$ replaced by $X_{t}f$ implies,%
\begin{align}
\frac{d}{d\varepsilon}|_{0}\frac{d}{dt}\left[  f\circ\mu_{t,s}^{X^{\varepsilon
}}\right]   &  =Y_{t}f\circ\mu_{t,s}^{X}+\frac{d}{d\varepsilon}|_{0}\left[
\left(  X_{t}f\right)  \circ\mu_{t,s}^{X^{\varepsilon}}\right] \nonumber\\
&  =Y_{t}f\circ\mu_{t,s}^{X}+V_{t,s}\left[  X_{t}f\circ\mu_{t,s}^{X}\right]  .
\label{e.2.31}%
\end{align}
Comparing Eqs. (\ref{e.2.30}) and (\ref{e.2.31}) shows,%
\[
\left(  \mu_{t,s\ast}^{X}\dot{V}_{t,s}\right)  f=\dot{V}_{t,s}\left[
f\circ\mu_{t,s}^{X}\right]  =Y_{t}f\circ\mu_{t,s}^{X}=\left(  Y_{t}\circ
\mu_{t,s}^{X}\right)  f\text{ ~}\forall~f\in C^{\infty}\left(  M\right)
\]
which implies,
\begin{equation}
\dot{V}_{t,s}=\mu_{s,t\ast}^{X}Y_{t}\circ\mu_{t,s}^{X}. \label{e.2.32}%
\end{equation}
Since $\mu_{s,s}^{X}=Id_{M}$ we know that $\partial_{Y}\mu_{s,s}^{X}=0$ and
hence $V_{s,s}=0$ and so integrating Eq. (\ref{e.2.32}) implies,
\[
V_{t,s}=\int_{s}^{t}\mu_{s,\tau\ast}^{X}Y_{t}\circ\mu_{\tau,s}^{X}d\tau
=\int_{s}^{t}\operatorname{Ad}_{\mu_{s,\tau}^{X}}Y_{\tau}d\tau.
\]
This equality along with Eq. (\ref{e.2.28}) proves Eq. (\ref{e.2.27}). The
proofs of Eqs. (\ref{e.2.25}) and (\ref{e.2.26}) now easily follows since
\begin{align*}
\left(  \mu_{t,s}^{X}\right)  _{\ast}V_{t,s}  &  =\int_{s}^{t}\left(
\mu_{t,s}^{X}\right)  _{\ast}\mu_{s,\tau\ast}^{X}Y_{\tau}\circ\mu_{\tau,s}%
^{X}d\tau\\
&  =\int_{s}^{t}\mu_{t,\tau\ast}^{X}Y_{\tau}\circ\mu_{\tau,s}^{X}d\tau\\
&  =\int_{s}^{t}\mu_{t,\tau\ast}^{X}Y_{\tau}\circ\mu_{\tau,t}^{X}\circ
\mu_{t,\tau}^{X}\circ\mu_{\tau,s}^{X}d\tau=\left(  \int_{s}^{t}%
\operatorname{Ad}_{\mu_{t,\tau}^{X}}Y_{\tau}d\tau\right)  \circ\mu_{t,s}^{X}.
\end{align*}

\end{proof}

The following theorem is an important special case of Theorem \ref{thm.2.21}.

\begin{theorem}
[Differential of $e^{tX}$ in $X$]\label{thm.2.22}Suppose that $M$ is a smooth
manifold and for each $\sigma\in\mathbb{R},$ $\left\{  X^{\varepsilon
}\right\}  \subset\Gamma\left(  TM\right)  $ is a smooth varying one parameter
family of complete vector fields on $M$ and let $X:=X^{0}$ and $Y:=\frac
{d}{d\varepsilon}|_{0}X^{\varepsilon}.$ Then%
\begin{align}
\partial_{Y}e^{tX}  &  =\frac{d}{d\varepsilon}|_{0}e^{tX^{\varepsilon}%
}=e_{\ast}^{tX}\int_{0}^{t}e_{\ast}^{-\tau X}Y\circ e^{\tau X}d\tau
\label{e.2.33}\\
&  =\int_{0}^{t}e_{\ast}^{\left(  t-\tau\right)  X}Y\circ e^{\tau X}%
d\tau\label{e.2.34}\\
&  =\left[  \int_{0}^{t}\operatorname{Ad}_{e^{\tau X}}Yd\tau\right]  \circ
e^{tX}. \label{e.2.35}%
\end{align}

\end{theorem}

\begin{notation}
\label{not.2.23}To each smooth path, $t\rightarrow Z_{t}\in\Gamma\left(
TM\right)  ,$ of complete vector fields, let
\begin{equation}
W_{t}^{Z}:=\int_{0}^{1}e_{\ast}^{sZ_{t}}\dot{Z}_{t}\circ e^{-sZ_{t}}%
ds=\int_{0}^{1}\operatorname{Ad}_{e^{sZ_{t}}}\dot{Z}_{t}~ds \label{e.2.39}%
\end{equation}

\end{notation}

\begin{corollary}
\label{cor.2.24}If $t\rightarrow Z_{t}\in\Gamma\left(  TM\right)  $ is a
smooth path of complete vector fields, then
\begin{equation}
\frac{d}{dt}e^{Z_{t}}=W_{t}^{Z}\circ e^{Z_{t}}. \label{e.2.40}%
\end{equation}

\end{corollary}

\begin{proof}
Theorem \ref{thm.2.22} with $t=1$ and $X_{s}=Z_{t+s},$ gives%
\begin{equation}
\frac{d}{dt}e^{Z_{t}}=\left[  \int_{0}^{1}\operatorname{Ad}_{e^{\tau X_{0}}%
}X_{0}^{\prime}d\tau\right]  \circ e^{X_{0}}=\left[  \int_{0}^{1}%
\operatorname{Ad}_{e^{\tau Z_{t}}}\dot{Z}_{t}d\tau\right]  \circ e^{Z_{t}%
}.\nonumber
\end{equation}

\end{proof}

\iffalse \textbf{Note: }stuff that used to be here has been moved to Section
\ref{sec.50.2}. \fi

\subsection{Jacobian formulas and estimates for flows\label{sec.2.5}}

\begin{notation}
\label{not.2.25}Let $\mathbb{R}\ni t\rightarrow W_{t}\in\Gamma\left(
TM\right)  $ be a smoothly varying time dependent vector field and suppose
that $\mathbb{R}\times M\ni\left(  t,m\right)  \rightarrow\nu_{t}\left(
m\right)  \in M$ is in $C^{\infty}\left(  \mathbb{R}\times M,M\right)  $
satisfies the ordinary differential equation,
\begin{equation}
\dot{\nu}_{t}=W_{t}\circ\nu_{t}. \label{e.2.43}%
\end{equation}

\end{notation}

Notice that if $W_{\left(  \cdot\right)  }$ is complete, then $\nu_{t}%
=\mu_{t,s}^{W}\circ\nu_{s}$ for any $s\in\mathbb{R}.$ The general goal of this
section is to find estimates on $\nu_{t},$ $\nu_{t\ast},$ and $\nabla
\nu_{t\ast}$ (see Definition \ref{def.5.26} below) expressed in terms of the
geometry of $M$ and $W_{t}.$ The next key proposition records the ordinary
differential equation satisfied by $\nu_{t\ast}.$

\begin{proposition}
\label{pro.2.26}If $W_{t}\in\Gamma\left(  TM\right)  $ and $\nu_{t}\in
C^{\infty}\left(  M,M\right)  $ are as in Notation \ref{not.2.25}, then
\begin{equation}
\frac{\nabla}{dt}\nu_{t\ast}v=\nabla_{\nu_{t\ast}v}W_{t}\text{ }\forall~v\in
TM. \label{e.2.44}%
\end{equation}

\end{proposition}

\begin{proof}
If $\sigma\left(  s\right)  $ is a curve in $M$ so that $\sigma^{\prime
}\left(  0\right)  =\frac{d}{ds}|_{0}\sigma\left(  s\right)  =v,$ then%
\begin{align*}
\frac{\nabla}{dt}\nu_{t\ast}v  &  =\frac{\nabla}{dt}\frac{d}{ds}|_{0}\nu
_{t}\left(  \sigma\left(  s\right)  \right)  =\frac{\nabla}{ds}|_{0}\frac
{d}{dt}\nu_{t}\left(  \sigma\left(  s\right)  \right) \\
&  =\frac{\nabla}{ds}|_{0}W_{t}\left(  \nu_{t}\left(  \sigma\left(  s\right)
\right)  \right)  =\nabla_{\nu_{t\ast}v}W_{t}.
\end{align*}

\end{proof}

\begin{corollary}
\label{cor.2.27}If $W_{t}\in\Gamma\left(  TM\right)  $ and $\nu_{t}\in
C^{\infty}\left(  M,M\right)  $ are as in Notation \ref{not.2.25}, then
\begin{equation}
\left\vert \nu_{t\ast}\right\vert _{m}\leq\left\vert \nu_{s\ast}\right\vert
_{m}\cdot e^{\int_{J\left(  s,t\right)  }\left\vert \nabla W_{\tau}\right\vert
_{\nu_{\tau}\left(  m\right)  }d\tau}\leq\left\vert \nu_{s\ast}\right\vert
_{m}\cdot e^{\left\vert \nabla W_{\cdot}\right\vert _{J}^{\ast}}
\label{e.2.45}%
\end{equation}
and in particular,%
\begin{equation}
\operatorname{Lip}\left(  \nu_{t}\right)  =\left\vert \nu_{t\ast}\right\vert
_{M}\leq\left\vert \nu_{s\ast}\right\vert _{M}\cdot e^{\left\vert \nabla
W_{\cdot}\right\vert _{J\left(  s,t\right)  }^{\ast}}. \label{e.2.46}%
\end{equation}
We also have the following time derivative estimates,%
\begin{equation}
\left\vert \frac{\nabla}{dt}\nu_{t\ast}\right\vert _{m}\leq\left\vert
\nu_{s\ast}\right\vert _{m}\left\vert \nabla W_{t}\right\vert _{\nu_{t}\left(
m\right)  }\cdot e^{\int_{J\left(  s,t\right)  }\left\vert \nabla W_{\tau
}\right\vert _{\nu_{\tau}\left(  m\right)  }d\tau} \label{e.2.47}%
\end{equation}
and
\begin{equation}
\left\vert \frac{\nabla}{dt}\nu_{t\ast}\right\vert _{M}\leq\left\vert
\nu_{s\ast}\right\vert _{M}\left\vert \nabla W_{t}\right\vert _{M}\cdot
e^{\left\vert \nabla W_{\cdot}\right\vert _{J}^{\ast}}. \label{e.2.48}%
\end{equation}

\end{corollary}

\begin{proof}
If we define $H_{t}v:=\nabla_{v}W_{t}$ for all $v\in TM,$ then Proposition
\ref{pro.2.26} states, for any $v\in TM,$ that%
\begin{equation}
\frac{\nabla}{dt}\nu_{t\ast}v=H_{t}\nu_{t\ast}v. \label{e.2.49}%
\end{equation}
Therefore by the geometric Bellman-Gronwall's inequality in Corollary
\ref{cor.9.3} (with $G\equiv0)$,%
\[
\left\vert \nu_{t\ast}v\right\vert \leq e^{\int_{J\left(  s,t\right)
}\left\Vert H_{\tau}\right\Vert _{op}d\tau}\left\vert \nu_{s\ast}v\right\vert
=e^{\int_{J\left(  s,t\right)  }\left\vert \nabla W_{\tau}\right\vert
_{\nu_{\tau}\left(  m\right)  }d\tau}\left\vert \nu_{s\ast}v\right\vert
\]
which proves Eqs. (\ref{e.2.45}) and (\ref{e.2.46}). By Eq. (\ref{e.2.49}),%
\begin{align*}
\left\vert \frac{\nabla}{dt}\nu_{t\ast}v\right\vert  &  \leq\left\Vert
H_{t}\right\Vert _{op}\cdot\left\vert \nu_{t\ast}v\right\vert =\left\vert
\nabla W_{t}\right\vert _{\nu_{t}\left(  m\right)  }\cdot\left\vert \nu
_{t\ast}v\right\vert \\
&  \leq\left\vert \nabla W_{t}\right\vert _{\nu_{t}\left(  m\right)
}\left\vert \nu_{s\ast}v\right\vert \cdot e^{\int_{J\left(  s,t\right)
}\left\vert \nabla W_{\tau}\right\vert _{\nu_{\tau}\left(  m\right)  }d\tau
}\leq\left\vert \nabla W_{t}\right\vert _{\nu_{t}\left(  m\right)  }\left\vert
\nu_{s\ast}v\right\vert \cdot e^{\left\vert \nabla W_{\cdot}\right\vert
_{J\left(  s,t\right)  }^{\ast}}.
\end{align*}
Taking the supremum of this inequality over $v\in T_{m}M$ with $\left\vert
v\right\vert =1$ gives the estimate in Eq. (\ref{e.2.47}) which then easily
implies Eq. (\ref{e.2.48}).
\end{proof}

The following corollary records the results in Proposition \ref{pro.2.26} and
Corollary \ref{cor.2.27} when $W_{t}=X\in\Gamma\left(  TM\right)  $ is a
complete vector field and $\nu_{t}=e^{tX}.$

\begin{corollary}
\label{cor.2.28}If $X\in\Gamma\left(  TM\right)  $ is a complete vector field
and $t\in\mathbb{R}$, then%
\begin{equation}
\frac{\nabla}{dt}e_{\ast}^{tX}v_{m}=\nabla_{e_{\ast}^{tX}v_{m}}X\text{ with
}e_{\ast}^{0X}v_{m}=v_{m}, \label{e.2.52}%
\end{equation}%
\begin{align*}
\left\vert e_{\ast}^{tX}\right\vert _{m}  &  \leq e^{\int_{J\left(
0,t\right)  }\left\vert \nabla X\right\vert _{e^{\tau X}\left(  m\right)
}d\tau}\leq e^{\left\vert t\right\vert \left\vert \nabla X\right\vert _{M}},\\
\left\vert \frac{\nabla}{dt}e_{\ast}^{tX}\right\vert _{m}  &  \leq\left\vert
\nabla X\right\vert _{e^{tX}\left(  m\right)  }\cdot e^{\int_{J\left(
0,t\right)  }\left\vert \nabla X\right\vert _{e^{\tau X}\left(  m\right)
}d\tau},\\
\operatorname{Lip}\left(  e^{tX}\right)   &  =\left\vert e_{\ast}%
^{tX}\right\vert _{M}\leq e^{\left\vert t\right\vert \left\vert \nabla
X\right\vert _{M}},\text{ and}\\
\left\vert \frac{\nabla}{dt}e_{\ast}^{tX}\right\vert _{M}  &  \leq\left\vert
\nabla X\right\vert _{M}\cdot e^{\left\vert t\right\vert \left\vert \nabla
X\right\vert _{M}}.
\end{align*}

\end{corollary}

\begin{corollary}
\label{cor.2.29}If $\mathbb{R}\ni t\rightarrow W_{t}\in\Gamma\left(
TM\right)  $ is a complete time dependent vector field and $Z\in\Gamma\left(
TM\right)  ,$ then
\begin{align}
\left\vert Ad_{\mu_{t,s}^{W}}Z\right\vert _{m}  &  \leq e^{\int_{J\left(
s,t\right)  }\left\vert \nabla W_{\tau}\right\vert _{\mu_{\tau,t}^{W}\left(
m\right)  }d\tau}\cdot\left\vert Z\right\vert _{\mu_{s,t}^{W}\left(  m\right)
}\nonumber\\
&  \leq e^{\int_{J\left(  s,t\right)  }\left\vert \nabla W_{\tau}\right\vert
_{M}d\tau}\cdot\left\vert Z\right\vert _{M}. \label{e.2.53}%
\end{align}
As a special case, if $X\in\Gamma\left(  TM\right)  $ is complete, then
\begin{equation}
\left\vert Ad_{e^{X}}Z\right\vert _{m}\leq e^{\int_{0}^{1}\left\vert \nabla
X\right\vert _{e^{-\tau X}\left(  m\right)  }d\tau}\cdot\left\vert
Z\right\vert _{e^{-X}\left(  m\right)  }\leq e^{\left\vert \nabla X\right\vert
_{M}}\cdot\left\vert Z\right\vert _{M}. \label{e.2.54}%
\end{equation}

\end{corollary}

\begin{proof}
Let $\nu_{t}:=\mu_{t,s}^{W},$ then $\nu_{t}^{-1}=\mu_{s,t}^{W},$ $\nu
_{s}=Id_{M},$ and
\begin{align*}
\left\vert \left(  Ad_{\mu_{t,s}^{W}}Z\right)  \left(  m\right)  \right\vert
=  &  \left\vert \left(  \nu_{t}\right)  _{\ast}Z\left(  \nu_{t}^{-1}\left(
m\right)  \right)  \right\vert \leq\left\vert \left(  \nu_{t}\right)  _{\ast
}\right\vert _{\nu_{t}^{-1}\left(  m\right)  }\cdot\left\vert Z\right\vert
_{\nu_{t}^{-1}\left(  m\right)  }\\
\leq &  \left\vert \nu_{s\ast}\right\vert _{\nu_{t}^{-1}\left(  m\right)
}\cdot e^{\int_{J\left(  s,t\right)  }\left\vert \nabla W_{\tau}\right\vert
_{\nu_{\tau}\left(  \nu_{t}^{-1}\left(  m\right)  \right)  }d\tau}%
\cdot\left\vert Z\right\vert _{\nu_{t}^{-1}\left(  m\right)  }\\
&  =e^{\int_{J\left(  s,t\right)  }\left\vert \nabla W_{\tau}\right\vert
_{\mu_{\tau,t}^{W}\left(  m\right)  }d\tau}\cdot\left\vert Z\right\vert
_{\mu_{s,t}^{W}\left(  m\right)  }.
\end{align*}
For the second assertion we take Eq. (\ref{e.2.53}) with $s=0,$ $t=1,$ and
$W_{t}=X$ for all $t$ to find,%
\[
\left\vert Ad_{e^{X}}Z\right\vert _{m}\leq e^{\int_{0}^{1}\left\vert \nabla
X\right\vert _{e^{\left(  \tau-1\right)  X}\left(  m\right)  }d\tau}%
\cdot\left\vert Z\right\vert _{e^{-X}\left(  m\right)  }=e^{\int_{0}%
^{1}\left\vert \nabla X\right\vert _{e^{-\tau X}\left(  m\right)  }d\tau}%
\cdot\left\vert Z\right\vert _{e^{-X}\left(  m\right)  }%
\]

\end{proof}

\subsection{Distance estimates for flows\label{sec.2.6}}

This subsection is devoted to estimating the distance between two flows,
$\mu^{X}$ and $\mu^{Y}.$ A key observation in the proofs to follow is, given
$t\in\left[  0,T\right]  ,$ that
\begin{equation}
\left[  0,t\right]  \ni s\rightarrow\Theta_{s}\left(  m\right)  :=\mu
_{t,s}^{X}\circ\mu_{s,0}^{Y}\left(  m\right)  \label{e.2.55}%
\end{equation}
is a natural path in $M$ which interpolates between $\Theta_{0}\left(
m\right)  =\mu_{t,0}^{X}\left(  m\right)  $ at $s=0$ and $\Theta_{t}\left(
m\right)  =\mu_{t,0}^{Y}\left(  m\right)  $ at $s=t.$

\begin{theorem}
\label{thm.2.30}Let $J=\left[  0,T\right]  \ni t\rightarrow X_{t,}Y_{t}%
\in\Gamma\left(  TM\right)  $ be two smooth complete time dependent vector
fields on $M$ and $\mu^{X}$ and $\mu^{Y}$ be their corresponding flows. Then
for $t>0$ (for notational simplicity)
\begin{equation}
d\left(  \mu_{t,0}^{X}\left(  m\right)  ,\mu_{t,0}^{Y}\left(  m\right)
\right)  \leq\int_{0}^{t}e^{\int_{s}^{t}\left\vert \nabla X_{\sigma
}\right\vert _{\mu_{\sigma,s}^{X}\left(  m\right)  }d\sigma}\cdot\left\vert
Y_{s}-X_{s}\right\vert _{\mu_{s,0}^{Y}\left(  m\right)  }~ds \label{e.2.56}%
\end{equation}
and in particular,%
\begin{equation}
d_{M}\left(  \mu_{t,0}^{X},\mu_{t,0}^{Y}\right)  \leq e^{\left\vert \nabla
X\right\vert _{t}^{\ast}}\cdot\left\vert Y-X\right\vert _{t}^{\ast}.
\label{e.2.57}%
\end{equation}

\end{theorem}

\begin{proof}
Fix $t\in\left[  0,T\right]  .$ If $\Theta_{s}\left(  m\right)  $ is as in Eq.
(\ref{e.2.55}), then
\begin{equation}
d\left(  \mu_{t,0}^{X}\left(  m\right)  ,\mu_{t,0}^{Y}\left(  m\right)
\right)  \leq\int_{0}^{t}\left\vert \Theta_{s}^{\prime}\left(  m\right)
\right\vert ds. \label{e.2.58}%
\end{equation}
Making use of Theorem \ref{thm.2.14} we find,%
\begin{align}
\Theta_{s}^{\prime}\left(  m\right)   &  =\left(  \frac{d}{ds}\mu_{t,s}%
^{X}\right)  \circ\mu_{s,0}^{Y}\left(  m\right)  +\left(  \mu_{t,s}%
^{X}\right)  _{\ast}\left(  \frac{d}{ds}\mu_{s,0}^{Y}\left(  m\right)  \right)
\nonumber\\
&  =\left(  \mu_{t,s}^{X}\right)  _{\ast}\left[  -X_{s}+Y_{s}\right]  \circ
\mu_{s,0}^{Y}\left(  m\right)  . \label{e.2.59}%
\end{align}
The Jacobian estimate in Corollary \ref{cor.2.27} with $\nu_{t}=\mu_{t,s}^{X}$
states that,%
\begin{equation}
\left\vert \mu_{t,s\ast}^{X}\right\vert _{m}\leq e^{\int_{s}^{t}\left\vert
\nabla X_{\sigma}\right\vert _{\mu_{\sigma,s}^{X}\left(  m\right)  }d\sigma
}\leq e^{\int_{s}^{t}\left\vert \nabla X_{\sigma}\right\vert _{M}d\sigma}.
\label{e.2.60}%
\end{equation}
By this Jacobian estimate and Eq. (\ref{e.2.59}), we find that%
\begin{align}
\left\vert \Theta_{s}^{\prime}\left(  m\right)  \right\vert  &  \leq\left\vert
\left(  \mu_{t,s}^{X}\right)  _{\ast}\right\vert _{\mu_{s,0}^{Y}\left(
m\right)  }\left\vert Y_{s}-X_{s}\right\vert _{\mu_{s,0}^{Y}\left(  m\right)
}\nonumber\\
&  \leq e^{\int_{s}^{t}\left\vert \nabla X_{\sigma}\right\vert _{\mu
_{\sigma,s}^{X}\left(  m\right)  }d\sigma}\cdot\left\vert Y_{s}-X_{s}%
\right\vert _{\mu_{s,0}^{Y}\left(  m\right)  }\leq e^{\left\vert \nabla
X\right\vert _{t}^{\ast}}\left\vert Y_{s}-X_{s}\right\vert _{M} \label{e.2.61}%
\end{align}
which then substituted back into Eq. (\ref{e.2.58}) completes the proof.
\end{proof}

\begin{corollary}
\label{cor.2.31}Let $J=\left[  0,T\right]  \ni t\rightarrow Y_{t}\in
\Gamma\left(  TM\right)  $ be a smooth complete time dependent vector field on
$M$ and $\mu^{Y}$ be the corresponding flow. Then for $t>0$
\begin{equation}
d\left(  m,\mu_{t,0}^{Y}\left(  m\right)  \right)  \leq\int_{0}^{t}\left\vert
Y_{s}\right\vert _{\mu_{s,0}^{Y}\left(  m\right)  }~ds\leq\left\vert
Y\right\vert _{t}^{\ast} \label{e.2.62}%
\end{equation}
and
\begin{equation}
d\left(  \mu_{t,0}^{Y}\left(  m\right)  ,m\right)  \leq\int_{0}^{t}e^{\int%
_{s}^{t}\left\vert \nabla Y_{\sigma}\right\vert _{\mu_{\sigma,s}^{Y}\left(
m\right)  }d\sigma}\cdot\left\vert Y_{s}\right\vert _{m}~ds\leq e^{\left\vert
\nabla Y\right\vert _{t}^{\ast}}\int_{0}^{t}\left\vert Y_{s}\left(  m\right)
\right\vert ds. \label{e.2.63}%
\end{equation}

\end{corollary}

\begin{proof}
This corollary easily follows from Theorem \ref{thm.2.11}. Alternatively,
taking $X_{\cdot}\equiv0$ in Theorem \ref{thm.2.30} gives Eq. (\ref{e.2.62})
while taking $Y_{\cdot}\equiv0$ shows%
\[
d\left(  \mu_{t,0}^{X}\left(  m\right)  ,m\right)  \leq\int_{0}^{t}e^{\int%
_{s}^{t}\left\vert \nabla X_{\sigma}\right\vert _{\mu_{\sigma,s}^{X}\left(
m\right)  }d\sigma}\cdot\left\vert X_{s}\right\vert _{m}~ds\leq e^{\left\vert
\nabla X\right\vert _{t}^{\ast}}\int_{0}^{t}\left\vert X_{s}\left(  m\right)
\right\vert ds.
\]
Equation (\ref{e.2.63}) now follows by relabeling $X$ to $Y.$
\end{proof}

\section{Nilpotent Lie Algebras (Group) Results\label{sec.3}}

Suppose that $\mathcal{A}$ is a non-commutative associative algebra with unit,
$1,$ over $\mathbb{R}$ such that: 1) $\dim_{\mathbb{R}}\mathcal{A}<\infty,$ 2)
$\mathcal{A}=\mathbb{R}\cdot1\oplus\mathfrak{g}$ where $\mathfrak{g}$ is a
sub-algebra of $\mathcal{A}$ without unit, and 3) there exists $\kappa
\in\mathbb{N}$ such that $\xi_{1}\dots\xi_{\kappa+1}=0$ whenever $\xi
_{1},\dots,\xi_{\kappa+1}\in\mathfrak{g.}$ We make $\mathcal{A}$ into a Lie
algebra using the commutator, $\left[  \xi,\eta\right]  :=\xi\eta-\eta\xi$ for
all $\xi,\eta\in\mathcal{A},$ as the Lie bracket. Note that $\mathfrak{g}$ is
a Lie-subalgebra of $\mathcal{A}$ and as usual we let $\operatorname{ad}_{\xi
}:\mathcal{A}\rightarrow\mathcal{A}$ be the linear operator defined by
$\operatorname{ad}_{\xi}\eta=\left[  \xi,\eta\right]  $. See Example
\ref{sec.3.2} below for the key example of this setup that is used in the bulk
of this paper.

\subsection{Calculus and functional calculus on $\mathcal{A}$\label{sec.3.1}}

\begin{definition}
\label{def.3.1}Let $\mathcal{H}_{0}$ denote the germs of functions which are
analytic in a neighborhood of $0\in\mathbb{C}$ and for $f\left(  z\right)
=\sum_{k=0}^{\infty}a_{k}z^{k}\in\mathcal{H}_{0}$ and $\xi\in\mathfrak{g},$
let
\[
f\left(  \xi\right)  :=\sum_{k=0}^{\infty}a_{k}\xi^{k}=\sum_{k=0}^{\kappa
}a_{k}\xi^{k}\in\mathcal{A}%
\]
and
\[
f\left(  \operatorname{ad}_{\xi}\right)  :=\sum_{k=0}^{\infty}a_{k}%
\operatorname{ad}_{\xi}^{k}=\sum_{k=0}^{\kappa-1}a_{k}\operatorname{ad}_{\xi
}^{k}:\mathcal{A}\rightarrow\mathcal{A}.
\]

\end{definition}

In most of the results below, we describe properties of $f\left(  \xi\right)
$ for $f\in\mathcal{H}_{0}$ with the understanding that similar results hold
equally as well for $f\left(  \operatorname{ad}_{\xi}\right)  .$

\begin{proposition}
\label{pro.3.2}For each fixed $\xi\in\mathfrak{g},$ the map
\begin{equation}
\mathcal{H}_{0}\ni f\rightarrow f\left(  \xi\right)  \in\mathcal{A}
\label{e.3.1}%
\end{equation}
is an algebra homomorphism and for each fixed $f\in\mathcal{A}_{0},$ the map
\[
\mathfrak{g}\ni\xi\rightarrow f\left(  \xi\right)  \in\mathcal{A}%
\]
is smooth, i.e. it is infinitely continuously differentiable. Moreover, if
$J:=\left(  a,b\right)  \ni t\rightarrow\xi\left(  t\right)  \in\mathfrak{g}$
is differentiable with $\left[  \xi\left(  t\right)  ,\xi\left(  s\right)
\right]  =0$ for $s,t\in J,$ then
\begin{equation}
\frac{d}{dt}f\left(  \xi\left(  t\right)  \right)  =f^{\prime}\left(
\xi\left(  t\right)  \right)  \dot{\xi}\left(  t\right)  =\dot{\xi}\left(
t\right)  f^{\prime}\left(  \xi\left(  t\right)  \right)  \text{~}%
\forall~\text{ }t\in J. \label{e.3.2}%
\end{equation}

\end{proposition}

\begin{proof}
The standard fact that the map in Eq. (\ref{e.3.1}) is an algebra homomorphism
is easily seen to be a direct consequence of the multiplication rules for
power series. The smoothness of $f:\mathfrak{g}\rightarrow\mathcal{A}$ is a
consequence of the fact that $f\left(  \xi\right)  $ is a finite linear
combination of the smooth multi-linear maps, $\mathfrak{g\ni\xi}\rightarrow
\xi^{k}\in\mathcal{A},$ for each $k\in\left\{  0,1,2,\dots,\kappa\right\}  .$
For arbitrary $\xi,\eta\in\mathfrak{g}$ we have
\[
\partial_{\eta}\xi^{k}=\sum_{j=0}^{k-1}\xi^{j}\eta\xi^{k-1-j}%
\]
which simplifies to
\[
\partial_{\eta}\xi^{k}=k\eta\xi^{k-1}=k\xi^{k-1}\eta\text{ when }\left[
\xi,\eta\right]  =0.
\]

With these observations the proof of Eq. (\ref{e.3.2}) is a consequence of the
following simple computation,
\begin{align*}
\frac{d}{dt}f\left(  \xi\left(  t\right)  \right)   &  =\sum_{k=0}^{\kappa
}a_{k}\frac{d}{dt}\xi\left(  t\right)  ^{k}=\sum_{k=0}^{\kappa}a_{k}k\dot{\xi
}\left(  t\right)  \xi\left(  t\right)  ^{k-1}\\
&  =\sum_{k=0}^{\infty}a_{k}k\dot{\xi}\left(  t\right)  \xi\left(  t\right)
^{k-1}=f^{\prime}\left(  \xi\left(  t\right)  \right)  \dot{\xi}\left(
t\right)  =\dot{\xi}\left(  t\right)  f^{\prime}\left(  \xi\left(  t\right)
\right)  .
\end{align*}

\end{proof}

For our purposes, the functions, $e^{z},$ $\left(  1+z\right)  ^{-1}%
,\log\left(  1+z\right)  ,$%
\begin{align}
\psi\left(  z\right)   &  :=\frac{e^{z}-1}{z}=\sum_{k=1}^{\infty}\frac
{z^{k-1}}{k!}=\sum_{k=0}^{\infty}\frac{z^{k}}{\left(  k+1\right)
!},\label{e.3.3}\\
\psi_{-}\left(  z\right)   &  :=\frac{1}{\psi\left(  -z\right)  }=\frac
{z}{1-e^{-z}}=\frac{e^{z}\cdot z}{e^{z}-1},\text{ and}\label{e.3.4}\\
\mathcal{L}\left(  z\right)   &  =\psi_{-}\left(  \log\left(  1+z\right)
\right)  :=\frac{\left(  1+z\right)  \log\left(  1+z\right)  }{z}%
\label{e.3.5}\\
&  =1+\sum_{j=2}^{\infty}\frac{\left(  -1\right)  ^{j}}{j\cdot\left(
j-1\right)  }z^{j-1} \label{e.3.6}%
\end{align}
are the most important functions in $\mathcal{H}_{0}.$

\begin{lemma}
\label{lem.3.3}The subset,
\[
G:=1+\mathfrak{g}=\left\{  1+\xi:\xi\in\mathfrak{g}\right\}  ,
\]
equipped with the algebra multiplication law forms a group where the inverse
operation is given by
\[
\left(  1+\xi\right)  ^{-1}=\frac{1}{1+\xi}.
\]

\end{lemma}

The following corollary follows directly from Proposition \ref{pro.3.2}.

\begin{corollary}
\label{cor.3.4}The three map, $\mathfrak{g\ni\xi}\rightarrow e^{\xi}\in G,$
$\mathfrak{g}\ni\xi\rightarrow\left(  1+\xi\right)  ^{-1}\in G,$ and
$\mathfrak{g}\ni\xi\rightarrow\log\left(  1+\xi\right)  \in\mathfrak{g}$ are
smooth and these maps satisfy the following natural identities.

\begin{enumerate}
\item For all $\xi\in\mathfrak{g},$%
\begin{align*}
\frac{d}{dt}e^{t\xi}  &  =\xi e^{t\xi},~\text{and}\\
\frac{d}{dt}\log\left(  1+t\xi\right)   &  =\xi\frac{1}{1+\xi}=\frac{\xi
}{1+\xi}.
\end{align*}
Moreover generally, if $t\rightarrow\xi\left(  t\right)  \in\mathfrak{g}$ is
differentiable near $t_{0}\in\mathbb{R}$ and $\left[  \xi\left(  t\right)
,\xi\left(  s\right)  \right]  =0$ for $s$ and $t$ near $t_{0},$ then
\begin{align*}
\frac{d}{dt}e^{\xi\left(  t\right)  }  &  =\dot{\xi}\left(  t\right)
e^{\xi\left(  t\right)  }\text{ and }\\
\frac{d}{dt}\log\left(  1+\xi\left(  t\right)  \right)   &  =\frac{\dot{\xi
}\left(  t\right)  }{1+\xi\left(  t\right)  }=\dot{\xi}\left(  t\right)
\left(  1+\xi\left(  t\right)  \right)  ^{-1}.
\end{align*}

\item For all $s,t\in\mathbb{R},$ $e^{t\xi}e^{s\xi}=e^{\left(  t+s\right)
\xi}$ and $e^{-t\xi}=\left[  e^{t\xi}\right]  ^{-1}.$
\end{enumerate}
\end{corollary}

\begin{proposition}
\label{pro.3.5}The map,%
\[
\mathfrak{g\ni\xi}\rightarrow e^{\xi}\in G,
\]
is a diffeomorphism and the map,%
\[
G\ni g=1+\xi\rightarrow\log\left(  g\right)  =\log\left(  1+\xi\right)
\in\mathfrak{g,}%
\]
is its inverse map.
\end{proposition}

\begin{proof}
By Corollary \ref{cor.3.4},%
\[
\frac{d}{dt}\log\left(  e^{t\xi}\right)  =\left[  e^{t\xi}\right]  ^{-1}%
\frac{d}{dt}e^{t\xi}=e^{-t\xi}e^{t\xi}\xi=\xi,
\]
from which it follows that $\log\left(  e^{t\xi}\right)  =\log\left(
1\right)  +t\xi=t\xi.$ Taking $t=1$ in this identity shows $\log\left(
e^{\xi}\right)  =\xi.$

Similarly, by Corollary \ref{cor.3.4}, if we let $g\left(  t\right)
:=e^{\log\left(  1+t\xi\right)  }\in G,$ then
\[
\dot{g}\left(  t\right)  =\frac{d}{dt}e^{\log\left(  1+t\xi\right)  }%
=e^{\log\left(  1+t\xi\right)  }\cdot\frac{d}{dt}\log\left(  1+t\xi\right)
=g\left(  t\right)  \frac{\xi}{1+t\xi}\text{ with }g\left(  0\right)  =1.
\]
Since $t\rightarrow\left(  1+t\xi\right)  $ satisfies the same equation as
$g\left(  t\right)  ,$ by uniqueness of solutions we conclude that $g\left(
t\right)  =1+t\xi$ for all $t\in\mathbb{R}$ and in particular taking $t=1$
shows%
\[
e^{\log\left(  1+\xi\right)  }=g\left(  1\right)  =1+\xi.
\]

\end{proof}

For $k\in G,$ let $L_{k}\in\mathrm{Diff}\left(  G\right)  $ be defined by
$L_{k}g=kg$ for all $g\in G.$ For $\xi\in\mathfrak{g},$ let $G\ni
g\rightarrow\tilde{\xi}\left(  g\right)  :=L_{g\ast}\xi\in T_{g}G$ be the left
invariant vector field on $G$ associated to $\xi\in\mathfrak{g.}$ If
$f:G+1+\mathfrak{g}\cong\mathfrak{g}\rightarrow\mathbb{R}$ is a smooth
function, then
\[
\left(  \tilde{\xi}f\right)  \left(  g\right)  =\frac{d}{dt}|_{0}f\left(
ge^{t\xi}\right)  =\left(  \partial_{g\xi}f\right)  \left(  g\right)
=f^{\prime}\left(  g\right)  g\xi.
\]
Thus if $\xi,\eta\in\mathfrak{g},$
\[
\left(  \tilde{\eta}\tilde{\xi}f\right)  \left(  g\right)  =f^{\prime\prime
}\left(  g\right)  \left[  g\eta\otimes g\xi\right]  +f^{\prime}\left(
g\right)  g\eta\xi
\]
and since $f^{\prime\prime}\left(  g\right)  $ is symmetric,
\[
\left(  \left[  \tilde{\eta},\tilde{\xi}\right]  f\right)  \left(  g\right)
=f^{\prime}\left(  g\right)  g\left(  \eta\xi-\xi\eta\right)  =\left(  \left(
\eta\xi-\xi\eta\right)  ^{\sim}f\right)  \left(  g\right)  .
\]
Therefore the standard left invariant vector-field Lie algebra associated to
$G$ has bracket,%
\[
\left[  \eta,\xi\right]  =\left[  \tilde{\eta},\tilde{\xi}\right]  \left(
1\right)  =\eta\xi-\xi\eta
\]
which is the same as the Lie algebra associated to the algebra multiplication law.

\begin{definition}
\label{def.3.6}For $\xi\in C^{1}\left(  \left[  0,T\right]  ,\mathfrak{g}%
\right)  ,$ let $g^{\xi}\left(  t\right)  \in G$ denote the unique solution to
the linear differential equation,%
\begin{equation}
\dot{g}^{\xi}\left(  t\right)  =g^{\xi}\left(  t\right)  \dot{\xi}\left(
t\right)  =\widetilde{\xi\left(  t\right)  }\left(  g\left(  t\right)
\right)  \text{ with }g^{\xi}\left(  0\right)  =1\in G \label{e.3.7}%
\end{equation}
and further let
\begin{equation}
C^{\xi}\left(  t\right)  =\log\left(  g^{\xi}\left(  t\right)  \right)  .
\label{e.3.8}%
\end{equation}

\end{definition}

\begin{remark}
\label{rem.3.7}If $\xi\in C^{1}\left(  \left[  0,T\right]  ,\mathfrak{g}%
\right)  ,$ then $t\rightarrow\widetilde{\dot{\xi}\left(  t\right)  }\in
\Gamma\left(  TG\right)  $ is a $C^{0}$-varying vector field on $G$ with
associated flow, $\mu_{t,s}^{\widetilde{\dot{\xi}}},$ which satisfies
$L_{k}\circ\mu_{t,s}^{\widetilde{\dot{\xi}}}=\mu_{t,s}^{\widetilde{\dot{\xi}}%
}\circ L_{k}.$ Applying this equation to$1\in G$ shows
\[
\mu_{t,s}^{\widetilde{\dot{\xi}}}\left(  k\right)  =k\cdot\mu_{t,s}%
^{\widetilde{\dot{\xi}}}\left(  1\right)  \in G\text{ for all }k\in G
\]
where $\mu_{t,s}^{\widetilde{\dot{\xi}}}\left(  1\right)  \in G$ satisfies the
ODE,
\[
\frac{d}{dt}\mu_{t,s}^{\widetilde{\dot{\xi}}}\left(  1\right)
=\widetilde{\dot{\xi}\left(  t\right)  }\circ\mu_{t,s}^{\widetilde{\dot{\xi}}%
}\left(  1\right)  =\mu_{t,s}^{\widetilde{\dot{\xi}}}\left(  1\right)
\dot{\xi}\left(  t\right)  \text{ with }\mu_{s,s}^{\widetilde{\dot{\xi}}%
}\left(  1\right)  =1.
\]
As $g^{\xi}\left(  s\right)  ^{-1}g^{\xi}\left(  t\right)  $ satisfies this
same differential equation, it follows that
\[
\mu_{t,s}^{\widetilde{\dot{\xi}}}\left(  k\right)  =kg^{\xi}\left(  s\right)
^{-1}g^{\xi}\left(  t\right)  =R_{g^{\xi}\left(  s\right)  ^{-1}g^{\xi}\left(
t\right)  }k\text{ }\forall~s,t\in\left[  0,T\right]  .
\]
In particular if $\xi\in\mathfrak{g}$ is constant, then $e^{t\tilde{\xi}%
}=R_{e^{t\xi}},$ i.e.
\[
e^{t\tilde{\xi}}\left(  k\right)  =ke^{t\xi}\text{ for all }k\in G
\]

\end{remark}

\begin{proposition}
\label{pro.3.8}For $\xi,\eta\in\mathfrak{g},$
\[
\partial_{\xi}e^{\eta}=e^{\eta}\int_{0}^{1}\operatorname{Ad}_{e^{-t\eta}}\xi
dt=\left[  \int_{0}^{1}\operatorname{Ad}_{e^{t\eta}}\xi dt\right]  e^{\eta}.
\]
Consequently if $C\left(  t\right)  \in\mathfrak{g}$ is a smooth curve then
\begin{align*}
\frac{d}{dt}e^{C\left(  t\right)  }  &  =\left[  \int_{0}^{1}\operatorname{Ad}%
_{e^{sC\left(  t\right)  }}\dot{C}\left(  t\right)  ds\right]  e^{C\left(
t\right)  }\\
&  =e^{C\left(  t\right)  }\left[  \int_{0}^{1}\operatorname{Ad}%
_{e^{-sC\left(  t\right)  }}\dot{C}\left(  t\right)  ds\right]
\end{align*}

\end{proposition}

\begin{proof}
\textbf{First proof. }Differentiating the identity,%
\[
\frac{d}{dt}e^{t\left(  \eta+s\xi\right)  }=\left(  \eta+s\xi\right)
e^{t\left(  \eta+s\xi\right)  },
\]
in $s$ shows
\begin{align*}
\frac{d}{dt}\frac{d}{ds}|_{0}e^{t\left(  \eta+s\xi\right)  }  &  =\frac{d}%
{ds}|_{0}\frac{d}{dt}e^{t\left(  \eta+s\xi\right)  }=\frac{d}{ds}|_{0}\left[
\left(  \eta+s\xi\right)  e^{t\left(  \eta+s\xi\right)  }\right] \\
&  =\xi e^{t\eta}+\eta\frac{d}{ds}|_{0}e^{t\left(  \eta+s\xi\right)  }\text{
with }\frac{d}{ds}|_{0}e^{0\left(  \eta+s\xi\right)  }=0.
\end{align*}
Solving this equation by Duhamel's principle gives,
\[
\frac{d}{ds}|_{0}e^{\left(  \eta+s\xi\right)  }=\int_{0}^{1}e^{\left(
1-t\right)  \eta}\xi e^{t\eta}dt,
\]
i.e.
\[
\partial_{\xi}e^{\eta}=e^{\eta}\int_{0}^{1}\operatorname{Ad}_{e^{-t\eta}}\xi
dt=\left[  \int_{0}^{1}\operatorname{Ad}_{e^{t\eta}}\xi dt\right]  e^{\eta}.
\]

\textbf{Second proof. } This proof relies on the fact that the statement of
this proposition is in fact a special case of Theorem \ref{thm.2.22}. Indeed
using this theorem along with Remark \ref{rem.3.7} shows,%
\begin{align*}
\partial_{\eta}e^{t\xi}  &  =\partial_{\tilde{\eta}}e^{t\tilde{\xi}}\left(
1\right)  =\left[  \int_{0}^{t}\operatorname{Ad}_{e^{\tau\tilde{\xi}}}%
\tilde{\eta}d\tau\right]  \circ e^{t\tilde{\xi}}\left(  1\right) \\
&  =\left[  \int_{0}^{t}\operatorname{Ad}_{e^{\tau\tilde{\xi}}}\tilde{\eta
}d\tau\right]  \left(  e^{t\xi}\right)
\end{align*}
where
\begin{align*}
\left(  \operatorname{Ad}_{e^{\tau\tilde{\xi}}}\tilde{\eta}\right)  \left(
k\right)   &  =\left(  e_{\ast}^{\tau\tilde{\xi}}\tilde{\eta}\circ
e^{-\tau\tilde{\xi}}\right)  \left(  k\right)  =e_{\ast}^{\tau\tilde{\xi}%
}\left(  L_{e^{-\tau\tilde{\xi}}\left(  k\right)  }\right)  _{\ast}\eta\\
&  =e_{\ast}^{\tau\tilde{\xi}}\left(  L_{ke^{-\tau\xi}}\right)  _{\ast}%
\eta=\left(  R_{e^{\tau\xi}}\right)  _{\ast}\left(  L_{ke^{-\tau\xi}}\right)
_{\ast}\eta\\
&  =ke^{-\tau\xi}\eta e^{\tau\xi}.
\end{align*}
Since%
\[
\left[  \int_{0}^{t}\operatorname{Ad}_{e^{\tau\tilde{\xi}}}\tilde{\eta}%
d\tau\right]  \left(  e^{t\xi}\right)  =e^{t\xi}\int_{0}^{t}e^{-\tau\xi}\eta
e^{\tau\xi}d\tau=\int_{0}^{t}e^{\left(  t-\tau\right)  \xi}\eta e^{\tau\xi
}d\tau,
\]
the result is again proved.
\end{proof}

\begin{lemma}
\label{lem.3.9}For $\eta\in\mathfrak{g}$ and $t\in\mathbb{R},$
$\operatorname{Ad}_{e^{t\eta}}=e^{t\operatorname{ad}_{\eta}}$ and
\begin{equation}
\int_{0}^{1}\operatorname{Ad}_{e^{t\eta}}dt=\psi\left(  \operatorname{ad}%
_{\eta}\right)  \label{e.3.9}%
\end{equation}
where $\psi$ is as in Eq. (\ref{e.3.3}).
\end{lemma}

\begin{proof}
Let $\xi\in\mathfrak{g}.$ Since
\begin{align*}
\frac{d}{dt}\left[  \operatorname{Ad}_{e^{t\eta}}\xi\right]   &  =\frac{d}%
{dt}\left[  e^{t\eta}\xi e^{-t\eta}\right]  =\eta e^{t\eta}\xi e^{-t\eta
}-e^{t\eta}\xi e^{-t\eta}\eta\\
&  =\operatorname{ad}_{\eta}\operatorname{Ad}_{e^{t\eta}}\xi
\end{align*}
and $e^{t\operatorname{ad}_{\eta}}\xi$ solves the same equation with the same
initial condition of $\xi$ at $t=0,$ we conclude that $\operatorname{Ad}%
_{e^{t\eta}}\xi=e^{t\operatorname{ad}_{\eta}}\xi.$ As this is true for all
$\xi\in\mathfrak{g,}$ it follows that $\operatorname{Ad}_{e^{t\eta}%
}=e^{t\operatorname{ad}_{\eta}}.$ The last equality is now proved by
integrating the series expansion for $e^{t\operatorname{ad}_{\eta}};$
\begin{align*}
\int_{0}^{1}\operatorname{Ad}_{e^{t\eta}}dt  &  =\int_{0}^{1}%
e^{t\operatorname{ad}_{\eta}}dt=\int_{0}^{1}\sum_{n=0}^{\infty}\frac{t^{n}%
}{n!}\operatorname{ad}_{\eta}^{n}dt\\
&  =\sum_{n=0}^{\infty}\int_{0}^{1}\frac{t^{n}}{n!}\operatorname{ad}_{\eta
}^{n}dt=\sum_{n=0}^{\infty}\frac{1}{\left(  n+1\right)  !}\operatorname{ad}%
_{\eta}^{n}=\psi\left(  \operatorname{ad}_{\eta}\right)  .
\end{align*}

\end{proof}

\begin{corollary}
\label{cor.3.10}Let $g\left(  t\right)  $ be a smooth curve in $G,$
\[
\dot{\xi}\left(  t\right)  :=L_{g\left(  t\right)  ^{-1}\ast}\dot{g}\left(
t\right)  =g\left(  t\right)  ^{-1}\dot{g}\left(  t\right)  \in\mathfrak{g}%
\text{ and }C\left(  t\right)  :=\log\left(  g\left(  t\right)  \right)
\in\mathfrak{g.}%
\]
Then $C\left(  t\right)  :=\log\left(  g\left(  t\right)  \right)
\in\mathfrak{g}$ is the unique solution to the ODE,%
\begin{equation}
\psi\left(  -\operatorname{ad}_{C\left(  t\right)  }\right)  \dot{C}\left(
t\right)  =\int_{0}^{1}\operatorname{Ad}_{e^{-sC\left(  t\right)  }}\dot
{C}\left(  t\right)  ds=\dot{\xi}\left(  t\right)  \text{ with }C\left(
0\right)  =\log\left(  g\left(  0\right)  \right)  \label{e.3.10}%
\end{equation}
or equivalently (in more standard form) $C\left(  t\right)  $ satisfies,%
\begin{equation}
\dot{C}\left(  t\right)  =\psi_{-}\left(  \operatorname{ad}_{C\left(
t\right)  }\right)  \dot{\xi}\left(  t\right)  =\text{\textquotedblleft}%
\frac{\operatorname{ad}_{C\left(  t\right)  }}{I-e^{-\operatorname{ad}%
_{C\left(  t\right)  }}}\dot{\xi}\left(  t\right)  \text{\textquotedblright%
\ with }C\left(  0\right)  =\log\left(  g\left(  0\right)  \right)  ,
\label{e.3.11}%
\end{equation}
where $\psi_{-}\left(  z\right)  =1/\psi\left(  -z\right)  $ as in Eq.
(\ref{e.3.4}).
\end{corollary}

\begin{proof}
Since $g\left(  t\right)  =e^{C\left(  t\right)  },$ it follows from
Proposition \ref{pro.3.8} and Lemmas \ref{lem.3.9} that
\begin{align*}
g\left(  t\right)  \dot{\xi}\left(  t\right)   &  =\dot{g}\,\left(  t\right)
=\frac{d}{dt}e^{C\left(  t\right)  }=e^{C\left(  t\right)  }\int_{0}%
^{1}\operatorname{Ad}_{e^{-sC\left(  t\right)  }}\dot{C}\left(  t\right)  ds\\
&  =g\left(  t\right)  \int_{0}^{1}\operatorname{Ad}_{e^{-sC\left(  t\right)
}}\dot{C}\left(  t\right)  ds=g\left(  t\right)  \psi\left(
-\operatorname{ad}_{C\left(  t\right)  }\right)  \dot{C}\left(  t\right)  .
\end{align*}
Multiplying this identity on the left by $g\left(  t\right)  ^{-1}$ gives the
Eq. (\ref{e.3.10}) while Eq. (\ref{e.3.11}) then follows by multiplying Eq.
(\ref{e.3.10}) on the left by $\psi_{-}\left(  \operatorname{ad}_{C\left(
t\right)  }\right)  .$
\end{proof}

\begin{definition}
\label{def.3.11}Let $\Gamma:\mathfrak{g}\times\mathfrak{g\rightarrow g}$ be
the function defined by
\[
\Gamma\left(  \xi,\eta\right)  :=\log\left(  e^{\xi}e^{\eta}\right)
\in\mathfrak{g}\text{ for all }\xi,\eta\in\mathfrak{g}.
\]

\end{definition}

The next proposition deals with Lie sub-algebras of $\mathfrak{g}$ and simply
connected Lie subgroups of $G.$

\begin{proposition}
\label{pro.3.12}Let $\mathfrak{g}_{0}$ be a Lie subalgebra of $\mathfrak{g}$
and $G_{0}\subset G$ be the unique connected Lie subgroup of $G$ which has
$\mathfrak{g}_{0}$ as its Lie algebra. If $g\left(  t\right)  \in G_{0}$ is a
smooth curve connecting $1$ to $g\in G_{0}$ and $\dot{\xi}\left(  t\right)
:=g\left(  t\right)  ^{-1}\dot{g}\left(  t\right)  \in\mathfrak{g}_{0},$ then
$C\left(  t\right)  :=\log\left(  g\left(  t\right)  \right)  \in
\mathfrak{g}_{0}.$
\end{proposition}

\begin{proof}
We know that $C\left(  t\right)  $ may be characterized as the solution to the
ODE%
\[
\dot{C}\left(  t\right)  =F_{\xi}\left(  t,C\left(  t\right)  \right)  \text{
with }C\left(  0\right)  =0,
\]
where
\[
F_{\xi}\left(  t,\eta\right)  \frac{\operatorname{ad}_{\eta}}%
{I-e^{-\operatorname{ad}_{\eta}}}\dot{\xi}\left(  t\right)  =\frac{1}{\psi
}\left(  -\operatorname{ad}_{\eta}\right)  \dot{\xi}\left(  t\right)  .
\]
As $F_{\xi}\left(  t,\cdot\right)  :\mathfrak{g}_{0}\rightarrow\mathfrak{g}%
_{0},$ it follows that $C\left(  t\right)  \in\mathfrak{g}_{0}$ as required.
\end{proof}

\begin{corollary}
\label{cor.3.13}If we continue the assumptions and notation in Proposition
\ref{pro.3.12}, then $\log\left(  G_{0}\right)  =\mathfrak{g}_{0}$ and
$\log|_{G_{0}}:G_{0}\rightarrow\mathfrak{g}_{0}$ is a diffeomorphism with
inverse given by
\[
\mathfrak{g}_{0}\ni A\rightarrow e^{A}\in G_{0}.
\]

\end{corollary}

\begin{proof}
These assertions follow directly using $e^{A}\in G_{0}$ for $A\in
\mathfrak{g}_{0}$ along with Proposition \ref{pro.3.5} and Proposition
\ref{pro.3.12}.
\end{proof}

For later purposes it is useful to record an \textquotedblleft
explicit\textquotedblright\ formula for $g^{\xi}\left(  t\right)  $ as defined
in Definition \ref{def.3.6}.

\begin{proposition}
\label{pro.3.14}The path, $g^{\xi}\left(  t\right)  \in G,$ as in Definition
\ref{def.3.6} may be expressed as%
\begin{equation}
g^{\xi}\left(  t\right)  =1+\sum_{k=1}^{\kappa}\int_{0\leq s_{1}\leq s_{2}%
\leq\dots\leq s_{k}\leq t}\dot{\xi}\left(  s_{1}\right)  \dots\dot{\xi}\left(
s_{k}\right)  ds_{1}\dots ds_{k}. \label{e.3.12}%
\end{equation}

\end{proposition}

\begin{proof}
From Definition \ref{def.3.6} and the fundamental theorem of calculus,%
\[
g^{\xi}\left(  t\right)  =1+\int_{0}^{t}g^{\xi}\left(  \tau\right)  \dot{\xi
}\left(  \tau\right)  d\tau.
\]
Feeding this equation back into itself then shows,%
\begin{align}
g^{\xi}\left(  t\right)   &  =1+\int_{0}^{t}\left[  1+\int_{0}^{\tau}g^{\xi
}\left(  s\right)  \dot{\xi}\left(  s\right)  ds\right]  \dot{\xi}\left(
\tau\right)  d\tau\nonumber\\
&  =1+\int_{0}^{t}\dot{\xi}\left(  \tau\right)  d\tau+\int_{0}^{t}d\tau
\int_{0}^{\tau}dsg^{\xi}\left(  s\right)  \dot{\xi}\left(  s\right)  \dot{\xi
}\left(  \tau\right)  .\nonumber
\end{align}
Continuing this way inductively shows for any $m\in\mathbb{N}$ that,
\begin{equation}
g^{\xi}\left(  t\right)  =1+\sum_{k=1}^{m-1}\int_{0\leq s_{1}\leq s_{2}%
\leq\dots\leq s_{k}\leq t}\dot{\xi}\left(  s_{1}\right)  \dots\dot{\xi}\left(
s_{k}\right)  d\mathbf{s+}R_{m}\left(  t\right)  \label{e.3.13}%
\end{equation}
where $\sum_{k=1}^{m-1}\left[  \dots\right]  \equiv0$ when $m=1$ and
\[
R_{m}\left(  t\right)  :=\int_{0\leq s_{1}\leq s_{2}\leq\dots\leq s_{m}\leq
t}g^{\xi}\left(  s_{1}\right)  \dot{\xi}\left(  s_{1}\right)  \dots\dot{\xi
}\left(  s_{m}\right)  d\mathbf{s}%
\]
where $d\mathbf{s}$ is short hand for $ds_{1}\dots ds_{m}$ in the above
formula. Since
\[
\dot{\xi}\left(  s_{1}\right)  \dots\dot{\xi}\left(  s_{\kappa+1}\right)
=0\in\mathcal{A},
\]
it follows that $R_{\kappa+1}\left(  t\right)  =0$ and so Eq. (\ref{e.3.13})
with $m=\kappa+1$ gives Eq. (\ref{e.3.12}).
\end{proof}

\begin{corollary}
\label{cor.3.15}If $g^{\xi}\left(  t\right)  \in G$ is as in Definition
\ref{def.3.6}, then
\begin{equation}
\operatorname{Ad}_{g^{\xi}\left(  t\right)  }=I+\sum_{k=1}^{\kappa}\int_{0\leq
s_{1}\leq s_{2}\leq\dots\leq s_{k}\leq t}\operatorname{ad}_{\dot{\xi}\left(
s_{1}\right)  }\dots\operatorname{ad}_{\dot{\xi}\left(  s_{k}\right)  }%
ds_{1}\dots ds_{k}. \label{e.3.14}%
\end{equation}

\end{corollary}

\begin{proof}
Since
\[
\frac{d}{dt}\operatorname{Ad}_{g^{\xi}\left(  t\right)  }=\operatorname{Ad}%
_{g^{\xi}\left(  t\right)  }\operatorname{ad}_{\dot{\xi}\left(  t\right)
}\text{ with }\operatorname{Ad}_{g^{\xi}\left(  t\right)  }=Id_{\mathfrak{g}%
},
\]
the proof of Eq. (\ref{e.3.14}) is exactly the same as the proof of Eq.
(\ref{e.3.12}) provided the reader changes $g^{\xi}$ to $\operatorname{Ad}%
_{g^{\xi}}$ and $\dot{\xi}$ to $\operatorname{ad}_{\dot{\xi}}$ everywhere.
\end{proof}

\begin{notation}
\label{not.3.16}For $j\in\mathbb{N},$ $a:[0,\infty)^{j}\rightarrow\mathbb{R}$
is a bounded measurable function, $t\in\left[  0,T\right]  ,$ and $\xi\in
L^{1}\left(  \left[  0,T\right]  ,\mathfrak{g}\right)  ,$ let
\[
\hat{a}_{t}\left(  \xi\right)  =\int_{\left[  0,t\right]  ^{j}}a\left(
s_{1},\dots,s_{j}\right)  \xi\left(  s_{1}\right)  \dots\xi\left(
s_{j}\right)  ds_{1}\dots ds_{j}\in\mathfrak{g}%
\]
and $\hat{a}_{t}\left(  \operatorname{ad}_{\xi}\right)  :\mathfrak{g}%
\rightarrow\mathfrak{g}$ be the linear transformation defined by%
\[
\hat{a}_{t}\left(  \operatorname{ad}_{\xi}\right)  :=\int_{\left[  0,t\right]
^{j}}a\left(  s_{1},\dots,s_{j}\right)  \operatorname{ad}_{\xi\left(
s_{1}\right)  }\dots\operatorname{ad}_{\xi\left(  s_{j}\right)  }ds_{1}\dots
ds_{j}.
\]
Note that $\hat{a}_{t}\left(  \xi\right)  =0$ if $j>\kappa$ and $\hat{a}%
_{t}\left(  \operatorname{ad}_{\xi}\right)  \equiv0$ if $j\geq\kappa.$
\end{notation}

The proof of the following lemma is elementary and is left to the reader.

\begin{lemma}
\label{lem.3.17}If $a:[0,\infty)^{j}\rightarrow\mathbb{R}$ and $b:[0,\infty
)^{k}\rightarrow\mathbb{R}$ are bounded and measurable functions, $t\in\left[
0,T\right]  ,$ and $\xi\in L^{1}\left(  \left[  0,T\right]  ,\mathfrak{g}%
\right)  ,$ then
\begin{align*}
\hat{a}_{t}\left(  \xi\right)  \hat{b}_{t}\left(  \xi\right)   &
=\widehat{\left[  a\otimes b\right]  }_{t}\left(  \xi\right)  \text{ and }\\
\hat{a}_{t}\left(  \operatorname{ad}_{\xi}\right)  \hat{b}_{t}\left(
\operatorname{ad}_{\xi}\right)   &  =\widehat{\left[  a\otimes b\right]  }%
_{t}\left(  \operatorname{ad}_{\xi}\right)
\end{align*}
where $a\otimes b:[0,\infty)^{j+k}\rightarrow\mathbb{R}$ is the bounded
measurable function defined by
\[
a\otimes b\left(  s_{1},\dots,s_{j},t_{1},\dots,t_{k}\right)  =a\left(
s_{1},\dots,s_{j}\right)  b\left(  t_{1},\dots,t_{k}\right)  .
\]

\end{lemma}

\begin{proposition}
\label{pro.3.18}If $g\left(  t\right)  =g^{\xi}\left(  t\right)  \in G$ and
$C\left(  t\right)  =C^{\xi}\left(  t\right)  =\log\left(  g\left(  t\right)
\right)  \in\mathfrak{g}$ are as in Definition \ref{def.3.6} and
$f\in\mathcal{H}_{0},\ $then for each $j\in\mathbb{N\cap}\left[
1,\kappa-1\right]  ,$ there exists bounded measurable functions,
$\mathbf{f}^{j}:[0,\infty)^{j}\rightarrow\mathbb{R}$ such that
\begin{equation}
f\left(  \operatorname{ad}_{C\left(  t\right)  }\right)  =f\left(  0\right)
Id_{\mathfrak{g}}+\sum_{j=1}^{\kappa-1}\widehat{\mathbf{f}^{j}}_{t}\left(
\operatorname{ad}_{\dot{\xi}}\right)  . \label{e.3.15}%
\end{equation}
Moreover, each function $\mathbf{f}^{j}$ depends linearly on $\left(  f\left(
0\right)  ,\dots,f^{\left(  \kappa-1\right)  }\left(  0\right)  \right)
.$\footnote{The fact that the $\mathbf{f}^{j}$ depend linearly on first
$\left(  \kappa-1\right)  $-derivatives of $f$ is easily understood from the
identity,$f\left(  \operatorname{ad}_{C\left(  t\right)  }\right)  =\sum
_{j=0}^{\kappa-1}\left(  f^{\left(  j\right)  }\left(  0\right)  /j!\right)
\operatorname{ad}_{C\left(  t\right)  }^{j}.$}
\end{proposition}

\begin{proof}
For $\lambda\in\mathbb{R},$ let $u\left(  w\right)  :=f\left(  \log\left(
1+w\right)  \right)  $ and observe by a simple exercise in differentiation
shows there exists $\alpha_{n,k}\in\mathbb{Z}$ such that $u^{\left(  n\right)
}\left(  0\right)  =\sum_{k=0}^{n}\alpha_{n,k}f^{\left(  k\right)  }\left(
0\right)  .$ [For example, one has $u\left(  0\right)  =f\left(  0\right)  ,$
$u^{\prime}\left(  0\right)  =f^{\prime}\left(  0\right)  ,$ $u^{\prime\prime
}\left(  0\right)  =f^{\prime\prime}\left(  0\right)  -f^{\prime}\left(
0\right)  ,$ and $u^{\left(  3\right)  }\left(  0\right)  =f^{\left(
3\right)  }\left(  0\right)  -3f^{\prime\prime}\left(  0\right)  +2f^{\prime
}\left(  0\right)  .]$

Since $g\left(  t\right)  =e^{C\left(  t\right)  },$ it follows that
$\operatorname{Ad}_{g\left(  t\right)  }=\operatorname{Ad}_{e^{C\left(
t\right)  }}=e^{\operatorname{ad}_{C\left(  t\right)  }}$ and therefore
\[
\operatorname{ad}_{C\left(  t\right)  }=\log\left(  \operatorname{Ad}%
_{g\left(  t\right)  }\right)  =\log\left(  Id_{\mathfrak{g}}+\left[
\operatorname{Ad}_{g\left(  t\right)  }-Id_{\mathfrak{g}}\right]  \right)
\]
and hence,
\begin{align*}
f\left(  \operatorname{ad}_{C\left(  t\right)  }\right)   &  =f\circ
\log\left(  Id_{\mathfrak{g}}+\left[  \operatorname{Ad}_{g\left(  t\right)
}-Id_{\mathfrak{g}}\right]  \right) \\
&  =u\left(  I+\left[  \operatorname{Ad}_{g\left(  t\right)  }-I\right]
\right) \\
&  =\sum_{j=0}^{\infty}\frac{u^{\left(  j\right)  }\left(  0\right)  }%
{j!}\left(  \operatorname{Ad}_{g\left(  t\right)  }-I\right)  ^{j}\\
&  =f\left(  0\right)  Id_{\mathfrak{g}}+\sum_{j=1}^{\kappa-1}\frac{u^{\left(
j\right)  }\left(  0\right)  }{j!}\left(  \operatorname{Ad}_{g\left(
t\right)  }-Id_{\mathfrak{g}}\right)  ^{j}.
\end{align*}
By Corollary \ref{cor.3.15},
\[
\operatorname{Ad}_{g\left(  t\right)  }-Id_{\mathfrak{g}}=\sum_{k=1}%
^{\kappa-1}\hat{b}_{t}^{k}\left(  \operatorname{ad}_{\dot{\xi}}\right)
\]
where
\[
b^{k}\left(  s_{1},\dots,s_{k}\right)  :=1_{0\leq s_{1}\leq s_{2}\leq\dots\leq
s_{k}}%
\]
and so it follows that%
\[
f\left(  \operatorname{ad}_{C\left(  t\right)  }\right)  =f\left(  0\right)
I+\sum_{j=1}^{\kappa-1}\frac{u^{\left(  j\right)  }\left(  0\right)  }{j!}%
\sum_{k_{1},\dots,k_{j}=1}^{\kappa-1}\hat{b}_{t}^{k_{1}}\left(
\operatorname{ad}_{\dot{\xi}}\right)  \dots\hat{b}_{t}^{k_{j}}\left(
\operatorname{ad}_{\dot{\xi}}\right)  .
\]
By repeated use of Lemma \ref{lem.3.17}, the last identity may be written in
the form described in Eq. (\ref{e.3.15}).

The formula for $\dot{C}^{\xi}\left(  t\right)  $ now follows directly from
Eqs. (\ref{e.3.11}) and (\ref{e.3.16}).
\end{proof}

\begin{corollary}
\label{cor.3.19}If $g\left(  t\right)  =g^{\xi}\left(  t\right)  \in G$ and
$C\left(  t\right)  =C^{\xi}\left(  t\right)  =\log\left(  g\left(  t\right)
\right)  \in\mathfrak{g}$ are as in Definition \ref{def.3.6}, then there
exists bounded measurable functions, $\Delta^{j}:[0,\infty)^{j-1}%
\rightarrow\mathbb{R}$ for $j\in\mathbb{N\cap}\left[  2,\kappa\right]  $ such
that%
\begin{equation}
\psi_{-}\left(  \operatorname{ad}_{C\left(  t\right)  }\right)
=Id_{\mathfrak{g}}+\sum_{j=2}^{\kappa}\hat{\Delta}_{t}^{j}\left(
\operatorname{ad}_{\xi}\right)  \label{e.3.16}%
\end{equation}
and
\begin{equation}
\dot{C}^{\xi}\left(  t\right)  =\dot{\xi}\left(  t\right)  +\sum_{j=2}%
^{\kappa}\hat{\Delta}_{t}^{j}\left(  \operatorname{ad}_{\xi}\right)  \dot{\xi
}\left(  t\right)  \label{e.3.17}%
\end{equation}

\end{corollary}

\begin{proof}
Applying Proposition \ref{pro.3.18} with $f=\psi_{-}$ and $\lambda=1$ gives
Eq. (\ref{e.3.16}). Equation (\ref{e.3.17}) then follows from Eq.
(\ref{e.3.16}) and Eq. (\ref{e.3.11}).
\end{proof}

\begin{remark}
\label{rem.3.20}It is possible, see for example \cite{Strichartz1987a}, to
work out explicit formula for the functions $\Delta^{j}$ in Corollary
\ref{cor.3.19} and this would lead to a proof of Eq. (\ref{e.1.2}). For our
purposes, these explicit formula are not needed.
\end{remark}

\begin{corollary}
\label{cor.3.21}If $g\left(  t\right)  =g^{\xi}\left(  t\right)  \in G$ and
$C\left(  t\right)  =C^{\xi}\left(  t\right)  =\log\left(  g\left(  t\right)
\right)  \in\mathfrak{g}$ are as in Definition \ref{def.3.6}, there exists
bounded measurable functions, $c^{j}:[0,\infty)^{j}\rightarrow\mathbb{R}$ for
$j\in\mathbb{N\cap}\left[  2,\kappa\right]  $ such that
\[
C^{\xi}\left(  t\right)  =C\left(  0\right)  +\xi\left(  t\right)  +\sum
_{j=2}^{\kappa}\hat{c}_{t}^{j}\left(  \dot{\xi}\right)  .
\]

\end{corollary}

\begin{proof}
Integrating Eq. (\ref{e.3.17}) shows,%
\begin{equation}
C^{\xi}\left(  t\right)  =C\left(  0\right)  +\xi\left(  t\right)  +\sum
_{j=2}^{\kappa}\int_{0}^{t}\hat{\Delta}_{\tau}^{j}\left(  \operatorname{ad}%
_{\dot{\xi}}\right)  \dot{\xi}\left(  \tau\right)  d\tau\label{e.3.18}%
\end{equation}
where%
\begin{align*}
\int_{0}^{t}  &  \hat{\Delta}_{\tau}^{j}\left(  \operatorname{ad}_{\dot{\xi}%
}\right)  \dot{\xi}\left(  \tau\right)  d\tau\\
&  =\int_{0}^{t}d\tau\int_{\left[  0,\tau\right]  ^{j-1}}ds_{1}\dots
ds_{j-1}\Delta^{j}\left(  s_{1},\dots,s_{j-1}\right)  \operatorname{ad}%
_{\dot{\xi}\left(  s_{1}\right)  }\dots\operatorname{ad}_{\dot{\xi}\left(
s_{j-1}\right)  }\dot{\xi}\left(  \tau\right) \\
&  =\int_{\left[  0,t\right]  ^{j}}\tilde{\Delta}^{j}\left(  s_{1},\dots
,s_{j}\right)  \operatorname{ad}_{\dot{\xi}\left(  s_{1}\right)  }%
\dots\operatorname{ad}_{\dot{\xi}\left(  s_{j-1}\right)  }\dot{\xi}\left(
s_{j}\right)  ds_{1}\dots ds_{j}%
\end{align*}
and
\[
\tilde{\Delta}^{j}\left(  s_{1},\dots,s_{j}\right)  :=\Delta^{j}\left(
s_{1},\dots,s_{j-1}\right)  1_{\max\left\{  s_{1},\dots s_{j-1}\right\}  \leq
s_{j}}.
\]
By expanding out all of the commutators and permuting the variables of
integration in each of the resulting terms we may rewrite the previous
expression in the form $\hat{c}_{t}^{j}\left(  \xi\right)  $ for some bounded
measurable function, $c^{j}:[0,\infty)^{j}\rightarrow\mathbb{R}.$

\textbf{Alternatively}: simply apply $\log$ to Eq. (\ref{e.3.12}) and then
repeatedly use Lemma \ref{lem.3.17} to arrive at the stated assertion.
\end{proof}

\subsection{Truncated tensor algebra estimates\label{sec.3.2}}

We now apply the above results with $\mathfrak{g}=\mathfrak{g}^{\left(
\kappa\right)  }\subset\mathcal{A}=T^{\left(  \kappa\right)  }\left(
\mathbb{R}^{d}\right)  $ as in Notation \ref{not.1.14}. In what follows we
will make use of the simple estimates in the following remark without further mention.

\begin{remark}
\label{rem.3.22}For any $m,n\in\mathbb{N\cap}\left[  1,2\kappa\right]  $ with
$m<n,$ it is easy to show, for $\mu,\lambda\geq0,$ that
\begin{align}
Q_{(m,n]}\left(  \lambda\right)   &  \asymp\sum_{k=m+1}^{n}\lambda
^{k},\nonumber\\
Q_{\left[  m,n\right]  }\left(  \lambda\right)   &  \asymp\sum_{k=m}%
^{n}\lambda^{k},\text{ and}\nonumber\\
Q_{(m,n]}\left(  \lambda+\mu\right)   &  \asymp Q_{(m,n]}\left(
\lambda\right)  +Q_{(m,n]}\left(  \mu\right)  . \label{e.3.19}%
\end{align}
For example, the first estimate follows from the more precise estimate,%
\[
Q_{(m,n]}\left(  \lambda\right)  \leq\sum_{k=m+1}^{n}\lambda^{k}\leq\left(
n-m\right)  Q_{(m,n]}\left(  \lambda\right)  .
\]

\end{remark}

Recalling from Definition \ref{def.1.16} that $N\left(  A\right)
:=\max_{1\leq k\leq\kappa}\left\vert A_{k}\right\vert ^{1/k}$ for
$A\in\mathfrak{g}^{\left(  \kappa\right)  },$ we find
\begin{equation}
\left\vert A\right\vert \leq\sum_{k=1}^{\kappa}\left\vert A_{k}\right\vert
\leq\sum_{k=1}^{\kappa}N\left(  A\right)  ^{k}\leq\kappa Q_{\left[
1,\kappa\right]  }\left(  N\left(  A\right)  \right)  . \label{e.3.20}%
\end{equation}
Similarly if $f\in C\left(  \left[  0,t\right]  ,\mathfrak{g}^{\left(
\kappa\right)  }\right)  ,$ then
\begin{equation}
\left\vert f\right\vert _{t}^{\ast}\leq\sum_{k=1}^{\kappa}\left\vert
f_{k}\right\vert _{t}^{\ast}\leq\sum_{k=1}^{\kappa}N_{t}^{\ast}\left(
f\right)  ^{k}\leq\kappa Q_{\left[  1,\kappa\right]  }\left(  N_{t}^{\ast
}\left(  f\right)  \right)  . \label{e.3.21}%
\end{equation}
Let us also recall that if $a=\left(  a_{j}\right)  _{j=1}^{N}$ is a sequence
($N=\infty$ allowed), then
\[
\left\Vert a\right\Vert _{p}:=\left(  \sum_{j=1}^{N}\left\vert a_{j}%
\right\vert ^{p}\right)  ^{1/p}%
\]
is a decreasing function of $p\in\lbrack1,\infty).$ In particular using
$\left\Vert a\right\Vert _{p}\leq\left\Vert a\right\Vert _{1}$ with $a_{j}$
replaced by $a_{j}^{1/p}$ it follows (as is easily proved directly) that
\begin{equation}
\left(  \sum_{i=1}^{m}a_{j}\right)  ^{1/p}\leq\sum_{j=1}^{N}a_{j}^{1/p}\text{
when }a_{j}\geq0\text{ and }p\geq1. \label{e.3.22}%
\end{equation}

\begin{lemma}
\label{lem.3.23}If $\left\{  A\left(  j\right)  \right\}  _{j=1}^{r}%
\subset\mathfrak{g}^{\left(  \kappa\right)  },$ then%
\begin{equation}
N\left(  \sum_{j=1}^{r}A\left(  j\right)  \right)  \leq\sum_{j=1}^{r}N\left(
A\left(  j\right)  \right)  . \label{e.3.23}%
\end{equation}
If $A,B\in\mathfrak{g}^{\left(  \kappa\right)  }$ and $2\leq k\leq2\kappa,$
then%
\begin{equation}
\left\vert \left[  A\otimes B\right]  _{k}\right\vert \leq N\left(  A\right)
N\left(  B\right)  \cdot\left(  N\left(  A\right)  +N\left(  B\right)
\right)  ^{k-2}. \label{e.3.24}%
\end{equation}

\end{lemma}

\begin{proof}
For $1\leq k\leq\kappa,$%
\[
\left\vert \left[  \sum_{j=1}^{r}A\left(  j\right)  \right]  _{k}\right\vert
^{1/k}\leq\left(  \sum_{j=1}^{r}\left\vert A\left(  j\right)  _{k}\right\vert
\right)  ^{1/k}\leq\sum_{j=1}^{r}\left\vert A\left(  j\right)  _{k}\right\vert
^{1/k}\leq\sum_{j=1}^{r}N\left(  A\left(  j\right)  \right)
\]
wherein we have used Eq. (\ref{e.3.22}) with $p=k$ for the second inequality.
Since this is true for all $1\leq k\leq\kappa,$ Eq. (\ref{e.3.23}) is proved.
The proof of the second inequality follows by the simple estimates;%
\begin{align*}
\left\vert \left[  A\otimes B\right]  _{k}\right\vert  &  =\left\vert
\sum_{m,n=1}^{\kappa}1_{m+n=k}\cdot A_{m}\otimes B_{n}\right\vert \leq
\sum_{m,n=1}^{\kappa}1_{m+n=k}\cdot\left\vert A_{m}\otimes B_{n}\right\vert \\
&  \leq\sum_{m,n=1}^{\kappa}1_{m+n=k}\cdot\left\vert A_{m}\right\vert
\left\vert B_{n}\right\vert \leq\sum_{m,n=1}^{\kappa}1_{m+n=k}\cdot N\left(
A\right)  ^{m}N\left(  B\right)  ^{n}\\
&  =N\left(  A\right)  N\left(  B\right)  \cdot\sum_{m,n=0}^{\kappa
-1}1_{m+n=k-2}\cdot N\left(  A\right)  ^{m}N\left(  B\right)  ^{n}\\
&  \leq N\left(  A\right)  N\left(  B\right)  \left(  N\left(  A\right)
+N\left(  B\right)  \right)  ^{k-2},
\end{align*}
wherein we have used all the coefficients in the binomial formula are greater
than or equal to $1$ for the last inequality.
\end{proof}

Recall from Notation \ref{not.3.16} with $\mathfrak{g}=\mathfrak{g}^{\left(
\kappa\right)  }$ that if $1\leq\ell\leq\kappa,$ $\Delta:\left[  0,T\right]
^{\ell}\rightarrow\mathbb{R}$ is a bounded measurable function, and $\xi\in
L^{1}\left(  \left[  0,T\right]  ,\mathfrak{g}^{\left(  \kappa\right)
}\right)  ,$ then we let%
\begin{equation}
\hat{\Delta}_{t}\left(  \xi\right)  :=\int_{\left[  0,t\right]  ^{\ell}}%
\Delta\left(  s_{1},\dots,s_{\ell}\right)  \xi\left(  s_{1}\right)  \dots
\xi\left(  s_{\ell}\right)  d\mathbf{s}\in\mathfrak{g}^{\left(  \kappa\right)
}\text{ }\forall~t\in\left[  0,T\right]  , \label{e.3.25}%
\end{equation}
where $d\mathbf{s}:=ds_{1}\dots ds_{\ell}.$

\begin{proposition}
\label{pro.3.24}Suppose that $1\leq\ell\leq\kappa,$ $\Delta:\left[
0,T\right]  ^{\ell}\rightarrow\mathbb{R},$ and $\xi\in L^{1}\left(  \left[
0,T\right]  ,\mathfrak{g}^{\left(  \kappa\right)  }\right)  ,$ and
\[
\hat{\Delta}_{t}\left(  \xi\right)  =\sum_{k=1}^{\kappa}\left[  \hat{\Delta
}_{t}\left(  \xi\right)  \right]  _{k}\in\oplus_{k=\ell}^{\kappa}\left[
\mathbb{R}^{d}\right]  ^{\otimes k}%
\]
are as above. Then%
\begin{equation}
\left\vert \left[  \hat{\Delta}_{t}\left(  \xi\right)  \right]  _{k}%
\right\vert \leq\#\left(  \Lambda_{k,\ell}\right)  \cdot\left\Vert
\Delta\right\Vert _{\infty}\cdot N_{t}^{\ast}\left(  \xi\right)  ^{k}
\label{e.3.26}%
\end{equation}%
\begin{equation}
N\left(  \hat{\Delta}_{t}\left(  \xi\right)  \right)  \leq C\left(
\kappa\right)  \max\left(  \left\Vert \Delta\right\Vert _{\infty}^{1/\ell
},\left\Vert \Delta\right\Vert _{\infty}^{1/\kappa}\right)  \cdot N_{t}^{\ast
}\left(  \xi\right)  \text{ for all }0\leq t\leq T \label{e.3.27}%
\end{equation}
where $\left\Vert \Delta\right\Vert _{\infty}$ is the essential supremum of
$\Delta$ on $\left[  0,T\right]  ^{\ell},$
\[
\Lambda_{k,\ell}:=\left\{  \left(  j_{1},\dots j_{\ell}\right)  \in
\mathbb{N}^{\ell}:\sum_{i=1}^{k}j_{i}=k\right\}  ,
\]
and%
\[
C\left(  \kappa\right)  :=\max_{1\leq\ell\leq\kappa}\max_{\ell\leq k\leq
\kappa}\left[  \left[  \#\left(  \Lambda_{k,\ell}\right)  \right]
^{1/k}\right]  .
\]

\end{proposition}

\begin{proof}
For $k\in\left[  \ell,\kappa\right]  \cap\mathbb{N},$
\begin{align*}
\left\vert \left[  \hat{\Delta}_{t}\left(  \xi\right)  \right]  _{k}%
\right\vert  &  =\left\vert \sum_{\left(  j_{1},\dots j_{\ell}\right)
\in\Lambda_{k,\ell}}\int_{\left[  0,t\right]  ^{\ell}}\Delta\left(
s_{1},\dots,s_{\ell}\right)  \xi_{j_{1}}\left(  s_{1}\right)  \dots
\xi_{j_{\ell}}\left(  s_{\ell}\right)  d\mathbf{s}\right\vert \\
&  \leq\left\Vert \Delta\right\Vert _{\infty}\sum_{\left(  j_{1},\dots
j_{\ell}\right)  \in\Lambda_{k,\ell}}\int_{\left[  0,t\right]  ^{\ell}%
}\left\vert \xi_{j_{1}}\left(  s_{1}\right)  \dots\xi_{j_{\ell}}\left(
s_{\ell}\right)  \right\vert d\mathbf{s}\\
&  =\left\Vert \Delta\right\Vert _{\infty}\sum_{\left(  j_{1},\dots j_{\ell
}\right)  \in\Lambda_{k,\ell}}\prod_{i=1}^{\ell}\left\vert \xi_{j_{i}%
}\right\vert _{t}^{\ast}\leq\left\Vert \Delta\right\Vert _{\infty}%
\sum_{\left(  j_{1},\dots j_{\ell}\right)  \in\Lambda_{k,\ell}}\prod
_{i=1}^{\ell}N_{t}^{\ast}\left(  \xi\right)  ^{j_{i}}\\
&  \leq\left\Vert \Delta\right\Vert _{\infty}\cdot\#\left(  \Lambda_{k,\ell
}\right)  \cdot N_{t}^{\ast}\left(  \xi\right)  ^{k}%
\end{align*}
which proves Eq. (\ref{e.3.26}). Equation (\ref{e.3.27}) is an easy
consequence of Eq. (\ref{e.3.26}) and the observation that
\[
\left[  \#\left(  \Lambda_{k,\ell}\right)  \right]  ^{1/k}\cdot\left\Vert
\Delta\right\Vert _{\infty}^{1/k}\leq C\left(  \kappa\right)  \max\left(
\left\Vert \Delta\right\Vert _{\infty}^{1/\ell},\left\Vert \Delta\right\Vert
_{\infty}^{1/\kappa}\right)  .
\]

\end{proof}

\begin{proposition}
\label{pro.3.25}Suppose that $1\leq\ell\leq\kappa,$ $\Delta:\left[
0,T\right]  ^{\ell}\rightarrow\mathbb{R},$ and $\xi\in L^{1}\left(  \left[
0,T\right]  ,\mathfrak{g}^{\left(  \kappa\right)  }\right)  ,$ then%
\[
\int_{0}^{T}\left\vert \left[  \hat{\Delta}_{t}\left(  \operatorname{ad}_{\xi
}\right)  \xi\left(  t\right)  \right]  _{k}\right\vert dt\lesssim\left\Vert
\Delta\right\Vert _{\infty}N_{T}^{\ast}\left(  \xi\right)  ^{k}.
\]

\end{proposition}

\begin{proof}
For $k\in(\ell,\kappa]\cap\mathbb{N},$ let
\[
\Lambda_{k,\ell}:=\left\{  \left(  j_{0},j_{1},\dots j_{\ell}\right)
\in\mathbb{N}^{\ell+1}:\sum_{i=0}^{k}j_{i}=k\right\}  .
\]
We then have%
\begin{align*}
&  \left\vert \left[  \hat{\Delta}_{t}\left(  \operatorname{ad}_{\xi}\right)
\xi\left(  t\right)  \right]  _{k}\right\vert \\
&  =\left\vert \sum_{\left(  j_{0},j_{1},\dots j_{\ell}\right)  \in\Lambda
}\int_{\left[  0,t\right]  ^{\ell}}\Delta\left(  s_{1},\dots,s_{\ell}\right)
\operatorname{ad}_{\xi_{j_{1}}\left(  s_{1}\right)  }\dots\operatorname{ad}%
_{\xi_{j_{\ell}}\left(  s_{\ell}\right)  }\xi_{j_{0}}\left(  t\right)
d\mathbf{s}\right\vert \\
&  \leq2^{\ell}\left\Vert \Delta\right\Vert _{\infty}\sum_{\left(  j_{0}%
,j_{1},\dots j_{\ell}\right)  \in\Lambda_{k,\ell}}\int_{\left[  0,t\right]
^{\ell}}\left\vert \xi_{j_{1}}\left(  s_{1}\right)  \right\vert \dots
\left\vert \xi_{j_{\ell}}\left(  s_{\ell}\right)  \right\vert \left\vert
\xi_{j_{0}}\left(  t\right)  \right\vert d\mathbf{s}\\
&  \leq2^{\ell}\left\Vert \Delta\right\Vert _{\infty}\sum_{\left(  j_{0}%
,j_{1},\dots j_{\ell}\right)  \in\Lambda_{k,\ell}}\int_{\left[  0,T\right]
^{\ell}}\left\vert \xi_{j_{1}}\left(  s_{1}\right)  \right\vert \dots
\left\vert \xi_{j_{\ell}}\left(  s_{\ell}\right)  \right\vert \left\vert
\xi_{j_{0}}\left(  t\right)  \right\vert d\mathbf{s}%
\end{align*}
Integrating this estimate on $t\in\left[  0,T\right]  $ shows,%
\begin{align*}
\int_{0}^{T}  &  \left\vert \left[  \hat{\Delta}_{t}\left(  \operatorname{ad}%
_{\xi}\right)  \xi\left(  t\right)  \right]  _{k}\right\vert dt\\
&  \leq2^{\ell}\left\Vert \Delta\right\Vert _{\infty}\sum_{\left(  j_{0}%
,j_{1},\dots j_{\ell}\right)  \in\Lambda_{k,\ell}}\int_{0}^{T}dt\int_{\left[
0,T\right]  ^{\ell}}\left\vert \xi_{j_{1}}\left(  s_{1}\right)  \right\vert
\dots\left\vert \xi_{j_{\ell}}\left(  s_{\ell}\right)  \right\vert \left\vert
\xi_{j_{0}}\left(  t\right)  \right\vert d\mathbf{s}\\
&  =2^{\ell}\left\Vert \Delta\right\Vert _{\infty}\sum_{\left(  j_{0}%
,j_{1},\dots j_{\ell}\right)  \in\Lambda_{k,\ell}}\prod_{i=0}^{\ell}\left\vert
\xi_{j_{i}}\right\vert _{T}^{\ast}\\
&  \leq2^{\ell}\left\Vert \Delta\right\Vert _{\infty}\sum_{\left(  j_{0}%
,j_{1},\dots j_{\ell}\right)  \in\Lambda_{k,\ell}}N_{t}^{\ast}\left(
\xi\right)  ^{k}=2^{\ell}\#\left(  \Lambda_{k,\ell}\right)  \left\Vert
\Delta\right\Vert _{\infty}\cdot N_{t}^{\ast}\left(  \xi\right)  ^{k}.
\end{align*}

\end{proof}

We end this section with a few key estimates that we will need in the
remainder of the paper. In each of the next three results we assume that
$\xi\in C^{1}\left(  \left[  0,T\right]  ,F^{\left(  \kappa\right)  }\left(
\mathbb{R}^{d}\right)  \right)  $ and $C\left(  t\right)  =C^{\xi}\left(
t\right)  =\log\left(  g^{\xi}\left(  t\right)  \right)  \in F^{\left(
\kappa\right)  }\left(  \mathbb{R}^{d}\right)  $ are as in Definition
\ref{def.1.25}.

\begin{proposition}
\label{pro.3.26}To each $f\in\mathcal{H}_{0}$ there exists $K\left(  f\right)
>0$ depending linearly on $\left(  \left\vert f\left(  0\right)  \right\vert
,\dots,\left\vert f^{\left(  \kappa-1\right)  }\left(  0\right)  \right\vert
\right)  $ and independent of $\xi$ such that%
\begin{equation}
\int_{0}^{T}\left\vert \left(  f\left(  \operatorname{ad}_{C\left(  t\right)
}\right)  \dot{\xi}\left(  t\right)  \right)  _{n}\right\vert dt\leq K\left(
f\right)  N_{T}^{\ast}\left(  \dot{\xi}\right)  ^{n} \label{e.3.28}%
\end{equation}

\end{proposition}

\begin{proof}
Recall that Proposition \ref{pro.3.18} asserts that%
\[
f\left(  \operatorname{ad}_{C\left(  t\right)  }\right)  \dot{\xi}\left(
t\right)  =f\left(  0\right)  \dot{\xi}\left(  t\right)  +\sum_{j=1}%
^{\kappa-1}\widehat{\mathbf{f}^{j}}_{t}\left(  \operatorname{ad}_{\dot{\xi}%
}\right)  \dot{\xi}\left(  t\right)
\]
where the functions $\mathbf{f}^{j}$ depend linearly on $\left(  f\left(
0\right)  ,\dots,f^{\left(  \kappa-1\right)  }\left(  0\right)  \right)  .$ So
by repeated application of Proposition \ref{pro.3.25} with $\xi$ replaced by
$\dot{\xi}$ shows,%
\begin{align*}
\int_{0}^{T}  &  \left\vert \left(  f\left(  \operatorname{ad}_{C\left(
t\right)  }\right)  \dot{\xi}\left(  t\right)  \right)  _{n}\right\vert dt\\
&  \leq\left\vert f\left(  0\right)  \right\vert \int_{0}^{T}\left\vert
\dot{\xi}_{n}\left(  t\right)  \right\vert dt+\sum_{j=1}^{\kappa-1}\int%
_{0}^{T}\left\vert \left[  \widehat{\mathbf{f}^{j}}_{t}\left(
\operatorname{ad}_{\dot{\xi}}\right)  \dot{\xi}\left(  t\right)  \right]
_{n}\right\vert dt\\
&  \lesssim\left\vert f\left(  0\right)  \right\vert N_{T}^{\ast}\left(
\dot{\xi}\right)  ^{n}+\sum_{j=1}^{\kappa-1}\left\Vert \mathbf{f}%
^{j}\right\Vert _{\infty}N_{T}^{\ast}\left(  \dot{\xi}\right)  ^{n}\leq
K\left(  f\right)  N_{T}^{\ast}\left(  \dot{\xi}\right)  ^{n}%
\end{align*}
where $K\left(  f\right)  >0$ may be chosen to depend linearly on $\left(
\left\vert f\left(  0\right)  \right\vert ,\dots,\left\vert f^{\left(
\kappa-1\right)  }\left(  0\right)  \right\vert \right)  .$
\end{proof}

\begin{corollary}
\label{cor.3.27}If $\xi\in C^{1}\left(  \left[  0,T\right]  ,F^{\left(
\kappa\right)  }\left(  \mathbb{R}^{d}\right)  \right)  $ and $C^{\xi}\left(
t\right)  =\log\left(  g^{\xi}\left(  t\right)  \right)  \in F^{\left(
\kappa\right)  }\left(  \mathbb{R}^{d}\right)  $ are as above, then%
\begin{equation}
N_{T}^{\ast}\left(  \dot{C}^{\xi}\right)  \lesssim N_{T}^{\ast}\left(
\dot{\xi}\right)  , \label{e.3.29}%
\end{equation}%
\begin{equation}
\left\vert C^{\xi}\left(  \cdot\right)  _{n}\right\vert _{\infty,T}\lesssim
N_{T}^{\ast}\left(  \dot{\xi}\right)  ^{n}\text{ }\forall~n\in\left[
1,\kappa\right]  \cap\mathbb{N},\text{ and} \label{e.3.30}%
\end{equation}%
\begin{equation}
\left\vert C^{\xi}\left(  \cdot\right)  \right\vert _{\infty,T}\lesssim
Q_{\left[  1,\kappa\right]  }\left(  N_{T}^{\ast}\left(  \dot{\xi}\right)
\right)  . \label{e.3.31}%
\end{equation}

\end{corollary}

\begin{proof}
By Corollary \ref{cor.3.10}, $\dot{C}^{\xi}\left(  t\right)  =f\left(
\operatorname{ad}_{C\left(  t\right)  }\right)  \dot{\xi}\left(  t\right)  $
where $f\left(  z\right)  =1/\psi\left(  -z\right)  $ and so by Proposition
\ref{pro.3.26},%
\[
\left\vert \dot{C}_{n}^{\xi}\right\vert _{T}^{\ast}=\int_{0}^{T}\left\vert
\left(  f\left(  \operatorname{ad}_{C\left(  t\right)  }\right)  \dot{\xi
}\left(  t\right)  \right)  _{n}\right\vert dt\leq K\left(  f\right)
N_{T}^{\ast}\left(  \dot{\xi}\right)  ^{n}\text{ for }1\leq n\leq\kappa.
\]
This proves Eq. (\ref{e.3.29}) and also Eqs. (\ref{e.3.30}) and (\ref{e.3.31})
since (as $C^{\xi}\left(  0\right)  =0),$%
\[
\left\vert C^{\xi}\left(  \cdot\right)  _{n}\right\vert _{\infty,T}%
\leq\left\vert \dot{C}_{n}^{\xi}\right\vert _{T}^{\ast}\leq K\left(  f\right)
N_{T}^{\ast}\left(  \dot{\xi}\right)  ^{n}%
\]
and hence
\[
\left\vert C^{\xi}\left(  \cdot\right)  \right\vert _{\infty,T}\leq\sum
_{n=1}^{\kappa}\left\vert C_{n}^{\xi}\left(  \cdot\right)  \right\vert
_{\infty,T}\lesssim\sum_{n=1}^{\kappa}N_{T}^{\ast}\left(  \dot{\xi}\right)
^{n}\lesssim Q_{\left[  1,\kappa\right]  }\left(  N_{T}^{\ast}\left(  \dot
{\xi}\right)  \right)  .
\]

\end{proof}

\begin{corollary}
\label{cor.3.28}Suppose $\xi\in C^{1}\left(  \left[  0,T\right]  ,F^{\left(
\kappa\right)  }\left(  \mathbb{R}^{d}\right)  \right)  $ and $C\left(
t\right)  =C^{\xi}\left(  t\right)  =\log\left(  g^{\xi}\left(  t\right)
\right)  \in F^{\left(  \kappa\right)  }\left(  \mathbb{R}^{d}\right)  $ are
as in Definition \ref{def.1.25} and $f\in\mathcal{H}_{0}.$ Then there exists
$K\left(  f\right)  >0$ such that $K\left(  f\right)  $ depends linearly on
$\left(  \left\vert f\left(  0\right)  \right\vert ,\dots,\left\vert
f^{\left(  \kappa-1\right)  }\left(  0\right)  \right\vert \right)  $ and on
$\kappa$ such that%
\begin{equation}
\int_{0}^{T}\left\vert C^{\xi}\left(  t\right)  _{m}\right\vert \left\vert
\left(  f\left(  \operatorname{ad}_{C\left(  t\right)  }\right)  \dot{\xi
}\left(  t\right)  \right)  _{n}\right\vert dt\leq K\left(  f\right)
N_{T}^{\ast}\left(  \dot{\xi}\right)  ^{m+n}\text{ }\forall~m,n\in\left[
1,\kappa\right]  \cap\mathbb{N}. \label{e.3.32}%
\end{equation}

\end{corollary}

\begin{proof}
Making use of the estimates in Proposition \ref{pro.3.26} and Corollary
\ref{cor.3.28} we find,
\begin{align*}
\int_{0}^{T}  &  \left\vert C^{\xi}\left(  t\right)  _{m}\right\vert
\left\vert \left(  f\left(  \operatorname{ad}_{C\left(  t\right)  }\right)
\dot{\xi}\left(  t\right)  \right)  _{n}\right\vert dt\\
&  \leq\left\vert C^{\xi}\left(  \cdot\right)  _{m}\right\vert _{\infty
,T}\cdot\int_{0}^{T}\left\vert \left(  f\left(  \operatorname{ad}_{C\left(
t\right)  }\right)  \dot{\xi}\left(  t\right)  \right)  _{n}\right\vert dt\\
&  \lesssim N_{T}^{\ast}\left(  \dot{\xi}\right)  ^{m}\cdot K\left(  f\right)
N_{T}^{\ast}\left(  \dot{\xi}\right)  ^{n}=K\left(  f\right)  N_{T}^{\ast
}\left(  \dot{\xi}\right)  ^{m+n}.
\end{align*}

\end{proof}

\section{Logarithm Approximation Problem\label{sec.4}}

Recall from Definition \ref{def.1.20} that a $d$\textbf{-dimensional dynamical
system} on $M$ is a linear map, $\mathbb{R}^{d}\ni w\rightarrow V_{w}\in
\Gamma\left(  TM\right)  .$ For $A\in F^{\left(  \kappa\right)  }\left(
\mathbb{R}^{d}\right)  $ we know that $V_{A}\in\Gamma\left(  TM\right)  $ by
Example \ref{ex.1.21}. Let us again emphasize that we assume Assumption
\ref{ass.1} is in force, i.e. $V$ is $\kappa$-complete.

To help motivate the next key theorem, let $A$ and $B$ be in the full free Lie
algebra, $F\left(  \mathbb{R}^{d}\right)  $. Working heuristically (using
$\sim$ to indicate equality of formal series), we should have
\begin{align*}
\operatorname{Ad}_{e^{sV_{A}}}V_{B}  &  =e^{s\operatorname{ad}_{V_{A}}}%
V_{B}=e^{-sL_{V_{A}}}V_{B}\\
&  \sim\sum_{k=0}^{\infty}\frac{\left(  -1\right)  ^{k}s^{k}}{k!}L_{V_{A}}%
^{k}V_{B}\sim\sum_{k=0}^{\infty}\frac{\left(  -1\right)  ^{k}s^{k}}%
{k!}V_{\operatorname{ad}_{A}^{k}B}\\
&  \sim V_{\sum_{k=0}^{\infty}\frac{\left(  -1\right)  ^{k}s^{k}}%
{k!}\operatorname{ad}_{A}^{k}B}=V_{e^{-s\operatorname{ad}_{A}}B}.
\end{align*}
Integrating this formal identity then suggests,%
\[
\int_{0}^{1}\operatorname{Ad}_{e^{sV_{A}}}V_{B}ds\sim\int_{0}^{1}%
e^{-sL_{V_{A}}}V_{B}ds\sim\int_{0}^{1}V_{e^{-s\operatorname{ad}_{A}}B}ds\sim
V_{\int_{0}^{1}e^{-s\operatorname{ad}_{A}}Bds}.
\]
Although the above series need not converge, this computation is suggestive of
the following key Taylor type approximation theorem for $\int_{0}%
^{1}\operatorname{Ad}_{e^{sV_{A}}}V_{B}ds\in\Gamma\left(  TM\right)  $ when
$A,B\in F^{\left(  \mathfrak{\kappa}\right)  }\left(  \mathbb{R}^{d}\right)
.$

\begin{theorem}
\label{thm.4.1}Let $\psi\left(  z\right)  $ be as in Eq. (\ref{e.3.3}) and
$V:\mathbb{R}^{d}\rightarrow\Gamma\left(  TM\right)  $ be a dynamical system
satisfying Assumption \ref{ass.1} so that in particular, $V_{A}\in
\Gamma\left(  TM\right)  $ is complete for all $A\in\mathfrak{g}%
_{0}=F^{\left(  \kappa\right)  }\left(  \mathbb{R}^{d}\right)  .$ Then for all
$A,B\in\mathfrak{g}_{0},$%
\begin{align}
\int_{0}^{1}  &  \operatorname{Ad}_{e^{sV_{A}}}V_{B}ds\nonumber\\
&  =V_{\left[  \int_{0}^{1}\operatorname{Ad}_{e^{-sA}}B~ds\right]  }+\int%
_{0}^{1}\operatorname{Ad}_{e^{sV_{A}}}V_{\pi_{>\kappa}\left[  A,\psi\left(
\left(  s-1\right)  \operatorname{ad}_{A}\right)  B\right]  _{\otimes}}\left(
s-1\right)  ds\label{e.4.1}\\
&  =V_{\psi\left(  -\operatorname{ad}_{A}\right)  B}+\int_{0}^{1}%
\operatorname{Ad}_{e^{sV_{A}}}V_{\pi_{>\kappa}\left[  A,\psi\left(  \left(
s-1\right)  \operatorname{ad}_{A}\right)  B\right]  _{\otimes}}\left(
s-1\right)  ds. \label{e.4.2}%
\end{align}

\end{theorem}

\begin{proof}
The heart of the proof is to show, for $0\leq l\leq\kappa,$ that%
\begin{align}
\int_{0}^{1}\operatorname{Ad}_{e^{sV_{A}}}V_{B}ds=  &  V_{\sum_{k=1}%
^{l}\left(  -1\right)  ^{k+1}\frac{\operatorname{ad}_{A}^{k-1}}{k!}B}+\frac
{1}{l!}\int_{0}^{1}\operatorname{Ad}_{e^{sV_{A}}}V_{\operatorname{ad}_{A}%
^{l}B}\left(  s-1\right)  ^{l}ds\nonumber\\
&  +\int_{0}^{1}\operatorname{Ad}_{e^{sV_{A}}}V_{\pi_{>\kappa}\left[
A,\left(  \sum_{k=1}^{l}\frac{\left(  s-1\right)  ^{k}}{k!}\operatorname{ad}%
_{A}^{k-1}\right)  B\right]  _{\otimes}}ds, \label{e.4.3}%
\end{align}
where $\sum_{k=1}^{l}\left[  \dots\right]  =0$ and $l!=1$ when $l=0.$ The
proof of these identities will be by induction on $l.$ In the proof of this
identity we will use Corollary \ref{cor.2.19} which in this context implies,%
\[
\frac{d}{ds}\operatorname{Ad}_{e^{sV_{A}}}V_{C}=\operatorname{Ad}_{e^{sV_{A}}%
}\operatorname{ad}_{V_{A}}V_{C}=-\operatorname{Ad}_{e^{sV_{A}}}\left[
V_{A},V_{C}\right]  =-\operatorname{Ad}_{e^{sV_{A}}}V_{\left[  A,C\right]
_{\otimes}}%
\]
for all $A,C\in F^{\left(  \kappa\right)  }\left(  \mathbb{R}^{d}\right)  .$
In the proof to follow, $C=\operatorname{ad}_{A}^{l}B$ for some $l.$

When $l=0,$ there is nothing to prove. For the induction step, we integrate by
parts the middle term on the right side of Eq. (\ref{e.4.3}),
\begin{align*}
\frac{1}{l!}\int_{0}^{1}  &  \operatorname{Ad}_{e^{sV_{A}}}%
V_{\operatorname{ad}_{A}^{l}B}\left(  s-1\right)  ^{l}ds\\
=  &  \frac{1}{\left(  l+1\right)  !}\int_{0}^{1}\operatorname{Ad}_{e^{sV_{A}%
}}V_{\operatorname{ad}_{A}^{l}B}d\left(  s-1\right)  ^{l+1}\\
=  &  \frac{1}{\left(  l+1\right)  !}\operatorname{Ad}_{e^{sV_{A}}%
}V_{\operatorname{ad}_{A}^{l}B}\left(  s-1\right)  ^{l+1}|_{0}^{1}\\
&  -\frac{1}{\left(  l+1\right)  !}\int_{0}^{1}\left(  \frac{d}{ds}%
\operatorname{Ad}_{e^{sV_{A}}}V_{\operatorname{ad}_{A}^{l}B}\right)  \left(
s-1\right)  ^{l+1}ds\\
=  &  \frac{\left(  -1\right)  ^{l}}{\left(  l+1\right)  !}\operatorname{Ad}%
_{e^{sV_{A}}}V_{\operatorname{ad}_{A}^{l}B}\\
&  +\frac{1}{\left(  l+1\right)  !}\int_{0}^{1}\operatorname{Ad}_{e^{sV_{A}}%
}\left[  V_{A},V_{\operatorname{ad}_{A}^{l}B}\right]  \left(  s-1\right)
^{l+1}ds\\
=  &  \frac{\left(  -1\right)  ^{l}}{\left(  l+1\right)  !}\operatorname{Ad}%
_{e^{sV_{A}}}V_{\operatorname{ad}_{A}^{l}B}\\
&  +\frac{1}{\left(  l+1\right)  !}\int_{0}^{1}\operatorname{Ad}_{e^{sV_{A}}%
}V_{\left[  A,\operatorname{ad}_{A}^{l}B\right]  _{\otimes}}\left(
s-1\right)  ^{l+1}ds.
\end{align*}
Combining this result with Eq. (\ref{e.4.3}) and the fact that
\[
\left[  A,\operatorname{ad}_{A}^{l}B\right]  _{\otimes}=\left[
A,\operatorname{ad}_{A}^{l}B\right]  +\pi_{>\kappa}\left[  A,\operatorname{ad}%
_{A}^{l}B\right]  _{\otimes}%
\]
completes the inductive step.

To finish the proof observe that
\[
\left(  s-1\right)  \psi\left(  \left(  s-1\right)  \operatorname{ad}%
_{A}\right)  =\sum_{k=1}^{\kappa}\frac{\left(  s-1\right)  ^{k}}%
{k!}\operatorname{ad}_{A}^{k-1}%
\]
and taking $s=0$ in this equation also shows,%
\[
\sum_{k=1}^{\kappa}\frac{\left(  -1\right)  ^{k+1}}{k!}\operatorname{ad}%
_{A}^{k-1}=\psi\left(  -\operatorname{ad}_{A}\right)  =\int_{0}^{1}%
\operatorname{Ad}_{e^{-sA}}ds,
\]
where the last equality comes from Eq. (\ref{e.4.1}). So from the last two
displayed equations and Eq. (\ref{e.4.3}) with $l=\kappa$ and the fact that
$\operatorname{ad}_{A}^{\kappa}B=0$ so that
\[
\frac{1}{\kappa!}\int_{0}^{1}\operatorname{Ad}_{e^{sV_{A}}}%
V_{\operatorname{ad}_{A}^{\kappa}B}\left(  s-1\right)  ^{\kappa}ds=0,
\]
we find
\begin{align*}
\int_{0}^{1}  &  \operatorname{Ad}_{e^{sV_{A}}}V_{B}\,ds\\
&  =V_{\psi\left(  -\operatorname{ad}_{A}\right)  B}+\int_{0}^{1}%
\operatorname{Ad}_{e^{sV_{A}}}V_{\pi_{>\kappa}\left[  A,\psi\left(  \left(
s-1\right)  \operatorname{ad}_{A}\right)  B\right]  _{\otimes}}\left(
s-1\right)  ds\\
&  =V_{\left[  \int_{0}^{1}\operatorname{Ad}_{e^{-sA}}B~ds\right]  }+\int%
_{0}^{1}\operatorname{Ad}_{e^{sV_{A}}}V_{\pi_{>\kappa}\left[  A,\psi\left(
\left(  s-1\right)  \operatorname{ad}_{A}\right)  B\right]  _{\otimes}}\left(
s-1\right)  ds.
\end{align*}
\end{proof}

\begin{remark}
[Signs]\label{rem.4.2}The expression, $\operatorname{Ad}_{e^{sV_{A}}}V_{B},$
appears on the left side of Eq. (\ref{e.4.1}) while on the right side we have
the expression, $\operatorname{Ad}_{e^{-sA}}B$ which involves a change of $s$
to $-s.$ This change of sign is a simple consequence of the fact that
vector-fields on $M$ may naturally be identified with \textbf{right invariant}
vector fields on $\mathrm{\mathrm{Diff}}\left(  M\right)  $ while on the other
hand we have chosen to view $A,B\in F^{\left(  \kappa\right)  }\left(
\mathbb{R}^{d}\right)  $ as \textbf{left invariant }vector fields on
$G_{0}=G_{geo}^{\left(  \kappa\right)  }\left(  \mathbb{R}^{d}\right)  .$ This
left right interchange is the reason for the sign changes in Eq. (\ref{e.4.1}).
\end{remark}

\begin{notation}
\label{not.4.4}To each $C\in C^{1}\left(  \left[  0,T\right]  ,F^{\left(
\kappa\right)  }\left(  \mathbb{R}^{d}\right)  \right)  ,$let%
\begin{equation}
W_{t}^{C}:=\int_{0}^{1}\operatorname{Ad}_{e^{sV_{C\left(  t\right)  }}}%
V_{\dot{C}\left(  t\right)  }~ds\in\Gamma\left(  TM\right)  . \label{e.4.4}%
\end{equation}

\end{notation}

\begin{notation}
\label{not.4.5}For $\xi\in C^{1}\left(  \left[  0,T\right]  ,F^{\left(
\kappa\right)  }\left(  \mathbb{R}^{d}\right)  \right)  ,$ let $g\left(
t\right)  =g^{\xi}\left(  t\right)  \in G_{0}$ be as in Definition
\ref{def.3.6}, $C^{\xi}\left(  t\right)  :=\log\left(  g^{\xi}\left(
t\right)  \right)  \in F^{\left(  \kappa\right)  }\left(  \mathbb{R}%
^{d}\right)  ,$ and $\mu_{t,s}^{\xi}:=\mu_{t,s}^{V_{\dot{\xi}}}\in
\mathrm{Diff}\left(  M\right)  $ denote the flow defined by%
\[
\dot{\mu}_{t,s}^{\xi}=V_{\dot{\xi}\left(  t\right)  }\circ\mu_{t,s}^{\xi
}\text{ with }\mu_{s,s}^{\xi}=Id_{M}.
\]

\end{notation}

Our goal is now to estimate the distance between $\mu_{t,0}^{\xi}$ and
$e^{V_{\log\left(  g^{\xi}\left(  t\right)  \right)  }}=e^{V_{C^{\xi}\left(
t\right)  }}.$ Since (by Corollary \ref{cor.2.24} with $Z_{t}=V_{C^{\xi
}\left(  t\right)  }\in\Gamma\left(  TM\right)  )$
\begin{equation}
\frac{d}{dt}e^{V_{C^{\xi}\left(  t\right)  }}=W_{t}^{C^{\xi}}\circ
e^{V_{C^{\xi}\left(  t\right)  }},\nonumber
\end{equation}
the desired distance estimates will be a consequence of applying Theorem
\ref{thm.2.30} with $X_{t}=V_{\dot{\xi}\left(  t\right)  }$ and $Y_{t}%
=W_{t}^{C^{\xi}}.$ Before carrying out the details we need to develop a few
auxiliary results first.

\begin{notation}
\label{not.4.6}For $0\leq s\leq1,$ let $u\left(  s,\cdot\right)
\in\mathcal{H}_{0}$ be defined by $u\left(  s,z\right)  :=\psi\left(  \left(
s-1\right)  z\right)  /\psi\left(  -z\right)  .$
\end{notation}

\begin{lemma}
\label{lem.4.7}Let $\xi$ and $C=C^{\xi}$ be as in Notation \ref{not.4.5} and
$W^{C^{\xi}}$ be as in Eq. (\ref{e.4.4}). Then the difference vector field,
\begin{equation}
U_{t}^{\xi}:=Y_{t}-X_{t}=W_{t}^{C^{\xi}}-V_{\dot{\xi}\left(  t\right)  }%
\in\Gamma\left(  TM\right)  , \label{e.4.5}%
\end{equation}
may be expresses as%
\begin{equation}
U_{t}^{\xi}=\int_{0}^{1}\operatorname{Ad}_{e^{sV_{C\left(  t\right)  }}}%
V_{\pi_{>\kappa}\left[  C\left(  t\right)  ,u\left(  s,\operatorname{ad}%
_{C\left(  t\right)  }\right)  \dot{\xi}\left(  t\right)  \right]  _{\otimes}%
}\left(  s-1\right)  ds. \label{e.4.6}%
\end{equation}

\end{lemma}

\begin{proof}
By Corollary \ref{cor.3.10},%
\begin{equation}
\psi\left(  -\operatorname{ad}_{C\left(  t\right)  }\right)  \dot{C}\left(
t\right)  =\int_{0}^{1}\operatorname{Ad}_{e^{-sC\left(  t\right)  }}\dot
{C}\left(  t\right)  ds=\dot{\xi}\left(  t\right)  \text{ with }C\left(
0\right)  =0, \label{e.4.7}%
\end{equation}
which combined with Theorem \ref{thm.4.1} with $A=C\left(  t\right)  $ and
$B=\dot{C}\left(  t\right)  $ implies%
\begin{equation}
W_{t}^{C}=V_{\dot{\xi}\left(  t\right)  }+\int_{0}^{1}\operatorname{Ad}%
_{e^{sV_{C\left(  t\right)  }}}V_{\pi_{>\kappa}\left[  C\left(  t\right)
,\psi\left(  \left(  s-1\right)  \operatorname{ad}_{C\left(  t\right)
}\right)  \dot{C}\left(  t\right)  \right]  _{\otimes}}\left(  s-1\right)  ds.
\label{e.4.8}%
\end{equation}
Since $\psi\left(  \left(  s-1\right)  z\right)  =u\left(  s,z\right)
\psi\left(  -z\right)  ,$ it follows (with the aid of Eq. (\ref{e.4.7}) that%
\[
\psi\left(  \left(  s-1\right)  \operatorname{ad}_{C\left(  t\right)
}\right)  \dot{C}\left(  t\right)  =u\left(  s,\operatorname{ad}_{C\left(
t\right)  }\right)  \psi\left(  -\operatorname{ad}_{C\left(  t\right)
}\right)  \dot{C}\left(  t\right)  =u\left(  s,\operatorname{ad}_{C\left(
t\right)  }\right)  \dot{\xi}\left(  t\right)
\]
which combined with Eq. (\ref{e.4.8}) gives Eq. (\ref{e.4.6}).
\end{proof}

\begin{corollary}
\label{cor.4.8}If $\left\{  V_{a}:a\in\mathbb{R}^{d}\right\}  $ generates a
\textbf{step-}$\kappa$\textbf{ nilpotent Lie sub-algebra} of $\Gamma\left(
TM\right)  ,$ then
\begin{equation}
\mu_{t,0}^{V_{\dot{\xi}}}=e^{V_{C^{\xi}\left(  t\right)  }}\text{ for all }%
\xi\in C^{1}\left(  \left[  0,T\right]  ,F^{\left(  \kappa\right)  }\left(
\mathbb{R}^{d}\right)  \right)  . \label{e.4.9}%
\end{equation}
Moreover, for any $A,B\in F^{\left(  \kappa\right)  }\left(  \mathbb{R}%
^{d}\right)  ,$ we have%
\begin{equation}
e^{V_{B}}\circ e^{V_{A}}=e^{V_{\log\left(  e^{A}e^{B}\right)  }}.
\label{e.4.10}%
\end{equation}

\end{corollary}

\begin{proof}
The given assumption implies $V_{\pi_{>\kappa}\left[  C\left(  t\right)
,\psi\left(  \left(  s-1\right)  \operatorname{ad}_{C\left(  t\right)
}\right)  \dot{C}\left(  t\right)  \right]  _{\otimes}}\equiv0$ and hence
$U^{\xi}\equiv0$ and the Eq. (\ref{e.4.9}) now follows from Theorem
\ref{thm.2.30}. To prove the second assertion $\xi:[0,\infty)\rightarrow
F^{\left(  \kappa\right)  }\left(  \mathbb{R}^{d}\right)  $ be defined by%
\begin{equation}
\xi\left(  t\right)  :=\left\{
\begin{array}
[c]{ccc}%
tA & \text{if} & 0\leq t\leq1\\
A+\left(  t-1\right)  B & \text{if} & 1\leq t<\infty
\end{array}
.\right.  \label{e.4.11}%
\end{equation}
With this choice of $\xi$ we have; $\dot{\xi}\left(  t\right)  =1_{t\leq
1}A+1_{t>1}B$ (for $t\neq1),$%
\[
g^{\xi}\left(  t\right)  =\left\{
\begin{array}
[c]{ccc}%
e^{tA} & \text{if} & 0\leq t\leq1\\
e^{A}e^{\left(  t-1\right)  B} & \text{if} & 1\leq t<\infty,
\end{array}
\right.
\]%
\[
\mu_{t,0}^{\xi}=\left\{
\begin{array}
[c]{ccc}%
e^{tV_{A}} & \text{if} & 0\leq t\leq1\\
e^{\left(  t-1\right)  V_{B}}\circ e^{V_{A}} & \text{if} & 1\leq t<\infty,
\end{array}
\right.
\]
all of which is valid where $V$ is step-$\kappa$ nilpotent or not. If $V$ is
step-$\kappa$ nilpotent we \textquotedblleft apply\textquotedblright\ Eq.
(\ref{e.4.9}) at $t=2,$ to find,%
\[
e^{V_{B}}\circ e^{V_{A}}=\mu_{2,0}^{\xi}=e^{V_{C^{\xi}\left(  2\right)  }%
}=e^{V_{\log\left(  e^{A}e^{B}\right)  }}.
\]

The slight flaw in this argument is that $\xi\left(  \cdot\right)  $ is not
continuously differentiable at $t=1.$ To correct this flaw, choose $\varphi\in
C_{c}^{\infty}\left(  \mathbb{R},[0,\infty)\right)  $ which is supported in
$\left(  0,1\right)  $ and satisfies $\int_{0}^{1}\varphi\left(  t\right)
dt=1.$ We then run the above argument with $\xi\in C^{\infty}\left(
[0,\infty),F^{\left(  \kappa\right)  }\left(  \mathbb{R}^{d}\right)  \right)
$ defined so that%
\begin{equation}
\dot{\xi}\left(  t\right)  =\varphi\left(  t\right)  A+\varphi\left(
t-1\right)  B\text{ with }\xi\left(  0\right)  =0. \label{e.4.12}%
\end{equation}
In more detail, if we let%
\begin{equation}
\bar{\varphi}\left(  t\right)  :=\int_{-\infty}^{t}\varphi\left(  \tau\right)
d\tau, \label{e.4.13}%
\end{equation}
then
\begin{equation}
\xi\left(  t\right)  =\bar{\varphi}\left(  t\right)  A+\bar{\varphi}\left(
t-1\right)  B, \label{e.4.14}%
\end{equation}%
\begin{align}
g^{\xi}\left(  t\right)   &  =e^{\bar{\varphi}\left(  t\right)  \cdot
A}e^{\bar{\varphi}\left(  t-1\right)  B},\text{ and }\label{e.4.15}\\
\mu_{t,0}^{\xi}  &  =e^{\bar{\varphi}\left(  t-1\right)  V_{B}}\circ
e^{\bar{\varphi}\left(  t\right)  V_{A}} \label{e.4.16}%
\end{align}
and in particular at $t=2$ we again have,%
\begin{equation}
e^{V_{B}}\circ e^{V_{A}}=\mu_{2,0}^{\xi}\text{ and }\,C^{\xi}\left(  2\right)
=\log\left(  g^{\xi}\left(  2\right)  \right)  =\log\left(  e^{A}e^{B}\right)
. \label{e.4.17}%
\end{equation}
Thus when $V$ is step-$\kappa$ nilpotent we are now justified in applying Eq.
(\ref{e.4.9}) at $t=2$ to arrive at Eq. (\ref{e.4.10}).
\end{proof}

\begin{notation}
[Commutator bounds]\label{not.4.9}If $V:\mathbb{R}^{d}\rightarrow\Gamma\left(
TM\right)  $ is a dynamical system and $m,n\in\mathbb{N}$ with $\kappa
<m+n\leq2\kappa,$ let%
\[
\mathcal{S}_{m,n}:=\left\{  \left(  A,B\right)  \in F_{m}^{\left(
\kappa\right)  }\left(  \mathbb{R}^{d}\right)  \times F_{n}^{\left(
\kappa\right)  }\left(  \mathbb{R}^{d}\right)  :\text{ }\left\vert
A\right\vert =1=\left\vert B\right\vert \right\}  ,
\]
\begin{align*}
\mathcal{C}_{m,n}^{0}\left(  V^{\left(  \kappa\right)  }\right)   &
:=\sup\left\{  \left\vert \left[  V_{A},V_{B}\right]  \right\vert _{M}:\left(
A,B\right)  \in\mathcal{S}_{m,n}\right\}  ,\\
\mathcal{C}_{m,n}^{1}\left(  V^{\left(  \kappa\right)  }\right)   &
:=\sup\left\{  \left\vert \nabla\left[  V_{A},V_{B}\right]  \right\vert
_{M}:\left(  A,B\right)  \in\mathcal{S}_{m,n}\right\}  ,
\end{align*}
and%
\[
\mathcal{C}^{j}\left(  V^{\left(  \kappa\right)  }\right)  :=\sum
_{m,n=1}^{\kappa}1_{m+n>\kappa}\mathcal{C}_{m,n}^{j}\left(  V^{\left(
\kappa\right)  }\right)  \text{ for }j=0,1.
\]

\end{notation}

Since
\begin{align*}
\left[  V_{A},V_{B}\right]   &  =\nabla_{V_{A}}V_{B}-\nabla_{V_{B}}V_{A}\text{
and}\\
\nabla_{v}\left[  V_{A},V_{B}\right]   &  =\nabla_{v\otimes V_{A}}^{2}%
V_{B}+\nabla_{\nabla_{v}V_{A}}V_{B}-\left(  A\longleftrightarrow B\right)
\end{align*}
it follows that
\begin{align*}
\mathcal{C}_{m,n}^{0}\left(  V^{\left(  \kappa\right)  }\right)   &
\leq2\left\vert V^{\left(  \kappa\right)  }\right\vert _{M}\left\vert \nabla
V^{\left(  \kappa\right)  }\right\vert _{M}\text{ and }\\
\mathcal{C}_{m,n}^{1}\left(  V^{\left(  \kappa\right)  }\right)   &
\leq2\left(  \left\vert \nabla^{2}V^{\left(  \kappa\right)  }\right\vert
_{M}\cdot\left\vert V^{\left(  \kappa\right)  }\right\vert _{M}+\left\vert
\nabla V^{\left(  \kappa\right)  }\right\vert _{M}^{2}\right)
\end{align*}
and therefore%
\begin{align}
\mathcal{C}^{0}\left(  V^{\left(  \kappa\right)  }\right)   &  \leq
\kappa\left(  \kappa+1\right)  \left\vert V^{\left(  \kappa\right)
}\right\vert _{M}\left\vert \nabla V^{\left(  \kappa\right)  }\right\vert
_{M}\text{ and }\label{e.4.18}\\
\mathcal{C}^{1}\left(  V^{\left(  \kappa\right)  }\right)   &  \leq
\kappa\left(  \kappa+1\right)  \left(  \left\vert \nabla^{2}V^{\left(
\kappa\right)  }\right\vert _{M}\cdot\left\vert V^{\left(  \kappa\right)
}\right\vert _{M}+\left\vert \nabla V^{\left(  \kappa\right)  }\right\vert
_{M}^{2}\right)  . \label{e.4.19}%
\end{align}
The previous estimates are in general not sharp. For example if $V$ is
$\kappa$-nilpotent, then $\mathcal{C}^{0}\left(  V^{\left(  \kappa\right)
}\right)  \equiv0$ while $2\left\vert V^{\left(  \kappa\right)  }\right\vert
_{M}\left\vert \nabla V^{\left(  \kappa\right)  }\right\vert _{M}$ will
typically be positive.

\begin{lemma}
\label{lem.4.10}If $\xi\in C^{1}\left(  \left[  0,T\right]  ,F^{\left(
\kappa\right)  }\left(  \mathbb{R}^{d}\right)  \right)  $ and $C^{\xi}\left(
t\right)  =\log\left(  g^{\xi}\left(  t\right)  \right)  \in F^{\left(
\kappa\right)  }\left(  \mathbb{R}^{d}\right)  $ are as in Definition
\ref{def.1.25} or Notation \ref{not.4.5} and $u\left(  s,z\right)  $ is as in
Notation \ref{not.4.6}, then%
\begin{equation}
\int_{0}^{1}ds\left(  1-s\right)  \int_{0}^{T}dt\left\vert V_{\pi_{>\kappa
}\left[  C\left(  t\right)  ,u\left(  s,\operatorname{ad}_{C\left(  t\right)
}\right)  \dot{\xi}\left(  t\right)  \right]  _{\otimes}}\right\vert
_{M}\lesssim\mathcal{C}^{0}\left(  V^{\left(  \kappa\right)  }\right)
Q_{(\kappa,2\kappa]}\left(  N_{T}^{\ast}\left(  \dot{\xi}\right)  \right)  .
\label{e.4.20}%
\end{equation}
and
\begin{equation}
\int_{0}^{1}ds\left(  1-s\right)  \int_{0}^{T}dt\left\vert \nabla
V_{\pi_{>\kappa}\left[  C\left(  t\right)  ,u\left(  s,\operatorname{ad}%
_{C\left(  t\right)  }\right)  \dot{\xi}\left(  t\right)  \right]  _{\otimes}%
}\right\vert _{M}\lesssim\mathcal{C}^{1}\left(  V^{\left(  \kappa\right)
}\right)  Q_{(\kappa,2\kappa]}\left(  N_{T}^{\ast}\left(  \dot{\xi}\right)
\right)  . \label{e.4.21}%
\end{equation}

\end{lemma}

\begin{proof}
Applying the triangle inequality to the identity,
\begin{align}
V_{\pi_{>\kappa}\left[  C\left(  t\right)  ,u\left(  s,\operatorname{ad}%
_{C\left(  t\right)  }\right)  \dot{\xi}\left(  t\right)  \right]  _{\otimes
}}  &  =\sum_{m,n=1}^{\kappa}1_{m+n>\kappa}V_{\left[  C\left(  t\right)
_{m},\left(  u\left(  s,\operatorname{ad}_{C\left(  t\right)  }\right)
\dot{\xi}\left(  t\right)  \right)  _{n}\right]  _{\otimes}}\nonumber\\
&  =\sum_{m,n=1}^{\kappa}1_{m+n>\kappa}\left[  V_{C\left(  t\right)  _{m}%
},V_{\left(  u\left(  s,\operatorname{ad}_{C\left(  t\right)  }\right)
\dot{\xi}\left(  t\right)  \right)  _{n}}\right]  , \label{e.4.22}%
\end{align}
while using Corollaries \ref{cor.3.27} and \ref{cor.3.28} and the definition
of $\mathcal{C}^{0}\left(  V^{\left(  \kappa\right)  }\right)  $ shows,%
\begin{align}
\int_{0}^{T}  &  \left\vert V_{\pi_{>\kappa}\left[  C\left(  t\right)
,\psi\left(  \left(  s-1\right)  \operatorname{ad}_{C\left(  t\right)
}\right)  \dot{C}\left(  t\right)  \right]  _{\otimes}}\right\vert
_{M}dt\nonumber\\
&  \leq\sum_{m,n=1}^{\kappa}1_{m+n>\kappa}\int_{0}^{T}\left\vert \left[
V_{C\left(  t\right)  _{m}},V_{\left(  u\left(  s,\operatorname{ad}_{C\left(
t\right)  }\right)  \dot{\xi}\left(  t\right)  \right)  _{n}}\right]
\right\vert _{M}dt\nonumber\\
&  \leq\sum_{m,n=1}^{\kappa}1_{m+n>\kappa}\mathcal{C}_{m,n}^{0}\left(
V^{\left(  \kappa\right)  }\right)  \int_{0}^{T}\left\vert C_{m}\right\vert
_{\infty,T}\cdot\left\vert \left(  u\left(  s,\operatorname{ad}_{C\left(
t\right)  }\right)  \dot{\xi}\left(  t\right)  \right)  _{n}\right\vert
dt\nonumber\\
&  \lesssim K\left(  u\left(  s,\cdot\right)  \right)  \sum_{m,n=1}^{\kappa
}1_{m+n>\kappa}\mathcal{C}_{m,n}^{0}\left(  V^{\left(  \kappa\right)
}\right)  N_{T}^{\ast}\left(  \dot{\xi}\right)  ^{m+n}\nonumber\\
&  \quad\leq K\left(  u\left(  s,\cdot\right)  \right)  \sum_{m,n=1}^{\kappa
}1_{m+n>\kappa}\mathcal{C}_{m,n}^{0}\left(  V^{\left(  \kappa\right)
}\right)  Q_{(\kappa,2\kappa]}\left(  N_{T}^{\ast}\left(  \dot{\xi}\right)
\right) \nonumber\\
&  \quad=K\left(  u\left(  s,\cdot\right)  \right)  \mathcal{C}^{0}\left(
V^{\left(  \kappa\right)  }\right)  Q_{(\kappa,2\kappa]}\left(  N_{T}^{\ast
}\left(  \dot{\xi}\right)  \right)  . \label{e.4.23}%
\end{align}

A simple differentiation exercise shows $p_{n}\left(  s\right)  :=\left(
\frac{d}{dz}\right)  ^{n}u\left(  s,z\right)  |_{z=0}$ is a degree $n$
-polynomial function of $s$ with $p_{0}\left(  s\right)  =1.$ As $K\left(
u\left(  s,\cdot\right)  \right)  $ depends linearly on $\left\{  \left(
\frac{d}{dz}\right)  ^{j}u\left(  s,z\right)  |_{z=0}\right\}  _{j=0}%
^{\kappa-1}$ it follows that $K\left(  u\left(  s,\cdot\right)  \right)  $ is
bounded by a polynomial function of $s$ and in particular,%
\[
\int_{0}^{1}K\left(  u\left(  s,\cdot\right)  \right)  \left(  1-s\right)
ds<\infty.
\]
Thus multiplying Eq. (\ref{e.4.23}) by $\left(  1-s\right)  $ and then
integrating on $s\in\left[  0,1\right]  $ completes the proof of Eq.
(\ref{e.4.20}). The proof of Eq. (\ref{e.4.21}) is very similar. Simply apply
$\nabla$ to both sides of Eq. (\ref{e.4.22}) and then continue the estimates
as above with $\mathcal{C}_{m,n}^{0}\left(  V^{\left(  \kappa\right)
}\right)  $ and $\mathcal{C}^{0}\left(  V^{\left(  \kappa\right)  }\right)  $
replaced by $\mathcal{C}_{m,n}^{1}\left(  V^{\left(  \kappa\right)  }\right)
$ and $\mathcal{C}^{1}\left(  V^{\left(  \kappa\right)  }\right)  $ respectively.
\end{proof}

\begin{theorem}
\label{thm.4.11}If $\xi\in C^{1}\left(  \left[  0,T\right]  ,F^{\left(
\kappa\right)  }\left(  \mathbb{R}^{d}\right)  \right)  $ and $C^{\xi}\left(
t\right)  =\log\left(  g^{\xi}\left(  t\right)  \right)  \in F^{\left(
\kappa\right)  }\left(  \mathbb{R}^{d}\right)  $ be as in Definition
\ref{def.1.25} or Notation \ref{not.4.5} and $U_{t}^{\xi}\in\Gamma\left(
TM\right)  $ as in Eq. (\ref{e.4.6}) of Lemma \ref{lem.4.7}, then
\begin{equation}
\left\vert U^{\xi}\right\vert _{T}^{\ast}\lesssim\mathcal{C}^{0}\left(
V^{\left(  \kappa\right)  }\right)  e^{\left\vert \nabla V^{\left(
\kappa\right)  }\right\vert _{M}\left\vert C^{\xi}\right\vert _{\infty,T}%
}Q_{(\kappa,\kappa+1]}\left(  N_{T}^{\ast}\left(  \dot{\xi}\right)  \right)
\label{e.4.24}%
\end{equation}
which combined with Eq. (\ref{e.3.31}) shows there exists $C\left(
\kappa\right)  <\infty$ such that
\begin{equation}
\left\vert U^{\xi}\right\vert _{T}^{\ast}\lesssim\mathcal{C}^{0}\left(
V^{\left(  \kappa\right)  }\right)  e^{C\left(  \kappa\right)  \left\vert
\nabla V^{\left(  \kappa\right)  }\right\vert _{M}Q_{\left[  1,\kappa\right]
}\left(  N_{T}^{\ast}\left(  \dot{\xi}\right)  \right)  }Q_{(\kappa,\kappa
+1]}\left(  N_{T}^{\ast}\left(  \dot{\xi}\right)  \right)  . \label{e.4.25}%
\end{equation}

\end{theorem}

\begin{proof}
By Corollary \ref{cor.2.28}, if $Y\in\Gamma\left(  TM\right)  ,$ then%
\begin{align*}
\left\vert \operatorname{Ad}_{e^{sV_{C\left(  \tau\right)  }}}Y\right\vert
_{M}  &  =\left\vert e_{\ast}^{sV_{C\left(  \tau\right)  }}Y\circ
e^{-sV_{C\left(  \tau\right)  }}\right\vert _{M}=\left\vert e_{\ast
}^{sV_{C\left(  \tau\right)  }}Y\right\vert _{M}\\
&  \leq e^{s\left\vert \nabla V_{C\left(  \tau\right)  }\right\vert _{M}%
}\left\vert Y\right\vert _{M}\leq e^{s\left\vert \nabla V^{\left(
\kappa\right)  }\right\vert _{M}\left\vert C\left(  \tau\right)  \right\vert
}\left\vert Y\right\vert _{M}%
\end{align*}
and so (see Eq. (\ref{e.4.6})),%
\begin{align*}
\left\vert U_{t}^{\xi}\right\vert _{M}  &  \leq\int_{0}^{1}\left\vert
\operatorname{Ad}_{e^{sV_{C\left(  t\right)  }}}V_{\pi_{>\kappa}\left[
C\left(  t\right)  ,u\left(  s,\operatorname{ad}_{C\left(  t\right)  }\right)
\dot{\xi}\left(  t\right)  \right]  _{\otimes}}\right\vert _{M}\left(
s-1\right)  ds\\
&  \leq\int_{0}^{1}e^{s\left\vert \nabla V^{\left(  \kappa\right)
}\right\vert _{M}\left\vert C\left(  t\right)  \right\vert }\left\vert
V_{\pi_{>\kappa}\left[  C\left(  t\right)  ,u\left(  s,\operatorname{ad}%
_{C\left(  t\right)  }\right)  \dot{\xi}\left(  t\right)  \right]  _{\otimes}%
}\right\vert _{M}\left(  s-1\right)  ds
\end{align*}
and so
\begin{equation}
\left\vert U^{\xi}\right\vert _{T}^{\ast}\leq e^{\left\vert \nabla V^{\left(
\kappa\right)  }\right\vert _{M}\left\vert C\right\vert _{\infty,T}}\int%
_{0}^{1}ds\left(  1-s\right)  \int_{0}^{T}dt\left\vert V_{\pi_{>\kappa}\left[
C\left(  t\right)  ,u\left(  s,\operatorname{ad}_{C\left(  t\right)  }\right)
\dot{\xi}\left(  t\right)  \right]  _{\otimes}}\right\vert _{M} \label{e.4.26}%
\end{equation}
which combined with Lemma \ref{lem.4.10} proves Eq. (\ref{e.4.24}).
\end{proof}

\begin{theorem}
[Approximate log-estimate]\label{thm.4.12}If $\xi\in C^{1}\left(  \left[
0,T\right]  ,F^{\left(  \kappa\right)  }\left(  \mathbb{R}^{d}\right)
\right)  ,$ then\footnote{We will see in Theorem \ref{thm.8.4} below that a
similar estimate holds for the distance between the differentials of
$\mu_{T,0}^{V_{\dot{\xi}}}$ and $e^{V_{\log\left(  g^{\xi}\left(  T\right)
\right)  }}.$}%
\begin{equation}
d_{M}\left(  \mu_{T,0}^{V_{\dot{\xi}}},e^{V_{\log\left(  g^{\xi}\left(
T\right)  \right)  }}\right)  \lesssim\mathcal{C}^{0}\left(  V^{\left(
\kappa\right)  }\right)  e^{C\left(  \kappa\right)  \left\vert \nabla
V^{\left(  \kappa\right)  }\right\vert _{M}Q_{\left[  1,\kappa\right]
}\left(  N_{T}^{\ast}\left(  \dot{\xi}\right)  \right)  }Q_{(\kappa,\kappa
+1]}\left(  N_{T}^{\ast}\left(  \dot{\xi}\right)  \right)  . \label{e.4.27}%
\end{equation}

\end{theorem}

\begin{proof}
By Theorem \ref{thm.2.30} with $X_{t}=V_{\dot{\xi}\left(  t\right)  }$ and
$Y_{t}=W_{t}^{C},$ we know that%
\begin{equation}
d_{M}\left(  \mu_{T,0}^{V_{\dot{\xi}}},e^{V_{C\left(  T\right)  }}\right)
\leq e^{\left\vert \nabla V_{\dot{\xi}}\right\vert _{T}^{\ast}}\cdot\left\vert
U^{\xi}\right\vert _{T}^{\ast}\leq e^{\left\vert \nabla V^{\left(
\kappa\right)  }\right\vert _{M}\left\vert \dot{\xi}\right\vert _{T}^{\ast}%
}.\left\vert U^{\xi}\right\vert _{T}^{\ast}. \label{e.4.28}%
\end{equation}
Combining this estimate with the estimate for $\left\vert U^{\xi}\right\vert
_{T}^{\ast}$ in Theorem \ref{thm.4.11} and the estimate for $\left\vert
\dot{\xi}\right\vert _{T}^{\ast}$ in Eq. (\ref{e.3.21}) gives Eq.
(\ref{e.4.27}).
\end{proof}

For the rest of this section we assume that $\xi\in C^{\infty}\left(
[0,\infty),F^{\left(  \kappa\right)  }\left(  \mathbb{R}^{d}\right)  \right)
$ is defined as in Eq. (\ref{e.4.12}) of the proof of Corollary \ref{cor.4.8},
i.e.
\begin{equation}
\xi\left(  t\right)  =\bar{\varphi}\left(  t\right)  A+\bar{\varphi}\left(
t-1\right)  B\in F^{\left(  \kappa\right)  }\left(  \mathbb{R}^{d}\right)  ,
\label{e.4.29}%
\end{equation}
where
\[
\bar{\varphi}\left(  t\right)  =\int_{-\infty}^{t}\varphi\left(  \tau\right)
d\tau
\]
and $\varphi\in C_{c}^{\infty}\left(  \mathbb{R},[0,\infty)\right)  $ with
$\bar{\varphi}\left(  1\right)  =\bar{\varphi}\left(  \infty\right)  =1.$

\begin{corollary}
\label{cor.4.13}If $A,B\in F^{\left(  \kappa\right)  }\left(  \mathbb{R}%
^{d}\right)  ,$ then
\begin{align}
d_{M}  &  \left(  e^{V_{B}}\circ e^{V_{A}},e^{V_{\log\left(  e^{A}%
e^{B}\right)  }}\right) \nonumber\\
&  \lesssim\mathcal{C}^{0}\left(  V^{\left(  \kappa\right)  }\right)
e^{C\left(  \kappa\right)  \left\vert \nabla V^{\left(  \kappa\right)
}\right\vert _{M}Q_{\left[  1,\kappa\right]  }\left(  N\left(  A\right)
+N\left(  B\right)  \right)  }Q_{(\kappa,\kappa+1]}\left(  N\left(  A\right)
+N\left(  B\right)  \right)  . \label{e.4.30}%
\end{align}

\end{corollary}

\begin{proof}
From the definition of $\xi$ in Eq. (\ref{e.4.29}), we find%
\begin{equation}
\left\vert \dot{\xi}_{k}\right\vert _{2}^{\ast}=\left\vert A_{k}\right\vert
+\left\vert B_{k}\right\vert \text{ for }1\leq k\leq\kappa\label{e.4.31}%
\end{equation}
and hence with the aid of Eq. (\ref{e.3.22}),%
\begin{equation}
N_{2}^{\ast}\left(  \dot{\xi}\right)  \leq N\left(  A\right)  +N\left(
B\right)  . \label{e.4.32}%
\end{equation}
Moreover, by the identities in Eq. (\ref{e.4.17}) we know that
\begin{equation}
d_{M}\left(  e^{V_{B}}\circ e^{V_{A}},e^{V_{\log\left(  e^{A}e^{B}\right)  }%
}\right)  =d_{M}\left(  \mu_{2,0}^{V_{\dot{\xi}}},e^{V_{\log\left(  g^{\xi
}\left(  2\right)  \right)  }}\right)  . \label{e.4.33}%
\end{equation}
So an application of Theorem \ref{thm.4.12} for this $\xi$ and taking $T=2$
gives Eq. (\ref{e.4.30}).
\end{proof}

The estimate in Eq. (\ref{e.4.30}) is not as sharp as we would like. For
example the right side of Eq. (\ref{e.4.30}) is only $0$ when $A=0=B$ while
the left side is $0$ when either $A=0$ or $B=0.$ To improve upon the estimate
in Eq. (\ref{e.4.30}) (see Corollary \ref{cor.4.16}) we need to examine the
form of the difference vector field, $U_{t}^{\xi},$ for $\xi$ in Eq.
(\ref{e.4.29}). We begin with a couple of lemmas.

\begin{lemma}
\label{lem.4.14}If $f\in\mathcal{H}_{0}$ satisfies, $f\left(  0\right)  =0,$
then $\left[  f\left(  \operatorname{ad}_{C\left(  t\right)  }\right)
B\right]  _{1}=0$ and for $2\leq k\leq\kappa,$%
\[
\max_{1\leq t\leq2}\left\vert \left[  f\left(  \operatorname{ad}_{C\left(
t\right)  }\right)  B\right]  _{k}\right\vert \leq K\left(  f\right)  N\left(
A\right)  N\left(  B\right)  \left(  N\left(  A\right)  +N\left(  B\right)
\right)  ^{k-2}%
\]
where $K\left(  f\right)  <\infty$ is a constant which depends linearly on
$\left\{  \left\vert f^{\left(  j\right)  }\left(  0\right)  \right\vert
\right\}  _{j=1}^{\kappa-1}.$
\end{lemma}

\begin{proof}
By Proposition \ref{pro.3.18}, there exists bounded measurable functions,
$\mathbf{f}^{j}:[0,\infty)^{j}\rightarrow\mathbb{R}$ depending linearly on
$\left(  f\left(  0\right)  ,\dots,f^{\left(  \kappa-1\right)  }\left(
0\right)  \right)  $ such that%
\[
f\left(  \operatorname{ad}_{C\left(  t\right)  }\right)  =\sum_{j=1}%
^{\kappa-1}\widehat{\mathbf{f}^{j}}_{t}\left(  \operatorname{ad}_{\dot{\xi}%
}\right)  .
\]
As $\left[  \widehat{\mathbf{f}^{j}}_{t}\left(  \operatorname{ad}_{\dot{\xi}%
}\right)  B\right]  _{k}=0$ if $j\geq k,$ to finish the proof it suffices to
show for each $1\leq j<k$ that
\begin{equation}
\max_{1\leq t\leq2}\left\vert \left[  \widehat{\mathbf{f}^{j}}_{t}\left(
\operatorname{ad}_{\dot{\xi}}\right)  B\right]  _{k}\right\vert \lesssim
\left\Vert \mathbf{f}^{j}\right\Vert _{\infty}N\left(  A\right)  N\left(
B\right)  \left(  N\left(  A\right)  +N\left(  B\right)  \right)  ^{k-2}.
\label{e.4.34}%
\end{equation}
Let us now fix $1\leq j<k.$

For $1\leq t\leq2,$%
\begin{align*}
\widehat{\mathbf{f}^{j}}_{t}\left(  \operatorname{ad}_{\dot{\xi}}\right)  B
&  =\int_{\left[  0,t\right]  ^{j}}\mathbf{f}^{j}\left(  t_{1},\dots
,t_{j}\right)  \operatorname{ad}_{\dot{\xi}\left(  t_{1}\right)  }%
\dots\operatorname{ad}_{\dot{\xi}\left(  t_{j-1}\right)  }\operatorname{ad}%
_{\dot{\xi}\left(  t_{j}\right)  }Bdt_{1}\dots dt_{j}\\
&  =\int_{\left[  0,t\right]  ^{j-1}}\mathbf{f}^{j}\left(  t_{1},\dots
,t_{j}\right)  \varphi\left(  t_{j}\right)  \operatorname{ad}_{\dot{\xi
}\left(  t_{1}\right)  }\dots\operatorname{ad}_{\dot{\xi}\left(
t_{j-1}\right)  }\left[  A,B\right]  dt_{1}\dots dt_{j-1}dt_{j},
\end{align*}
wherein we have used $\operatorname{ad}_{\dot{\xi}\left(  t_{j}\right)
}B=\varphi\left(  t_{j}\right)  \left[  A,B\right]  $ for all $t\geq0.$ Since
$\int\varphi\left(  t\right)  dt=1,$ it is simple to verify that%
\begin{equation}
\left\vert \left(  \widehat{\mathbf{f}^{j}}_{t}\left(  \operatorname{ad}%
_{\dot{\xi}}\right)  B\right)  _{k}\right\vert \leq\left\Vert \mathbf{f}%
^{j}\right\Vert _{\infty}\int_{\left[  0,t\right]  ^{j-1}}\left\vert \left(
\operatorname{ad}_{\dot{\xi}\left(  t_{1}\right)  }\dots\operatorname{ad}%
_{\dot{\xi}\left(  t_{j-1}\right)  }\left[  A,B\right]  \right)
_{k}\right\vert d\mathbf{t} \label{e.4.35}%
\end{equation}
where $d\mathbf{t}:=dt_{1}\dots dt_{j-1}.$ We now estimate the integral in the
usual way, namely;%
\begin{align}
\int_{\left[  0,t\right]  ^{j-1}}  &  \left\vert \left(  \operatorname{ad}%
_{\dot{\xi}\left(  t_{1}\right)  }\dots\operatorname{ad}_{\dot{\xi}\left(
t_{j-1}\right)  }\left[  A,B\right]  \right)  _{k}\right\vert d\mathbf{t}%
\nonumber\\
&  \leq\sum\int_{\left[  0,t\right]  ^{j-1}}\left\vert \operatorname{ad}%
_{\dot{\xi}_{k_{1}}\left(  t_{1}\right)  }\dots\operatorname{ad}_{\dot{\xi
}_{k_{j-1}}\left(  t_{j-1}\right)  }\left[  A_{m},B_{n}\right]  \right\vert
d\mathbf{t} \label{e.4.36}%
\end{align}
where the sum is over $\left(  m,n,k_{1},\dots,k_{j-1}\right)  \in
\mathbb{N}^{j+1}$ such that $\sum_{i=1}^{j-1}k_{i}+m+n=k.$ Using $\left\vert
\left[  A,B\right]  \right\vert \leq2\left\vert A\right\vert \left\vert
B\right\vert $ for all $A,B\in F^{\left(  \kappa\right)  }\left(
\mathbb{R}^{d}\right)  ,$ each term on the right side of Eq. (\ref{e.4.36})
may be estimated by
\begin{align}
2^{j}  &  \int_{\left[  0,t\right]  ^{j-1}}\left\vert \dot{\xi}_{k_{1}}\left(
t_{1}\right)  \right\vert \dots\left\vert \dot{\xi}_{k_{j-1}}\left(
t_{j-1}\right)  \right\vert \left\vert A_{m}\right\vert \left\vert
B_{n}\right\vert d\mathbf{t}\nonumber\\
&  \leq2^{j}\prod_{i=1}^{j-1}\left\vert \dot{\xi}_{k_{i}}\right\vert
_{2}^{\ast}\left\vert A_{m}\right\vert \left\vert B_{n}\right\vert \leq
2^{j}\prod_{i=1}^{j-1}N_{t}^{\ast}\left(  \dot{\xi}\right)  ^{k_{i}}N\left(
A\right)  ^{m}N\left(  B\right)  ^{n}\nonumber\\
&  \leq2^{j}\left(  N\left(  A\right)  +N\left(  B\right)  \right)
^{k-m-n}N\left(  A\right)  ^{m}N\left(  B\right)  ^{n}\nonumber\\
&  \leq2^{j}N\left(  A\right)  N\left(  B\right)  \left(  N\left(  A\right)
+N\left(  B\right)  \right)  ^{k-2}. \label{e.4.37}%
\end{align}
Combining the estimates in Eqs. (\ref{e.4.35}) -- (\ref{e.4.37}) completes the
proof of Eq. (\ref{e.4.34}) and hence the proof of the lemma.
\end{proof}

\begin{proposition}
\label{pro.4.15}If $A,B\in F^{\left(  \kappa\right)  }\left(  \mathbb{R}%
^{d}\right)  $ and $\xi\in C^{\infty}\left(  [0,\infty),F^{\left(
\kappa\right)  }\left(  \mathbb{R}^{d}\right)  \right)  $ is as in Eq.
(\ref{e.4.29}), then%
\begin{align}
&  \left[  C^{\xi}\left(  t\right)  ,u\left(  s,\operatorname{ad}_{C^{\xi
}\left(  t\right)  }\right)  \dot{\xi}\left(  t\right)  \right]  _{\otimes
}\nonumber\\
&  =\varphi\left(  t-1\right)  \left(  \left[  A,B\right]  _{\otimes}+\left[
\bar{C}^{\xi}\left(  t\right)  ,B\right]  _{\otimes}+\left[  C^{\xi}\left(
t\right)  ,\bar{u}\left(  s,\operatorname{ad}_{C^{\xi}\left(  t\right)
}\right)  B\right]  _{\otimes}\right)  \label{e.4.38}%
\end{align}
where
\begin{align*}
\bar{C}^{\xi}\left(  t\right)   &  :=C^{\xi}\left(  t\right)  -\xi\left(
t\right)  \text{ and }\\
\bar{u}\left(  s,z\right)   &  :=u\left(  s,z\right)  -u\left(  s,0\right)
=u\left(  s,z\right)  -1.
\end{align*}
Moreover for $k\geq2,$ the following estimates hold;%
\begin{align}
\max_{0\leq t\leq2}\left\vert \bar{C}_{k}^{\xi}\left(  t\right)  \right\vert
&  \lesssim N\left(  A\right)  N\left(  B\right)  \left(  N\left(  A\right)
+N\left(  B\right)  \right)  ^{k-2}\text{ and}\label{e.4.39}\\
\max_{0\leq t\leq2}\max_{0\leq s\leq1}\left\vert \left[  \bar{u}\left(
s,\operatorname{ad}_{C^{\xi}\left(  t\right)  }\right)  B\right]
_{k}\right\vert  &  \lesssim N\left(  A\right)  N\left(  B\right)  \left(
N\left(  A\right)  +N\left(  B\right)  \right)  ^{k-2}. \label{e.4.40}%
\end{align}

\end{proposition}

\begin{proof}
From Eq. (\ref{e.4.15}), $C^{\xi}\left(  t\right)  =\bar{\varphi}\left(
t\right)  A=\xi\left(  t\right)  $ when $t\leq1$ and therefore%
\begin{align*}
\left[  C^{\xi}\left(  t\right)  ,u\left(  s,\operatorname{ad}_{C^{\xi}\left(
t\right)  }\right)  \dot{\xi}\left(  t\right)  \right]  _{\otimes}  &
=\bar{\varphi}\left(  t\right)  \varphi\left(  t\right)  \left[  A,u\left(
s,\operatorname{ad}_{\bar{\varphi}\left(  t\right)  A}\right)  A\right]
_{\otimes}\\
&  =\bar{\varphi}\left(  t\right)  \varphi\left(  t\right)  \left[  A,u\left(
s,0\right)  A\right]  _{\otimes}=0.
\end{align*}
which proves Eq. (\ref{e.4.38}) for $t\leq1.$ When $t\geq1,$ $\xi\left(
t\right)  =A+\bar{\varphi}\left(  t-1\right)  B,$ $\dot{\xi}\left(  t\right)
=\varphi\left(  t-1\right)  B,$ and
\[
u\left(  s,\operatorname{ad}_{C^{\xi}\left(  t\right)  }\right)  \dot{\xi
}\left(  t\right)  =u\left(  s,\operatorname{ad}_{C^{\xi}\left(  t\right)
}\right)  B=B+\bar{u}\left(  s,\operatorname{ad}_{C^{\xi}\left(  t\right)
}\right)  B
\]
and hence%
\begin{align*}
&  \left[  C^{\xi}\left(  t\right)  ,u\left(  s,\operatorname{ad}_{C^{\xi
}\left(  t\right)  }\right)  \dot{\xi}\left(  t\right)  \right]  _{\otimes}\\
&  \quad=\left[  C^{\xi}\left(  t\right)  ,\dot{\xi}\left(  t\right)  +\bar
{u}\left(  s,\operatorname{ad}_{C^{\xi}\left(  t\right)  }\right)  \dot{\xi
}\left(  t\right)  \right]  _{\otimes}\\
&  \quad=\varphi\left(  t-1\right)  \left(  \left[  A+\bar{\varphi}\left(
t-1\right)  B+\bar{C}^{\xi}\left(  t\right)  ,B\right]  _{\otimes}+\left[
C^{\xi}\left(  t\right)  ,\bar{u}\left(  s,\operatorname{ad}_{C^{\xi}\left(
t\right)  }\right)  B\right]  _{\otimes}\right)
\end{align*}
which easily gives Eq. (\ref{e.4.38}) for $t\geq1.$

By Eq. (\ref{cor.3.10}),
\[
\dot{C}^{\xi}\left(  t\right)  =\frac{1}{\psi}\left(  -\operatorname{ad}%
_{C\left(  t\right)  }\right)  \dot{\xi}\left(  t\right)  =\dot{\xi}\left(
t\right)  +g\left(  \operatorname{ad}_{C\left(  t\right)  }\right)  \dot{\xi
}\left(  t\right)
\]
wherein the last equality we used $1/\psi\left(  0\right)  =1$ and have set%
\[
g\left(  z\right)  :=\frac{1}{\psi\left(  -z\right)  }-\frac{1}{\psi\left(
0\right)  }=\frac{1}{\psi\left(  -z\right)  }-1.
\]
Thus it follows that
\[
\bar{C}^{\xi}\left(  t\right)  =\int_{0}^{t}g\left(  \operatorname{ad}%
_{C^{\xi}\left(  \tau\right)  }\right)  \dot{\xi}\left(  \tau\right)
d\tau=\int_{0}^{t}\varphi\left(  \tau-1\right)  g\left(  \operatorname{ad}%
_{C^{\xi}\left(  \tau\right)  }\right)  Bd\tau.
\]
By Lemma \ref{lem.4.14},%
\[
\max_{1\leq\tau\leq2}\left\vert \left[  g\left(  \operatorname{ad}_{C^{\xi
}\left(  \tau\right)  }\right)  B\right]  _{k}\right\vert \leq K\left(
g\right)  N\left(  A\right)  N\left(  B\right)  \left(  N\left(  A\right)
+N\left(  B\right)  \right)  ^{k-2}%
\]
and so it now easily follows that%
\[
\left\vert \bar{C}_{k}^{\xi}\left(  t\right)  \right\vert \leq K\left(
g\right)  N\left(  A\right)  N\left(  B\right)  \left(  N\left(  A\right)
+N\left(  B\right)  \right)  ^{k-2}\text{ for all }0\leq t\leq2.
\]
By another application of Lemma \ref{lem.4.14},%
\[
\max_{0\leq t\leq2}\left\vert \left[  \bar{u}\left(  s,\operatorname{ad}%
_{C^{\xi}\left(  t\right)  }\right)  B\right]  _{k}\right\vert \leq K\left(
\bar{u}\left(  s,\cdot\right)  \right)  N\left(  A\right)  N\left(  B\right)
\left(  N\left(  A\right)  +N\left(  B\right)  \right)  ^{k-2}%
\]
where $K\left(  \bar{u}\left(  s,\cdot\right)  \right)  $ is bounded in
$s\in\left[  0,1\right]  $ as the derivatives of $\bar{u}\left(  s,z\right)  $
as $z=0$ are polynomial functions in $s.$ These last two inequalities verify
Eqs. (\ref{e.4.39}) and (\ref{e.4.40}) and hence complete the proof.
\end{proof}

\begin{corollary}
\label{cor.4.16}If $A,B\in F^{\left(  \kappa\right)  }\left(  \mathbb{R}%
^{d}\right)  ,$ then%
\begin{equation}
d_{M}\left(  e^{V_{B}},Id_{M}\right)  \leq\left\vert V^{\left(  \kappa\right)
}\right\vert \left\vert B\right\vert \leq\left\vert V^{\left(  \kappa\right)
}\right\vert Q_{\left[  1,\kappa\right]  }\left(  N\left(  B\right)  \right)
\label{e.4.41}%
\end{equation}
and there exists $C\left(  \kappa\right)  <\infty$ such that
\begin{equation}
d_{M}\left(  e^{V_{B}}\circ e^{V_{A}},e^{V_{\log\left(  e^{A}e^{B}\right)  }%
}\right)  \lesssim\mathcal{K}_{0}N\left(  A\right)  N\left(  B\right)
Q_{\left[  \kappa-1,2\kappa-2\right]  }\left(  N\left(  A\right)  +N\left(
B\right)  \right)  \label{e.4.42}%
\end{equation}
where
\begin{align}
\mathcal{K}_{0}  &  :=\mathcal{C}^{0}\left(  V^{\left(  \kappa\right)
}\right)  e^{C\left(  \kappa\right)  \left\vert \nabla V^{\left(
\kappa\right)  }\right\vert _{M}Q_{\left[  1,\kappa\right]  }\left(  N\left(
A\right)  +N\left(  B\right)  \right)  }\label{e.4.43}\\
&  \leq2\left\vert V^{\left(  \kappa\right)  }\right\vert _{M}\left\vert
\nabla V^{\left(  \kappa\right)  }\right\vert _{M}e^{c\left(  \kappa\right)
\left\vert \nabla V^{\left(  \kappa\right)  }\right\vert _{M}Q_{\left[
1,\kappa\right]  }\left(  N\left(  A\right)  +N\left(  B\right)  \right)  }.
\label{e.4.44}%
\end{align}

\end{corollary}

\begin{proof}
The first inequality follows as an application of Corollary \ref{cor.2.31}
with $Y_{t}:=V_{B}$ using
\[
\left\vert Y\right\vert _{1}^{\ast}=\left\vert V_{B}\right\vert _{M}%
\leq\left\vert V^{\left(  \kappa\right)  }\right\vert \left\vert B\right\vert
\]
To prove the second inequality we let $\xi\left(  t\right)  $ be as in
Proposition \ref{pro.4.15}. By Eq. (\ref{e.4.28}) in the proof of Theorem
\ref{thm.4.12} we then have, to find,
\begin{align}
d_{M}\left(  e^{V_{B}}\circ e^{V_{A}},e^{V_{\log\left(  e^{A}e^{B}\right)  }%
}\right)   &  =d_{M}\left(  \mu_{2,0}^{V_{\dot{\xi}}},e^{V_{C\left(  2\right)
}}\right) \nonumber\\
&  \leq e^{\left\vert \nabla V^{\left(  \kappa\right)  }\right\vert
_{M}\left\vert \dot{\xi}\right\vert _{2}^{\ast}}.\left\vert U^{\xi}\right\vert
_{2}^{\ast}=e^{\left\vert \nabla V^{\left(  \kappa\right)  }\right\vert
_{M}\left(  \left\vert A\right\vert +\left\vert B\right\vert \right)
}.\left\vert U^{\xi}\right\vert _{2}^{\ast}. \label{e.4.45}%
\end{align}
where, by Eq. (\ref{e.4.26}) of the proof of Theorem \ref{thm.4.11},%
\begin{equation}
\left\vert U^{\xi}\right\vert _{2}^{\ast}\leq e^{\left\vert \nabla V^{\left(
\kappa\right)  }\right\vert _{M}\left\vert C\right\vert _{\infty,2}}\int%
_{0}^{1}ds\left(  1-s\right)  \int_{0}^{2}dt\left\vert V_{\pi_{>\kappa}\left[
C\left(  t\right)  ,u\left(  s,\operatorname{ad}_{C\left(  t\right)  }\right)
\dot{\xi}\left(  t\right)  \right]  _{\otimes}}\right\vert _{M}.
\label{e.4.46}%
\end{equation}

From Proposition \ref{pro.4.15}%
\begin{align}
&  \left\vert V_{\left(  \left[  C\left(  t\right)  ,u\left(
s,\operatorname{ad}_{C\left(  t\right)  }\right)  \dot{\xi}\left(  t\right)
\right]  _{\otimes}\right)  _{k}}\right\vert _{M}\nonumber\\
&  \leq\varphi\left(  t-1\right)  \left(
\begin{array}
[c]{c}%
\left\vert V_{\left(  \left[  A,B\right]  _{\otimes}\right)  _{k}}\right\vert
_{M}+\left\vert V_{\left(  \left[  \bar{C}^{\xi}\left(  t\right)  ,B\right]
_{\otimes}\right)  _{k}}\right\vert _{M}\\
+\left\vert V_{\left(  \left[  C^{\xi}\left(  t\right)  ,\bar{u}\left(
s,\operatorname{ad}_{C^{\xi}\left(  t\right)  }\right)  B\right]  _{\otimes
}\right)  _{k}}\right\vert _{M}%
\end{array}
\right)  \label{e.4.47}%
\end{align}
We now estimate each of the three terms appearing on the right side of Eq.
(\ref{e.4.47}).

\begin{enumerate}
\item Since, for $m,n\in\left[  1,\kappa\right]  $ with $m+n=k,$
\[
\left\vert A_{m}\right\vert \left\vert B_{n}\right\vert \leq N\left(
A\right)  ^{m}N\left(  B\right)  ^{n}\leq N\left(  A\right)  N\left(
B\right)  \left(  N\left(  A\right)  +N\left(  B\right)  \right)  ^{k-2},
\]
we find%
\begin{align*}
\left\vert V_{\left(  \left[  A,B\right]  _{\otimes}\right)  _{k}}\right\vert
_{M}  &  \leq\sum_{m,n=1}^{\kappa}1_{m+n=k}\left\vert \left[  V_{A_{m}%
},V_{B_{n}}\right]  \right\vert \\
&  \leq\sum_{m,n=1}^{\kappa}1_{m+n=k}\mathcal{C}_{m,n}^{0}\left(  V^{\left(
\kappa\right)  }\right)  \left\vert A_{m}\right\vert \left\vert B_{n}%
\right\vert \\
&  \leq\mathcal{C}^{0}\left(  V^{\left(  \kappa\right)  }\right)  N\left(
A\right)  N\left(  B\right)  \left(  N\left(  A\right)  +N\left(  B\right)
\right)  ^{k-2}.
\end{align*}

\item Using Eq. (\ref{e.4.39}) and (by definition) $\bar{C}_{1}^{\xi}=0,$ it
follows that
\begin{align*}
&  \left\vert V_{\left(  \left[  \bar{C}^{\xi}\left(  t\right)  ,B\right]
_{\otimes}\right)  _{k}}\right\vert _{M}\\
&  \quad\leq\sum_{m,n=1}^{\kappa}1_{m+n=k}\mathcal{C}_{m,n}^{0}\left(
V^{\left(  \kappa\right)  }\right)  \left\vert \bar{C}_{m}^{\xi}\left(
t\right)  \right\vert \left\vert B_{n}\right\vert \\
&  \quad\lesssim\sum_{n=1}^{\kappa}\sum_{m=2}^{\kappa}1_{m+n=k}\mathcal{C}%
_{m,n}^{0}\left(  V^{\left(  \kappa\right)  }\right)  N\left(  A\right)
N\left(  B\right)  \left(  N\left(  A\right)  +N\left(  B\right)  \right)
^{m-2}N\left(  B\right)  ^{n-1}\\
&  \quad\lesssim\mathcal{C}^{0}\left(  V^{\left(  \kappa\right)  }\right)
N\left(  A\right)  N\left(  B\right)  \left(  N\left(  A\right)  +N\left(
B\right)  \right)  ^{k-2}.
\end{align*}

\item Similarly using Eqs. (\ref{e.3.30}) and (\ref{e.4.40}),%
\begin{align*}
&  \left\vert V_{\left[  C^{\xi}\left(  t\right)  ,\bar{u}\left(
s,\operatorname{ad}_{C^{\xi}\left(  t\right)  }\right)  B\right]  _{\otimes
k}}\right\vert _{M}\\
&  \quad\leq\sum_{m,n=1}^{\kappa}1_{m+n=k}\mathcal{C}_{m,n}^{0}\left(
V^{\left(  \kappa\right)  }\right)  \left\vert C_{m}^{\xi}\left(  t\right)
\right\vert \left\vert \left(  \bar{u}\left(  s,\operatorname{ad}_{C^{\xi
}\left(  t\right)  }\right)  B\right)  _{n}\right\vert \\
&  \quad\lesssim\sum_{n=2}^{\kappa}\sum_{m=1}^{\kappa}1_{m+n=k}\mathcal{C}%
_{m,n}^{0}\left(  V^{\left(  \kappa\right)  }\right)  \left\vert C_{m}^{\xi
}\left(  t\right)  \right\vert \cdot N\left(  A\right)  N\left(  B\right)
\left(  N\left(  A\right)  +N\left(  B\right)  \right)  ^{n-2}\\
&  \quad\lesssim\sum_{n=2}^{\kappa}\sum_{m=1}^{\kappa}1_{m+n=k}\mathcal{C}%
_{m,n}^{0}\left(  V^{\left(  \kappa\right)  }\right)  \cdot N\left(  A\right)
N\left(  B\right)  \left(  N\left(  A\right)  +N\left(  B\right)  \right)
^{m+n-2}\\
&  \quad\lesssim\mathcal{C}^{0}\left(  V^{\left(  \kappa\right)  }\right)
N\left(  A\right)  N\left(  B\right)  \left(  N\left(  A\right)  +N\left(
B\right)  \right)  ^{k-2}.
\end{align*}

\end{enumerate}

Combining the last three estimates with Eqs. (\ref{e.4.47}) and (\ref{e.4.46})
shows (with $\mathcal{K}_{1}$ having the form as in Eq. (\ref{e.4.43})),%
\begin{align}
\left\vert U^{\xi}\right\vert _{2}^{\ast}  &  \leq e^{\left\vert \nabla
V^{\left(  \kappa\right)  }\right\vert _{M}\left\vert C\right\vert _{\infty
,2}}\sum_{k=\kappa+1}^{2\kappa}\int_{0}^{1}ds\left(  1-s\right)  \int_{0}%
^{2}dt\left\vert V_{\left(  \left[  C\left(  t\right)  ,u\left(
s,\operatorname{ad}_{C\left(  t\right)  }\right)  \dot{\xi}\left(  t\right)
\right]  _{\otimes}\right)  _{k}}\right\vert _{M}\nonumber\\
&  \lesssim e^{\left\vert \nabla V^{\left(  \kappa\right)  }\right\vert
_{M}\left\vert C\right\vert _{\infty,2}}\sum_{k=\kappa+1}^{2\kappa}%
\mathcal{C}^{0}\left(  V^{\left(  \kappa\right)  }\right)  N\left(  A\right)
N\left(  B\right)  \left(  N\left(  A\right)  +N\left(  B\right)  \right)
^{k-2}\nonumber\\
&  \lesssim\mathcal{C}^{0}\left(  V^{\left(  \kappa\right)  }\right)
e^{\left\vert \nabla V^{\left(  \kappa\right)  }\right\vert _{M}\left\vert
C\right\vert _{\infty,2}}N\left(  A\right)  N\left(  B\right)  Q_{\left[
\kappa-1,2\kappa-2\right]  }\left(  N\left(  A\right)  +N\left(  B\right)
\right) \nonumber\\
&  \lesssim\mathcal{K}_{1}N\left(  A\right)  N\left(  B\right)  Q_{\left[
\kappa-1,2\kappa-2\right]  }\left(  N\left(  A\right)  +N\left(  B\right)
\right)  , \label{e.4.48}%
\end{align}
wherein the last inequality we have also used the estimate in Eq.
(\ref{e.3.31}). This estimate combined with Eq. (\ref{e.4.45}), while using
Eqs. (\ref{e.3.20}) and (\ref{e.3.19}) in order to show $\left\vert
A\right\vert +\left\vert B\right\vert \lesssim Q_{\left[  1,\kappa\right]
}\left(  N\left(  A\right)  +N\left(  B\right)  \right)  ,$ completes the proof.
\end{proof}

This completes part I. of the paper. The second remaining part of the paper is
devoted to developing estimates for the distance between the differentials of
$\mu_{T,0}^{V_{\dot{\xi}}}$ and $e^{V_{\log\left(  g^{\xi}\left(  T\right)
\right)  }}.$ In order to formulate our results we must first define a
distance between $f_{\ast}$ and $g_{\ast}$ for $f,g\in C^{1}\left(
M,M\right)  .$ To do so we will use the metric on $M$ to endow $TM$ with a
Riemannian metric and then make use of this metric to construct the desired
distance. It will also be necessary to develop some of the basic properties of
the induced distance function on $TM$ which is the topic of the next section.

\section{Riemannian Distances on $TM$\label{sec.5}}

\subsection{Riemannian distances on vector bundles\label{sec.5.1}}

For clarity of exposition (and since it is no harder), it is convenient to
carry out these constructions in the more general context of an arbitrary
Hermitian vector bundle, $\pi:E\rightarrow M,$ with metric compatible
covariant derivative, $\nabla,$ as in Notation \ref{not.2.2}. Later we will
specialize to the case of interest where $E=TM.$

\begin{definition}
[Riemannian metric on $E$]\label{def.5.1}Continuing the setup in Notation
\ref{not.2.2}, we define a Riemannian metric on $TE$ by defining
\[
\left\langle \dot{\xi}\left(  0\right)  ,\dot{\eta}\left(  0\right)
\right\rangle _{TE}:=\left\langle \pi_{\ast}\dot{\xi}\left(  0\right)
,\pi_{\ast}\dot{\eta}\left(  0\right)  \right\rangle _{g}+\left\langle
\frac{\nabla\xi}{dt}\left(  0\right)  ,\frac{\nabla\eta}{dt}\left(  0\right)
\right\rangle _{E}%
\]
whenever $\xi\left(  t\right)  $ and $\eta\left(  t\right)  $ are two smooth
curves in $E$ such $\pi\left(  \xi\left(  0\right)  \right)  =\pi\left(
\eta\left(  0\right)  \right)  .$
\end{definition}

\begin{remark}
\label{rem.5.2}Let $\sigma\left(  t\right)  $ and $\gamma\left(  t\right)  $
be two smooth paths in $M$ so that $\sigma\left(  0\right)  =m=\gamma\left(
0\right)  $ and suppose that $\alpha\left(  t\right)  $ and $\beta\left(
t\right)  $ are two smooth paths in $\mathbb{R}^{D}.$ Then in the local model
described in Remark \ref{rem.2.3} we have,%
\begin{align*}
\pi_{\ast}\dot{\xi}\left(  0\right)   &  =\pi_{\ast}\left(  \dot{\sigma
}\left(  0\right)  ,\dot{\alpha}\left(  0\right)  _{\alpha\left(  0\right)
}\right)  =\dot{\sigma}\left(  0\right)  ,\\
\pi_{\ast}\dot{\eta}\left(  0\right)   &  =\pi_{\ast}\left(  \dot{\gamma
}\left(  0\right)  ,\dot{\beta}\left(  0\right)  _{\beta\left(  0\right)
}\right)  =\dot{\gamma}\left(  0\right)  ,\\
\frac{\nabla\xi}{dt}\left(  0\right)   &  =\left(  m,\dot{\alpha}\left(
0\right)  +\Gamma\left(  \dot{\sigma}\left(  0\right)  \right)  \alpha\left(
0\right)  \right)  ,\\
\frac{\nabla\eta}{dt}\left(  0\right)   &  =\left(  m,\dot{\beta}\left(
0\right)  +\Gamma\left(  \dot{\gamma}\left(  0\right)  \right)  \beta\left(
0\right)  \right)  ,\text{ and}\\
\left\langle \dot{\xi}\left(  0\right)  ,\dot{\eta}\left(  0\right)
\right\rangle _{TE}=  &  \left\langle \dot{\sigma}\left(  0\right)
,\dot{\gamma}\left(  0\right)  \right\rangle _{g}\\
&  +\left(  \dot{\alpha}\left(  0\right)  +\Gamma\left(  \dot{\sigma}\left(
0\right)  \right)  \alpha\left(  0\right)  \right)  \cdot\left(  \dot{\beta
}\left(  0\right)  +\Gamma\left(  \dot{\gamma}\left(  0\right)  \right)
\beta\left(  0\right)  \right)  .
\end{align*}
From this expression we see that $\left\langle \cdot,\cdot\right\rangle _{TE}$
is indeed a Riemannian metric on $E.$ For example, $\left\vert \dot{\xi
}\left(  0\right)  \right\vert _{TE}^{2}=0$ implies
\[
0=\left\vert \dot{\sigma}\left(  0\right)  \right\vert _{g}^{2}+\left\vert
\dot{\alpha}\left(  0\right)  +\Gamma\left(  \dot{\sigma}\left(  0\right)
\right)  \alpha\left(  0\right)  \right\vert _{\mathbb{R}^{D}}^{2}%
\]
from which it follows that $\dot{\sigma}\left(  0\right)  =0$ and then
$\left\vert \dot{\alpha}\left(  0\right)  \right\vert _{\mathbb{R}^{D}}^{2}=0$
so that $\dot{\alpha}\left(  0\right)  =0,$ i.e. $\dot{\xi}\left(  0\right)
=0\in T_{\xi\left(  0\right)  }E.$
\end{remark}

\begin{definition}
\label{def.5.3}As usual, the length of a smooth path, $t\rightarrow\xi\left(
t\right)  \in E,$ is defined by
\[
\ell_{E}\left(  \xi\right)  =\int_{0}^{1}\left\vert \dot{\xi}\left(  t\right)
\right\vert dt=\int_{0}^{1}\sqrt{\left\vert \pi_{\ast}\dot{\xi}\left(
t\right)  \right\vert ^{2}+\left\vert \frac{\nabla\xi\left(  t\right)  }%
{dt}\right\vert _{E}^{2}}dt
\]
and the distance, $d^{E},$ is then the distance associated to this length.
\end{definition}

Our first goal is to give a more practical way (see Eq. (\ref{e.5.7}) of
Corollary \ref{cor.5.8} below) of computing $d^{E}\left(  e,e^{\prime}\right)
$ for $e,e^{\prime}\in E,$

\begin{notation}
\label{not.5.4}Given a path $\sigma:\left[  0,1\right]  \rightarrow M$, let%
\[
L_{\sigma}\left(  e,e^{\prime}\right)  :=\sqrt{\ell_{M}\left(  \sigma\right)
^{2}+\left\vert \pt_{1}\left(  \sigma\right)  ^{-1}e^{\prime}-e\right\vert
^{2}}\text{ }\forall~e\in E_{\sigma\left(  0\right)  }\text{ and }e^{\prime
}\in E_{\sigma\left(  1\right)  }%
\]
with the convention that $L_{\sigma}\left(  e,e^{\prime}\right)  =\infty$ if
$\sigma$ is not absolutely continuous.
\end{notation}

\begin{theorem}
\label{thm.5.6}If $\sigma\in AC\left(  \left[  0,1\right]  ,M\right)  ,$
$\xi\in AC_{\sigma}\left(  \left[  0,1\right]  ,E\right)  ,$ and
\[
s\left(  t\right)  :=\int_{0}^{t}\left\vert \dot{\sigma}\left(  \tau\right)
\right\vert d\tau\text{ -- arc-length of }\sigma|_{\left[  0,t\right]  },
\]
then%
\begin{align}
L_{\sigma}\left(  \xi\left(  0\right)  ,\xi\left(  1\right)  \right)  \leq &
\sqrt{\ell_{M}\left(  \sigma\right)  ^{2}+\left[  \int_{0}^{1}\left\vert
\nabla_{t}\xi\left(  t\right)  \right\vert dt\right]  ^{2}}\nonumber\\
&  \qquad\qquad\leq\ell_{E}\left(  \xi\right)  \leq\int_{0}^{1}\left[
\left\vert \dot{\sigma}\left(  t\right)  \right\vert +\left\vert \nabla_{t}%
\xi\left(  t\right)  \right\vert \right]  dt. \label{e.5.3}%
\end{align}
and moreover
\begin{equation}
L_{\sigma}\left(  \xi\left(  0\right)  ,\xi\left(  1\right)  \right)
=\sqrt{\ell_{M}\left(  \sigma\right)  ^{2}+\left[  \int_{0}^{1}\left\vert
\nabla_{t}\xi\left(  t\right)  \right\vert dt\right]  ^{2}}=\ell_{E}\left(
\xi\right)  \label{e.5.4}%
\end{equation}
when \footnote{If $0=s\left(  1\right)  =\ell\left(  \sigma\right)  ,$ then
necessarily $m=\sigma\left(  0\right)  =\sigma\left(  1\right)  =p.$}%
\begin{align}
\xi\left(  t\right)   &  =\xi\left(  0\right)  +t\left(  \xi\left(  1\right)
-\xi\left(  0\right)  \right)  \text{ if }s\left(  1\right)  =0\text{
or}\label{e.5.5}\\
\xi\left(  t\right)   &  =\pt_{t}\left(  \sigma\right)  \left[  e_{m}%
+\frac{s\left(  t\right)  }{s\left(  1\right)  }\left(  \pt_{1}\left(
\sigma\right)  ^{-1}\xi\left(  1\right)  -\xi\left(  0\right)  \right)
\right]  \text{ if }s\left(  1\right)  >0. \label{e.5.6}%
\end{align}

\end{theorem}

\begin{proof}
If we let $w\left(  t\right)  :=\pt_{t}\left(  \sigma\right)  ^{-1}\xi\left(
t\right)  \in E_{\sigma\left(  0\right)  },$ then $\left\vert \nabla_{t}%
\xi\left(  t\right)  \right\vert =\left\vert \dot{w}\left(  t\right)
\right\vert $ and so
\[
\int_{0}^{1}\left\vert \nabla_{t}\xi\left(  t\right)  \right\vert dt=\int%
_{0}^{1}\left\vert \dot{w}\left(  t\right)  \right\vert dt\geq\left\vert
w\left(  1\right)  -w\left(  0\right)  \right\vert =\left\vert \pt_{1}\left(
\sigma\right)  ^{-1}\xi\left(  1\right)  -\xi\left(  0\right)  \right\vert ,
\]
wherein we have used the length of $w$ is greater than or equal $\left\vert
w\left(  1\right)  -w\left(  0\right)  \right\vert .$ The last inequality is
equivalent to the first inequality in Eq. (\ref{e.5.3}).

If we let%
\[
u\left(  t\right)  :=\int_{0}^{t}\left\vert \nabla_{\tau}\xi\left(
\tau\right)  \right\vert d\tau,
\]
then $t\rightarrow\left(  s\left(  t\right)  ,u\left(  t\right)  \right)
\in\mathbb{R}^{2}$ is an absolutely continuous path in $\mathbb{R}^{2}$ from
$\left(  0,0\right)  $ and so the length of this path,%
\[
\int_{0}^{1}\sqrt{\dot{s}\left(  t\right)  ^{2}+\dot{u}\left(  t\right)  ^{2}%
}dt=\int_{0}^{1}\sqrt{\left\vert \dot{\sigma}\left(  t\right)  \right\vert
^{2}+\left\vert \nabla_{t}\xi\left(  t\right)  \right\vert ^{2}}dt=\ell
_{E}\left(  \xi\right)  ,
\]
is is greater than or equal to
\[
\left\Vert \left(  s\left(  1\right)  ,u\left(  1\right)  \right)  \right\Vert
_{\mathbb{R}^{2}}=\sqrt{\ell_{M}\left(  \sigma\right)  ^{2}+\left[  \int%
_{0}^{1}\left\vert \nabla_{t}\xi\left(  t\right)  \right\vert dt\right]  ^{2}%
}.
\]
This proves the second inequality in Eq. (\ref{e.5.3}). To prove the last
inequality in Eq. (\ref{e.5.3}) simply observe (see Eq. (\ref{e.3.22}) with
$p=2)$ that%
\[
\sqrt{\left\vert \dot{\sigma}\left(  t\right)  \right\vert ^{2}+\left\vert
\nabla_{t}\xi\left(  t\right)  \right\vert ^{2}}\leq\left\vert \dot{\sigma
}\left(  t\right)  \right\vert +\left\vert \nabla_{t}\xi\left(  t\right)
\right\vert .
\]

If $s\left(  1\right)  >0$ and $\xi$ is given as in Eq. (\ref{e.5.6}), then
\[
\left\vert \nabla_{t}\xi\left(  t\right)  \right\vert =\left\vert
\pt_{t}\left(  \sigma\right)  \frac{\dot{s}\left(  t\right)  }{s\left(
1\right)  }\left(  \pt_{1}\left(  \sigma\right)  ^{-1}e_{p}^{\prime}%
-e_{m}\right)  \right\vert =\frac{\left\vert \dot{\sigma}\left(  t\right)
\right\vert }{\ell_{M}\left(  \sigma\right)  }\left\vert \pt_{1}\left(
\sigma\right)  ^{-1}e_{p}^{\prime}-e_{m}\right\vert
\]
and hence%
\begin{align*}
\ell_{E}\left(  \xi\right)   &  =\int_{0}^{1}\sqrt{\left\vert \dot{\sigma
}\left(  t\right)  \right\vert ^{2}+\frac{\left\vert \dot{\sigma}\left(
t\right)  \right\vert ^{2}}{s^{2}\left(  1\right)  }\left\vert \pt_{1}\left(
\sigma\right)  ^{-1}\xi\left(  1\right)  -\xi\left(  0\right)  \right\vert
^{2}}dt\\
&  =\int_{0}^{1}\frac{\left\vert \dot{\sigma}\left(  t\right)  \right\vert
}{\ell_{M}\left(  \sigma\right)  }\sqrt{\ell_{M}^{2}\left(  \sigma\right)
+\left\vert \pt_{1}\left(  \sigma\right)  ^{-1}e_{p}^{\prime}-e_{m}\right\vert
^{2}}dt=L_{\sigma}\left(  \xi\left(  0\right)  ,\xi\left(  1\right)  \right)
\end{align*}
which verifies Eq. (\ref{e.5.4}) in this case. Similarly by a simple
calculation, Eq. (\ref{e.5.4}) holds when $\ell_{M}\left(  \sigma\right)
=s\left(  1\right)  =0$ and $\xi$ is given as in Eq. (\ref{e.5.5}).
\end{proof}

\begin{notation}
\label{not.5.7}To each $\sigma\in AC\left(  \left[  0,1\right]  ,M\right)  ,$
let $AC_{\sigma}\left(  \left[  0,1\right]  ,E\right)  $ denote those $\xi\in
AC\left(  \left[  0,1\right]  ,E\right)  $ such that $\xi\left(  t\right)  \in
T_{\sigma\left(  t\right)  }M$ for all $0\leq t\leq1.$
\end{notation}

\begin{corollary}
\label{cor.5.8}If $e_{m}\in E_{m},$ and $e_{p}^{\prime}\in E_{p},$ then%
\begin{equation}
d^{E}\left(  e_{m},e_{p}^{\prime}\right)  =\inf\left\{  L_{\sigma}\left(
e_{m},e_{p}^{\prime}\right)  :\sigma\in AC\left(  \left[  0,1\right]
,M\right)  ,\text{ }\sigma\left(  0\right)  =m\text{~\& }\sigma\left(
1\right)  =p\right\}  \label{e.5.7}%
\end{equation}
where for $\sigma\in AC\left(  \left[  0,1\right]  ,M\right)  $ with
$\sigma\left(  0\right)  =m$ and $\sigma\left(  1\right)  =p,$
\begin{equation}
L_{\sigma}\left(  e_{m},e_{p}^{\prime}\right)  =\min\left\{  \ell_{E}\left(
\xi\right)  :\xi\in AC_{\sigma}\left(  \left[  0,1\right]  ,E\right)  ,\text{
}\xi\left(  0\right)  =e_{m},\text{\& }\xi\left(  1\right)  =e_{p}^{\prime
}\right\}  . \label{e.5.8}%
\end{equation}

\end{corollary}

\begin{proof}
The first equation is an easy consequence of the second. For the second
equation, if $\xi\in AC_{\sigma}\left(  \left[  0,1\right]  ,E\right)  $ with
$\xi\left(  0\right)  =e_{m},$ and $\xi\left(  1\right)  =e_{p}^{\prime},$
then by Theorem \ref{thm.5.6}, $L_{\sigma}\left(  e_{m},e_{p}^{\prime}\right)
\leq\ell_{E}\left(  \xi\right)  $ with equality occurring when $\xi$ is given
by Eq. (\ref{e.5.5}) if $\ell_{M}\left(  \sigma\right)  =0$ or by Eq.
(\ref{e.5.6}) if $\ell_{M}\left(  \sigma\right)  >0.$
\end{proof}

\begin{remark}
\label{rem.5.9}One might suspect that if $e,e^{\prime}\in E_{m},$ then
$d^{E}\left(  e,e^{\prime}\right)  =\left\vert e-e^{\prime}\right\vert .$
However this is not necessarily the case unless the holonomy group of
$\nabla^{E}$ at $m$ is trivial (in particular this implies the curvature of
$\nabla^{E}=0.)$ For example of $\left\vert e\right\vert =\left\vert
e^{\prime}\right\vert =1,$ there may be a very short loop, $\sigma,$ starting
and ending at $m,$ so that $\pt_{1}\left(  \sigma\right)  ^{-1}e^{\prime}=e$
in which case it would follow that $d^{E}\left(  e,e^{\prime}\right)  \leq
\ell_{M}\left(  \sigma\right)  $ which can easily be smaller that $\left\vert
e-e^{\prime}\right\vert $ which could be as large at $\sqrt{2}.$ If $\sigma$
is the constant loop sitting at $m,$ then $L_{\sigma}\left(  e,e^{\prime
}\right)  =\left\vert e^{\prime}-e\right\vert $ and hence $d^{E}\left(
e,e^{\prime}\right)  \leq\left\vert e^{\prime}-e\right\vert $ whenever
$e,e^{\prime}\in E_{m}$ for some $m\in M.$
\end{remark}

\begin{proposition}
[$\left\vert \cdot\right\vert _{E}$ is Lipschitz ]\label{pro.5.10}If
$e,e^{\prime}\in E,$ then
\begin{equation}
\left\vert \left\vert e\right\vert _{E}-\left\vert e^{\prime}\right\vert
_{E}\right\vert \leq d^{E}\left(  e,e^{\prime}\right)  , \label{e.5.9}%
\end{equation}
i.e. fiber metric on $E$, $\left\vert \cdot\right\vert _{E},$ is $1$-Lipschitz
relative to $d^{E}.$
\end{proposition}

\begin{proof}
Let $e_{m}$ and $e_{p}^{\prime}$ in $E$ and $\sigma$ be an absolutely
continuous path joining $m$ to $p.$ Then by Lemma \ref{lem.2.4},
\[
\left\vert \left\vert e_{m}\right\vert _{E}-\left\vert e_{p}^{\prime
}\right\vert _{E}\right\vert \leq\left\vert e_{m}-\pt_{1}\left(
\sigma\right)  ^{-1}e_{p}^{\prime}\right\vert _{E_{m}}\leq L_{\sigma}\left(
e_{m},e_{p}^{\prime}\right)
\]
and therefore by Corollary \ref{cor.5.8},
\[
\left\vert \left\vert e_{m}\right\vert _{E}-\left\vert e_{p}^{\prime
}\right\vert _{E}\right\vert \leq\inf_{\sigma}L_{\sigma}\left(  e_{m}%
,e_{p}^{\prime}\right)  =d^{E}\left(  e_{m},e_{p}^{\prime}\right)  ,
\]
where the infimum is over all paths, $\sigma,$ joining $m$ to $p.$
\end{proof}

\begin{proposition}
[Completeness of $E$]\label{pro.5.11}If $\left(  M,g\right)  $ is a complete
Riemannian manifold then the vector bundle, $E,$ with the Riemannian structure
in Definition \ref{def.5.1} is again a complete Riemannian manifold.
\end{proposition}

\begin{proof}
Let $\pi:E\rightarrow M$ be the natural projection map and observe that
\[
\left\vert \pi_{\ast}\dot{\xi}\left(  0\right)  \right\vert _{M}\leq\left\vert
\dot{\xi}\left(  0\right)  \right\vert _{E}\text{ }\forall~\dot{\xi}\left(
0\right)  \in TE.
\]
If $e_{0},e_{1}\in E$ and $e\left(  \cdot\right)  \in AC\left(  \left[
0,1\right]  ,E\right)  $ is a path joining $e_{0}$ to $e_{1},$ then $\pi\circ
e\in AC\left(  \left[  0,1\right]  ,M\right)  $ is path joining $\pi\left(
e_{0}\right)  $ to $\pi\left(  e_{1}\right)  $ and
\begin{align*}
d_{M}\left(  \pi\left(  e_{0}\right)  ,\pi\left(  e_{1}\right)  \right)   &
\leq\ell_{M}\left(  \pi\circ e\right)  =\int_{0}^{1}\left\vert \pi_{\ast}%
\dot{u}\left(  t\right)  \right\vert _{M}dt\\
&  \leq\int_{0}^{1}\left\vert \dot{u}\left(  t\right)  \right\vert _{E}%
dt=\ell_{E}\left(  e\right)  .
\end{align*}
Minimizing this inequality over $e$ as described above shows
\[
d_{M}\left(  \pi\left(  e_{0}\right)  ,\pi\left(  e_{1}\right)  \right)  \leq
d^{E}\left(  e_{0},e_{1}\right)  .
\]
Hence if $\left\{  e_{n}\right\}  _{n=1}^{\infty}$ is a Cauchy sequence in
$E,$ then $\left\{  p_{n}=\pi\left(  e_{n}\right)  \right\}  _{n=1}^{\infty}$
is a Cauchy sequence in $M.$ As $M$ is complete we know that $p=\lim
_{n\rightarrow\infty}p_{n}$ exists in $M.$ Let $W$ be an open neighborhood of
$p$ in $M$ over which $M$ is trivial and let $U$ be a local orthonormal frame
(as described after Notation \ref{not.2.2}) of $E$ defined over $W$ and, for
large enough $n,$ let $v_{n}:=U\left(  p_{n}\right)  ^{-1}e_{n}\in
\mathbb{R}^{N}$ where $N$ is the fiber dimension of $E.$ From Proposition
\ref{pro.5.10}, we know $\left\{  \left\vert e_{n}\right\vert _{E}=\left\vert
v_{n}\right\vert _{\mathbb{R}^{N}}\right\}  _{n=1}^{\infty}$ is a Cauchy
sequence in $\mathbb{R}$ and hence bounded and hence there exists a
subsequence, $\left\{  v_{n_{k}}\right\}  _{k=1}^{\infty}$ of $\left\{
v_{n}\right\}  $ so that $v:=\lim_{k\rightarrow\infty}v_{n_{k}}\ $exists in
$\mathbb{R}^{N}.$ It then follows that
\[
\lim_{k\rightarrow\infty}e_{n_{k}}=\lim_{k\rightarrow\infty}U\left(  p_{n_{k}%
}\right)  v_{n_{k}}=U\left(  p\right)  v\text{ exists.}%
\]
As $\left\{  e_{n}\right\}  _{n=1}^{\infty}$ was Cauchy in $E$ and has a
convergent subsequence, it follows that $\lim_{n\rightarrow\infty}%
e_{n}=U\left(  p\right)  v$ exists in $E$ and hence $E$ is complete.
\end{proof}

\begin{theorem}
\label{thm.5.12}Let $\pi:E\rightarrow M$ be a vector bundle equipped with a
fiber metric and metric compatible covariant derivative as above. If
$\lambda\geq0$ and $e_{m},e_{p}^{\prime}\in E,$ then
\begin{equation}
d^{E}\left(  \lambda e_{m},\lambda e_{p}^{\prime}\right)  \leq\left(
\lambda\vee1\right)  d^{E}\left(  e_{m},e_{p}^{\prime}\right)  .
\label{e.5.12}%
\end{equation}

\end{theorem}

\begin{proof}
Let $\sigma$ be a curve joining $m$ to $p,$ then
\begin{align*}
d^{E}\left(  \lambda e_{m},\lambda e_{p}^{\prime}\right)   &  \leq L_{\sigma
}\left(  \lambda e_{m},\lambda e_{p}^{\prime}\right)  =\sqrt{\ell_{M}\left(
\sigma\right)  ^{2}+\left\vert \lambda\pt_{1}\left(  \sigma\right)
^{-1}e^{\prime}-\lambda e\right\vert ^{2}}\\
&  =\sqrt{\ell_{M}\left(  \sigma\right)  ^{2}+\left\vert \lambda\right\vert
^{2}\left\vert \pt_{1}\left(  \sigma\right)  ^{-1}e^{\prime}-e\right\vert
^{2}}\leq\lambda\vee1\cdot L_{\sigma}\left(  e_{m},e_{p}^{\prime}\right)
\end{align*}
and the result now follows from Corollary \ref{cor.5.8} as $\sigma$ was arbitrary.
\end{proof}

\begin{definition}
[Bundle maps]\label{def.5.13}A smooth function, $F:E\rightarrow E,$ is a
\textbf{bundle map} provided there exist a smooth map, $f:M\rightarrow M$ such
that $F\left(  E_{m}\right)  \subset E_{f\left(  m\right)  }$ for all $m\in M$
and $F|_{E_{m}}:E_{m}\rightarrow E_{f\left(  m\right)  }$ is linear. We will
refer to such an $F$ as a \textbf{bundle map covering }$f.$
\end{definition}

We are interested in measuring the distance between two bundle maps,
$F,G:E\rightarrow E.$ For such maps we can no longer define $d_{\infty}%
^{E}\left(  F,G\right)  :=\sup_{e\in E}d^{E}\left(  Fe,Ge\right)  $ since
\[
\sup_{\lambda>0}d^{E}\left(  F\lambda e,G\lambda e\right)  =\infty\text{ if
}\left\vert Fe\right\vert \neq\left\vert Ge\right\vert .
\]
Indeed if $\sigma\in C^{1}\left(  \left[  0,1\right]  ,M\right)  $ is any path
such that $Fe\in E_{\sigma\left(  0\right)  }$ and $Ge\in E_{\sigma\left(
1\right)  },$ then%
\begin{align*}
L_{\sigma}\left(  \lambda Fe,\lambda Ge\right)   &  =\sqrt{\ell_{M}\left(
\sigma\right)  ^{2}+\left\vert \pt_{1}\left(  \sigma\right)  ^{-1}\lambda
Ge-\lambda Fe\right\vert ^{2}}\\
&  \geq\left\vert \left\vert \pt_{1}\left(  \sigma\right)  ^{-1}\lambda
Ge\right\vert -\left\vert \lambda Fe\right\vert \right\vert =\left\vert
\lambda\right\vert \left\vert \left\vert Ge\right\vert -\left\vert
Fe\right\vert \right\vert
\end{align*}
and hence by Corollary \ref{cor.5.8},
\[
d^{E}\left(  F\lambda e,G\lambda e\right)  \geq\left\vert \lambda\right\vert
\left\vert \left\vert Ge\right\vert -\left\vert Fe\right\vert \right\vert
\rightarrow\infty\text{ as }\lambda\uparrow\infty.
\]
On the other hand, as bundle maps are fiber linear they are determined by
their values, $\left\{  Fe:e\in E\text{ with }\left\vert e\right\vert
=1\right\}  .$ With these comments in mind we make the following definition.

\begin{definition}
[Bundle map norms and distances]\label{def.5.14}Given a bundle maps,
$F:E\rightarrow E,$ $m\in M,$ and $\sigma\in C\left(  \left[  0,1\right]
,M\right)  ,$ let
\begin{align*}
\left\vert F\right\vert _{m}  &  :=\sup_{e\in E_{m}:\left\vert e\right\vert
=1}\left\vert Fe\right\vert ,\\
\left\vert F\right\vert _{\sigma}  &  :=\sup_{t\in\left[  0,1\right]
}\left\vert F\right\vert _{\sigma\left(  t\right)  },\text{ and}\\
\left\vert F\right\vert _{M}  &  :=\sup_{m\in M}\left\vert F\right\vert
_{m}=\sup_{e\in E_{m}:\left\vert e\right\vert =1}\left\vert Fe\right\vert .
\end{align*}
If $G:E\rightarrow E$ is another bundle map, let
\[
d_{\infty}^{E}\left(  F,G\right)  :=\sup_{e\in E:\left\vert e\right\vert
=1}d^{E}\left(  Fe,Ge\right)  .
\]

\end{definition}

\begin{remark}
\label{rem.5.15}Let us note that $d_{\infty}^{E}\left(  F,G\right)  =0$ iff
$Fe=Ge$ for all $\left\vert e\right\vert =1$ which suffices to show $F\equiv
G$ since both $F$ and $G$ are fiber linear.
\end{remark}

\begin{lemma}
\label{lem.5.16}If $F,G:E\rightarrow E$ are bundle maps, then
\begin{equation}
\left\vert \left\vert F\right\vert _{M}-\left\vert G\right\vert _{M}%
\right\vert \leq d_{\infty}^{E}\left(  F,G\right)  . \label{e.5.14}%
\end{equation}

\end{lemma}

\begin{proof}
If $e\in E$ with $\left\vert e\right\vert =1,$ then (by Proposition
\ref{pro.5.10})
\[
\left\vert \left\vert Fe\right\vert -\left\vert Ge\right\vert \right\vert \leq
d^{E}\left(  Fe,Ge\right)  \leq d_{\infty}^{E}\left(  F,G\right)
\]
and therefore
\[
\left\vert Fe\right\vert \leq\left\vert Ge\right\vert +d_{\infty}^{E}\left(
F,G\right)  \leq\left\vert G\right\vert _{M}+d_{\infty}^{E}\left(  F,G\right)
.
\]
As this is true for all $e\in E$ with $\left\vert e\right\vert =1$ we may
further conclude that
\[
\left\vert F\right\vert _{M}\leq\left\vert G\right\vert _{M}+d_{\infty}%
^{E}\left(  F,G\right)  .
\]
Reversing the roles of $F$ and $G$ also shows%
\[
\left\vert G\right\vert _{M}\leq\left\vert F\right\vert _{M}+d_{\infty}%
^{E}\left(  F,G\right)
\]
and together the last two displayed equations proves Eq. (\ref{e.5.14}).
\end{proof}

The next proposition contains the typical mechanism we will use for estimating
$d_{\infty}^{E}\left(  F,G\right)  .$

\begin{proposition}
\label{pro.5.17}If $\left\{  F_{t}\right\}  _{0\leq t\leq1}$ is a smoothly
varying one parameter family of bundle maps from $E$ to $E$ covering $\left\{
f_{t}\right\}  _{0\leq t\leq1}\subset C^{\infty}\left(  M,M\right)  ,$ then
for any $e\in E_{m},$%
\begin{equation}
d^{E}\left(  F_{0}e,F_{1}e\right)  \leq\sqrt{\ell_{M}^{2}\left(  f_{\left(
\cdot\right)  }\left(  m\right)  \right)  +\left[  \int_{0}^{1}\left\vert
\frac{\nabla F_{t}}{dt}e\right\vert dt\right]  ^{2}}, \label{e.5.15}%
\end{equation}
and
\begin{align}
d_{\infty}^{E}\left(  F_{0},F_{1}\right)   &  \leq\sqrt{\sup_{m\in M}\ell
_{M}^{2}\left(  f_{\left(  \cdot\right)  }\left(  m\right)  \right)  +\left[
\int_{0}^{1}\left\vert \frac{\nabla F_{t}}{dt}\right\vert _{M}dt\right]  ^{2}%
}\label{e.5.16}\\
&  \leq\sup_{m\in M}\ell_{M}\left(  f_{\left(  \cdot\right)  }\left(
m\right)  \right)  +\int_{0}^{1}\left\vert \frac{\nabla F_{t}}{dt}\right\vert
_{M}dt\label{e.5.17}\\
&  \leq\int_{0}^{1}\left[  \left\vert \dot{f}_{t}\right\vert _{M}+\left\vert
\frac{\nabla F_{t}}{dt}\right\vert _{M}\right]  dt. \label{e.5.18}%
\end{align}

\end{proposition}

\begin{proof}
If $e\in E_{m},$ then, by Corollary \ref{cor.5.8} with $\sigma\left(
t\right)  =f_{t}\left(  m\right)  ,$%
\begin{align}
d^{E}\left(  F_{0}e,F_{1}e\right)  \leq &  L_{\sigma}\left(  F_{0}%
e,F_{1}e\right) \nonumber\\
&  =\sqrt{\ell_{M}^{2}\left(  t\rightarrow f_{t}\left(  m\right)  \right)
+\left\vert \pt_{1}\left(  f_{\left(  \bullet\right)  }\left(  m\right)
\right)  ^{-1}F_{1}e-F_{0}e\right\vert ^{2}}. \label{e.5.19}%
\end{align}
This inequality along with Lemma \ref{lem.2.4} applied with $\xi\left(
t\right)  =F_{t}e$ then gives Eq. (\ref{e.5.15}) and Eq. (\ref{e.5.15}) along
with Eq. (\ref{e.3.22}) with $p=2$ then show,%
\begin{equation}
d^{E}\left(  F_{0}e,F_{1}e\right)  \leq\ell_{M}\left(  f_{\left(
\cdot\right)  }\left(  m\right)  \right)  +\int_{0}^{1}\left\vert \frac{\nabla
F_{t}}{dt}e\right\vert dt. \label{e.5.20}%
\end{equation}
Taking the supremum of these estimates over $\left\vert e\right\vert =1$ then
give the remaining stated estimates since, Definition \ref{def.5.14},%
\begin{equation}
\left\vert \frac{\nabla F_{t}}{dt}\right\vert _{M}:=\sup\left\{  \left\vert
\frac{\nabla F_{t}}{dt}e\right\vert :e\in E\text{ with }\left\vert
e\right\vert =1\right\}  . \label{e.5.21}%
\end{equation}

\end{proof}

\begin{remark}
\label{rem.5.18}A more elementary way to arrive at Eq. (\ref{e.5.20}) is again
to let $\sigma\left(  t\right)  =f_{t}\left(  m\right)  $ and $\xi\left(
t\right)  =F_{t}e$ and then observe that
\begin{align*}
d^{E}\left(  F_{0}e,F_{1}e\right)   &  \leq\ell_{E}\left(  \xi\right)
=\int_{0}^{1}\sqrt{\left\vert \dot{\sigma}\left(  t\right)  \right\vert
^{2}+\left\vert \frac{\nabla\xi\left(  t\right)  }{dt}\right\vert ^{2}}dt\\
&  \leq\int_{0}^{1}\left(  \left\vert \dot{\sigma}\left(  t\right)
\right\vert +\left\vert \frac{\nabla\xi\left(  t\right)  }{dt}\right\vert
\right)  dt=\ell_{M}\left(  f_{\left(  \cdot\right)  }\left(  m\right)
\right)  +\int_{0}^{1}\left\vert \frac{\nabla F_{t}}{dt}e\right\vert dt
\end{align*}
wherein we have used Eq. (\ref{e.3.22}) with $p=2$ for the last inequality.
\end{remark}

Lastly we turn our attention to estimating $d^{E}\left(  Fe,Fe^{\prime
}\right)  ,$ where $e,e^{\prime}\in E$ and $F:E\rightarrow E$ is a bundle map
covering $f:M\rightarrow M.$ As a warm up let us begin with the following flat
special case.

\begin{lemma}
\label{lem.5.19}Suppose $M=\mathbb{R}^{n}$ with the standard metric, $\left(
W,\left\langle \cdot,\cdot\right\rangle \right)  $ is a finite dimensional
inner product space, and $E=M\times W$ which is equipped with flat covariant
derivative, i.e. $\Gamma\equiv0$ in this trivialization. [We denote $e=\left(
m,w\right)  \in E=M\times W$ by $w_{m}.]$ If $f\in C^{\infty}\left(
M,M\right)  $ and $\hat{F}\in C^{\infty}\left(  M,\operatorname*{End}\left(
W\right)  \right)  ,$ then $Fw_{m}:=\left[  \hat{F}\left(  m\right)  w\right]
_{f\left(  m\right)  }$ is a bundle map covering $f$ and this map satisfies,%
\begin{equation}
d^{E}\left(  Fw_{m},Fw_{p}^{\prime}\right)  \leq\left(  \max\left\{
\operatorname{Lip}\left(  f\right)  ,\left\Vert \hat{F}\left(  p\right)
\right\Vert \right\}  +\left\vert \hat{F}^{\prime}\right\vert _{M}\left\Vert
w\right\Vert \right)  d^{E}\left(  w_{m},w_{p}^{\prime}\right)  .
\label{e.5.22}%
\end{equation}

\end{lemma}

\begin{proof}
Written in this form we find%
\begin{align*}
\left(  d^{E}\right)  ^{2}  &  \left(  Fw_{m},Fw_{p}^{\prime}\right) \\
&  =\left\Vert f\left(  m\right)  -f\left(  p\right)  \right\Vert
^{2}+\left\Vert \hat{F}\left(  m\right)  w-\hat{F}\left(  p\right)  w^{\prime
}\right\Vert ^{2}\\
&  \leq\operatorname{Lip}^{2}\left(  f\right)  \left\Vert m-p\right\Vert
^{2}+\left(  \left\Vert \hat{F}\left(  m\right)  w-\hat{F}\left(  p\right)
w\right\Vert +\left\Vert \hat{F}\left(  p\right)  \left(  w-w^{\prime}\right)
\right\Vert \right)  ^{2}\\
&  \leq\operatorname{Lip}^{2}\left(  f\right)  \left\Vert m-p\right\Vert
^{2}+\left(  \left\vert \hat{F}^{\prime}\right\vert _{M}\left\Vert
w\right\Vert \left\Vert m-p\right\Vert +\left\Vert \hat{F}\left(  p\right)
\right\Vert \left\Vert w-w^{\prime}\right\Vert \right)  ^{2}\\
&  =\operatorname{Lip}^{2}\left(  f\right)  \left\Vert m-p\right\Vert
^{2}+\left\Vert \hat{F}\left(  p\right)  \right\Vert ^{2}\left\Vert
w-w^{\prime}\right\Vert ^{2}+\left\vert \hat{F}^{\prime}\right\vert _{M}%
^{2}\left\Vert w\right\Vert ^{2}\left\Vert m-p\right\Vert ^{2}\\
&  \qquad+2\left\vert \hat{F}^{\prime}\right\vert _{M}\left\Vert w\right\Vert
\left\Vert \hat{F}\left(  p\right)  \right\Vert \left\Vert m-p\right\Vert
\left\Vert w-w^{\prime}\right\Vert .
\end{align*}
Using $\rho=\max\left\{  \operatorname{Lip}\left(  f\right)  ,\left\Vert
\hat{F}\left(  p\right)  \right\Vert \right\}  $ the above estimate implies,%
\begin{align*}
\left(  d^{E}\right)  ^{2}\left(  Fw_{m},Fw_{p}^{\prime}\right)   &
\leq\left[  \rho^{2}+\left\vert \hat{F}^{\prime}\right\vert _{M}^{2}\left\Vert
w\right\Vert ^{2}+2\left\vert \hat{F}^{\prime}\right\vert _{M}\left\Vert
w\right\Vert \rho\right]  \left(  d^{E}\right)  ^{2}\left(  w_{m}%
,w_{p}^{\prime}\right) \\
&  =\left(  \rho+\left\vert \hat{F}^{\prime}\right\vert _{M}\left\Vert
w\right\Vert \right)  ^{2}\left(  d^{E}\right)  ^{2}\left(  w_{m}%
,w_{p}^{\prime}\right)
\end{align*}
which gives the estimate in Eq. (\ref{e.5.22}).
\end{proof}

By swapping $w_{m}$ with $w_{p}^{\prime}$ in Eq. (\ref{e.5.22}) we of course
also have%
\begin{equation}
d^{E}\left(  Fw_{m},Fw_{p}^{\prime}\right)  \leq\left[  \max\left\{
\operatorname{Lip}\left(  f\right)  ,\left\Vert \hat{F}\left(  m\right)
\right\Vert \right\}  +\left\vert \hat{F}^{\prime}\right\vert _{M}\left\Vert
w^{\prime}\right\Vert \right]  d^{E}\left(  w_{m},w_{p}^{\prime}\right)  .
\label{e.5.23}%
\end{equation}
In Theorem \ref{thm.5.23} below, we will show that the analogue of Eq.
(\ref{e.5.23}) holds in full generality. The following notation will be used
in the statement of this theorem.

\begin{definition}
\label{def.5.20}Suppose that $f\in C^{\infty}\left(  M,M\right)  $ and
$F:E\rightarrow E$ is a bundle map covering $f.$ For $v\in T_{m}M,$ let
$\nabla_{v}F\in\operatorname{Hom}\left(  E_{m},E_{f\left(  m\right)  }\right)
$ be defined by;%
\[
\nabla_{v}F:=\frac{d}{dt}|_{0}\pt_{t}\left(  f\circ\sigma\right)
^{-1}F_{\sigma\left(  t\right)  }\pt_{t}\left(  \sigma\right)
\]
where $\sigma$ is any $C^{1}$-cure in $M$ such that $\dot{\sigma}\left(
0\right)  =v$ and $F_{\sigma\left(  t\right)  }:=F|_{E_{\sigma\left(
t\right)  }}.$
\end{definition}

\begin{lemma}
[Product rule]\label{lem.5.21}If $F:E\rightarrow E$ is a bundle map covering
$f,$ $S\in\Gamma\left(  E\right)  $ and $\sigma\left(  t\right)  \in
M,\ $then
\begin{equation}
\frac{\nabla}{dt}|_{0}\left(  FS\right)  \left(  \sigma\left(  t\right)
\right)  =\left(  \nabla_{\dot{\sigma}\left(  0\right)  }F\right)  S\left(
m\right)  +F_{\sigma\left(  0\right)  }\nabla_{\dot{\sigma}\left(  0\right)
}S. \label{e.5.24}%
\end{equation}

\end{lemma}

\begin{proof}
This result is easily reduced to the standard product rule matrices and
vectors as follows;%
\begin{align*}
\frac{\nabla}{dt}|_{0}\left(  FS\right)  \left(  \sigma\left(  t\right)
\right)  =  &  \frac{d}{dt}|_{0}\left(  \pt_{t}\left(  f\circ\sigma\right)
^{-1}\left[  F_{\sigma\left(  t\right)  }S\left(  \sigma\left(  t\right)
\right)  \right]  \right) \\
=  &  \frac{d}{dt}|_{0}\left(  \pt_{t}\left(  f\circ\sigma\right)
^{-1}\left[  F_{\sigma\left(  t\right)  }\pt_{t}\left(  \sigma\right)
\pt_{t}\left(  \sigma\right)  ^{-1}S\left(  \sigma\left(  t\right)  \right)
\right]  \right) \\
=  &  \frac{d}{dt}|_{0}\left(  \pt_{t}\left(  f\circ\sigma\right)
^{-1}\left[  F_{\sigma\left(  t\right)  }\pt_{t}\left(  \sigma\right)
S\left(  m\right)  \right]  \right) \\
&  \quad+\frac{d}{dt}|_{0}\left(  \left[  F_{\sigma\left(  0\right)  }%
\pt_{t}\left(  \sigma\right)  ^{-1}S\left(  \sigma\left(  t\right)  \right)
\right]  \right)
\end{align*}
which is equivalent to Eq. (\ref{e.5.24}).
\end{proof}

\begin{notation}
\label{not.5.22}Given $m\in M,$ $\sigma\in C\left(  \left[  0,1\right]
,M\right)  ,$ $f\in C^{\infty}\left(  M,M\right)  ,$ and a bundle map,
$F:E\rightarrow E,$ covering $f,$ let%
\begin{align*}
\left\vert \nabla F\right\vert _{m}  &  :=\sup_{v\in T_{m}M:\left\vert
v\right\vert =1}\left\vert \nabla_{v}F\right\vert _{op}:=\sup_{v\in
T_{m}M:\left\vert v\right\vert =1}\sup_{e\in E_{m}:\left\vert e\right\vert
=1}\left\vert \left(  \nabla_{v}F\right)  e\right\vert ,\\
\left\vert \nabla F\right\vert _{\sigma}  &  :=\sup_{t\in\left[  0,1\right]
}\left\vert \nabla F\right\vert _{\sigma\left(  t\right)  },\text{ and}\\
\left\vert \nabla F\right\vert _{M}  &  :=\sup_{m\in M}\left\vert \nabla
F\right\vert _{m}.
\end{align*}

\end{notation}

\begin{theorem}
\label{thm.5.23}Let $F:E\rightarrow E$ is a bundle map covering
$f:M\rightarrow M,$ $e\in E_{m},$ $e^{\prime}\in E_{p},$ and $\sigma\in
AC\left(  \left[  0,1\right]  ,M\right)  $ be a curve such that $\sigma\left(
0\right)  =m$ and $\sigma\left(  1\right)  =p.$ Then%
\begin{equation}
d^{E}\left(  Fe,Fe^{\prime}\right)  \leq\left(  \max\left(  \left\vert
f_{\ast}\right\vert _{\sigma},\left\vert F_{m}\right\vert \right)  +\left\vert
\nabla F\right\vert _{\sigma}\cdot\left\vert e^{\prime}\right\vert \right)
L_{\sigma}\left(  e,e^{\prime}\right)  , \label{e.5.25}%
\end{equation}
and in particular,%
\begin{equation}
d^{E}\left(  Fe,Fe^{\prime}\right)  \leq\left(  \max\left(  \operatorname{Lip}%
\left(  f\right)  ,\left\vert F_{m}\right\vert \right)  +\left\vert \nabla
F\right\vert _{M}\left\vert e^{\prime}\right\vert \right)  d^{E}\left(
e,e^{\prime}\right)  . \label{e.5.26}%
\end{equation}

\end{theorem}

\begin{proof}
To simplify notation in the proof below let
\begin{align*}
\rho &  :=\max\left(  \left\vert f_{\ast}\right\vert _{\sigma},\left\vert
F_{m}\right\vert \right)  ,\\
\tilde{e}  &  :=\pt_{1}\left(  \sigma\right)  ^{-1}e^{\prime}\in E_{m},\text{
and }\\
A_{t}  &  :=\pt_{t}\left(  f\circ\sigma\right)  ^{-1}F_{\sigma\left(
t\right)  }\pt_{t}\left(  \sigma\right)  :E_{m}\rightarrow E_{f\left(
m\right)  }.
\end{align*}
By Corollary \ref{cor.5.8}, it follows that
\begin{align*}
d^{E}\left(  Fe,Fe^{\prime}\right)  \leq L_{f\circ\sigma}\left(
Fe,Fe^{\prime}\right)   &  =\sqrt{\ell_{M}^{2}\left(  f\circ\sigma\right)
+\left\vert \pt_{1}\left(  f\circ\sigma\right)  ^{-1}Fe^{\prime}-Fe\right\vert
^{2}}.\\
&  =\sqrt{\ell_{M}^{2}\left(  f\circ\sigma\right)  +\left\vert A_{1}\tilde
{e}-A_{0}e\right\vert ^{2}}.
\end{align*}
The first term in the square root is estimated by,%
\[
\ell_{M}\left(  f\circ\sigma\right)  =\int_{0}^{1}\left\vert f_{\ast}%
\dot{\sigma}\left(  t\right)  \right\vert dt\leq\left\vert f_{\ast}\right\vert
_{\sigma}\ell_{M}\left(  \sigma\right)  .
\]
For the second term, we note that%
\[
\left\vert \frac{d}{dt}A_{t}\right\vert =\left\vert \pt_{t}\left(
f\circ\sigma\right)  ^{-1}\left(  \nabla_{\dot{\sigma}\left(  t\right)
}F\right)  \pt_{t}\left(  \sigma\right)  \right\vert =\left\vert \nabla
_{\dot{\sigma}\left(  t\right)  }F\right\vert \leq\left\vert \nabla
F\right\vert _{\sigma}\left\vert \dot{\sigma}\left(  t\right)  \right\vert
\]
and hence%
\[
\left\vert A_{1}-A_{0}\right\vert _{op}=\int_{0}^{1}\left\vert \frac{d}%
{dt}A_{t}\right\vert dt\leq\left\vert \nabla F\right\vert _{\sigma}\ell
_{M}\left(  \sigma\right)  .
\]
Thus we conclude that
\begin{align*}
\left\vert A_{1}\tilde{e}-A_{0}e\right\vert  &  \leq\left\vert A_{1}\tilde
{e}-A_{0}\tilde{e}\right\vert +\left\vert A_{0}\left[  \tilde{e}-e\right]
\right\vert \\
&  \leq\left\vert \nabla F\right\vert _{\sigma}\ell_{M}\left(  \sigma\right)
\left\vert \tilde{e}\right\vert +\left\vert F_{m}\right\vert \left\vert
\tilde{e}-e\right\vert \\
&  =\left\vert \nabla F\right\vert _{\sigma}\ell_{M}\left(  \sigma\right)
\left\vert e^{\prime}\right\vert +\left\vert F_{m}\right\vert \left\vert
\pt_{1}\left(  \sigma\right)  ^{-1}e^{\prime}-e\right\vert .
\end{align*}
Combining the previous estimates then shows,%
\begin{align*}
\left(  d^{E}\right)  ^{2}  &  \left(  Fe,Fe^{\prime}\right) \\
\leq &  \left\vert f_{\ast}\right\vert _{\sigma}^{2}\ell_{M}^{2}\left(
\sigma\right)  +\left[  \left\vert \nabla F\right\vert _{\sigma}\ell
_{M}\left(  \sigma\right)  \left\vert e^{\prime}\right\vert +\left\vert
F_{m}\right\vert \left\vert \pt_{1}\left(  \sigma\right)  ^{-1}e^{\prime
}-e\right\vert \right]  ^{2}\\
=  &  \left\vert f_{\ast}\right\vert _{\sigma}^{2}\ell_{M}^{2}\left(
\sigma\right)  +\operatorname{Lip}_{\sigma}^{2}\left(  F\right)  \left\vert
e^{\prime}\right\vert ^{2}\ell_{M}^{2}\left(  \sigma\right)  +\left\vert
F_{m}\right\vert ^{2}\left\vert \pt_{1}\left(  \sigma\right)  ^{-1}e^{\prime
}-e\right\vert ^{2}\\
&  \qquad+2\left\vert F_{m}\right\vert \left\vert \pt_{1}\left(
\sigma\right)  ^{-1}e^{\prime}-e\right\vert \cdot\left\vert \nabla
F\right\vert _{\sigma}\ell_{M}\left(  \sigma\right)  \left\vert e^{\prime
}\right\vert \\
\leq &  \rho^{2}L_{\sigma}^{2}\left(  e,e^{\prime}\right)  +\operatorname{Lip}%
_{\sigma}^{2}\left(  F\right)  \left\vert e^{\prime}\right\vert ^{2}L_{\sigma
}^{2}\left(  e,e^{\prime}\right)  +2\left\vert F_{m}\right\vert \left\vert
\nabla F\right\vert _{\sigma}\left\vert e^{\prime}\right\vert L_{\sigma}%
^{2}\left(  e,e^{\prime}\right) \\
\leq &  \rho^{2}L_{\sigma}^{2}\left(  e,e^{\prime}\right)  +\operatorname{Lip}%
_{\sigma}^{2}\left(  F\right)  \left\vert e^{\prime}\right\vert ^{2}L_{\sigma
}^{2}\left(  e,e^{\prime}\right)  +2\rho\left\vert \nabla F\right\vert
_{\sigma}\left\vert e^{\prime}\right\vert L_{\sigma}^{2}\left(  e,e^{\prime
}\right) \\
&  \qquad=\left(  \rho+\left\vert \nabla F\right\vert _{\sigma}\left\vert
e^{\prime}\right\vert \right)  ^{2}L_{\sigma}^{2}\left(  e,e^{\prime}\right)
\end{align*}
which proves Eq. (\ref{e.5.25}). Moreover, Eq. (\ref{e.5.25} implies%
\[
d^{E}\left(  Fe,Fe^{\prime}\right)  \leq\left(  \max\left(  \operatorname{Lip}%
\left(  f\right)  ,\left\vert F_{m}\right\vert \right)  +\left\vert \nabla
F\right\vert _{M}\cdot\left\vert e^{\prime}\right\vert \right)  L_{\sigma
}\left(  e,e^{\prime}\right)
\]
and so taking the infimum of this last inequality over $\sigma\in AC\left(
\left[  0,1\right]  ,M\right)  $ such that $\sigma\left(  0\right)  =m$ and
$\sigma\left(  1\right)  =p$ gives (see Corollary \ref{cor.5.8}) Eq.
(\ref{e.5.26}).
\end{proof}

\subsection{Metrics on $TM$\label{sec.5.2}}

From now we are going to restrict our attention to the case of interest where
$E=TM$ and $F=f_{\ast}$ where $f\in C^{2}\left(  M,M\right)  .$ Before stating
the main result in Theorem \ref{thm.5.30} below, let us record that relevant
notions of covariant differentiation in this context.

\begin{definition}
[Vector-fields along $f$]\label{def.5.24}For $f\in C^{\infty}\left(
M,M\right)  ,$ let $\Gamma_{f}\left(  TM\right)  $ denote the \textbf{vector
fields} \textbf{along} $f,$ i.e. $U\in\Gamma_{f}\left(  TM\right)  $ iff
$U:M\rightarrow TM$ is a smooth function such that $U\left(  m\right)  \in
T_{f\left(  m\right)  }M$ for all $m\in M.$
\end{definition}

\begin{example}
\label{ex.5.25}If $Z\in\Gamma\left(  TM\right)  $ and $f\in C^{\infty}\left(
M,M\right)  ,$ then $f_{\ast}Z$ and $Z\circ f$ are both vector fields along
$f.$
\end{example}

\begin{definition}
\label{def.5.26}For $f\in C^{\infty}\left(  M,M\right)  ,$ $U\in\Gamma
_{f}\left(  TM\right)  ,$ and $v=v_{m}\in T_{m}M,$ let $\nabla_{v}U\in
T_{f\left(  m\right)  }M$ and $\nabla_{v}f_{\ast}$ be the linear map from
$T_{m}M$ to $T_{f\left(  m\right)  }M\ $be defined by,%
\begin{align*}
\nabla_{v}U  &  =\frac{\nabla}{dt}|_{0}U\left(  \sigma\left(  t\right)
\right)  =\frac{d}{dt}|_{0}\left[  \pt_{t}\left(  f\circ\sigma\right)
^{-1}U\left(  \sigma\left(  t\right)  \right)  \right]  \text{ and }\\
\nabla_{v}f_{\ast}  &  =\frac{d}{dt}|_{0}\left[  \pt_{t}\left(  f\circ
\sigma\right)  ^{-1}f_{\ast\sigma\left(  t\right)  }\pt_{t}\left(
\sigma\right)  \right]
\end{align*}
where is any $C^{1}$-curve in $M$ such that $\dot{\sigma}\left(  0\right)
=v_{m}.$ [It is easily verified by working in local trivializations of $TM$
that $\nabla_{v}U$ and $\nabla_{v}f_{\ast}$ are well defined independent of
the choice of $\sigma$ such that $\dot{\sigma}\left(  0\right)  =v_{m}.]$
\end{definition}

\begin{proposition}
[Chain and product rules]\label{pro.5.27}If $f\in C^{\infty}\left(
M,M\right)  ,$ $Z\in\Gamma\left(  TM\right)  ,$ and $v\in T_{m}M,$ then%
\begin{align}
\nabla_{v}\left[  Z\circ f\right]   &  =\nabla_{f_{\ast}v}Z\text{ and
}\label{e.5.27}\\
\nabla_{v}\left[  f_{\ast}Z\right]   &  =\left(  \nabla_{v}f_{\ast}\right)
Z\left(  m\right)  +f_{\ast}\nabla_{v}Z. \label{e.5.28}%
\end{align}
More generally if $U\in\Gamma_{f}\left(  TM\right)  $ and $g\in C^{\infty
}\left(  M,M\right)  ,$ then $U\circ g\in\Gamma_{f\circ g}\left(  M\right)  ,$
$g_{\ast}U\in\Gamma_{g\circ f}\left(  M\right)  ,$%
\begin{align}
\nabla_{v}\left[  U\circ g\right]   &  =\nabla_{g_{\ast}v}U,\text{
and}\label{e.5.29}\\
\nabla_{v}\left[  g_{\ast}U\right]   &  =\left(  \nabla_{f_{\ast}v}g_{\ast
}\right)  U\left(  m\right)  +g_{\ast m}\nabla_{v}U. \label{e.5.30}%
\end{align}

\end{proposition}

\begin{proof}
If $\sigma\left(  t\right)  \in M$ is chosen so that $\dot{\sigma}\left(
0\right)  =v_{m},$ then%
\[
\nabla_{v}\left[  Z\circ f\right]  =\frac{d}{dt}|_{0}\left[  \pt_{t}\left(
f\circ\sigma\right)  ^{-1}\left(  Z\circ f\right)  \left(  \sigma\left(
t\right)  \right)  \right]  =\nabla_{f_{\ast}v}Z
\]
and
\begin{align*}
\nabla_{v}\left[  f_{\ast}Z\right]  =  &  \frac{d}{dt}|_{0}\left[
\pt_{t}\left(  f\circ\sigma\right)  ^{-1}f_{\ast\sigma\left(  t\right)
}Z\left(  \sigma\left(  t\right)  \right)  \right] \\
=  &  \frac{d}{dt}|_{0}\left[  \pt_{t}\left(  f\circ\sigma\right)
^{-1}f_{\ast\sigma\left(  t\right)  }\pt_{t}\left(  \sigma\right)
~\pt_{t}\left(  \sigma\right)  ^{-1}Z\left(  \sigma\left(  t\right)  \right)
\right] \\
=  &  \frac{d}{dt}|_{0}\left[  \pt_{t}\left(  f\circ\sigma\right)
^{-1}f_{\ast\sigma\left(  t\right)  }\pt_{t}\left(  \sigma\right)  \right]
Z\left(  m\right) \\
&  +f_{\ast m}\frac{d}{dt}|_{0}\left[  \pt_{t}\left(  \sigma\right)
^{-1}Z\left(  \sigma\left(  t\right)  \right)  \right] \\
=  &  \left(  \nabla_{v}f_{\ast}\right)  Z\left(  m\right)  +f_{\ast}%
\nabla_{v}Z.
\end{align*}

The more general cases are proved similarly;%
\begin{align*}
\nabla_{v}\left[  U\circ g\right]   &  =\frac{d}{dt}|_{0}\left[
\pt_{t}\left(  f\circ g\circ\sigma\right)  ^{-1}\left(  U\circ g\right)
\left(  \sigma\left(  t\right)  \right)  \right] \\
&  =\frac{d}{dt}|_{0}\left[  \pt_{t}\left(  f\circ\left(  g\circ\sigma\right)
\right)  ^{-1}U\left(  g\circ\sigma\right)  \left(  t\right)  \right] \\
&  =\nabla_{g_{\ast}v}U
\end{align*}
and
\begin{align*}
\nabla_{v}\left[  g_{\ast}U\right]  =  &  \frac{d}{dt}|_{0}\left[
\pt_{t}\left(  g\circ f\circ\sigma\right)  ^{-1}g_{\ast}U\left(  \sigma\left(
t\right)  \right)  \right] \\
=  &  \frac{d}{dt}|_{0}\left[  \pt_{t}\left(  g\circ f\circ\sigma\right)
^{-1}g_{\ast}\pt_{t}\left(  f\circ\sigma\right)  ^{-1}~\pt_{t}\left(
f\circ\sigma\right)  U\left(  \sigma\left(  t\right)  \right)  \right] \\
=  &  \frac{d}{dt}|_{0}\left[  \pt_{t}\left(  g\circ f\circ\sigma\right)
^{-1}g_{\ast}\pt_{t}\left(  f\circ\sigma\right)  ^{-1}~U\left(  m\right)
\right] \\
&  +\frac{d}{dt}|_{0}\left[  g_{\ast m}~\pt_{t}\left(  f\circ\sigma\right)
U\left(  \sigma\left(  t\right)  \right)  \right] \\
=  &  \left(  \nabla_{f_{\ast}v}g_{\ast}\right)  U\left(  m\right)  +g_{\ast
m}\nabla_{v}U.
\end{align*}

\end{proof}

\begin{corollary}
\label{cor.5.28}If $f\in\mathrm{Diff}\left(  M\right)  ,$ $Z\in\Gamma\left(
TM\right)  ,$ and $v\in T_{m}M,$ then
\[
\nabla_{v}\left[  \operatorname{Ad}_{f}Z\right]  =\left(  \nabla_{f_{\ast
}^{-1}v}f_{\ast}\right)  Z\left(  f^{-1}\left(  m\right)  \right)  +f_{\ast
}\nabla_{f_{\ast}^{-1}v}Z.
\]

\end{corollary}

\begin{proof}
Since $\operatorname{Ad}_{f}Z=\left(  f_{\ast}Z\right)  \circ f^{-1}$ with
$f_{\ast}Z\in\Gamma_{f}\left(  TM\right)  ,$ it follows by first applying Eq.
(\ref{e.5.29}) and then Eq. (\ref{e.5.30}) that
\[
\nabla_{v}\left[  \operatorname{Ad}_{f}Z\right]  =\nabla_{f_{\ast}^{-1}%
v}\left(  f_{\ast}Z\right)  =\left(  \nabla_{f_{\ast}^{-1}v}f_{\ast}\right)
Z\left(  f^{-1}\left(  m\right)  \right)  +f_{\ast}\nabla_{f_{\ast}^{-1}v}Z.
\]

\end{proof}

\begin{definition}
\label{def.5.29}Let $d^{TM}:TM\times TM\rightarrow\lbrack0,\infty)$ be the
metric on $TM$ associated to the Riemannian metric on $E=TM$ with the given
fiber Riemannian metric $g.$
\end{definition}

In this setting,
\[
\max\left(  \operatorname{Lip}\left(  f\right)  ,\left\vert F_{m}\right\vert
\right)  =\max\left(  \operatorname{Lip}\left(  f\right)  ,\left\vert f_{\ast
m}\right\vert \right)  =\operatorname{Lip}\left(  f\right)
\]
and hence the next theorem is an immediate consequence of Theorem
\ref{thm.5.23}.

\begin{theorem}
[$d_{TM}\left(  f_{\ast}v_{m},f_{\ast}w_{p}\right)  $ estimates]%
\label{thm.5.30}Let $v_{m},w_{p}\in TM$ and $f\in C^{2}\left(  M,M\right)  $
and for any path $\sigma\in AC\left(  \left[  0,1\right]  ,M\right)  $ with
$\sigma\left(  0\right)  =v_{m}$ and $\sigma\left(  1\right)  =w_{p},$ let%
\begin{equation}
L_{\sigma}\left(  v_{m},w_{p}\right)  :=\sqrt{\ell_{M}\left(  \sigma\right)
^{2}+\left\vert \pt_{1}\left(  \sigma\right)  ^{-1}w_{p}-v_{m}\right\vert
^{2}}. \label{e.5.31}%
\end{equation}
Then
\begin{equation}
d^{TM}\left(  f_{\ast}v_{m},f_{\ast}w_{p}\right)  \leq\left[  \left\vert
f_{\ast}\right\vert _{\sigma}+\left\vert \nabla f_{\ast}\right\vert _{\sigma
}\cdot\left\vert w_{p}\right\vert \right]  L_{\sigma}\left(  v_{m}%
,w_{p}\right)  \label{e.5.32}%
\end{equation}
and consequently,\footnote{The next inequality may be localized if necessary.
The point is we may assume that $\ell\left(  \sigma\right)  \leq d_{TM}\left(
v_{m},w_{p}\right)  $ and so we need compute $\operatorname{Lip}\left(
f\right)  $ and $\operatorname{Lip}\left(  f_{\ast}\right)  $ over the ball,
$B\left(  m,d_{TM}\left(  v_{m},w_{p}\right)  \right)  .$}
\begin{equation}
d^{TM}\left(  f_{\ast}v_{m},f_{\ast}w_{p}\right)  \leq\left(
\operatorname{Lip}\left(  f\right)  +\left\vert \nabla f_{\ast}\right\vert
_{M}\cdot\left\vert w_{p}\right\vert \right)  d^{TM}\left(  v_{m}%
,w_{p}\right)  . \label{e.5.33}%
\end{equation}

\end{theorem}

\section{First order derivative estimates\label{sec.6}}

\subsection{$\nabla\nu_{t\ast}$ -- estimates\label{sec.6.1}}

Suppose that $W_{t}\in\Gamma\left(  TM\right)  $ and $\nu_{t}\in C^{\infty
}\left(  M,M\right)  $ are as in Notation \ref{not.2.25}. Our next goal is to
estimate the local Lipschitz-norm of $\nu_{t\ast}.$ We will do this using
Theorem \ref{thm.5.30} which requires us to estimate $\nabla\nu_{t\ast}.$ We
begin by finding the differential equation solved by $\nabla\nu_{t\ast}.$

\begin{proposition}
\label{pro.6.1}If $W_{t}\in\Gamma\left(  TM\right)  $ and $\nu_{t}\in
C^{\infty}\left(  M,M\right)  $ are as in Notation \ref{not.2.25}, $m\in M,$
and $v_{m},\xi_{m}\in T_{m}M,$ then $\left(  \nabla_{v_{m}}\nu_{t\ast}\right)
\xi_{m}$ satisfies the covariant differential equation;%
\begin{align}
\nabla_{t}\left(  \nabla_{v_{m}}\nu_{t\ast}\right)  \xi_{m}  &  =\left(
\nabla W_{t}\right)  \left[  \left(  \nabla_{v_{m}}\nu_{t\ast}\right)  \xi
_{m}\right]  +\left(  \nabla^{2}W_{t}\right)  \left[  \nu_{t\ast}v_{m}%
\otimes\nu_{t\ast}\xi_{m}\right] \nonumber\\
&  +R\left(  W_{t}\left(  \nu_{t}\left(  m\right)  \right)  ,\nu_{t\ast}%
v_{m}\right)  \nu_{t\ast}\xi_{m}. \label{e.6.1}%
\end{align}

\end{proposition}

\begin{proof}
Let $\sigma\left(  s\right)  $ be a smooth curve in $M$ such that
$v_{m}:=\sigma^{\prime}\left(  0\right)  $ and define $\xi\left(  s\right)
:=\pt_{s}\left(  \sigma\right)  \xi_{m}.$ With this notation we have
\begin{align}
\frac{\nabla}{ds}|_{0}\left[  \nu_{t\ast}\xi\left(  s\right)  \right]   &
=\frac{d}{ds}|_{0}\left[  \pt_{s}\left(  \nu_{t}\circ\sigma\right)  ^{-1}%
\nu_{t\ast}\xi\left(  s\right)  \right] \nonumber\\
&  =\frac{d}{ds}|_{0}\left[  \pt_{s}\left(  \nu_{t}\circ\sigma\right)
^{-1}\nu_{t\ast}\pt_{s}\left(  \sigma\right)  \xi_{m}\right]  =\left(
\nabla_{v_{m}}\nu_{t\ast}\right)  \xi_{m}. \label{e.6.2}%
\end{align}

Using the relationship of curvature to the commutator of covariant
derivatives,
\[
\left[  \nabla_{t},\nabla_{s}\right]  =R\left(  \frac{d}{dt}\nu_{t}\left(
\sigma\left(  s\right)  \right)  ,\frac{d}{ds}\nu_{t}\left(  \sigma\left(
s\right)  \right)  \right)  =R\left(  W_{t}\left(  \nu_{t}\left(
\sigma\left(  s\right)  \right)  \right)  ,\nu_{t\ast}\sigma^{\prime}\left(
s\right)  \right)  ,
\]
it follows that
\begin{equation}
\nabla_{t}\nabla_{s}\left[  \nu_{t\ast}\xi\left(  s\right)  \right]
=\nabla_{s}\nabla_{t}\left[  \nu_{t\ast}\xi\left(  s\right)  \right]
+R\left(  W_{t}\left(  \nu_{t}\left(  \sigma\left(  s\right)  \right)
\right)  ,\nu_{t\ast}\sigma^{\prime}\left(  s\right)  \right)  \nu_{t\ast}%
\xi\left(  s\right)  . \label{e.6.3}%
\end{equation}
By Proposition \ref{pro.2.26} and the product rule for covariant derivatives
the first term in Eq. (\ref{e.6.3}) may be written as%
\begin{align}
\nabla_{s}\nabla_{t}\left[  \nu_{t\ast}\xi\left(  s\right)  \right]   &
=\nabla_{s}\left[  \nabla_{\nu_{t\ast}\xi\left(  s\right)  }W_{t}\right]
\nonumber\\
&  =\left(  \nabla^{2}W_{t}\right)  \left[  \nu_{t\ast}\sigma^{\prime}\left(
s\right)  \otimes\nu_{t\ast}\xi\left(  s\right)  \right]  +\left(  \nabla
W_{t}\right)  \nabla_{s}\nu_{t\ast}\xi\left(  s\right)  . \label{e.6.4}%
\end{align}
Combining Eqs. (\ref{e.6.2})--(\ref{e.6.4}) gives,%
\begin{align*}
\nabla_{t}\left(  \nabla_{v_{m}}\nu_{t\ast}\right)  \xi_{m}  &  =\nabla
_{t}\frac{\nabla}{ds}|_{0}\left[  \nu_{t\ast}\xi\left(  s\right)  \right] \\
&  =\left[  \left(  \nabla^{2}W_{t}\right)  \left[  \nu_{t\ast}\sigma^{\prime
}\left(  s\right)  \otimes\nu_{t\ast}\xi\left(  s\right)  \right]  +\left(
\nabla W_{t}\right)  \nabla_{s}\nu_{t\ast}\xi\left(  s\right)  \right]
_{s=0}\\
&  +\left[  R\left(  W_{t}\left(  \nu_{t}\left(  \sigma\left(  s\right)
\right)  \right)  ,\nu_{t\ast}\sigma^{\prime}\left(  s\right)  \right)
\nu_{t\ast}\xi\left(  s\right)  \right]  _{s=0}%
\end{align*}
which is the same as Eq. (\ref{e.6.1}).
\end{proof}

Recall from Notations \ref{not.1.3} and \ref{not.1.5} (also see Example
\ref{ex.1.6}) that
\begin{equation}
H_{m}\left(  W_{t}\right)  =\left\vert \nabla^{2}W_{t}\right\vert
_{m}+\left\vert R\left(  W_{t},\bullet\right)  \right\vert _{m} \label{e.6.5}%
\end{equation}
and for a closed interval, $J\subset\left[  0,T\right]  ,$ that
\begin{equation}
H\left(  W_{\cdot}\right)  _{J}^{\ast}=\int_{J}H_{M}\left(  W_{t}\right)
dt.=\int_{J}\sup_{m\in M}H_{m}\left(  W_{t}\right)  dt. \label{e.6.6}%
\end{equation}

\begin{corollary}
[$\left\vert \nabla\nu_{t\ast}\right\vert _{M}$ -estimate]\label{cor.6.2}If
$W_{t}\in\Gamma\left(  TM\right)  $ and $\nu_{t}\in C^{\infty}\left(
M,M\right)  $ are as in Notation \ref{not.2.25} and we let%
\begin{align}
k_{J}\left(  m\right)   &  :=\int_{J}\left\vert \nabla W\right\vert
_{\nu_{\tau}\left(  m\right)  }d\tau\leq\left\vert \nabla W\right\vert
_{J}^{\ast},\text{ and}\label{e.6.7}\\
K_{J}\left(  m\right)   &  :=\int_{J}H_{\nu_{\tau}\left(  m\right)  }\left(
W_{\tau}\right)  d\tau\leq H\left(  W_{\cdot}\right)  _{J}^{\ast},
\label{e.6.8}%
\end{align}
then
\begin{align}
\left\vert \nabla\nu_{t\ast}\right\vert _{m}  &  \leq e^{k_{J\left(
s,t\right)  }\left(  m\right)  }\left[  \left\vert \nabla\nu_{s\ast
}\right\vert _{m}+\left\vert \nu_{s\ast}\right\vert _{m}^{2}\int_{J\left(
s,t\right)  }H_{\nu_{\tau}\left(  m\right)  }\left(  W_{\tau}\right)
e^{k_{J\left(  s,\tau\right)  }\left(  m\right)  }d\tau\right]  \label{e.6.9}%
\\
&  \leq e^{k_{J\left(  s,t\right)  }\left(  m\right)  }\left\vert \nabla
\nu_{s\ast}\right\vert _{m}+e^{2k_{J\left(  s,t\right)  }\left(  m\right)
}K_{J\left(  s,t\right)  }\left(  m\right)  \left\vert \nu_{s\ast}\right\vert
_{m}^{2}. \label{e.6.10}%
\end{align}
If we further assume that $\nu_{s}=Id_{M},$ then the above estimate reduces to%
\begin{align}
\left\vert \nabla\nu_{t\ast}\right\vert _{m}  &  \leq e^{k_{J\left(
s,t\right)  }\left(  m\right)  }\cdot\int_{J\left(  s,t\right)  }H_{\nu_{\tau
}\left(  m\right)  }\left(  W_{\tau}\right)  e^{k_{J\left(  s,\tau\right)
}\left(  m\right)  }d\tau\label{e.6.11}\\
&  \leq e^{2k_{J\left(  s,t\right)  }\left(  m\right)  }\cdot K_{J\left(
s,t\right)  }\left(  m\right)  \label{e.6.12}%
\end{align}
and in particular,%
\begin{equation}
\left\vert \nabla\nu_{t\ast}\right\vert _{M}\leq e^{2\left\vert \nabla
W\right\vert _{J\left(  s,t\right)  }^{\ast}}.H\left(  W_{\cdot}\right)
_{J\left(  s,t\right)  }^{\ast}. \label{e.6.13}%
\end{equation}

\end{corollary}

\begin{proof}
To shorten notation in this proof, let%
\[
h_{t}=H_{\nu_{t}\left(  m\right)  }\left(  W_{t}\right)  :=\left\vert
\nabla^{2}W_{t}\right\vert _{\nu_{t}\left(  m\right)  }+\left\vert R\left(
W_{t},\bullet\right)  \right\vert _{\nu_{t}\left(  m\right)  }.
\]
Starting with Eq. (\ref{e.6.1}) while using the estimate in Eq. (\ref{e.2.45})
allows us to easily conclude that
\begin{align*}
\left\vert \nabla_{t}\left(  \nabla_{v_{m}}\nu_{t\ast}\right)  \right\vert  &
\leq\left\vert \nabla W_{t}\right\vert _{\nu_{t}\left(  m\right)  }\left\vert
\nabla_{v_{m}}\nu_{t\ast}\right\vert +\left\vert \nabla^{2}W_{t}\right\vert
_{\nu_{t}\left(  m\right)  }\left\vert \nu_{t\ast}v_{m}\right\vert \left\vert
\nu_{t\ast}\right\vert _{m}\\
&  +\left\vert R\left(  W_{t}\left(  \nu_{t}\left(  m\right)  \right)
,\nu_{t\ast}v_{m}\right)  \right\vert \left\vert \nu_{t\ast}\right\vert
_{m}.\\
&  \leq\left\vert \nabla W_{t}\right\vert _{\nu_{t}\left(  m\right)
}\left\vert \nabla_{v_{m}}\nu_{t\ast}\right\vert +e^{2k_{J\left(  s,t\right)
}\left(  m\right)  }h_{t}\left\vert \nu_{s\ast}\right\vert _{m}^{2}%
\cdot\left\vert v_{m}\right\vert .
\end{align*}
It follows by the Bellman-Gronwall inequality in Corollary \ref{cor.9.3} of
the appendix that
\begin{align*}
\left\vert \nabla\nu_{t\ast}\right\vert _{m}\leq &  e^{\int_{J\left(
s,t\right)  }\left\vert \nabla W_{s}\right\vert _{\nu_{s}\left(  m\right)
}ds}\left\vert \nabla\nu_{s\ast}\right\vert _{m}\\
&  +\int_{J\left(  s,t\right)  }e^{\int_{J\left(  \tau,t\right)  }\left\vert
\nabla W_{s}\right\vert _{\nu_{s}\left(  m\right)  }ds}e^{2k_{J\left(
s,\tau\right)  }\left(  m\right)  }h_{\tau}\left\vert \nu_{s\ast}\right\vert
_{m}^{2}d\tau\\
&  =e^{k_{J\left(  s,t\right)  }\left(  m\right)  }\left\vert \nabla\nu
_{s\ast}\right\vert _{m}+\int_{J\left(  s,t\right)  }e^{k_{J\left(
\tau,t\right)  }\left(  m\right)  }e^{2k_{J\left(  s,\tau\right)  }\left(
m\right)  }h_{\tau}\left\vert \nu_{s\ast}\right\vert _{m}^{2}d\tau\\
&  =e^{k_{J\left(  s,t\right)  }\left(  m\right)  }\left\vert \nabla\nu
_{s\ast}\right\vert _{m}+\int_{J\left(  s,t\right)  }e^{k_{J\left(
s,t\right)  }\left(  m\right)  }e^{k_{J\left(  s,\tau\right)  }\left(
m\right)  }h_{\tau}\left\vert \nu_{s\ast}\right\vert _{m}^{2}d\tau\\
&  \qquad\leq e^{k_{J\left(  s,t\right)  }\left(  m\right)  }\left\vert
\nabla\nu_{s\ast}\right\vert _{m}+e^{2k_{J\left(  s,t\right)  }\left(
m\right)  }\int_{J\left(  s,t\right)  }h_{\tau}\left\vert \nu_{s\ast
}\right\vert _{m}^{2}d\tau.
\end{align*}
Lastly if $\nu_{s}=Id_{M}$ then $\nu_{s\ast}=Id_{TM}$ in which case
$\left\vert \nu_{s\ast}\right\vert _{m}^{2}=1$ and $\nabla\nu_{s\ast}=0$ and
so Eq. (\ref{e.6.10}) reduces to Eq. (\ref{e.6.11}).

\iffalse The last displayed inequality written out in detail is:%
\[
\left\vert \nabla\nu_{t\ast}\right\vert _{m}\leq e^{k_{J\left(  s,t\right)
}\left(  m\right)  }\left[  \left\vert \nabla\nu_{s\ast}\right\vert _{m}%
+\int_{J\left(  s,t\right)  }\left[  \left\vert \nabla^{2}W_{\tau}\right\vert
_{\nu_{\tau}\left(  m\right)  }+\left\vert R\left(  W_{\tau},\bullet\right)
\right\vert _{\nu_{\tau}\left(  m\right)  }\right]  \left\vert \nu_{s\ast
}\right\vert _{m}^{2}e^{k_{J\left(  s,\tau\right)  }\left(  m\right)  }%
d\tau\right]  .
\]
\fi
\end{proof}

\begin{corollary}
\label{cor.6.3}If $W_{t}\in\Gamma\left(  TM\right)  $ and $\nu_{t}\in
C^{\infty}\left(  M,M\right)  $ are as in Notation \ref{not.2.25} and further
assuming $\nu_{0}=Id_{M},$ then
\[
d^{TM}\left(  \nu_{t\ast}v_{m},\nu_{t\ast}w_{p}\right)  \leq e^{2\left\vert
\nabla W\right\vert _{t}^{\ast}}\left(  1+H\left(  W_{\cdot}\right)
_{t}^{\ast}\cdot\left\vert w_{p}\right\vert \right)  d^{TM}\left(  v_{m}%
,w_{p}\right)
\]

\end{corollary}

\begin{proof}
By Theorem \ref{thm.5.30} with $f=\nu_{t}$ along with Corollaries
\ref{cor.2.27} and Corollary \ref{cor.6.2} we find,
\begin{align*}
d^{TM}\left(  \nu_{t\ast}v_{m},\nu_{t\ast}w_{p}\right)   &  \leq\left(
\operatorname{Lip}\left(  \nu_{t}\right)  +\left\vert \nabla\nu_{t\ast
}\right\vert _{M}\cdot\left\vert w_{p}\right\vert \right)  d^{TM}\left(
v_{m},w_{p}\right) \\
&  \leq\left(  e^{\left\vert \nabla W\right\vert _{t}^{\ast}}+e^{2\left\vert
\nabla W\right\vert _{t}^{\ast}}.H\left(  W_{\cdot}\right)  _{t}^{\ast}%
\cdot\left\vert w_{p}\right\vert \right)  d^{TM}\left(  v_{m},w_{p}\right) \\
&  \leq e^{2\left\vert \nabla W\right\vert _{t}^{\ast}}\left(  1+H\left(
W_{\cdot}\right)  _{t}^{\ast}\cdot\left\vert w_{p}\right\vert \right)
d^{TM}\left(  v_{m},w_{p}\right)  .
\end{align*}

\end{proof}

The next corollary is the special case of Corollaries \ref{cor.6.2} and
\ref{cor.6.3} when $W_{t}=X$ is a time independent vector field.

\begin{corollary}
[$\left\vert \nabla e_{\ast}^{tX}\right\vert _{M}$ -estimate]\label{cor.6.4}If
$X$ is a complete vector field and%
\begin{equation}
k_{t}\left(  X,m\right)  :=\int_{0}^{t}\left\vert \nabla X\right\vert
_{e^{\tau X}\left(  m\right)  }d\tau, \label{e.6.14}%
\end{equation}
then%
\begin{align}
\left\vert \nabla e_{\ast}^{tX}\right\vert _{m}  &  \leq e^{k_{t}\left(
X,m\right)  }\cdot\int_{0}^{t}H_{e^{\tau X}\left(  m\right)  }\left(
X\right)  e^{k_{\tau}\left(  X,m\right)  }d\tau\label{e.6.15}\\
&  \leq e^{2k_{t}\left(  X,m\right)  }\cdot\int_{0}^{t}H_{e^{\tau X}\left(
m\right)  }\left(  X\right)  d\tau\label{e.6.16}%
\end{align}
and, for $v_{m},w_{p}\in TM,$%
\begin{equation}
d^{TM}\left(  e_{\ast}^{X}v_{m},e_{\ast}^{X}w_{p}\right)  \leq e^{2\left\vert
\nabla X\right\vert _{M}}\left[  1+H_{M}\left(  X\right)  \left\vert
w_{p}\right\vert \right]  d^{TM}\left(  v_{m},w_{p}\right)  . \label{e.6.17}%
\end{equation}

\end{corollary}

\begin{notation}
\label{not.6.5}For $X\in\Gamma\left(  TM\right)  $ and $m\in M,$ let%
\[
\bar{H}_{m}\left(  X\right)  :=\int_{0}^{1}H_{e^{-\tau X}\left(  m\right)
}\left(  X\right)  d\tau\leq H_{M}\left(  X\right)  .
\]

\end{notation}

\begin{proposition}
\label{pro.6.6}If $X,Z\in\Gamma\left(  TM\right)  $ and $X$ is complete, then%
\begin{align}
\left\vert \nabla\left[  \operatorname{Ad}_{e^{X}}Z\right]  \right\vert _{m}
&  \leq e^{2k_{1}\left(  -X,m\right)  }\left\vert \nabla Z\right\vert
_{e^{-X}\left(  m\right)  }+\bar{H}_{m}\left(  X\right)  e^{3k_{1}\left(
-X,m\right)  }\left\vert Z\right\vert _{e^{-X}\left(  m\right)  }%
\label{e.6.18}\\
&  \leq e^{3k_{1}\left(  -X,m\right)  }\left[  \left\vert \nabla Z\right\vert
_{e^{-X}\left(  m\right)  }+\bar{H}_{m}\left(  X\right)  \left\vert
Z\right\vert _{e^{-X}\left(  m\right)  }\right]  \label{e.6.19}%
\end{align}
where, from Eq. (\ref{e.6.14}),
\begin{equation}
k_{1}\left(  -X,m\right)  =\int_{0}^{1}\left\vert \nabla X\right\vert
_{e^{-\tau X}\left(  m\right)  }d\tau. \label{e.6.20}%
\end{equation}
[It is possible, using \textquotedblleft transport methods,\textquotedblright%
\ to replace $e^{3k_{1}\left(  -X,m\right)  }$ by $e^{2k_{1}\left(
-X,m\right)  }$ in the previous inequalities but we do not bother doing so in
this paper.]
\end{proposition}

\begin{proof}
As a consequence of the flow property of $e^{tX}$ and a simple change of
variables, it is useful to record;%
\begin{equation}
k_{1-s}\left(  X,e^{-X}\left(  m\right)  \right)  =\int_{0}^{1-s}\left\vert
\nabla X\right\vert _{e^{-\left(  1-\tau\right)  X}\left(  m\right)  }%
d\tau=\int_{s}^{1}\left\vert \nabla X\right\vert _{e^{-uX}\left(  m\right)
}du \label{e.6.21}%
\end{equation}
for any $s\in\left[  0,1\right]  .$ When $\sigma=0$ this identity may be
stated as%
\begin{equation}
k_{1}\left(  X,e^{-X}\left(  m\right)  \right)  =k_{1}\left(  -X,m\right)  .
\label{e.6.22}%
\end{equation}
With this preparation in hand, we now go to the proof of the proposition.

By Corollary \ref{cor.5.28} with $f=e^{X},$%
\begin{align*}
\nabla_{v_{m}}\left[  \operatorname{Ad}_{e^{X}}Z\right]   &  =\nabla_{v_{m}%
}\left[  e_{\ast}^{X}\left[  Z\circ e^{-X}\right]  \right] \\
&  =\left(  \nabla_{e_{\ast}^{-X}v_{m}}e_{\ast}^{X}\right)  \left[  Z\circ
e^{-X}\left(  m\right)  \right]  +e_{\ast}^{X}\left[  \nabla_{e_{\ast}%
^{-X}v_{m}}Z\right]
\end{align*}
and so
\begin{align*}
&  \left\vert \nabla_{v_{m}}\left[  \operatorname{Ad}_{e^{X}}Z\right]
\right\vert \\
&  \quad\leq\left(  \left\vert \nabla e_{\ast}^{X}\right\vert _{e^{-X}\left(
m\right)  }\left\vert Z\right\vert _{e^{-X}\left(  m\right)  }+\left\vert
e_{\ast}^{X}\right\vert _{e^{-X}\left(  m\right)  }\left\vert \nabla
Z\right\vert _{e^{-X}\left(  m\right)  }\right)  \cdot\left\vert e_{\ast}%
^{-X}\right\vert _{m}\left\vert v_{m}\right\vert .
\end{align*}
By Corollary \ref{cor.2.28},%
\[
\left\vert e_{\ast}^{-X}\right\vert _{m}\leq e^{\int_{0}^{1}\left\vert \nabla
X\right\vert _{e^{-\tau X}\left(  m\right)  }d\tau}=e^{k_{1}\left(
-X,m\right)  }%
\]

and from this inequality with $X$ replaced by $-X$ and $m$ by $e^{-X}\left(
m\right)  $ we also have (using Eq. (\ref{e.6.22}) with $t=1)$ that%
\[
\left\vert e_{\ast}^{X}\right\vert _{e^{-X}\left(  m\right)  }\leq
e^{k_{1}\left(  X,e^{-X}\left(  m\right)  \right)  }=e^{k_{1}\left(
-X,m\right)  }.
\]

Similarly from Corollary \ref{cor.6.4} with $m$ replaced by $e^{-X}\left(
m\right)  ,$%
\begin{align*}
\left\vert \nabla e_{\ast}^{X}\right\vert _{e^{-X}\left(  m\right)  }\leq &
e^{k_{1}\left(  X,e^{-X}\left(  m\right)  \right)  }\cdot\int_{0}%
^{1}H_{e^{\tau X}\left(  e^{-X}\left(  m\right)  \right)  }\left(  X\right)
e^{k_{\tau}\left(  X,e^{-X}\left(  m\right)  \right)  }d\tau\\
&  =e^{k_{1}\left(  -X,m\right)  }\int_{0}^{1}H_{e^{-\left(  1-\tau\right)
X}\left(  m\right)  }\left(  X\right)  e^{k_{\tau}\left(  X,e^{-X}\left(
m\right)  \right)  }d\tau\\
&  =e^{k_{1}\left(  -X,m\right)  }\int_{0}^{1}H_{e^{-sX}\left(  m\right)
}\left(  X\right)  e^{k_{1-s}\left(  X,e^{-X}\left(  m\right)  \right)  }ds\\
&  =e^{k_{1}\left(  -X,m\right)  }\int_{0}^{1}H_{e^{-sX}\left(  m\right)
}\left(  X\right)  e^{\int_{s}^{1}\left\vert \nabla X\right\vert
_{e^{-uX}\left(  m\right)  }du}ds\\
&  \leq e^{2k_{1}\left(  -X,m\right)  }\int_{0}^{1}H_{e^{-sX}\left(  m\right)
}\left(  X\right)  ds=e^{2k_{1}\left(  -X,m\right)  }\bar{H}_{m}\left(
X\right)  .
\end{align*}
Combining these inequalities shows,%
\begin{align*}
&  \left\vert \nabla_{v_{m}}\left[  \operatorname{Ad}_{e^{X}}Z\right]
\right\vert \\
&  \quad\leq\left(  e^{3k_{1}\left(  -X,m\right)  }\cdot\bar{H}_{m}\left(
X\right)  \left\vert Z\right\vert _{e^{-X}\left(  m\right)  }+e^{2k_{1}\left(
-X,m\right)  }\left\vert \nabla Z\right\vert _{e^{-X}\left(  m\right)
}\right)  \left\vert v_{m}\right\vert
\end{align*}
from which Eq. (\ref{e.6.18}) immediately follows.
\end{proof}

\section{First order distance estimates\label{sec.7}}

The main goal of this section is to estimate (see Theorem \ref{thm.7.2}) the
distance between the differentials of $\mu_{t,0}^{X}$ and $\mu_{t,0}^{Y}.$ To
do so we will again need to estimate the time derivative of the interpolator
defined Eq. (\ref{e.2.55}) above.

\begin{proposition}
\label{pro.7.1}Let $\left[  0,T\right]  \ni t\rightarrow X_{t,}Y_{t}\in
\Gamma\left(  TM\right)  $ be smooth complete time dependent vector fields on
$M$ and $\mu^{X}$ and $\mu^{Y}$ be their corresponding flows. If $0<t\leq T,$
$\left[  0,t\right]  \ni s\rightarrow\Theta_{s}:=\mu_{t,s}^{X}\circ\mu
_{s,0}^{Y}$ is the interpolator defined in Eq. (\ref{e.2.55}), and $v_{m}\in
T_{m}M,$ then
\begin{equation}
\left\vert \frac{\nabla}{ds}\Theta_{s\ast}\right\vert _{M}\leq e^{2\left\vert
\nabla X\right\vert _{t}^{\ast}+\left\vert \nabla Y\right\vert _{t}^{\ast}%
}\cdot\left(  H\left(  X_{\cdot}\right)  _{t}^{\ast}\left\vert Y_{s}%
-X_{s}\right\vert _{M}+\left\vert \nabla\left[  Y_{s}-X_{s}\right]
\right\vert _{M}\right)  , \label{e.7.1}%
\end{equation}
where (as in Eq. (\ref{e.5.21}) with $f_{s}=\Theta_{s}$ and $F_{s}=\Theta)$
\[
\left\vert \frac{\nabla}{ds}\Theta_{s\ast}\right\vert _{M}=\sup\left\{
\left\vert \frac{\nabla}{ds}\left[  \left(  \Theta_{s}\right)  _{\ast}%
v_{m}\right]  \right\vert :v\in TM\text{ with }\left\vert v\right\vert
=1\right\}
\]
and (as in Notation \ref{not.1.5}) $H_{m}\left(  X_{t}\right)  =\left\vert
\nabla^{2}X_{t}\right\vert _{m}+\left\vert R\left(  X_{t},\bullet\right)
\right\vert _{m}.$
\end{proposition}

\begin{proof}
Choose $\sigma\left(  \tau\right)  \in M$ so that $\dot{\sigma}\left(
0\right)  =v_{m}.$ Then by the properties of the Levi-Civita covariant
derivatives, the formula for $\Theta_{s}^{\prime}\left(  m\right)  $ in Eq.
(\ref{e.2.59}), along with the product and chain rule in Proposition
\ref{pro.5.27}, it follows that%
\begin{align*}
\frac{\nabla}{ds}\left[  \left(  \Theta_{s}\right)  _{\ast}v_{m}\right]  =  &
\frac{\nabla}{ds}\frac{d}{d\tau}|_{0}\Theta_{s}\left(  \sigma\left(
\tau\right)  \right)  =\frac{\nabla}{d\tau}|_{0}\frac{d}{ds}\Theta_{s}\left(
\sigma\left(  \tau\right)  \right) \\
=  &  \frac{\nabla}{d\tau}|_{0}\left[  \left(  \mu_{t,s}^{X}\right)  _{\ast
}\left(  Y_{s}-X_{s}\right)  \circ\mu_{s,0}^{Y}\left(  \sigma\left(
\tau\right)  \right)  \right] \\
=  &  \nabla_{v}\left[  \left(  \mu_{t,s}^{X}\right)  _{\ast}\left(  \left[
Y_{s}-X_{s}\right]  \right)  \mu_{s,0}^{Y}\left(  \sigma\left(  \tau\right)
\right)  \right] \\
=  &  \nabla_{\left(  \mu_{s,0}^{Y}\right)  _{\ast}v_{m}}\left[  \left(
\mu_{t,s}^{X}\right)  _{\ast}\left(  Y_{s}-X_{s}\right)  \right] \\
=  &  \left[  \nabla_{\left(  \mu_{s,0}^{Y}\right)  _{\ast}v_{m}}\left(
\mu_{t,s}^{X}\right)  _{\ast}\right]  \left(  Y_{s}-X_{s}\right)  \circ
\mu_{s,0}^{Y}\left(  m\right) \\
&  +\left(  \mu_{t,s}^{X}\right)  _{\ast}\nabla_{\left(  \mu_{s,0}^{Y}\right)
_{\ast}v_{m}}\left(  Y_{s}-X_{s}\right)
\end{align*}
and consequently,%
\begin{align*}
&  \left\vert \frac{\nabla}{ds}\left[  \left(  \Theta_{s}\right)  _{\ast}%
v_{m}\right]  \right\vert \\
&  \quad\leq\left(  \left\vert \nabla\mu_{t,s\ast}^{X}\right\vert
_{M}\left\vert Y_{s}-X_{s}\right\vert _{M}+\left\vert \mu_{t,s\ast}%
^{X}\right\vert _{M}\left\vert \nabla\left(  Y_{s}-X_{s}\right)  \right\vert
_{M}\right)  \left\vert \mu_{s,0\ast}^{Y}\right\vert _{M}\left\vert
v_{m}\right\vert .
\end{align*}
By Eq. (\ref{e.6.13}) of Corollary \ref{cor.6.2} with $\nu_{t}=\mu_{t,s}^{X},$%
\[
\left\vert \nabla\mu_{t,s\ast}^{X}\right\vert _{M}\leq e^{2\int_{s}%
^{t}\left\vert \nabla X_{\tau}\right\vert _{M}d\tau}H\left(  X_{\cdot}\right)
_{J\left(  s,t\right)  }^{\ast}\leq e^{2\int_{s}^{t}\left\vert \nabla X_{\tau
}\right\vert _{M}d\tau}H\left(  X_{\cdot}\right)  _{t}^{\ast}.
\]
Using the estimate in Eq. (\ref{e.2.60}) twice shows,
\begin{align*}
\left\vert \nabla\mu_{t,s\ast}^{X}\right\vert _{M}  &  \leq e^{\int_{s}%
^{t}\left\vert \nabla X_{\tau}\right\vert _{M}d\tau}\leq e^{2\int_{0}%
^{t}\left\vert \nabla X_{\tau}\right\vert _{M}d\tau,}\text{ and }\\
\left\vert \mu_{s,0\ast}^{Y}\right\vert _{M}  &  \leq e^{\int_{0}%
^{s}\left\vert \nabla Y_{\tau}\right\vert _{M}d\tau}\leq e^{\int_{0}%
^{t}\left\vert \nabla Y_{\tau}\right\vert _{M}d\tau}.
\end{align*}
Combining the last four inequalities yields and taking the supremum of the
result over $v_{m}\in TM$ with $\left\vert v_{m}\right\vert =1$ yields Eq.
(\ref{e.7.1}).
\end{proof}

\begin{theorem}
\label{thm.7.2}If $\left[  0,T\right]  \ni t\rightarrow X_{t,}Y_{t}\in
\Gamma\left(  TM\right)  $ are smooth complete time dependent vector fields on
$M$ and $\mu^{X}$ and $\mu^{Y}$ be their corresponding flows, then%
\begin{equation}
d\left(  \left(  \mu_{t,0}^{X}\right)  _{\ast},\left(  \mu_{t,0}^{Y}\right)
_{\ast}\right)  \leq e^{2\left\vert \nabla X\right\vert _{t}^{\ast}+\left\vert
\nabla Y\right\vert _{t}^{\ast}}\cdot\left(  \left(  1+H\left(  X_{\cdot
}\right)  _{t}^{\ast}\right)  \left\vert Y-X\right\vert _{t}^{\ast}+\left\vert
\nabla\left[  Y-X\right]  \right\vert _{t}^{\ast}\right)  , \label{e.7.2}%
\end{equation}

\end{theorem}

\begin{proof}
Integrating the estimate in Eq. (\ref{e.2.61}) shows%
\[
\int_{0}^{t}\left\vert \Theta_{s}^{\prime}\right\vert _{M}ds\leq e^{\left\vert
\nabla X\right\vert _{t}^{\ast}}\left\vert Y-X\right\vert _{t}^{\ast}\leq
e^{2\left\vert \nabla X\right\vert _{t}^{\ast}+\left\vert \nabla Y\right\vert
_{t}^{\ast}}\left\vert Y-X\right\vert _{t}^{\ast}%
\]
and integrating the estimate in Eq. (\ref{e.7.1}) shows%
\[
\int_{0}^{t}\left\vert \frac{\nabla}{ds}\Theta_{s\ast}\right\vert _{M}ds\leq
e^{2\left\vert \nabla X\right\vert _{t}^{\ast}+\left\vert \nabla Y\right\vert
_{t}^{\ast}}\cdot\left(  H\left(  X_{\cdot}\right)  _{t}^{\ast}\left\vert
Y-X\right\vert _{t}^{\ast}+\left\vert \nabla\left[  Y-X\right]  \right\vert
_{t}^{\ast}\right)  .
\]
Adding these estimates while making use of an appropriately time scaled
version of Eq. (\ref{e.5.18}) of Proposition \ref{pro.5.17} with $E=TM,$
$f_{s}=\Theta_{s},$ and $F_{s}=\Theta_{s\ast}$ completes the proof of Eq.
(\ref{e.7.2}).
\end{proof}

\begin{corollary}
\label{cor.7.3}Let $J=\left[  0,T\right]  \ni t\rightarrow Y_{t}\in
\Gamma\left(  TM\right)  $ be a smooth complete time dependent vector field on
$M$ and $\mu^{Y}$ be the corresponding flow. Then for $t>0$ (for notational
simplicity)
\begin{equation}
d\left(  \left(  \mu_{t,0}^{Y}\right)  _{\ast},Id_{TM}\right)  \leq
e^{\left\vert \nabla Y\right\vert _{t}^{\ast}}\cdot\left(  \left\vert
Y\right\vert _{t}^{\ast}+\left\vert \nabla Y\right\vert _{t}^{\ast}\right)  .
\label{e.7.3}%
\end{equation}

\end{corollary}

\begin{proof}
Applying Theorem \ref{thm.7.2} with $X\equiv0$ gives Eq. (\ref{e.7.3}).
\end{proof}

\section{First order logarithm estimates\label{sec.8}}

The main purpose of this section is to give a first order version (see Theorem
\ref{thm.8.4} below) of the logarithm control estimate in Theorem
\ref{thm.4.12}. Before doing so we will first need to develop a few more
auxiliary estimates.

\begin{proposition}
\label{pro.8.1}If $C\left(  \cdot\right)  \in C\left(  \left[  0,T\right]
,F^{\left(  \kappa\right)  }\left(  \mathbb{R}^{d}\right)  \right)  $ and
$Z\in\Gamma\left(  TM\right)  ,$ then for $0\leq s\leq1,$%
\begin{equation}
\left\vert \nabla\left[  \operatorname{Ad}_{e^{sV_{C\left(  t\right)  }}%
}Z\right]  \right\vert _{M}\leq e^{3\left\vert \nabla V^{\left(
\kappa\right)  }\right\vert _{M}\left\vert C\left(  t\right)  \right\vert
}\cdot\left(  H_{M}\left(  V^{\left(  \kappa\right)  }\right)  \left\vert
Z\right\vert _{M}\left\vert C\left(  t\right)  \right\vert +\left\vert \nabla
Z\right\vert _{M}\right)  , \label{e.8.1}%
\end{equation}
where $H_{M}\left(  V^{\left(  \kappa\right)  }\right)  $ was defined in Eq.
(\ref{e.1.14}) of Definition \ref{def.1.24}.
\end{proposition}

\begin{proof}
This result follows directly as an application of Proposition \ref{pro.6.6}
with $X_{t}=-sV_{C}\left(  t\right)  .$
\end{proof}

\begin{corollary}
\label{cor.8.2}If $C\in C^{1}\left(  \left[  0,T\right]  ,F^{\left(
\kappa\right)  }\left(  \mathbb{R}^{d}\right)  \right)  $ and $W_{t}^{C}$ is
given as in Eq. (\ref{e.4.4}), then%
\begin{equation}
\left\vert \nabla W^{C}\right\vert _{t}^{\ast}\lesssim e^{3\left\vert \nabla
V^{\left(  \kappa\right)  }\right\vert _{M}\left\vert C\right\vert _{\infty
,t}}\cdot\left(  H_{M}\left(  V^{\left(  \kappa\right)  }\right)  \left\vert
V^{\left(  \kappa\right)  }\right\vert _{M}\left\vert C\right\vert _{\infty
,t}+\left\vert \nabla V^{\left(  \kappa\right)  }\right\vert _{M}\right)
\left\vert \dot{C}\right\vert _{t}^{\ast}. \label{e.8.2}%
\end{equation}
Moreover, there exists $c\left(  \kappa\right)  <\infty$ such that, whenever
$\xi\in C^{1}\left(  \left[  0,T\right]  ,F^{\left(  \kappa\right)  }\left(
\mathbb{R}^{d}\right)  \right)  $ and $C^{\xi}\left(  t\right)  =\log\left(
g^{\xi}\left(  t\right)  \right)  ,$%
\begin{equation}
\left\vert \nabla W^{C^{\xi}}\right\vert _{t}^{\ast}\lesssim\mathcal{K}%
_{t}\cdot\left(  H_{M}\left(  V^{\left(  \kappa\right)  }\right)  \left\vert
V^{\left(  \kappa\right)  }\right\vert _{M}Q_{\left[  1,\kappa\right]
}\left(  N_{t}^{\ast}\left(  \dot{\xi}\right)  \right)  +\left\vert \nabla
V^{\left(  \kappa\right)  }\right\vert _{M}\right)  \label{e.8.3}%
\end{equation}
where
\begin{equation}
\mathcal{K}_{t}:=e^{c\left(  \kappa\right)  \left\vert \nabla V^{\left(
\kappa\right)  }\right\vert _{M}Q_{\left[  1,\kappa\right]  }\left(
N_{t}^{\ast}\left(  \dot{\xi}\right)  \right)  }Q_{\left[  1,\kappa\right]
}\left(  N_{t}^{\ast}\left(  \dot{\xi}\right)  \right)  . \label{e.8.4}%
\end{equation}

\end{corollary}

\begin{proof}
Let $\tau\in\left[  0,t\right]  $ and $s\in\left[  0,1\right]  .$ Applying the
estimate in Eq. (\ref{e.8.1}) with $Z=V_{\dot{C}\left(  \tau\right)  }$
implies,%
\begin{align*}
&  \left\vert \nabla\left[  \operatorname{Ad}_{e^{sV_{C\left(  \tau\right)  }%
}}V_{\dot{C}\left(  \tau\right)  }\right]  \right\vert _{M}\\
&  \leq e^{3\left\vert \nabla V^{\left(  \kappa\right)  }\right\vert
_{M}\left\vert C\left(  \tau\right)  \right\vert }\cdot\left(  H_{M}\left(
V^{\left(  \kappa\right)  }\right)  \left\vert V^{\left(  \kappa\right)
}\right\vert _{M}\left\vert C\left(  \tau\right)  \right\vert +\left\vert
\nabla V^{\left(  \kappa\right)  }\right\vert _{M}\right)  \left\vert \dot
{C}\left(  \tau\right)  \right\vert \\
&  \leq e^{3\left\vert \nabla V^{\left(  \kappa\right)  }\right\vert
_{M}\left\vert C\right\vert _{\infty,t}}\cdot\left(  H_{M}\left(  V^{\left(
\kappa\right)  }\right)  \left\vert V^{\left(  \kappa\right)  }\right\vert
_{M}\left\vert C\right\vert _{\infty,t}+\left\vert \nabla V^{\left(
\kappa\right)  }\right\vert _{M}\right)  \left\vert \dot{C}\left(
\tau\right)  \right\vert .
\end{align*}
Integrating this inequality on $s\in\left[  0,1\right]  $ and $\tau\in\left[
0,t\right]  ,$ while using%
\[
\left\vert \nabla W^{C}\right\vert _{t}^{\ast}=\int_{0}^{T}\left\vert \nabla
W_{\tau}^{C}\right\vert d\tau\leq\int_{0}^{T}\left[  \int_{0}^{1}\left\vert
\nabla\left[  \operatorname{Ad}_{e^{sV_{C\left(  \tau\right)  }}}V_{\dot
{C}\left(  \tau\right)  }\right]  \right\vert ~ds\right]  d\tau,
\]
gives Eq. (\ref{e.8.2}).

Now suppose that $\xi\in C^{1}\left(  \left[  0,T\right]  ,F^{\left(
\kappa\right)  }\left(  \mathbb{R}^{d}\right)  \right)  $ and $C^{\xi}\left(
t\right)  =\log\left(  g^{\xi}\left(  t\right)  \right)  .$ Then from Eq.
(\ref{e.3.31}),
\begin{equation}
\left\vert C^{\xi}\left(  \cdot\right)  \right\vert _{\infty,t}\lesssim
Q_{\left[  1,\kappa\right]  }\left(  N_{t}^{\ast}\left(  \dot{\xi}\right)
\right)  \label{e.8.5}%
\end{equation}
and by the estimates in Eqs. (\ref{e.3.29}) and (\ref{e.3.21}),
\[
\left\vert \dot{C}\right\vert _{t}^{\ast}\lesssim Q_{\left[  1,\kappa\right]
}\left(  N_{t}^{\ast}\left(  \dot{\xi}\right)  \right)  .
\]
Using the previous two estimates in Eq. (\ref{e.8.2}) proves Eq. (\ref{e.8.3}).
\end{proof}

\begin{corollary}
\label{cor.8.3}If $\xi\in C^{1}\left(  \left[  0,T\right]  ,F^{\left(
\kappa\right)  }\left(  \mathbb{R}^{d}\right)  \right)  ,$ $C^{\xi}\left(
t\right)  =\log\left(  g^{\xi}\left(  t\right)  \right)  ,$ and $U_{t}^{\xi
}\in\Gamma\left(  TM\right)  $ is the difference vector field in Eq.
(\ref{e.4.6}), then there exist $c\left(  \kappa\right)  <\infty$ such that%
\begin{align}
\left\vert \nabla U\right\vert _{T}^{\ast}\lesssim &  \left[  \left(
\mathcal{C}^{0}\left(  V^{\left(  \kappa\right)  }\right)  H_{M}\left(
V^{\left(  \kappa\right)  }\right)  Q_{\left[  1,\kappa\right]  }\left(
N_{T}^{\ast}\left(  \dot{\xi}\right)  \right)  +\mathcal{C}^{1}\left(
V^{\left(  \kappa\right)  }\right)  \right)  \right]  \cdot\nonumber\\
&  \quad\cdot e^{c\left(  \kappa\right)  \left\vert \nabla V^{\left(
\kappa\right)  }\right\vert _{M}Q_{\left[  1,\kappa\right]  }\left(
N_{T}^{\ast}\left(  \dot{\xi}\right)  \right)  }\cdot Q_{(\kappa,2\kappa
]}\left(  N_{T}^{\ast}\left(  \dot{\xi}\right)  \right)  . \label{e.8.6}%
\end{align}

\end{corollary}

\begin{proof}
Let $t\in\left[  0,T\right]  $ and $s\in\left[  0,1\right]  .$ The estimate in
Eq. (\ref{e.8.1}) with
\[
Z=V_{\pi_{>\kappa}\left[  C^{\xi}\left(  t\right)  ,u\left(
s,\operatorname{ad}_{C^{\xi}\left(  t\right)  }\right)  \dot{\xi}\left(
t\right)  \right]  _{\otimes}}%
\]
becomes%
\begin{align*}
&  \left\vert \nabla\left[  \operatorname{Ad}_{e^{sV_{C^{\xi}\left(  t\right)
}}}V_{\pi_{>\kappa}\left[  C^{\xi}\left(  t\right)  ,u\left(
s,\operatorname{ad}_{C^{\xi}\left(  t\right)  }\right)  \dot{\xi}\left(
t\right)  \right]  _{\otimes}}\right]  \right\vert _{M}\\
&  \qquad\leq e^{3\left\vert \nabla V^{\left(  \kappa\right)  }\right\vert
_{M}\left\vert C^{\xi}\left(  t\right)  \right\vert }\cdot\left[
\alpha\left(  s,t\right)  +\beta\left(  s,t\right)  \right] \\
&  \qquad\leq e^{3\left\vert \nabla V^{\left(  \kappa\right)  }\right\vert
_{M}\left\vert C^{\xi}\right\vert _{\infty,T}}\cdot\left[  \alpha\left(
s,t\right)  +\beta\left(  s,t\right)  \right]  ,
\end{align*}
where
\begin{align*}
\alpha\left(  s,t\right)   &  =H_{M}\left(  V^{\left(  \kappa\right)
}\right)  \left\vert C^{\xi}\right\vert _{\infty,T}\cdot\left\vert
V_{\pi_{>\kappa}\left[  C^{\xi}\left(  t\right)  ,u\left(  s,\operatorname{ad}%
_{C^{\xi}\left(  t\right)  }\right)  \dot{\xi}\left(  t\right)  \right]
_{\otimes}}\right\vert _{M}\text{ and}\\
\beta\left(  s,t\right)   &  =\left\vert \nabla V_{\pi_{>\kappa}\left[
C^{\xi}\left(  t\right)  ,u\left(  s,\operatorname{ad}_{C^{\xi}\left(
t\right)  }\right)  \dot{\xi}\left(  t\right)  \right]  _{\otimes}}\right\vert
_{M}.
\end{align*}
In this notation we have
\[
\left\vert \nabla U\right\vert _{T}^{\ast}\leq e^{3\left\vert \nabla
V^{\left(  \kappa\right)  }\right\vert _{M}\left\vert C^{\xi}\right\vert
_{\infty,T}}\int_{0}^{T}dt\int_{0}^{1}ds\left(  1-s\right)  \left[
\alpha\left(  s,t\right)  +\beta\left(  s,t\right)  \right]
\]
where according to Lemma \ref{lem.4.10},
\[
\int_{0}^{T}dt\int_{0}^{1}ds\left(  1-s\right)  \alpha\left(  s,t\right)
\lesssim H_{M}\left(  V^{\left(  \kappa\right)  }\right)  \mathcal{C}%
^{0}\left(  V^{\left(  \kappa\right)  }\right)  \left\vert C^{\xi}\right\vert
_{\infty,T}Q_{(\kappa,2\kappa]}\left(  N_{T}^{\ast}\left(  \dot{\xi}\right)
\right)
\]
and
\[
\int_{0}^{T}dt\int_{0}^{1}ds\left(  1-s\right)  \beta\left(  s,t\right)
\lesssim\mathcal{C}^{1}\left(  V^{\left(  \kappa\right)  }\right)
Q_{(\kappa,2\kappa]}\left(  N_{T}^{\ast}\left(  \dot{\xi}\right)  \right)  .
\]
The proof is now completed by combining these estimate with the estimate for
$\left\vert C^{\xi}\left(  \cdot\right)  \right\vert _{\infty,T}$ in Eq.
(\ref{e.8.5}).
\end{proof}

\begin{theorem}
[Comparing differentials]\label{thm.8.4}If $\xi\in C^{1}\left(  \left[
0,T\right]  ,F^{\left(  \kappa\right)  }\left(  \mathbb{R}^{d}\right)
\right)  ,$ then%
\begin{equation}
d_{M}^{TM}\left(  \left(  \mu_{T,0}^{V_{\dot{\xi}}}\right)  _{\ast},e_{\ast
}^{V_{\log\left(  g^{\xi}\left(  T\right)  \right)  }}\right)  \leq
\mathcal{K}\cdot Q_{(\kappa,2\kappa]}\left(  N_{T}^{\ast}\left(  \dot{\xi
}\right)  \right)  . \label{e.8.7}%
\end{equation}
where
\[
\mathcal{K=K}\left(  T,\left\vert V^{\left(  \kappa\right)  }\right\vert
_{M},\left\vert \nabla V^{\left(  \kappa\right)  }\right\vert _{M}%
,H_{M}\left(  V^{\left(  \kappa\right)  }\right)  ,N_{T}^{\ast}\left(
\dot{\xi}\right)  \right)
\]
is a (fairly complicated) increasing function of each of its arguments.
\end{theorem}

\begin{proof}
Our proof of this result is similar to the proof of Theorem \ref{thm.4.12}
except that we will being using Theorem \ref{thm.7.2} in place of Theorem
\ref{thm.2.30}. Applying Theorem \ref{thm.7.2} with $X_{t}=V_{\dot{\xi}\left(
t\right)  }$ and $Y_{t}=W_{t}^{C^{\xi}}$ shows%
\begin{align*}
d_{M}^{TM}  &  \left(  \mu_{t,0\ast},e_{\ast}^{V_{C\left(  t\right)  }}\right)
\\
&  \leq e^{2\left\vert \nabla V_{\dot{\xi}\left(  t\right)  }\right\vert
_{t}^{\ast}+\left\vert \nabla W^{C}\right\vert _{t}^{\ast}}\cdot\left(
\left(  1+H\left(  V_{\dot{\xi}}\right)  _{t}^{\ast}\right)  \left\vert
U^{\xi}\right\vert _{t}^{\ast}+\left\vert \nabla U^{\xi}\right\vert _{t}%
^{\ast}\right) \\
&  \leq e^{2\left\vert \nabla V^{\left(  \kappa\right)  }\right\vert
_{M}\left\vert \dot{\xi}\right\vert _{t}^{\ast}+\left\vert \nabla
W^{C}\right\vert _{t}^{\ast}}\cdot\left(  \left[  1+H_{M}\left(  V^{\left(
\kappa\right)  }\right)  \right]  \left\vert \dot{\xi}\right\vert _{t}^{\ast
}\cdot\left\vert U^{\xi}\right\vert _{t}^{\ast}+\left\vert \nabla U^{\xi
}\right\vert _{t}^{\ast}\right)  .
\end{align*}
Recalling from Eq. (\ref{e.3.21}) that $\left\vert \dot{\xi}\right\vert
_{t}^{\ast}\lesssim Q_{\left[  1,\kappa\right]  }\left(  N_{T}^{\ast}\left(
\dot{\xi}\right)  \right)  $ and substituting the estimates for$\left\vert
\nabla W^{C^{\xi}}\right\vert _{t}^{\ast}$ in Corollary \ref{cor.8.2},
$\left\vert U^{\xi}\right\vert _{t}^{\ast}$ in Theorem \ref{thm.4.11}, and
$\left\vert \nabla U\right\vert _{T}^{\ast}$ in Corollary \ref{cor.8.3} into
the previous inequality gives the stated estimate in Eq. (\ref{e.8.7}).
\end{proof}

\begin{corollary}
\label{cor.8.5}If $A,B\in F^{\left(  \kappa\right)  }\left(  \mathbb{R}%
^{d}\right)  ,$ then%
\[
d_{M}^{TM}\left(  e_{\ast}^{V_{B}},Id_{TM}\right)  \leq\left[  \left\vert
V^{\left(  \kappa\right)  }\right\vert _{M}+\left\vert \nabla V^{\left(
\kappa\right)  }\right\vert _{M}e^{\left\vert \nabla V^{\left(  \kappa\right)
}\right\vert _{M}\left\vert B\right\vert }\right]  \left\vert B\right\vert ,
\]
and there exists $\mathcal{K}_{1}$ such that%
\begin{align*}
d_{M}^{TM}  &  \left(  \left[  e^{V_{B}}\circ e^{V_{A}}\right]  _{\ast
},e_{\ast}^{V_{\log\left(  e^{A}e^{B}\right)  }}\right) \\
&  \leq\mathcal{K}_{1}\cdot N\left(  A\right)  N\left(  B\right)
Q_{(\kappa-1,2\left(  \kappa-1\right)  ]}\left(  N\left(  A\right)  +N\left(
B\right)  \right)  .
\end{align*}
where
\[
\mathcal{K}_{1}=\mathcal{K}_{1}\left(  \left\vert V^{\left(  \kappa\right)
}\right\vert _{M},\left\vert \nabla V^{\left(  \kappa\right)  }\right\vert
_{M},H_{M}\left(  V^{\left(  \kappa\right)  }\right)  ,N\left(  A\right)  \vee
N\left(  B\right)  \right)  .
\]

\end{corollary}

\begin{proof}
From Corollary \ref{cor.7.3} with $Y_{t}=V_{B}$ we find
\begin{align*}
d_{M}^{TM}\left(  e_{\ast}^{V_{B}},Id_{TM}\right)   &  \leq e^{\left\vert
\nabla V_{B}\right\vert _{1}^{\ast}}\cdot\left(  \left\vert V_{B}\right\vert
_{1}^{\ast}+\left\vert \nabla V_{B}\right\vert _{1}^{\ast}\right) \\
&  \leq\left[  \left\vert V^{\left(  \kappa\right)  }\right\vert
_{M}+\left\vert \nabla V^{\left(  \kappa\right)  }\right\vert _{M}%
e^{\left\vert \nabla V^{\left(  \kappa\right)  }\right\vert _{M}\left\vert
B\right\vert }\right]  \left\vert B\right\vert .
\end{align*}
The proof of the second inequality is completely analogous to the proof of the
second inequality in Corollary \ref{cor.4.16} with the exception that we now
use Theorem \ref{thm.8.4} in place of Theorem \ref{thm.4.12} and we must
replace $V_{\left[  A,B\right]  _{\otimes}}=\left[  V_{A},V_{B}\right]  $ by
$\nabla V_{\left[  A,B\right]  _{\otimes}}=\nabla\left[  V_{A},V_{B}\right]  $
and $\mathcal{C}^{0}\left(  V^{\left(  \kappa\right)  }\right)  $ by
$\mathcal{C}^{1}\left(  V^{\left(  \kappa\right)  }\right)  $ appropriately.
\end{proof}

\section{Appendix: Gronwall Inequalities\label{sec.9}}

This appendix gathers a few rather standard differential inequalities that are
used in the body of the paper.

\subsection{Flat space Gronwall inequalities\label{sec.9.1}}

\begin{proposition}
[A Gronwall Inequality]\label{pro.9.1}Suppose that $\psi:\left[  0,T\right]
\rightarrow\mathbb{R}$ is absolutely continuous, $u\in C\left(  \left[
0,T\right]  ,\mathbb{R}\right)  ,$ and $h\in L^{1}\left(  \left[  0,T\right]
\right)  .$ If
\begin{equation}
\dot{\psi}\left(  t\right)  \leq u\left(  t\right)  +h\left(  t\right)
\psi\left(  t\right)  \text{ for a.e. }t, \label{e.9.1}%
\end{equation}
then
\[
\psi\left(  t\right)  \leq\psi\left(  0\right)  e^{\int_{0}^{t}h\left(
s\right)  ds}+\int_{0}^{t}e^{\int_{\tau}^{t}h\left(  s\right)  ds}\cdot
u\left(  \tau\right)  d\tau
\]

\end{proposition}

\begin{proof}
Here is the short proof of this standard result for the reader's convenience.
Let $H\left(  t\right)  :=\int_{0}^{t}h\left(  s\right)  ds$ so that $H$ is
absolutely continuous with $\dot{H}\left(  t\right)  =h\left(  t\right)  $ for
a.e. $t.$ We then have, for a.e. $t,$ that%
\[
\frac{d}{dt}\left[  e^{-H\left(  t\right)  }\psi\left(  t\right)  \right]
=e^{-H\left(  t\right)  }\left[  \dot{\psi}\left(  t\right)  -h\left(
t\right)  \psi\left(  t\right)  \right]  \leq e^{-H\left(  t\right)  }u\left(
t\right)  .
\]
Integrating this equation gives
\[
e^{-H\left(  t\right)  }\psi\left(  t\right)  -\psi\left(  0\right)  \leq
\int_{0}^{t}e^{-H\left(  \tau\right)  }u\left(  \tau\right)  d\tau.
\]
Multiplying this inequality by $e^{H\left(  t\right)  }$ completes the proof
as $H\left(  t\right)  -H\left(  \tau\right)  =\int_{\tau}^{t}h\left(
s\right)  ds.$
\end{proof}

The following corollary is the form of Gronwall's inequality which is most
useful to us.

\begin{corollary}
\label{cor.9.2}Let $\left(  V,\left\vert \cdot\right\vert \right)  $ be a
normed space, $-\infty<a<b<\infty,$ and $\left[  a,b\right]  \ni$
$t\rightarrow C\left(  t\right)  \in V$ be a $C^{1}$-function of $t.$ If there
exists continuous functions, $h\left(  t\right)  $ and $g\left(  t\right)  ,$
such that%
\begin{equation}
\left\vert \dot{C}\left(  t\right)  \right\vert \leq h\left(  t\right)
\left\vert C\left(  t\right)  \right\vert +g\left(  t\right)  \text{ }%
\forall~t\in\left[  a,b\right]  , \label{e.9.2}%
\end{equation}
then for any $s,t\in\left[  a,b\right]  ,$%
\begin{equation}
\left\vert C\left(  t\right)  \right\vert \leq\left\vert C\left(  s\right)
\right\vert e^{\int_{J\left(  s,t\right)  }h\left(  \sigma\right)  d\sigma
}+\int_{J\left(  s,t\right)  }g\left(  \sigma\right)  e^{\int_{J\left(
\sigma,t\right)  }h\left(  \sigma^{\prime}\right)  d\sigma^{\prime}}%
d\sigma\text{ }\forall~t\in\left[  0,T\right]  . \label{e.9.3}%
\end{equation}
where
\[
J\left(  s,t\right)  :=\left\{
\begin{array}
[c]{ccc}%
\left[  s,t\right]  & \text{if} & s\leq t\\
\left[  t,s\right]  & \text{if} & t\leq s
\end{array}
.\right.
\]

\end{corollary}

\begin{proof}
If $K:=\max_{a\leq t\leq b}\left\vert \dot{C}\left(  t\right)  \right\vert
<\infty,$ then for $a\leq s\leq t\leq b,$
\[
\left\vert \left\vert C\left(  t\right)  \right\vert -\left\vert C\left(
s\right)  \right\vert \right\vert \leq\left\vert C\left(  t\right)  -C\left(
s\right)  \right\vert =\left\vert \int_{s}^{t}\dot{C}\left(  \tau\right)
d\tau\right\vert \leq\int_{s}^{t}\left\vert \dot{C}\left(  \tau\right)
\right\vert d\tau\leq K\left\vert t-s\right\vert
\]
which shows $\left\vert C\left(  t\right)  \right\vert $ is Lipschitz and
hence absolutely continuous. Moreover, at $t,$ where $\left\vert C\left(
t\right)  \right\vert $ is differentiable, we have
\[
\left\vert \frac{d}{dt}\left\vert C\left(  t\right)  \right\vert \right\vert
=\lim_{s\rightarrow t}\frac{\left\vert \left\vert C\left(  t\right)
\right\vert -\left\vert C\left(  s\right)  \right\vert \right\vert
}{\left\vert t-s\right\vert }\leq\lim_{s\rightarrow t}\left\vert \frac
{\int_{s}^{t}\dot{C}\left(  \tau\right)  d\tau}{t-s}\right\vert =\left\vert
\dot{C}\left(  t\right)  \right\vert
\]
which combined with Eq. (\ref{e.9.2}) implies,%
\begin{equation}
\left\vert \frac{d}{dt}\left\vert C\left(  t\right)  \right\vert \right\vert
\leq h\left(  t\right)  \left\vert C\left(  t\right)  \right\vert +g\left(
t\right)  \text{ for a.e. }t\in\left[  a,b\right]  . \label{e.9.4}%
\end{equation}

If $s\leq t$ in Eq. (\ref{e.9.3}) let $\varepsilon=+1$ while if $t\leq s$ in
Eq. (\ref{e.9.3}) let $\varepsilon=-1$ and in either case let $\psi
_{\varepsilon}\left(  \tau\right)  :=\left\vert C\left(  s+\varepsilon
\tau\right)  \right\vert $ for $\tau\geq0$ so that $s+\varepsilon\tau
\in\left[  a,b\right]  .$ Then $\psi_{\varepsilon}\left(  \tau\right)  $ is
still absolutely continuous and satisfies,%
\begin{align*}
\dot{\psi}_{\varepsilon}\left(  \tau\right)   &  =\varepsilon\frac{d}%
{dt}\left\vert C\left(  t\right)  \right\vert |_{t=s+\varepsilon\tau}%
\leq\left\vert \frac{d}{dt}\left\vert C\left(  t\right)  \right\vert
|_{t=s+\varepsilon\tau}\right\vert \\
&  \leq h\left(  s+\varepsilon\tau\right)  \left\vert C\left(  s+\varepsilon
\tau\right)  \right\vert +g\left(  s+\varepsilon\tau\right)  =h\left(
s+\varepsilon\tau\right)  \psi_{\varepsilon}\left(  \tau\right)  +g\left(
s+\varepsilon\tau\right)  .
\end{align*}
Thus by Propositions \ref{pro.9.1},%
\[
\left\vert C\left(  s+\varepsilon\tau\right)  \right\vert =\psi_{\varepsilon
}\left(  \tau\right)  \leq e^{\int_{0}^{\tau}h\left(  s+\varepsilon r\right)
dr}\psi_{\varepsilon}\left(  0\right)  +\int_{0}^{\tau}e^{\int_{\rho}^{\tau
}h\left(  s+\varepsilon r\right)  dr}g\left(  s+\varepsilon\rho\right)
d\rho.
\]
We now choose $\tau$ so that $s+\varepsilon\tau=t$ (i.e. $\tau:=\varepsilon
\left(  t-s\right)  =\left\vert t-s\right\vert )$ to conclude,%
\[
\left\vert C\left(  t\right)  \right\vert \leq e^{\int_{0}^{\varepsilon\left(
t-s\right)  }h\left(  s+\varepsilon r\right)  dr}\left\vert C\left(  s\right)
\right\vert +\int_{0}^{\varepsilon\left(  t-s\right)  }e^{\int_{\rho
}^{\varepsilon\left(  t-s\right)  }h\left(  s+\varepsilon r\right)
dr}g\left(  s+\varepsilon\rho\right)  d\rho
\]
which after affine change of variables gives Eq. (\ref{e.9.3}).
\end{proof}

\subsection{A geometric form of Gronwall's inequality\label{sec.9.2}}

Recall that $\nabla$ denotes the Levi-Civita covariant derivative on $TM.$ We
also use $\nabla$ to denote the Levi-Civita covariant derivative extended (by
the product rule) to act on any associated vector bundle, let $\Lambda
^{k}\left(  TM\right)  ,$ $\Lambda^{k}\left(  T^{\ast}M\right)  ,$
$TM^{\otimes\ell}\otimes\left(  T^{\ast}M\right)  ^{\otimes k},$ etc. The
following geometric version of the classic Bellman-Gronwall inequality will be
used frequently in this section.

\begin{corollary}
[Covariant Bellman/Gronwall]\label{cor.9.3}Let $E:=TM^{\otimes k}\otimes
T^{\ast}M^{\otimes l}$ for some $k,l\in\mathbb{N}_{0},$ $\sigma\in
C^{1}\left(  \left(  a,b\right)  ,M\right)  ,$ and suppose that $T_{t}%
,G_{t}\in E_{\sigma\left(  t\right)  }$ and $H_{t}\in\operatorname*{End}%
\left(  E_{\sigma\left(  t\right)  }\right)  $ are given continuously
differentiable functions of $t.$ If $T_{t},$ $H_{t},$ and $G_{t}$ satisfy the
differential equation,%
\begin{equation}
\nabla_{t}T_{t}=H_{t}T_{t}+G_{t}, \label{e.9.5}%
\end{equation}
then, for all $s,t\in\left(  a,b\right)  ,$%
\begin{equation}
\left\vert T_{t}\right\vert \leq e^{\int_{J\left(  s,t\right)  }\left\Vert
H_{r}\right\Vert _{op}dr}\left\vert T_{s}\right\vert +\int_{J\left(
s,t\right)  }e^{\int_{J\left(  r,t\right)  }\left\Vert H_{r}\right\Vert
_{op}dr}\left\vert G_{\rho}\right\vert d\rho. \label{e.9.6}%
\end{equation}

\end{corollary}

\begin{proof}
The point is that, writing $\pt_{t}$ for $\pt_{t}\left(  \sigma\right)  ,$ we
have from Eq. (\ref{e.9.5}) that%
\begin{align*}
\frac{d}{dt}\left[  \pt_{t}^{-1}T_{t}\right]   &  =\pt_{t}^{-1}\nabla_{t}%
T_{t}=\pt_{t}^{-1}H_{t}T_{t}+\pt_{t}^{-1}G_{t}\\
&  =\left[  \pt_{t}^{-1}H_{t}\pt_{t}\right]  \pt_{t}^{-1}T_{t}+\pt_{t}%
^{-1}G_{t}%
\end{align*}
and therefore
\begin{align*}
\left\vert \frac{d}{dt}\left[  \pt_{t}^{-1}T_{t}\right]  \right\vert  &
\leq\left\Vert \pt_{t}^{-1}H_{t}\pt_{t}\right\Vert _{op}\left\vert
\pt_{t}^{-1}T_{t}\right\vert +\left\vert \pt_{t}^{-1}G_{t}\right\vert \\
&  =\left\Vert H_{t}\right\Vert _{op}\left\vert \pt_{t}^{-1}T_{t}\right\vert
+\left\vert G_{t}\right\vert
\end{align*}
and Eq. (\ref{e.9.6}) now follows directly from Corollary \ref{cor.9.2} above
with $C\left(  t\right)  :=\pt_{t}^{-1}T_{t}$ and the observation that
$\left\vert C\left(  t\right)  \right\vert =\left\vert T_{t}\right\vert $ for
all $t.$
\end{proof}

\providecommand{\bysame}{\leavevmode\hbox to3em{\hrulefill}\thinspace}
\providecommand{\MR}{\relax\ifhmode\unskip\space\fi MR }
\providecommand{\MRhref}[2]{%
  \href{http://www.ams.org/mathscinet-getitem?mr=#1}{#2}
}
\providecommand{\href}[2]{#2}

%\bibliographystyle{amsplain}
%\bibliography{rde}
\end{document}